\documentclass[11pt,letterpaper]{article}%
\usepackage{amsfonts}
\usepackage{amssymb}
\usepackage{graphicx}
\usepackage{amsmath}
\usepackage{version}
\usepackage{boxedminipage}%
\providecommand{\U}[1]{\protect\rule{.1in}{.1in}}
\newtheorem{theorem}{Theorem}[section]

\newtheorem{corollary}[theorem]{Corollary}

\newtheorem{definition}[theorem]{Definition}

\newtheorem{lemma}[theorem]{Lemma}

\newtheorem{proposition}[theorem]{Proposition}
\newtheorem{remark}[theorem]{Remark}

\newenvironment{proof}[1][Proof]{\textbf{#1.} }{\ \rule{0.5em}{0.5em}}
\def\i{{\rm i}\,}

\newenvironment{sfblock}{\sffamily\itshape}{}
\includeversion{privatenotes}
\excludeversion{privatenotes}
\numberwithin{equation}{section}
\begin{document}

\title{Asymptotic Inference in a Stationary Quantum Time Series}
\author{\textsc{Michael Nussbaum\thanks{Department of Mathematics, Cornell University,
Ithaca NY, USA} }\textsc{\bigskip\ and Arleta Szko{\l }a\thanks{Faculty of
Natural and Environmental Sciences, Zittau/ G\"{o}rlitz University of Applied
Sciences, Germany}}}
\maketitle

\begin{abstract}
We consider a statistical model of a n-mode quantum Gaussian state which is
shift invariant and also gauge invariant. Such models can be considered
analogs of classical Gaussian stationary time series, parametrized by their
spectral density. Defining an appropriate quantum spectral density as the
parameter, we establish that the quantum Gaussian time series model is
asymptotically equivalent to a classical nonlinear regression model given as a
collection of independent geometric random variables. The asymptotic
equivalence is established in the sense of the quantum Le Cam distance between
statistical models (experiments). The geometric regression model has a further
classical approximation as a certain Gaussian white noise model with a
transformed quantum spectral density as signal. In this sense, the result is a
quantum analog of the asymptotic equivalence of classical spectral density
estimation and Gaussian white noise, which is known for Gaussian stationary
time series. In a forthcoming version of this preprint, we will also identify
a quantum analog of the periodogram and provide optimal parametric and
nonparametric estimates of the quantum spectral density.

\end{abstract}
\tableofcontents

\section{\textbf{Main Results\label{Sec: Intro and main results}}}

\subsection{Introduction}

Quantum stationary time series models have arisen in the context of quantum
system identification and control theory \cite{MR3483526}, \cite{MR3764406}.
For some context, we will first describe some basic asymptotic inference
results for classical time series models in statistics.

Local asymptotic normality (LAN, Le Cam \cite{MR856411}) is a fundamental
property of a sequence of statistical experiments, which essentially reduces
inference for large sample size to the case of a normal location model. Let
$\left(  P_{n,\theta},\theta\in\Theta\right)  $ be a sequence of families of
p.m.'s on measurable spaces $\left(  X_{n},\Omega_{n}\right)  $ where
$\Theta\subset\mathbb{R}^{k}$; assume that for given $n$, all $P_{n,\theta}$
are mutually absolutely continuous. The sequence is LAN at $\theta
\in\mathrm{int}\left(  \Theta\right)  $ if there exists a positive $k\times k$
matrix $J_{\theta}$ and random $k$-vectors $\Delta_{n,\theta}$ on $\left(
X_{n},\Omega_{n}\right)  $ such that $\mathcal{L}\left(  \Delta_{n}%
|P_{n,\theta}\right)  \Longrightarrow_{d}N\left(  0,J_{\theta}\right)  $
(convergence in distribution), and for $h\in\mathbb{R}^{k}$ one has
\begin{equation}
\log\frac{dP_{n,\theta+h/\sqrt{n}}}{dP_{n,\theta}}=h^{\prime}\Delta_{n,\theta
}-\frac{1}{2}h^{\prime}J_{\theta}h+o_{P}\left(  1\right)  \text{ as
}n\rightarrow\infty, \label{LAN-basic}%
\end{equation}
with probability convergence taking place under the $P_{n,\theta}$ law,
uniformly over compacts in $h$. The underlying idea here is that the
log-likelihood ratio asymptotically, and locally in neighborhoods of $\theta$,
takes the form associated to a Gaussian shift experiment
\begin{equation}
\left(  N_{k}\left(  h,J_{\theta}^{-1}\right)  ,\;h\in\mathbb{R}^{k}\right)  .
\label{basic-Gaussian-shift-model}%
\end{equation}
The latter model then serves as a benchmark for optimal inference in the
original model $\left(  P_{n,\theta},\theta\in\Theta\right)  $, typically
giving risk bounds in terms of the Fisher information matrix $J_{\theta}$. One
of the earliest results establishing the LAN\ property, beyond the basic
i.i.d. case, has been Davies \cite{D73} for a stationary Gaussian time series
with spectral density depending on a parameter $\theta$. Later developments
and extensions within the framework of parametric statistical inference for
time series are summarized in the monographs \cite{MR812272} and \cite{TK00}.

When parameters are infinite dimensional, defining a framework of
nonparametric inference, the proper analog of LAN to describe risk benchmarks
for procedures \ is asymptotic equivalence in the sense of Le Cam's $\Delta
$-distance. To define it, assume all measurable sample spaces are Polish
(complete separable) metric spaces equipped with their Borel sigma-algebra.
For measures $P,$ $Q$ on the same sample space, let $\left\Vert P-Q\right\Vert
_{1}$ be $L_{1}$-distance. For the general case where $P,$ $Q$ are not
necessarily on the same sample space, suppose $K$ is a Markov kernel such that
$KP$ is a measure on the same sample space as $Q$. In that case, $\left\Vert
Q-KP\right\Vert _{1}$ is defined and will be used to measure the distance
between $Q$ and a Markov kernel randomization of $P$.

Consider now experiments (families of measures) $\mathcal{F}=\left(
Q_{\theta},\;\theta\in\Theta\right)  $ and $\mathcal{E}=\left(  P_{\theta
},\;\theta\in\Theta\right)  $, on possibly different sample spaces, but with
the same parameter space $\Theta$ (of arbitrary nature). All experiments here
are assumed dominated by a sigma-finite measure on their respective sample
space. The deficiency of $\mathcal{E}$ with respect to $\mathcal{F}$ is
defined as
\[
\delta\left(  \mathcal{E},\mathcal{F}\right)  =\inf_{K}\sup_{\theta\in\Theta
}\left\Vert Q_{\theta}-KP_{\theta}\right\Vert _{1}%
\]
where $\inf$ extends over all appropriate Markov kernels. Le Cam's
pseudodistance $\Delta\left(  \mathcal{\cdot},\cdot\right)  $ between
$\mathcal{E}$ and $\mathcal{F}$ then is
\begin{equation}
\Delta\left(  \mathcal{E},\mathcal{F}\right)  =\max\left(  \delta\left(
\mathcal{E},\mathcal{F}\right)  ,\delta\left(  \mathcal{F},\mathcal{E}\right)
\right)  . \label{Delta-distance}%
\end{equation}
It is well known that for two experiments $\mathcal{E}$ and $\mathcal{F}$
having the same parameter space, $\Delta(\mathcal{E},\mathcal{F})<\varepsilon$
implies that for any decision problem with loss bounded by $1$ and any
statistical procedure in the experiment $\mathcal{F}$ there is a (randomized)
procedure in $\mathcal{E}$, the risk of which evaluated in $\mathcal{E}$
nearly matches (within $\varepsilon$) the risk of the original procedure
evaluated in $\mathcal{F}$. In this statement the roles of $\mathcal{E}$ and
$\mathcal{F}$ can also be reversed. Two sequences $\mathcal{E}_{n}%
,\mathcal{F}_{n}$ are said to be \textit{asymptotically equivalent} if
$\Delta(\mathcal{E}_{n},\mathcal{F}_{n})\rightarrow0$.

A result on approximation in $\Delta$-distance of a classical Gaussian
stationary time series model has been obtained in \cite{MR2589320}. Assume a
sample $y^{(n)}=\left(  y(1),\ldots,y(n)\right)  ^{\prime}$ from a real
Gaussian stationary sequence $y(t)$ with zero mean, autocovariance function
$\gamma_{j}=\mathrm{E}y(t)y(t+j)$ and real spectral density $f$ on $[-\pi
,\pi]$ such that $f\left(  \omega\right)  =f\left(  -\omega\right)  $ and
\begin{equation}
\gamma_{j}=\frac{1}{2\pi}\int_{-\pi}^{\pi}\exp\left(  -ij\omega\right)
f\left(  \omega\right)  d\omega. \label{spec-density-classical}%
\end{equation}
Define a function set $\Sigma_{\alpha,M}=B^{\alpha}(M)\cap\mathcal{F}_{M},$
where $B^{\alpha}(M)$ is a Besov-Sobolev smoothness class of spectral
densities with smoothness coefficient $\alpha$ and $\mathcal{F}_{M}$ is the
set of real even positive functions $f$ on $[-\pi,\pi]$ such that $\left\vert
\log f\right\vert \leq M$. Then it is shown that observations $y^{(n)}$ with
spectral density $f$ are asymptotically equivalent to a white noise model
\begin{equation}
dZ_{\omega}=\log f(\omega)d\omega+2\pi^{1/2}n^{-1/2}dW_{\omega}\text{, }%
\omega\in\lbrack-\pi,\pi] \label{gwn-spec-density}%
\end{equation}
if the parameter space is given by $f\in$ $\Sigma_{\alpha,M}$ for some $M>0$
and $\alpha>1/2$. This represents the nonparametric (asymptotic equivalence)
version of the classical LAN property for parametric sets $\left(  f_{\theta
},\theta\in\Theta\right)  $ of spectral densities \cite{D73} \cite{MR812272}
\cite{TK00}. Here the Gaussian white noise model (\ref{gwn-spec-density})
represents an analog of the basic Gaussian location model
(\ref{basic-Gaussian-shift-model}), with the approximation valid globally
(over all spectral densities $f\in$ $\Sigma_{\alpha,M}$). Also established
were local approximations (via the connection to \cite{MR1633574}) around a
fixed spectral density $f_{0}$ like
\begin{equation}
dZ_{\omega}=f(\omega)d\omega+2\pi^{1/2}n^{-1/2}f_{0}(\omega)dW_{\omega}\text{,
}\omega\in\lbrack-\pi,\pi] \label{gwn-spec-density-local}%
\end{equation}
which are more suitable for obtaining risk bounds for estimation on $f$
itself, rather than $\log f$. Here the log-transformation plays the role of a
variance stabilizing transformation, removing the factor $f_{0}$ from the
noise term and allowing to proceed from the local asymptotic equivalence
(\ref{gwn-spec-density-local}) (valid for $f$ close to $f_{0}$) to the global
variant (\ref{gwn-spec-density}) (cf. \cite{MR1633574} for details).

The analog of the $\Delta$-distance for quantum statistical models has been
introduced and studied by several authors. In \cite{MR2229156},
\cite{MR2506764} it was used to define a (strong) quantum analog of the LAN
property (\ref{LAN-basic}) for tensor product models of qubits and finite
dimensional states. Alternative approaches to quantum LAN were pursued by
\cite{MR2346393} and \cite{YFG13}, via different definitions of a quantum
likelihood ratio. In \cite{BGN-QAE}\ the quantum Le Cam distance was used to
establish asymptotic equivalence of a tensor product model of infinite
dimensional pure states to a quantum Gaussian white noise model. Although the
approximation is local, valid in a neighborhood of a fixed pure state (and
thus is an analog of (\ref{gwn-spec-density-local})), it allows to establish a
number of results for nonparametric inference on pure states (estimation and testing).

The object of the present paper is to investigate, with regard to asymptotic
equivalence, a quantum Gaussian model studied earlier in \cite{MR2510896}. We
will consider an $n$-mode quantum Gaussian system to define a quantum Gaussian
time series of "length" $n$.

\bigskip A one mode quantum system is given by the Hilbert space $L_{2}\left(
\mathbb{R}\right)  $ and self-adjoint operators acting on appropriately
defined domains as
\[
\left(  Qf\right)  \left(  x\right)  =xf\left(  x\right)  ,\;\;\;\;\;\left(
Pf\right)  \left(  x\right)  =-i\frac{df\left(  x\right)  }{dx}%
\]
which satisfy the commutation relations%
\[
\left[  Q,P\right]  =QP-PQ=i\mathbf{1}.
\]
The Hilbert space of an $n$-mode system is $L_{2}^{\otimes n}\left(
\mathbb{R}\right)  \cong L_{2}\left(  \mathbb{R}^{n}\right)  $ on which
"canonical pairs" $\left(  Q_{j},P_{j}\right)  $ are defined acting on the
$j$th tensor factor as above, and as identity on the other tensor factors.
Thus the commutation relations on $L_{2}\left(  \mathbb{R}^{n}\right)  $ are
\begin{equation}
\left[  Q_{j},Q_{k}\right]  =\left[  P_{j},P_{k}\right]  =0,\left[
Q_{j},P_{k}\right]  =i\delta_{jk}\mathbf{1.} \label{comm-rela}%
\end{equation}
Write the vector of observables as $\mathbf{R}:=\left(  Q_{1},\ldots
,Q_{n},P_{1},\ldots,P_{n}\right)  $ and for $x\in\mathbb{R}^{2n}$ introduce
the Weyl unitaries as
\begin{equation}
W\left(  x\right)  =\exp\left(  i\mathbf{R}x\right)  .
\label{Weyl-as-function-of-P-Q}%
\end{equation}
For $x,y\in$ $\mathbb{R}^{2n}$ define a bilinear, antisymmetric (symplectic)
form as
\[
D\left(  x,y\right)  =\sum_{j=1}^{n}\left(  x_{j}y_{j+n}-x_{j+n}y_{j}\right)
.
\]
The operators $W\left(  x\right)  $, $x\in\mathbb{R}^{2n}$ satisfy $W\left(
x\right)  ^{\ast}=W\left(  -x\right)  $ and
\begin{equation}
W\left(  x\right)  W\left(  y\right)  =W\left(  x+y\right)  \exp\left(
-\frac{i}{2}D\left(  x,y\right)  \right)  ,\;x,y\in\mathbb{R}^{2n},
\label{CCR-1}%
\end{equation}
i.e. the Weyl canonical commutations relations, or CCR. The $C^{\ast}$-algebra
generated by $\left\{  W\left(  x\right)  ,x\in\mathbb{R}^{2n}\right\}  $
defines the Schr\"{o}dinger representation of $CCR\left(  \mathbb{R}%
^{2n},D\right)  $ (\cite{MR1441540}, 5.2.16). The von Neumann algebra
generated by $\left\{  W\left(  x\right)  ,x\in\mathbb{R}^{2n}\right\}  $ is
the full algebra $\mathcal{L}(L_{2}\left(  \mathbb{R}^{n}\right)  )$ of
bounded operators on $L_{2}\left(  \mathbb{R}^{n}\right)  $.

\subsection{Gaussian states\textit{\label{subsec: Gaussian states}}}

A state $\varphi$ on a von Neumann algebra $\mathcal{A}$ is a positive normal
linear functional $\varphi:$ $\mathcal{A}\rightarrow\mathbb{C}$ which takes
value $1$ on the unit of $\mathcal{A}$; cf. Section
\ref{subsec:States-channels-obs} for a short overview of the relevant
concepts. In the case of $\mathcal{A=L}(L_{2}\left(  \mathbb{R}^{n}\right)
)$, a state is entirely determined by its values on the Weyl unitaries, which
allows to define the characteristic function of $\varphi$ at argument
$x\in\mathbb{R}^{2n}$ as
\begin{equation}
\hat{W}\left[  \varphi\right]  \left(  x\right)  :=\varphi\left(  W\left(
x\right)  \right)  . \label{char-function-general}%
\end{equation}
Consider a real positive definite symmetric $2n\times2n$ matrix $\Sigma$
satisfying
\begin{equation}
\frac{1}{4}\left(  D\left(  x,y\right)  \right)  ^{2}\leq\left\langle x,\Sigma
x\right\rangle \left\langle y,\Sigma y\right\rangle \text{, }x,y\in
\mathbb{R}^{2n}. \label{uncertainty-implied}%
\end{equation}
Then there exists a unique state $\varphi\left(  0,\Sigma\right)  $ on
$\mathcal{L}(L_{2}\left(  \mathbb{R}^{n}\right)  )$ with characteristic
function
\begin{equation}
\hat{W}\left[  \varphi\left(  0,\Sigma\right)  \right]  \left(  x\right)
=\exp\left(  -\frac{1}{2}\left\langle x,\Sigma x\right\rangle \right)  \text{,
}x\in\mathbb{R}^{2n} \label{char-function-Gaussian}%
\end{equation}
(\cite{MR1057180}, Theorem 3.4). Such states are called \textit{centered
Gaussian} (or quasifree) with covariance matrix $\Sigma$. The inequality
(\ref{uncertainty-implied}) is required by Heisenberg's uncertainty relation
(\cite{MR2797301}, Theorem 5.5.1).

\subsection{Shift invariant states}

In a centered Gaussian state $\varphi\left(  0,\Sigma\right)  $, every
observable $R\left(  x\right)  =\mathbf{R}x$ has a normal distribution
\begin{equation}
R\left(  x\right)  \sim N\left(  0,\left\langle x,\Sigma x\right\rangle
\right)  . \label{observables-normal-distr}%
\end{equation}
Define the two vectors $\mathbf{R}_{s}:=\left(  Q_{1+s},\ldots,Q_{n-1+s}%
,P_{1+s},\ldots,P_{n-1+s}\right)  $, $s=0,1$. The state $\varphi\left(
0,\Sigma\right)  $ is \textit{shift invariant }if for every $t\in
\mathbb{R}^{2(n-1)}$ the observables
\[
R_{s}\left(  t\right)  =\mathbf{R}_{s}t\text{, }s=0,1
\]
have the same distribution. It is easily seen that this implies shift
invariance for the one mode subsystem $\left(  Q_{1},P_{1}\right)  $, and also
shift invariance for any $r$-mode subsystem $\left(  Q_{1,},\ldots,Q_{r,}%
P_{1},\ldots,P_{r}\right)  $, $1\leq r<n$. It follows that the covariance
matrix $\Sigma$ is such that all four $n\times n$ submatrices in
\[
\Sigma=\left(
\begin{array}
[c]{cc}%
\Sigma_{11} & \Sigma_{12}\\
\Sigma_{12}^{\prime} & \Sigma_{22}%
\end{array}
\right)
\]
are Toeplitz. Equivalently, if $\check{\Sigma}$ is the permutation of
$\Sigma\ $such that for $\mathbf{\check{R}}:=(Q_{1},P_{1},\ldots,Q_{n},P_{n})$
we have
\begin{equation}
\mathbf{\check{R}}x\sim N\left(  0,\left\langle x,\check{\Sigma}x\right\rangle
\right)  \text{, }x\in\mathbb{R}^{2n} \label{permuted-cov-matrix}%
\end{equation}
then $\check{\Sigma}$ is\textit{ block Toeplitz}, i.e. it is of form
$\Sigma=\left(  \Sigma_{j-k}^{0}\right)  _{j,k=1}^{n}$ where $\left\{
\Sigma_{k}^{0}\right\}  _{k=1-n}^{n-1}$ is a sequence of $2\times2$ matrices.
The block Toeplitz structure is familiar in the statistical theory for
classical multivariate time series \cite{TK00}.

\subsection{Gauge invariant states}

The Weyl unitaries $W\left(  x\right)  $, $x\in\mathbb{R}^{2n}$ can
equivalently be indexed by complex $u\in\mathbb{C}^{n}$ such that $V\left(
u\right)  :=W\left(  \underline{u}\right)  $ where $\underline{u}:=\left(
-\operatorname{Im}u\right)  \oplus\operatorname{Re}u$, whereupon the CCR
relation (\ref{CCR-1}) writes as
\begin{equation}
V\left(  u\right)  V\left(  v\right)  =V\left(  u+v\right)  \exp\left(
-\frac{i}{2}\operatorname{Im}\left\langle u,v\right\rangle \right)
,\;u,v\in\mathbb{C}^{n}. \label{Weyl-unitaries-complex-indexed}%
\end{equation}
A state $\rho$ is gauge invariant if for every $z\in\mathbb{C}$, $\left\vert
z\right\vert =1$ one has $\rho\left(  V\left(  zu\right)  \right)
=\rho\left(  V\left(  u\right)  \right)  ,$ $u\in\mathbb{C}^{n}$. A quasifree
state $\varphi\left(  0,\Sigma\right)  $ is gauge invariant if and only if
\[
\left\langle \underline{zu},\Sigma\underline{zu}\right\rangle =\left\langle
\underline{u},\Sigma\underline{u}\right\rangle ,u\in\mathbb{C}^{n}%
,z\in\mathbb{C},\left\vert z\right\vert =1
\]
or equivalently, if there exists a self-adjoint positive operator $A$ on
$\mathbb{C}^{n}$ such that
\[
\left\langle \underline{u},\Sigma\underline{u}\right\rangle =\frac{1}%
{2}\left\langle u,Au\right\rangle ,\;u\in\mathbb{C}^{n}.
\]
The matrix $A$ is called the \textit{symbol} of $\varphi\left(  0,\Sigma
\right)  $; it is related to the covariance matrix $\Sigma$ by
\begin{equation}
\Sigma=\Sigma\left(  A\right)  :=\frac{1}{2}\left(
\begin{array}
[c]{cc}%
\operatorname{Re}A & -\operatorname{Im}A\\
\operatorname{Im}A & \operatorname{Re}A
\end{array}
\right)  \label{symbol-related-to-covmatrix}%
\end{equation}
where $\operatorname{Re}A$ is symmetric and $\operatorname{Im}A$ is
antisymmetric ($\left(  \operatorname{Im}A\right)  ^{\prime}%
=-\operatorname{Im}A$). Relation (\ref{uncertainty-implied}) then can be
written
\begin{equation}
\left(  \operatorname{Im}\left\langle u,v\right\rangle \right)  ^{2}%
\leq\left\langle u,Au\right\rangle \left\langle v,Av\right\rangle \text{,
}u,v\in\mathbb{C}^{n}. \label{uncertainty-implied-symbols}%
\end{equation}
Upon setting $v=iu$, this implies $A\geq I_{n}$, and conversely every $n\times
n$ Hermitian matrix $A\geq I_{n}$ satisfies (\ref{uncertainty-implied-symbols}%
) and thus is the symbol of an $n$-mode gauge invariant centered Gaussian
state. For the gauge invariant centered Gaussian state $\rho=\varphi\left(
0,\Sigma\right)  $ with symbol $A$, covariance matrix $\Sigma=\Sigma\left(
A\right)  $ and characteristic function
\begin{equation}
\hat{W}\left[  \varphi\left(  0,\Sigma\right)  \right]  \left(  \underline
{u}\right)  =\rho\left(  V\left(  u\right)  \right)  =\exp\left(  -\frac{1}%
{4}\left\langle u,Au\right\rangle \right)  \text{, }u\in\mathbb{C}^{n}
\label{char-func-gauge-invar}%
\end{equation}
we write
\begin{equation}
\varphi\left(  0,\Sigma\left(  A\right)  \right)  =\mathfrak{N}_{n}\left(
0,A\right)  . \label{symbol-related-to-covmatrix-2}%
\end{equation}
With this notation we suggest an analogy to the $n$-variate centered normal
distribution with covariance matrix $M$, usually written $N_{n}\left(
0,M\right)  $. Note that for one mode ($n=1$), a gauge invariant centered
Gaussian state has symbol $a\in\mathbb{R}$, $a\geq1$ and covariance matrix
$\Sigma=aI_{2}/2$. Thus
\begin{equation}
\mathfrak{N}_{1}\left(  0,a\right)  =\varphi\left(  0,aI_{2}/2\right)
\label{thermal-state}%
\end{equation}
is the vacuum state for $a=1$ and a thermal state for $a>1$. If $A$ is
diagonal $A=\mathrm{diag}\left(  a_{1},\ldots,a_{n}\right)  >I_{n}$ then
$\mathfrak{N}_{n}\left(  0,A\right)  $ is the $n$-fold tensor product of
thermal states $\mathfrak{N}_{1}\left(  0,a_{j}\right)  $.

\subsection{The asymptotic setup}

The quantum statistical model for fixed $n$ is now given by a family of gauge
invariant and shift invariant centered Gaussian states $\left(  \mathfrak{N}%
_{n}\left(  0,A\right)  ,A\in\mathfrak{A}_{n}\right)  $ where $\mathfrak{A}%
_{n}$ is a set of $n\times n$ complex Hermitian Toeplitz matrices with $A\geq
I$. In accordance w\'{\i}th the usage in classical statistics, the model might
be described as a \textit{stationary quantum Gaussian time series.} For
asymptotic inference in that model, we assume that the $n\times n$ symbols
$A=\left(  a_{j,k}\right)  _{j,k=1}^{n}$ are related to a given positive
bounded measurable function $a:\left[  -\pi,\pi\right]  \rightarrow\mathbb{R}$
as follows:
\begin{equation}
a_{jk}=a_{k-j},\;a_{k}=\frac{1}{2\pi}\int_{-\pi}^{\pi}\exp\left(
-ik\omega\right)  a\left(  \omega\right)  d\omega,\;j,k\in\mathbb{Z}.
\label{symbol-generate}%
\end{equation}
such that
\begin{equation}
a\left(  \omega\right)  =\sum_{k=-\infty}^{\infty}a_{k}\phi_{k}\left(
\omega\right)  \text{ where }\phi_{k}\left(  \omega\right)  =\exp\left(
ik\omega\right)  \text{, }\omega\in\mathbb{R}\text{.}
\label{symbol-generate-2}%
\end{equation}
Here $a_{k}$ are analogs of the autocovariances of a classical stationary
complex valued time series, fulfilling $\overline{a_{k}}=a_{-k}$. Accordingly
the function $a\left(  \omega\right)  $ may be described as the
\textit{quantum spectral density.} We assume $a$ to be real and fulfilling
$a\geq1$, and we write $A_{n}\left(  a\right)  $ for the $n\times n$ Hermitian
Toeplitz matrix generated by (\ref{symbol-generate}) for given $a$. We then
have $A_{n}\left(  a\right)  \geq I_{n}$ (compare Lemma \ref{lem-toeplitz-EV}
below); our quantum statistical model is now a family of states
\begin{equation}
\left(  \mathfrak{N}_{n}\left(  0,A_{n}\left(  a\right)  \right)  ,a\in
\Theta\right)  \label{quantum-asy-setup-spec-density}%
\end{equation}
where $\Theta$ is a family of quantum spectral densities on $\left[  -\pi
,\pi\right]  $ fulfilling $a\geq1$. Note that if $f$ is a real function with
$f\geq0$ on $\left[  -\pi,\pi\right]  $ which is even (i.e. symmetric,
$f\left(  \omega\right)  =f\left(  -\omega\right)  $) then the matrix
$A_{n}\left(  f\right)  $ is real symmetric nonnegative definite, i.e. it is
the covariance matrix of a real random vector. As $A_{n}\left(  f\right)  $ is
also a sequence of Toeplitz matrices, it would describe the standard setup for
a sequence of covariance matrices in classical stationary real valued time
series \cite{MR1093459}, \cite{MR812272}, \cite{MR2589320} (except that the
standard setup defines the spectral density $f$ without a factor $1/2\pi$ in
(\ref{spec-density-classical})). Indeed comparing (\ref{symbol-generate}) with
(\ref{spec-density-classical}), we see that $a_{j}=\gamma_{j}$ if in
(\ref{symbol-generate}) we set $a\left(  \omega\right)  =f\left(
\omega\right)  ,$ $\omega\in\left[  -\pi,\pi\right]  $. Our asymptotic model
(\ref{quantum-asy-setup-spec-density}), where the spectral density $a$ is the
parameter, is thus a quantum analog of a classical time series, involving the
symbol matrices as analogs of covariance matrices. To our knowledge, the model
has first been treated in \cite{MR2510896} in the problem of discrimination
between two spectral densities $a_{1},a_{2}$. There the quantum Chernoff bound
has been computed for the specified quantum Gaussian models, based on the
general form of the quantum Chernoff bound as previously found in
\cite{MR2502660} and \cite{MR2377635}.

\subsection{Quantum Le Cam distance\label{subsec: Qu-Lecam-distance}}

We follow \cite{MR2346393} for defining the quantum analog of the $\Delta
$-distance (\ref{Delta-distance}). So far the quantum Gaussian states
$\mathfrak{N}_{n}\left(  0,A\right)  $ have been defined on the von Neumann
algebra $\mathcal{L}\left(  L_{2}\left(  \mathbb{R}^{n}\right)  \right)  $,
but in order to incorporate classical families of probability distributions
into this framework, one needs to consider commutative von Neumann algebras
defined by spaces $L^{\infty}\left(  \mu\right)  $ of functions on a $\sigma
$-finite measure space $\left(  X,\Omega,\mu\right)  $. In our appendix
section \ref{subsec:States-channels-obs} we clarify how states on a von
Neumann algebra $\mathcal{A}$ can be understood as elements of the predual
$\mathcal{A}_{\ast}$ of $\mathcal{A}$. The predual $\mathcal{A}_{\ast}$ is a
Banach space with norm $\left\Vert \cdot\right\Vert _{1}$ such that
$\mathcal{A}$ is its dual Banach space, and the states $\varphi$ are positive
elements of $\mathcal{A}_{\ast}$ which fulfill $\left\Vert \varphi\right\Vert
_{1}=1$. In the case $\mathcal{A}=\mathcal{L}\left(  L_{2}\left(
\mathbb{R}^{n}\right)  \right)  $, it is well known that a state $\varphi$ has
a density operator $\rho_{\varphi}$ (a positive operator on $L_{2}\left(
\mathbb{R}^{n}\right)  $ with unit trace) such that
\[
\varphi\left(  V\left(  x\right)  \right)  =\mathrm{Tr\;}\rho_{\varphi
}V\left(  x\right)  ,x\in\mathbb{C}^{n}.
\]
In that case $\left\Vert \varphi\right\Vert _{1}=\mathrm{Tr\;}\rho_{\varphi
}=1$ and for states $\varphi,\sigma$, the distance
\[
\left\Vert \varphi-\sigma\right\Vert _{1}=\mathrm{Tr\;}\left\vert
\rho_{\varphi}-\rho_{\sigma}\right\vert
\]
is the usual trace distance. In the case $\mathcal{A}=L^{\infty}\left(
\mu\right)  $, states are positive elements $f$ of $L^{1}\left(  \mu\right)  $
fulfilling $\left\Vert f\right\Vert _{1}=\int fd\mu=1$, i.e. probability
density functions, and for states $f,g$ on $L^{\infty}\left(  \mu\right)  $,
the distance
\[
\left\Vert f-g\right\Vert _{1}=\int\left\vert f-g\right\vert d\mu
\]
is the usual $L^{1}$-distance. In general we will refer to $\left\Vert
\cdot\right\Vert _{1}$ as the predual norm.

A \textit{quantum statistical experiment} $\mathcal{E}=\left(  \mathcal{A}%
,\rho_{\theta},\theta\in\Theta\right)  $ is given by a family of states
$\rho_{\theta},\theta\in\Theta$ on a von Neumann algebra $\mathcal{A}$ where
$\rho_{\theta}\in\mathcal{A}_{\ast}$. As a regularity condition, it is assumed
that experiments are homogeneous and in reduced form (cf. Subsection
\ref{subsubsec-qu-statist-experiments}). Let $\mathcal{F}:=\left(
\mathcal{B},\sigma_{\theta},\theta\in\Theta\right)  $ be another quantum
statistical experiment, indexed by the same parameter $\theta$. The deficiency
$\delta\left(  \mathcal{E},\mathcal{F}\right)  $ is defined as
\begin{equation}
\delta\left(  \mathcal{E},\mathcal{F}\right)  :=\inf_{\alpha}\sup_{\theta
}\left\Vert \rho_{\theta}\circ\alpha-\sigma_{\theta}\right\Vert _{1}
\label{deficiency-quantum-def}%
\end{equation}
where the infimum is taken over all quantum channels $\alpha:\mathcal{B}%
\rightarrow\mathcal{A}$ (see Appendix, \ref{subsec:States-channels-obs} for
the definition of channels). The channels $\alpha$ are certain linear and
(completely) positive maps between the von Neumann algebras; they give rise to
quantum state transitions (TP-CP maps) $T:\mathcal{A}_{\ast}\rightarrow
\mathcal{B}_{\ast}$ via the duality (\ref{dual-channels-1}). If $\mathcal{A}$
and $\mathcal{B}$ are of type $L^{\infty}\left(  \mu_{i}\right)  $, $i=1,2$
then the TP-CP maps are transitions in the sense of Le Cam between dominated
families of probability measures, which under regularity conditions are given
by Markov kernels (cf. (\ref{Markov-kernel-transition-exist-b})). In the mixed
case where $\mathcal{B=}L^{\infty}\left(  \mu\right)  $ and $\mathcal{A=L}%
\left(  \mathcal{H}\right)  ,$ the channel $\alpha$ is an observation channel
(measurement) which arises from a POVM\ (positive operator valued measure),
cf. Subsection \ref{subsubsec:measmt-obs-channel}.

The Le Cam distance between $\mathcal{E}$ and $\mathcal{F}$ is
\begin{equation}
\Delta\left(  \mathcal{E},\mathcal{F}\right)  =\max\left(  \delta\left(
\mathcal{E},\mathcal{F}\right)  ,\delta\left(  \mathcal{F},\mathcal{E}\right)
\right)  . \label{Q-Delta-distance}%
\end{equation}
We say that $\mathcal{E}$ is \textit{more informative} than $\mathcal{F}$ if
$\delta\left(  \mathcal{E},\mathcal{F}\right)  =0$; if the reverse also holds
(i.e. $\Delta\left(  \mathcal{E},\mathcal{F}\right)  =0$) then $\mathcal{E}%
,\mathcal{F}$ are said to be \textit{statistically equivalent}.

Consider now sequences of experiments, where the algebras and states depend on
$n$, but the parameter space $\Theta$ remains fixed. A sequence $\mathcal{E}%
_{n}=\left(  \mathcal{A}_{n},\rho_{n,\theta},\theta\in\Theta\right)  $ is said
to be \textit{asymptotically more informative} than $\mathcal{F}_{n}=\left(
\mathcal{B}_{n},\sigma_{n,\theta},\theta\in\Theta\right)  $ if
\[
\delta\left(  \mathcal{E}_{n},\mathcal{F}_{n}\right)  \rightarrow0\text{ as
}n\rightarrow\infty.
\]
We write $\mathcal{F}_{n}\precsim\mathcal{E}_{n}$ in this case. If the reverse
also holds, i.e. if
\[
\Delta\left(  \mathcal{E}_{n},\mathcal{F}_{n}\right)  \rightarrow0\text{ as
}n\rightarrow\infty
\]
then $\mathcal{E}_{n}$ and $\mathcal{F}_{n}$ are said to be
\textit{asymptotically equivalent}, written $\mathcal{E}_{n}\approx
\mathcal{F}_{n}$ \textit{.}

As to the statistical meaning of the relation $\mathcal{F}_{n}\precsim
\mathcal{E}_{n}$, it implies there is a sequence of dual channels (TP-CP maps,
state transitions) between preduals $T_{n}:\mathcal{A}_{n\ast}\rightarrow
\mathcal{B}_{n\ast}$ such that
\begin{equation}
\sup_{\theta}\left\Vert \sigma_{n,\theta}-T_{n}\left(  \rho_{n,\theta}\right)
\right\Vert _{1}\rightarrow0. \label{delta-closeness}%
\end{equation}
Assume that statistical decisions are to be made in the experiment
$\mathcal{F}_{n}$. Let $M_{n}$ be a dual observation channel (measurement) to
be applied in $\mathcal{F}_{n}$, such that $M_{n}:\mathcal{B}_{n\ast
}\rightarrow L^{1}\left(  \nu\right)  $ where $\nu$ is a sigma-finite measure
on $\left(  X,\Omega\right)  $. Then $p_{n,\theta}:=M_{n}\left(
\sigma_{n,\theta}\right)  $ is a $\nu$-probability density on $\left(
X,\Omega\right)  $, and combining the transitions $M_{n}$ and $T_{n}$, we
obtain a $\nu$-probability density $p_{n,\theta}^{\prime}:=M_{n}\left(
T_{n}\left(  \rho_{n,\theta}\right)  \right)  $. Then by the contraction
property (\ref{contraction-property}) of the dual channel $M_{n}$
\begin{equation}
\sup_{\theta}\left\Vert p_{n,\theta}-p_{n,\theta}^{\prime}\right\Vert _{1}%
\leq\sup_{\theta}\left\Vert \sigma_{n,\theta}-T_{n}\left(  \rho_{n,\theta
}\right)  \right\Vert _{1}\rightarrow0. \label{delta-closeness-3}%
\end{equation}
Let a set of $\Omega$-measurable loss functions $W_{n,\theta}:X\rightarrow
\left[  0,1\right]  $, $\theta\in\Theta$ be given. Then a measurement $M_{n}$
as above can be interpreted as a (randomized) decision rule in experiment
$\mathcal{F}_{n},$ where the aim is to make $\int W_{n,\theta}p_{n,\theta}%
d\nu$ small for every $\theta$ (or small in a worst case sense). Then
(\ref{delta-closeness-3}) implies%
\[
\sup_{\theta}\left\vert \int W_{n,\theta}p_{n,\theta}d\nu-\int W_{n,\theta
}p_{n,\theta}^{\prime}d\nu\right\vert \rightarrow0.
\]
In other words, if the sequence $\mathcal{E}_{n}$ is asymptotically more
informative than $\mathcal{F}_{n}$ ($\mathcal{F}_{n}\precsim\mathcal{E}_{n}$)
then for every randomized decision rule in $\mathcal{F}_{n}$ there exists one
in $\mathcal{E}_{n}$ which is asymptotically as good, uniformly in $\theta
\in\Theta$. Hence risk bounds attainable in $\mathcal{F}_{n}$ can also be
attained in $\mathcal{E}_{n}$. Conversely, decision rules in $\mathcal{F}_{n}$
cannot be asymptotically better than those in $\mathcal{E}_{n}$, i. e. the
relation provides lower asymptotic risk bounds.

\subsection{Main theorems}

For any set $\Theta$ of quantum spectral densities, i.e.%
\
real functions $a$ on $\left[  -\pi,\pi\right]  $ such that $a\left(
\omega\right)  \geq1$, $\omega\in\left[  -\pi,\pi\right]  $ consider the
quantum statistical experiment
\begin{equation}
\mathcal{E}_{n}\left(  \Theta\right)  :=\left(  \mathfrak{N}_{n}\left(
0,A_{n}\left(  a\right)  \right)  ,a\in\Theta\right)
\label{basic-quantum-experi-def}%
\end{equation}
where $A_{n}\left(  a\right)  $ is the $n\times n$ symbol matrix pertaining to
$a$. Define also a corresponding classical geometric regression experiment
$\mathcal{F}_{n}\left(  \Theta\right)  $ as follows. For any function
$a\in\Theta$ define a set of functionals (local averages on $\left[  -\pi
,\pi\right]  $) as
\begin{equation}
J_{j,n}\left(  a\right)  =n\int_{\left(  j-1\right)  /n}^{j/n}a\left(
2\pi\left(  x-1/2\right)  \right)  dx\text{.} \label{local-averages-def}%
\end{equation}
Also consider the geometric distribution $\mathrm{Geo}\left(  p\right)  $ with
probabilities $\left(  1-p\right)  p^{j}$, $j=0,1,\ldots$where the parameter
$p\in\left(  0,1\right)  $ depends on some $\lambda>1$ via $p\left(
\lambda\right)  =\left(  \lambda-1\right)  /\left(  \lambda+1\right)  $.
Define
\begin{equation}
\mathcal{F}_{n}\left(  \Theta\right)  :=\left(
{\displaystyle\bigotimes\limits_{j=1}^{n}}
\mathrm{Geo}\left(  p\left(  J_{j,n}\left(  a\right)  \right)  \right)
,a\in\Theta\right)  . \label{geometric-regression-experiment-basic}%
\end{equation}
Consider the set $\Theta_{1}\left(  \alpha,M\right)  $ of quantum spectral
densities $a$ defined as the set of real functions on $\left[  -\pi
,\pi\right]  $, such that for some $\alpha>0$, $M>1$%
\begin{align}
\Theta_{1}\left(  \alpha,M\right)   &  :=\left\{  a:\text{ }\left\vert
a_{0}\right\vert ^{2}+\sum_{j=-\infty}^{\infty}j^{2\alpha}\left\vert
a_{j}\right\vert ^{2}\leq M\right\}  \cap\mathcal{L}_{M}%
,\label{Theta-1-functionset-def}\\
\mathcal{L}_{M}  &  :=\left\{  a:a\left(  \omega\right)  \geq1+M^{-1}\text{,
}\omega\in\left[  -\pi,\pi\right]  \right\}  , \label{lowerbound-set-def}%
\end{align}
where $a_{j}$ are defined by (\ref{symbol-generate}).

\begin{theorem}
\label{theor-main-1}If $\Theta=\Theta_{1}\left(  \alpha,M\right)  $ for some
$\alpha>1/2,$ $M>1$ then
\[
\delta\left(  \mathcal{F}_{n}\left(  \Theta\right)  ,\mathcal{E}_{n}\left(
\Theta\right)  \right)  \rightarrow0\text{ as }n\rightarrow\infty,
\]
i.e. $\mathcal{F}_{n}\left(  \Theta\right)  $ is asymptotically more
informative than $\mathcal{E}_{n}\left(  \Theta\right)  $: $\mathcal{E}%
_{n}\left(  \Theta\right)  \precsim\mathcal{F}_{n}\left(  \Theta\right)  $.
\end{theorem}

\bigskip

Let us further introduce an experiment of the type "signal in Gaussian white
noise" on the interval $\left[  -\pi,\pi\right]  $. Consider the function
\begin{equation}
\mathrm{arc\cosh}\left(  x\right)  =\log\left(  x+\sqrt{x^{2}-1}\right)  ,x>1
\label{arccosh-def}%
\end{equation}
and let $Q_{n}\left(  a\right)  $ be the distribution of the stochastic
process $Y_{\omega},\omega\in\left[  -\pi,\pi\right]  $ given by the
stochastic differential equation
\begin{equation}
dY_{\omega}=\mathrm{arc\cosh}\left(  a\left(  \omega\right)  \right)
d\omega+\left(  2\pi/n\right)  ^{1/2}dW_{\omega}\text{, }\omega\in\left[
-\pi,\pi\right]  \label{SDE-1}%
\end{equation}
where $dW_{\omega},$ $\omega\in\left[  -\pi,\pi\right]  $ is Gaussian white
noise. Here $Q_{n}\left(  a\right)  $ is a distribution on the measurable
space $\left(  C_{\left[  -\pi,\pi\right]  },\mathcal{B}_{C_{\left[  -\pi
,\pi\right]  }}\right)  $ where $\mathcal{B}_{C_{\left[  -\pi,\pi\right]  }}$
is the pertaining Borel sigma-algebra. For $\Theta=\Theta_{1}\left(
\alpha,M\right)  $ consider the experiment $\mathcal{G}_{n}\left(
\Theta\right)  =\left(  Q_{n}\left(  a\right)  ,\;a\in\Theta\right)  $.

\begin{theorem}
\label{theor-GWN-main-1}If $\Theta=\Theta_{1}\left(  \alpha,M\right)  $ for
some $\alpha>1,$ $M>1$ then%
\[
\Delta\left(  \mathcal{F}_{n}\left(  \Theta\right)  ,\mathcal{G}_{n}\left(
\Theta\right)  \right)  \rightarrow0\text{ as }n\rightarrow\infty
\]
i.e. $\mathcal{F}_{n}\left(  \Theta\right)  $ and $\mathcal{G}_{n}\left(
\Theta\right)  $ are asymptotically equivalent: $\mathcal{F}_{n}\left(
\Theta\right)  \approx\mathcal{G}_{n}\left(  \Theta\right)  $.
\end{theorem}

This claim essentially follows from the results of \cite{MR1633574}. It
implies that for $\alpha>1$, for the quantum time series, the white noise
model $\mathcal{G}_{n}\left(  \Theta\right)  $ is an upper information bound
as well. Note that the function $\mathrm{arc\cosh}$ is the analog of the
log-transformation of the spectral density in (\ref{gwn-spec-density}).

\bigskip

Converse results can be established if the parameter space is restricted to be
finite dimensional. For a nonnegative integer $d$ and some $M>1$ define
\begin{equation}
\Theta_{2}\left(  d,M\right)  :=\left\{  a:\text{ }\sum_{j=-d}^{d}\left\vert
a_{j}\right\vert ^{2}\leq M,\;a_{j}=0\text{, }\left\vert j\right\vert
>d\right\}  \cap\mathcal{L}_{M}. \label{d-dependence-param-set}%
\end{equation}
Then the symbol matrices $A_{n}\left(  a\right)  $ are banded Toeplitz and the
quantum states $\mathfrak{N}_{n}\left(  0,A_{n}\left(  a\right)  \right)  $
form a $d$-dependent quantum time series.

\begin{theorem}
\label{theor-main-2}If $\Theta=\Theta_{2}\left(  d,M\right)  $ for an integer
$d\geq0$ and some $M>1$ then
\[
\delta\left(  \mathcal{E}_{n}\left(  \Theta\right)  ,\mathcal{G}_{n}\left(
\Theta\right)  \right)  \rightarrow0\text{ as }n\rightarrow\infty,
\]
i.e. $\mathcal{E}_{n}\left(  \Theta\right)  $ is asymptotically more
informative than $\mathcal{G}_{n}\left(  \Theta\right)  $: $\mathcal{G}%
_{n}\left(  \Theta\right)  \precsim\mathcal{E}_{n}\left(  \Theta\right)  $.
\end{theorem}

It is easy to see that for $\alpha>0$, one has $\Theta_{2}\left(  d,M\right)
\subset\Theta_{1}\left(  \alpha,M^{\prime}\right)  $ for $M^{\prime}%
=M\min\left(  1,d^{2\alpha}\right)  $. In view of Theorems \ref{theor-main-1}
and \ref{theor-GWN-main-1} this implies

\begin{corollary}
If $\Theta=\Theta_{2}\left(  d,M\right)  $ for $d\geq0$ and $M>1$ then
\[
\Delta\left(  \mathcal{E}_{n}\left(  \Theta\right)  ,\mathcal{G}_{n}\left(
\Theta\right)  \right)  \rightarrow0\text{ as }n\rightarrow\infty,
\]
i.e. $\mathcal{G}_{n}\left(  \Theta\right)  $ and $\mathcal{E}_{n}\left(
\Theta\right)  $ are asymptotically equivalent: $\mathcal{E}_{n}\left(
\Theta\right)  \approx\mathcal{G}_{n}\left(  \Theta\right)  $.
\end{corollary}

\bigskip The proofs of Theorems \ref{theor-main-1}, \ref{theor-GWN-main-1} and
\ref{theor-main-2} are in Subsections \ref{subsec:comparing-geom-reg},
\ref{Subsec-Geom-reg-and-white-noise} and \ref{subsec-proof-lower-info-bound}, respectively.

In a forthcoming version of this preprint, we will also identify a quantum
analog of the periodogram and provide optimal parametric and nonparametric
estimates of the quantum spectral density.%

\begin{privatenotes}%

\subsection{Discussion}

\textbf{Remark 1.} \textit{The Chernoff bound for discrimination.} In
\cite{MR2510896}, for the problem of disciminating between two spectral
densities $a_{0},a_{1}$, the optimal error exponent, i.e. the Chernoff bound
has been found. If $D_{n}$ is the minimum sum of error probabilities over all
tests based on a time series of length $n$, then
\begin{equation}
\lim\frac{1}{n}\log D_{n}=\inf_{0\leq t\leq1}\psi\left(  t\right)  \text{, }
\label{Cher-bound-gauge-inv}%
\end{equation}%
\begin{equation}
\psi\left(  t\right)  =-\frac{1}{2\pi}\int_{\left[  0,2\pi\right]  }\log
\frac{1}{2}\left[  \left(  a_{0}\left(  \omega\right)  +1\right)  ^{t}\left(
a_{1}\left(  \omega\right)  +1\right)  ^{1-t}-\left(  a_{0}\left(
\omega\right)  -1\right)  ^{t}\left(  a_{1}\left(  \omega\right)  -1\right)
^{1-t}\right]  d\omega. \label{Cher-bound-gauge-inv-3}%
\end{equation}
This form of the Chernoff bound can be recognized as being related to
geometric distributions. Indeed let $X$ be a r.v. with geometric law
$\mathrm{Geo}\left[  p\right]  $ for parameter $p\in\left(  0,1\right)  $,
given by
\[
P\left(  X=k\right)  =\mathrm{Geo}\left[  p\right]  \left(  k\right)  =\left(
1-p\right)  p^{k}\text{, }k=0,1,\ldots
\]
On the basis of $n$ i.i.d. observation $X_{1},\ldots,X_{n}$ with distribution
$\mathrm{Geo}\left[  p\right]  $, consider the problem of discriminating
between two parameter values $p_{0},p_{1}$. The classical Chernoff bound for
two discrete probability measures $q_{i,k}$, $k\geq0,$ $i=0,1$ states
(\ref{Cher-bound-gauge-inv}) with $\psi\left(  t\right)  $ replaced by
\[
\tilde{\psi}\left(  t\right)  =\sum_{k=0}^{\infty}q_{0,k}^{t}q_{1,k}^{1-t}%
\]
(cf. e.g. \cite{MR2502660} , Appendix). Setting $q_{i,k}=\mathrm{Geo}\left[
p_{i}\right]  \left(  k\right)  $ and $a_{i}=\left(  1+p_{i}\right)  /\left(
1-p_{i}\right)  $, $i=0,1$, a direct computation shows
\begin{equation}
\tilde{\psi}\left(  t\right)  =-\log\frac{1}{2}\left[  \left(  a_{0}+1\right)
^{t}\left(  a_{1}+1\right)  ^{1-t}-\left(  a_{0}-1\right)  ^{t}\left(
a_{1}-1\right)  ^{1-t}\right]  . \label{Cher-bound-geom-1}%
\end{equation}
This coincides with (\ref{Cher-bound-gauge-inv-3}) if both spectral densities
$a_{i}\left(  x\right)  $ are constant.

For the classical stationary Gaussian time series with spectral density $f$
according to (\ref{spec-density-classical}), the Chernoff bound has been found
in \cite{CD79}: for testing between $f_{0},f_{1}$, (\ref{Cher-bound-gauge-inv}%
) holds with
\[
\psi\left(  t\right)  =\frac{1}{4\pi}\int_{\left[  -\pi,\pi\right]  }\log
\frac{\left(  f_{0}\left(  \omega\right)  \right)  ^{t}\left(  f_{1}\left(
\omega\right)  \right)  ^{1-t}}{tf_{0}\left(  \omega\right)  +\left(
1-t\right)  f_{1}\left(  \omega\right)  }d\omega.
\]
This coincides with the Chernoff bound for a model of independent exponential
random variables $Y_{j}\sim\mathrm{Exp}\left(  f\left(  \omega_{j}\right)
\right)  $, $j=1,\ldots,\left[  n/2\right]  $, where $\omega_{j}$ are
equidistant points in $\left[  0,\pi\right]  $ and $\mathrm{Exp}\left(
\theta\right)  $ has density $\theta^{-1}\exp\left(  x\theta^{-1}\right)  $,
$x\geq0$. The asymptotic $\Delta$-equivalence of the latter model with the
stationary Gaussian time series has been shown in \cite{MR2589320}.

It should be noted that an exact relationship between asymptotic $\Delta
$-equivalence of models and equality of their binary Chernoff testing bounds
is not obvious and has not been established, as the Chernoff bound belongs to
the realm of large deviation theory. Nevertheless, the result
(\ref{Cher-bound-gauge-inv}) in \cite{MR2510896} and its relation to geometric
random variables gave rise to a conjecture of asymptotic $\Delta$-equivalence,
and thus to the results of the present paper.

\bigskip

2) \textit{An analog of the periodogram.} Efficient estimation of the spectral
density (i.e. of the autocovariance parameter $\theta$) in the case of
$m$-dependence. an analog of the periodogram: an $n$-vector of observables
$\left(  \Psi_{1},\ldots,\Psi_{n}\right)  $

\bigskip\textit{3) Asymptotically classical quantum models.} Discuss the paper
by Suzuki \cite{suzuki2019information}%

\end{privatenotes}%

\paragraph*{Further notation\textbf{. }}

Consider quantum statistical experiments $\mathcal{E}=\left(  \mathcal{A}%
,\rho_{\theta},\theta\in\Theta\right)  $ and $\mathcal{F}:=\left(
\mathcal{B},\sigma_{\theta},\theta\in\Theta\right)  $ having the same
parameter space. For the special case that $\mathcal{A}=\mathcal{B}$ define
their predual norm distance
\[
\Delta_{0}\left(  \mathcal{E},\mathcal{F}\right)  =\sup_{\theta}\left\Vert
\rho_{\theta}-\sigma_{\theta}\right\Vert _{1}.
\]
In general we will use the following notation involving quantum experiments
$\mathcal{E}$ and $\mathcal{F}$.%

\begin{tabular}
[t]{lllll}%
$\mathcal{E}$ & $\mathbf{\preceq}$ & $\mathcal{F}$ & ($\mathcal{F}$ more
informative than $\mathcal{E}$): & $\delta\left(  \mathcal{F},\mathcal{E}%
\right)  =0$\\
$\mathcal{E}$ & $\mathcal{\sim}$ & $\mathcal{F}$ & (equivalent): &
$\Delta\left(  \mathcal{F},\mathcal{E}\right)  =0$\\
$\mathcal{E}_{n}$ & $\simeq$ & $\mathcal{F}_{n}$ & (asymptotically norm
equivalent): & $\Delta_{0}\left(  \mathcal{F}_{n},\mathcal{E}_{n}\right)
\rightarrow0$\\
$\mathcal{E}_{n}$ & $\precsim$ & $\mathcal{F}_{n}$ & ($\mathcal{F}_{n}$
asymptotically more informative than $\mathcal{E}_{n}$): & $\delta\left(
\mathcal{F}_{n},\mathcal{E}_{n}\right)  \rightarrow0$\\
$\mathcal{E}_{n}$ & $\approx$ & $\mathcal{F}_{n}$ & (asymptotically
equivalent): & $\Delta\left(  \mathcal{F}_{n},\mathcal{E}_{n}\right)
\rightarrow0.$%
\end{tabular}

Note that "more informative" above is used in the sense of a semi-ordering,
i.e.%
\
its actual meaning is "at least as informative". If $\mathcal{E},\mathcal{F}$
are classical experiments, where the predual norm distance is a multiple of
the total variation distance between probability measures, the relation
$\mathcal{E}_{n}\simeq\mathcal{F}_{n}\mathcal{\ }$will also be described as
asymptotic equivalence in total variation.

\section{Upper informativity bound}

\subsection{Gaussian states on symmetric Fock space
\label{subsec:symmetric-Fock-space}}

Let $\mathcal{H}$ be a complex separable Hilbert space. Let $\vee
^{m}\mathcal{H}$ denote the $m$-fold symmetric tensor power, that is, the
subspace of $\mathcal{H}^{\otimes m}$ consisting of vectors which are
symmetric under permutations of the tensors, with $\vee^{0}\mathcal{H}%
:=\mathbb{C}$. The Fock space over $\mathcal{H}$ is the Hilbert space
\[
\mathfrak{F}\left(  \mathcal{H}\right)  :=%
{\displaystyle\bigoplus_{m\geq0}}
\vee^{m}\mathcal{H}.
\]
For each $x\in\mathcal{H}$ let
\begin{equation}
x_{F}:=%
{\displaystyle\bigoplus_{m\geq0}}
\frac{1}{\sqrt{m!}}x^{\otimes m} \label{exponential-vectors}%
\end{equation}
denote the corresponding exponential vector (or Fock vector). The exponential
vectors are linearly independent and their linear span is dense in
$\mathfrak{F}\left(  \mathcal{H}\right)  $. The Weyl unitaries $V\left(
x\right)  $, $x\in\mathcal{H}$ are defined by their action on exponential
vectors as
\begin{equation}
V\left(  x\right)  y_{F}:=\left(  y+2^{-1/2}x\right)  _{F}\exp\left(
-\frac{1}{4}\left\Vert x\right\Vert ^{2}-2^{-1/2}\left\langle x,y\right\rangle
\right)  \text{, }y\in\mathcal{H}. \label{Weyl-unitaries-complex-indexed-3}%
\end{equation}
These can be seen to satisfy the relation
\begin{equation}
V\left(  x\right)  V\left(  y\right)  =V\left(  x+y\right)  \exp\left(
-\frac{i}{2}\operatorname{Im}\left\langle x,y\right\rangle \right)
\label{Weyl-unitaries-complex-indexed-2}%
\end{equation}
and for $\mathcal{H}=\mathbb{C}^{n}$ this coincides with the CCR
(\ref{Weyl-unitaries-complex-indexed}) stated in the Schr\"{o}dinger
representation. Denote by $\left\{  V_{j}\left(  x\right)  ,x\in\mathbb{C}%
^{n}\right\}  $, $j=1,2$ these two versions of the Weyl unitaries ($i=2$
corresponding to (\ref{Weyl-unitaries-complex-indexed-3}) ); since both sets
of operators are irreducible, there is a linear isometric map $U$
$:L_{2}\left(  \mathbb{R}^{n}\right)  \mapsto\mathfrak{F}\left(
\mathbb{C}^{n}\right)  $ such that
\[
V_{1}\left(  x\right)  =U^{\ast}V_{2}\left(  x\right)  U\text{, }%
x\in\mathbb{C}^{n}\text{.}%
\]
The corresponding generated C*-algebras are hence *-isomorphic and are denoted
by $CCR\left(  \mathbb{C}^{n}\right)  $; henceforth in this section we will
work with the Fock representation $V\left(  x\right)  =V_{2}\left(  x\right)
$ of (\ref{Weyl-unitaries-complex-indexed-3}). A state $\varphi$ on
$CCR\left(  \mathcal{H}\right)  $ is a positive linear functional
$\varphi:CCR\left(  \mathcal{H}\right)  \mapsto\mathbb{C}$ that takes the
value $1$ on the unit on $CCR\left(  \mathcal{H}\right)  $.

Let $B\in\mathcal{B}\left(  \mathcal{H}\right)  $ be a bounded operator on
$\mathcal{H}$, and let $\vee^{m}B$ be the restriction of $B^{\otimes m}$ to
$\vee^{m}\mathcal{H}$. The Fock operator $B_{F}$ corresponding to $B$ is
\[
B_{F}:=%
{\displaystyle\bigoplus_{m\geq0}}
\vee^{m}B
\]
with an appropriate domain $\mathcal{D}\left(  B_{F}\right)  $ (cf.
\cite{MR2510896}, Appendix for more details). Then $\left(  Bx\right)
_{F}=B_{F}x_{F}$ holds for exponential vectors $x_{F}$, and for $A\in
\mathcal{B}\left(  \mathcal{H}\right)  $, the relation
\begin{equation}
A_{F}B_{F}=\left(  AB\right)  _{F} \label{multiplying-Fock-ops}%
\end{equation}
holds on a dense subset of $\mathfrak{F}\left(  \mathcal{H}\right)  $. Then,
for a gauge invariant centered Gaussian state with symbol matrix $A$, the
density operator on $\mathfrak{F}\left(  \mathbb{C}^{n}\right)  $ can be
described as follows (cp. \cite{MR2510896}, A5):%
\begin{equation}
\mathfrak{N}_{n}\left(  0,A\right)  =\frac{2^{n}}{\det\left(  A+I\right)
}\left(  \frac{A-I}{A+I}\right)  _{F}. \label{Fock-repre-1}%
\end{equation}
A proof is given in subsection \ref{subsec:density-op-GIV-Gaussians} below.%

\begin{privatenotes}
\begin{boxedminipage}{\textwidth}%

\begin{sfblock}
\texttt{Specialize the above formula down to }$n=1$\texttt{ and obtain thermal
states. Use some text from variant v45:}

The background for the connection to geometric laws is the following. Note
that if $\check{\Sigma}$ in (\ref{cov-matrix-gaug-inv}) is $2\times2$, i.e. is
the covariance matrix of a one-mode Gaussian state $\mathbb{N}_{2}\left(
0,\check{\Sigma}\right)  $ then by symmetry of $\check{\Sigma}$ one has
$\operatorname{Im}A_{1}=0$ and thus $\check{\Sigma}=aI_{2}/2$ where
$a=\operatorname{Re}A_{1}\geq1$. The state $\mathbb{N}_{2}\left(
0,aI_{2}/2\right)  $ is well known to be the thermal state of the harmonic
oscillator: the density operator is
\begin{equation}
\mathbb{N}_{2}\left(  0,\frac{a}{2}I_{2}\right)  =\left(  1-p\right)
\sum_{k=0}^{\infty}p^{k}\left\vert k\right\rangle \left\langle k\right\vert
\text{, }\mathbb{\;\;\;}p=\left(  a-1\right)  /\left(  a+1\right)
\end{equation}
where $\left\vert k\right\rangle $, $k=0,1,\ldots$is the Fock ONB of
$L_{2}\left(  \mathbb{R}\right)  $, and $p=0$ if $a=1$. So if the spectral
density $a\left(  \omega\right)  $ of the gauge invariant time series is
constant: $a\left(  \omega\right)  =a$ then one measures $n$ independent gauge
invariant, zero mean Gaussian states, i.e. $n$ copies of the thermal state
(\ref{thermal-state}). The thermal states for different $p\geq0$ all commute,
so the problems of estimating $p$ becomes classical, and is equivalent to
observing $n$ i.i.d. observation $X_{1},\ldots,X_{n}$ with distribution
$\mathrm{Geo}\left[  p\right]  $..
\end{sfblock}

%

\end{boxedminipage}
\end{privatenotes}%

\subsection{Distance of states in terms of symbols}

Our model is the quantum statistical experiment $\mathcal{E}_{n}\left(
\Theta_{1}\left(  \alpha,M\right)  \right)  $ described in Theorem
\ref{theor-main-1}. To characterize the parameter space $\Theta_{1}\left(
\alpha,M\right)  $ for $\alpha>1/2$, define for any real valued $a\in
L_{2}(-\pi,\pi)$ and its Fourier coefficients (\ref{symbol-generate})
\begin{equation}
\left\vert a\right\vert _{2,\alpha}^{2}:=\sum_{k=-\infty}^{\infty}\left\vert
k\right\vert ^{2\alpha}\left\vert a_{k}\right\vert ^{2},\;\left\Vert
a\right\Vert _{2,\alpha}^{2}:=a_{0}^{2}+\left\vert a\right\vert _{2,\alpha
}^{2} \label{normdef}%
\end{equation}
provided the r.h.s. is finite. The set of real functions
\begin{equation}
W^{\alpha}(M)=\left\{  a\in L_{2}(-\pi,\pi):\left\Vert a\right\Vert
_{2,\alpha}^{2}\leq M\right\}  . \label{sobol-ball}%
\end{equation}
then describes a ball in the scale of periodic fractional Sobolev spaces with
smoothness coefficient $\alpha$. Note that for $\alpha>1/2$, by an embedding
theorem (\cite{GolNussbZ-specpaper-preprint}, Lemma 5.6) , functions in
$W^{\alpha}(M)$ are also uniformly bounded. For $M>0$, define a set of real
valued functions on $[-\pi,\pi]$
\begin{equation}
\mathcal{L}_{M}=\left\{  a:M^{-1}\leq a(\omega)-1,\omega\in\lbrack-\pi
,\pi]\right\}  . \label{FM-def}%
\end{equation}
Then for the parameter space $\Theta_{1}\left(  \alpha,M\right)  $ of Theorem
\ref{theor-main-1} we have
\begin{equation}
\Theta_{1}\left(  \alpha,M\right)  =W^{\alpha}(M)\cap\mathcal{L}_{M}.
\label{Theta-1-functionset-def-a}%
\end{equation}
Therefore we can assume there exists $C=C_{M,\alpha}>0$ such that
\begin{equation}
1+C^{-1}\leq a\left(  \omega\right)  \leq C,\text{ }\omega\in\left[  -\pi
,\pi\right]  \label{cond-boundedness-specdensity}%
\end{equation}
holds for all $a\in\Theta_{1}\left(  \alpha,M\right)  $. Introducing notation
\begin{equation}
Q:=\left(  A-I\right)  /2\text{, }R:=\frac{Q}{Q+I} \label{Q-R-def}%
\end{equation}
we obtain from (\ref{Fock-repre-1})
\begin{align}
\mathfrak{N}_{n}\left(  0,A\right)   &  =\frac{1}{\det\left(  I+Q\right)
}\left(  \frac{Q}{I+Q}\right)  _{F}\label{state-repre-1}\\
&  =\frac{1}{\det\left(  I+Q\right)  }R_{F}. \label{state-repre-2}%
\end{align}
In the sequel we will approximate a state $\mathfrak{N}_{n}\left(
0,A_{1}\right)  $, given by symbol $A_{1}$ by the corresponding state for a
symbol $A_{2}$. Specifically, $A_{1}$ will be taken as the Hermitian Toeplitz
matrix $A_{n}\left(  a\right)  $ and $A_{2}$ will be a (truncated) circulant
matrix. We assume that $A_{i}$, $i=1,2$ are Hermitian $n\times n$ such that
there exists $c>0,$ independent of $n$, such that
\[
\lambda_{\min}\left(  A_{i}-I\right)  \geq c.
\]
This assumption will be justified later for the cases at hand, on the basis of
(\ref{cond-boundedness-specdensity}). In the Fock representation
(\ref{Fock-repre-1}) it then follows from Lemma
\ref{Lem-spec-decompos-Fock-op} that
\[
\lambda_{\min}\left(  \frac{A_{i}-I}{I+A_{i}}\right)  _{F}>0
\]
(cp. (\ref{symm-basis}) below), hence the states $\mathfrak{N}_{n}\left(
0,A_{i}\right)  $ are faithful.

We begin with a bound for the trace norm in terms of relative entropy. The
trace norm between states $\rho,\sigma$ is defined as
\[
\left\Vert \rho-\sigma\right\Vert _{1}:=\mathrm{Tr}\;\left\vert \rho
-\sigma\right\vert .
\]
For finite dimensional states $\rho$ and $\sigma$, the relative entropy is
\begin{equation}
S\left(  \rho||\sigma\right)  =\left\{
\begin{tabular}
[c]{l}%
$\mathrm{Tr}\;\rho\left(  \log\rho-\log\sigma\right)  $ if\textrm{
supp}$\;\sigma\supseteq\;$\textrm{supp}$\;\rho$\\
$\infty$ otherwise.
\end{tabular}
\ \ \ \ \ \right.  \label{rel-entropy-def}%
\end{equation}
This formula extends to faithful Gaussian states with density operators
$\rho,$ $\sigma$, (\ref{rel-entropy-def}), in the sense of agreeing with the
definition of relative entropy for normal states on a von Neumann algebra
(\cite{MR2363070}, sec 3.4). As we argued above, both our states $\rho,\sigma$
are assumed faithful, so \textrm{supp}$\;\sigma\supseteq\;$\textrm{supp}%
$\;\rho$ holds and $K\left(  \rho,\sigma\right)  $ can be computed from the
first line of (\ref{rel-entropy-def}). Then a quantum analog of Pinsker's
inequality holds (\cite{MR1230389}, Theorem 5.5): for the trace norm distance
between $\rho$ and $\sigma$ one has
\begin{equation}
\left\Vert \rho-\sigma\right\Vert _{1}^{2}\leq2S\left(  \rho||\sigma\right)  .
\label{inequality-trace-norm-relative-entropy}%
\end{equation}
Consider symbols $A_{j}$, $j=1,2$ and let $\rho_{j}=\mathfrak{N}_{n}\left(
0,A_{j}\right)  $, $j=1,2$ be the corresponding Gaussian states. Our purpose
in this section is to obtain an upper bound on the trace norm distance in
terms of the symbols, by using (\ref{inequality-trace-norm-relative-entropy})
and an appropriate upper bound on $S\left(  \rho||\sigma\right)  $.

For general Gaussian states, explicit expression for $S\left(  \rho
||\sigma\right)  $ in terms of the first two moments have been obtained
(\cite{pirandola2017fundamental} and references therein). Below we give a
special formula which focuses on the zero mean gauge invariant case, and
writes out $S\left(  \rho||\sigma\right)  $ directly in terms of the symbols
rather than the covariance matrices.

Consider the relative entropy between two Bernoulli laws $\left(
1-p_{j},p_{j}\right)  $ with $p_{j}\in\left(  0,1\right)  $, $j=1,2$:%
\[
S_{2}\left(  p_{1}||p_{2}\right)  =p_{1}\log\frac{p_{1}}{p_{2}}+\left(
1-p_{1}\right)  \log\frac{1-p_{1}}{1-p_{2}}.
\]
An analog for $n\times n$ Hermitian $R_{j}$ satisfying $0<R_{i}<I$ is
\begin{equation}
S_{2}\left(  R_{1}||R_{2}\right)  :=R_{1}\left(  \log R_{1}-\log R_{2}\right)
+\left(  I-R_{1}\right)  \left(  \log\left(  I-R_{1}\right)  -\log\left(
I-R_{2}\right)  \right)  . \label{S2-def-n}%
\end{equation}

\begin{proposition}
\label{prop-rel-entrop-form}Let $A_{j}$, $i=1,2$ be Hermitian $n\times n$ such
that $\lambda_{\min}\left(  A_{j}-I\right)  >0$, and let
\[
\rho_{j}=\mathfrak{N}_{n}\left(  0,A_{j}\right)  =\frac{2^{n}}{\det\left(
I+A_{j}\right)  }\left(  \frac{A_{j}-I}{A_{j}+I}\right)  _{F}.
\]
be the corresponding Gaussian states. Let $Q_{j}$ and $R_{j}$ be defined by
\[
Q_{j}:=\left(  A_{j}-I\right)  /2\text{, }R_{j}:=\frac{Q_{j}}{Q_{j}+I}%
=\frac{A_{j}-I}{A_{j}+I}\text{, }j=1,2.
\]
Then for the relative entropy one has
\begin{equation}
S\left(  \rho_{1}||\rho_{2}\right)  =\mathrm{Tr}\;\left(  I+Q_{1}\right)
S_{2}\left(  R_{1}||R_{2}\right)  \label{rel-entrop-formula-1}%
\end{equation}
where $S_{2}\left(  \cdot||\cdot\right)  $ is defined by (\ref{S2-def-n}).
\end{proposition}

\begin{proof}
Assume a Gaussian state is given by $\rho=\frac{1}{\det\left(  I+Q\right)
}R_{F}$ according to (\ref{state-repre-2}). Then
\begin{align}
\log\rho &  =-\log\det\left(  I+Q\right)  I_{F}+\log R_{F}\nonumber\\
&  =-\log\det\left(  I+Q\right)  I_{F}+\oplus_{m=0}^{\infty}\log\vee^{m}R.
\label{log-rho-express}%
\end{align}
Using Lemma \ref{Lem-log-Gamma-m} we find%
\begin{align*}
\log\rho &  =-\log\det\left(  I+Q\right)  I_{F}+\oplus_{m=0}^{\infty}%
\Gamma_{m}\left(  \log R\right) \\
&  =-\log\det\left(  I+Q\right)  I_{F}+\Gamma\left(  \log R\right)
\end{align*}
with the definition of $\Gamma\left(  \log R\right)  $ given in Lemma
\ref{Lem-pre-log-formula}. Setting $\rho=\rho_{2}$ and applying this lemma for
the case $A=R_{1}$, $B=\log R_{2}$, we obtain%
\begin{align}
\mathrm{Tr}\;\rho_{1}\log\rho_{2}  &  =\frac{1}{\det\left(  I+Q_{1}\right)
}\mathrm{Tr}\;\left(  R_{1}\right)  _{F}\left(  -\log\det\left(
I+Q_{2}\right)  I_{F}+\Gamma\left(  \log R_{2}\right)  \right) \nonumber\\
&  =-\log\det\left(  I+Q_{2}\right)  +\frac{1}{\det\left(  I+Q_{1}\right)
}\mathrm{Tr}\;\left(  R_{1}\right)  _{F}\Gamma\left(  \log R_{2}\right)
\nonumber\\
&  =-\log\det\left(  I+Q_{2}\right)  +\frac{1}{\det\left(  I+Q_{1}\right)
}\frac{1}{\det\left(  I-R_{1}\right)  }\mathrm{Tr}\;\frac{R_{1}}{I-R_{1}}\log
R_{2}. \label{comput-trace-1}%
\end{align}
In view of (\ref{Q-R-def}) we have
\begin{align*}
I-R_{1}  &  =I-\frac{Q_{1}}{I+Q_{1}}=\frac{I}{I+Q_{1}},\\
\det\left(  I-R_{1}\right)   &  =1/\det\left(  I+Q_{1}\right)  ,\\
-\log\det\left(  I+Q_{2}\right)   &  =\log\det\left(  I-R_{2}\right)
=\mathrm{Tr}\;\log\left(  I-R_{2}\right)  .
\end{align*}
Applied to (\ref{comput-trace-1}) this implies%
\begin{align}
\mathrm{Tr}\;\rho_{1}\log\rho_{2}  &  =\mathrm{Tr}\;\log\left(  I-R_{2}%
\right)  +\mathrm{Tr}\;\frac{R_{1}}{I-R_{1}}\log R_{2}\label{compo}\\
&  =\mathrm{Tr}\;\log\left(  I-R_{2}\right)  +\mathrm{Tr}\left(
I+Q_{1}\right)  \;R_{1}\log R_{2}\nonumber\\
&  =\mathrm{Tr}\;\left(  I+Q_{1}\right)  \left[  \left(  I-R_{1}\right)
\log\left(  I-R_{2}\right)  +\;R_{1}\log R_{2}\right]  .
\label{comput-trace-2}%
\end{align}
For the case $\rho_{1}=\rho_{2}$ we obtain
\begin{equation}
\mathrm{Tr}\;\rho_{1}\log\rho_{1}=\mathrm{Tr}\;\left(  I+Q_{1}\right)  \left[
\left(  I-R_{1}\right)  \log\left(  I-R_{1}\right)  +R_{1}\log R_{1}\right]
\label{comput-trace-3}%
\end{equation}
From (\ref{rel-entropy-def}), (\ref{comput-trace-2}) and (\ref{comput-trace-3}%
) we finally obtain
\[
S\left(  \rho_{1}||\rho_{2}\right)  =\mathrm{Tr}\;\rho_{1}\left(  \log\rho
_{1}-\log\rho_{2}\right)
\]%
\begin{align*}
&  =\mathrm{Tr}\;\left(  I+Q_{1}\right)  \left[  R_{1}\left(  \log R_{1}-\log
R_{2}\right)  +\left(  I-R_{1}\right)  \left(  \log\left(  I-R_{1}\right)
-\log\left(  I-R_{2}\right)  \right)  \right] \\
&  =\mathrm{Tr}\;\left(  I+Q_{1}\right)  S_{2}\left(  R_{1}||R_{2}\right)  .
\end{align*}

\end{proof}

\bigskip

Since $Q_{i}$ are positive definite $n\times n$ Hermitian, the matrices
$R_{i}$ and $I-R_{i}=I/\left(  Q_{i}+I\right)  $ also have these properties;
in particular
\begin{equation}
0<R_{i}<I\text{, }i=1,2. \label{R-bracket}%
\end{equation}
Hence $S_{2}\left(  R_{1}||R_{2}\right)  $ defined by
(\ref{rel-entrop-formula-1}) is finite, and thus $S\left(  \rho_{1}||\rho
_{2}\right)  $ is also finite. In order to achieve uniformity of estimates
over the $R_{i}$ considered, we assume a strengthened version of
(\ref{R-bracket}): there exists $\lambda\in\left(  1/2,1\right)  $ such that
\begin{equation}
\left(  1-\lambda\right)  I<R_{i}<\lambda I\text{, }\;i=1,2.
\label{R-bracket-2}%
\end{equation}
It is immediate, in view of (\ref{Q-R-def}), that this condition is equivalent
to each of the following two:
\begin{align}
\frac{1-\lambda}{\lambda}I  &  <Q_{i}<\frac{\lambda}{1-\lambda}%
I,\label{R-bracket-3}\\
\left(  \frac{2}{\lambda}-1\right)  I  &  <A_{i}<\frac{1+\lambda}{1-\lambda}I.
\label{R-bracket-4}%
\end{align}
Also, (\ref{R-bracket-3}) implies
\begin{equation}
I+Q_{i}<\frac{1}{1-\lambda}I\text{, }\;i=1,2. \label{R-bracket-5}%
\end{equation}
Our next task is to estimate (\ref{rel-entrop-formula-1}) in terms of the
difference $H=R_{1}-R_{2}$. To that end we use an expansion
\[
\log R_{2}=\log\left(  I-\left(  I-R_{2}\right)  \right)  =-\sum_{k=1}%
^{\infty}\frac{1}{k}\left(  I-R_{2}\right)  ^{k}.
\]
That is valid if $I-R_{2}$ has all eigenvalues contained in $\left(
-1,1\right)  $, which holds due to (\ref{R-bracket}). Similarly we expand
$\log R_{1}$ and obtain
\begin{align}
\log R_{1}-\log R_{2}  &  =\sum_{k=1}^{\infty}\frac{1}{k}\left[  \left(
I-R_{2}\right)  ^{k}-\left(  I-R_{1}\right)  ^{k}\right] \label{rearrange}\\
&  =\sum_{k=1}^{\infty}\frac{1}{k}\left[  \left(  I-R_{1}+H\right)
^{k}-\left(  I-R_{1}\right)  ^{k}\right]  . \label{rearrange1}%
\end{align}
Furthermore we obtain%
\[
\left(  I-R_{1}+H\right)  ^{k}-\left(  I-R_{1}\right)  ^{k}=
\]%
\begin{align}
&  =H\left(  I-R_{1}\right)  ^{k-1}+\left(  I-R_{1}\right)  H\left(
I-R_{1}\right)  ^{k-2}+\ldots+\left(  I-R_{1}\right)  ^{k-1}H\label{expa-1}\\
&  +H^{2}\left(  I-R_{1}\right)  ^{k-2}+H\left(  I-R_{1}\right)  H\left(
I-R_{1}\right)  ^{k-3}+\ldots+\left(  I-R_{1}\right)  ^{k-2}H^{2}\nonumber\\
&  \ldots\nonumber\\
&  +H^{k-1}\left(  I-R_{1}\right)  +H^{k-2}\left(  I-R_{1}\right)
H+\ldots+\left(  I-R_{1}\right)  H^{k-1}\nonumber\\
&  +H^{k}.\nonumber
\end{align}
A similar expansion holds for the log terms in the second summand of
(\ref{S2-def-n}): writing $G=-H$, we have $R_{2}=R_{1}+G$ and
\begin{equation}
\log\left(  I-R_{1}\right)  -\log\left(  I-R_{2}\right)  =\sum_{k=1}^{\infty
}\frac{1}{k}\left[  \left(  R_{1}+G\right)  ^{k}-R_{1}^{k}\right]  ,
\label{expa-1a}%
\end{equation}%
\[
\left(  R_{1}+G\right)  ^{k}-R_{1}^{k}=
\]%
\begin{align}
&  GR_{1}^{k-1}+R_{1}GR_{1}^{k-2}+\ldots+R_{1}^{k-1}G\label{expa-2}\\
&  +G^{2}R_{1}^{k-2}+GR_{1}GR_{1}^{k-3}+\ldots+R_{1}^{k-2}G^{2}\nonumber\\
&  \ldots\nonumber\\
&  +G^{k-1}R_{1}+G^{k-2}R_{1}G+\ldots+R_{1}G^{k-1}\nonumber\\
&  +G^{k}.\nonumber
\end{align}
We also denote
\begin{align*}
M_{1}  &  =\left(  I+Q_{1}\right)  R_{1},\\
M_{2}  &  =\left(  I+Q_{1}\right)  \left(  I-R_{1}\right)  .
\end{align*}
Furthermore, for matrices $A$ we write the operator norm $\left\vert
A\right\vert =\lambda_{\max}^{1/2}\left(  A^{\ast}A\right)  $, so that for
Hermitian positive $A$ we have $\left\vert A\right\vert =\lambda_{\max}\left(
A\right)  $. The Hilbert-Schmidt norm is written $\left\Vert A\right\Vert
_{2}=\left(  \mathrm{Tr}\;A^{\ast}A\right)  ^{1/2}$. We then have
\begin{align}
\left\Vert AB\right\Vert _{2}  &  \leq\left\vert A\right\vert \left\Vert
B\right\Vert _{2},\label{norm-inequality}\\
\left\vert AB\right\vert  &  \leq\left\vert A\right\vert \left\vert
B\right\vert . \label{norm-inequality-operator}%
\end{align}

Consider now the series expression for $S\left(  \rho_{1}||\rho_{2}\right)  $
given by (\ref{rel-entrop-formula-1}) and the expansions (\ref{expa-1}),
(\ref{expa-2}), i.e. the series obtained for $\mathrm{Tr}\;\left(
I+Q_{1}\right)  \;S_{2}\left(  R_{1}||R_{2}\right)  $. Consider first the
question whether it converges absolutely.

To that end we denote the generic term in the expansion (\ref{expa-1}) by
$T_{k,j,l}$, in such a way that

\begin{itemize}
\item $k$ is as indicated, i.e it pertains to a term in the expansion of
$\left(  I-R_{1}+H\right)  ^{k}-\left(  I-R_{1}\right)  ^{k}$, where $1\leq
k<\infty$

\item $j$ is the order in $H$, i.e. the total number of factors $H$ (such that
$k-j$ is the total number of factors $I-R_{1}$), and $1\leq j\leq k$

\item $l$ indicates the $l$-th summand in a given line of (\ref{expa-1}), for
any chosen systematic order of the summands pertaining to given $k,j$, where
$1\leq l\leq\left(
\genfrac{}{}{0pt}{}{k}{j}%
\right)  $.
\end{itemize}

In a similar way, we denote the generic term in the expansion (\ref{expa-2})
by $U_{k,j,l}$, in such a way that

\begin{itemize}
\item $k$ is as indicated, i.e it pertains to a term in the expansion of
$\left(  R_{1}+G\right)  ^{k}-R_{1}^{k}$, where $1\leq k<\infty$

\item $j$ is the order in $G$, i.e. the total number of factors $G$ (such that
$k-j$ is the total number of factors $R_{1}$), and $1\leq j\leq k$

\item $l$ indicates the $l$-th summand in a given line of (\ref{expa-1}), for
any chosen systematic order of the summands pertaining to given $k,j$, where
$1\leq l\leq\left(
\genfrac{}{}{0pt}{}{k}{j}%
\right)  $.
\end{itemize}

\begin{lemma}
For $\left\Vert H\right\Vert _{2}<1-\lambda$, with $\lambda$ from
(\ref{R-bracket-2}), the series
\[
\mathrm{Tr}\;\left(  I+Q_{1}\right)  \;S_{2}\left(  R_{1}||R_{2}\right)  =
\]%
\begin{equation}
=\sum_{k=1}^{\infty}\sum_{j=1}^{k}\sum_{l=1}^{\left(
\genfrac{}{}{0pt}{}{k}{j}%
\right)  }\frac{1}{k}\mathrm{Tr}\;M_{1}\;T_{k,j,l}+\sum_{k=1}^{\infty}%
\sum_{j=1}^{k}\sum_{l=1}^{\left(
\genfrac{}{}{0pt}{}{k}{j}%
\right)  }\frac{1}{k}\mathrm{Tr}\;M_{2}\;U_{k,j,l} \label{two-series}%
\end{equation}
converges absolutely.
\end{lemma}

\begin{proof}
Consider the first series and all terms with $j=1$. Since $\left(
I+Q_{1}\right)  $ and $R_{1}$ are commuting and positive, we have by
Cauchy-Schwartz, for $1\leq l\leq k$
\begin{align*}
\left\vert \mathrm{Tr}\;M_{1}\;T_{k,1,l}\right\vert  &  =\left\vert
\mathrm{Tr}\;\left(  I+Q_{1}\right)  R_{1}\left(  I-R_{1}\right)
^{k-1}H\right\vert \\
&  \leq\left\Vert \left(  I+Q_{1}\right)  R_{1}\left(  I-R_{1}\right)
^{k-1}\right\Vert _{2}\left\Vert H\right\Vert _{2}\\
&  \leq\lambda^{k}\left\Vert \left(  I+Q_{1}\right)  \right\Vert
_{2}\left\Vert H\right\Vert _{2}%
\end{align*}
where we used (\ref{norm-inequality}) and $\left\vert R_{1}\right\vert
<\lambda$, $\left\vert I-R_{1}\right\vert <\lambda$ due to (\ref{R-bracket-2}%
). Consequently%
\begin{align}
\sum_{k=1}^{\infty}\sum_{l=1}^{k}\frac{1}{k}\left\vert \mathrm{Tr}%
\;M_{1}\;T_{k,1,l}\right\vert  &  \leq\sum_{k=1}^{\infty}\lambda^{k}\left\Vert
\left(  I+Q_{1}\right)  \right\Vert _{2}\left\Vert H\right\Vert _{2}%
\nonumber\\
&  =\frac{\lambda}{1-\lambda}\left\Vert \left(  I+Q_{1}\right)  \right\Vert
_{2}\left\Vert H\right\Vert _{2}<\infty. \label{abs-conv-linear}%
\end{align}
Next consider all quadratic terms in $H$, i.e. the case $j=2$. The general
form of such a term, with $k\geq2$, is
\[
\mathrm{Tr}\;M_{1}\;T_{k,2,l}=\mathrm{Tr}\;\left(  I+Q_{1}\right)
R_{1}\left(  I-R_{1}\right)  ^{a}H\left(  I-R_{1}\right)  ^{b}H
\]
where $a$ and $b$ depend on $k$ and $l$, with $a+b=k-2$, $a,b\geq0$. By
Cauchy-Schwartz we obtain%
\[
\left\vert \mathrm{Tr}\;M_{1}\;T_{k,2,l}\right\vert \leq
\]%
\begin{equation}
\left\Vert \left(  I+Q_{1}\right)  R_{1}\left(  I-R_{1}\right)  ^{a}%
H\right\Vert _{2}\cdot\left\Vert \left(  I-R_{1}\right)  ^{b}H\right\Vert
_{2}. \label{factors}%
\end{equation}
Setting $\beta:=1/\left(  1-\lambda\right)  $ and using the bound
(\ref{R-bracket-5}), the first factor above can be upper bounded as
$\beta\lambda^{a+1}\left\Vert H\right\Vert _{2}$. Similarly the second factor
in (\ref{factors}) can be bounded by $\lambda^{b}\left\Vert H\right\Vert _{2}%
$. As a result we get
\begin{equation}
\left\vert \mathrm{Tr}\;M_{1}\;T_{k,2,l}\right\vert \leq\beta\lambda
^{k-1}\left\Vert H\right\Vert _{2}^{2}. \label{quad-term-bound}%
\end{equation}
Thus for the totality of second order terms we have
\begin{align}
\sum_{k=2}^{\infty}\sum_{l=1}^{\left(
\genfrac{}{}{0pt}{}{k}{2}%
\right)  }\frac{1}{k}\left\vert \mathrm{Tr}\;M_{1}\;T_{k,2,l}\right\vert  &
\leq\sum_{k=2}^{\infty}\frac{1}{k}\left(
\genfrac{}{}{0pt}{}{k}{2}%
\right)  \beta\lambda^{k-1}\left\Vert H\right\Vert _{2}^{2}\nonumber\\
&  =\beta\left\Vert H\right\Vert _{2}^{2}\sum_{k=2}^{\infty}\frac{k-1}%
{2}\lambda^{k-1}=\frac{\beta}{2}\left\Vert H\right\Vert _{2}^{2}\frac{\lambda
}{\left(  1-\lambda\right)  ^{2}} \label{quad-term-bound-2}%
\end{align}
using relation (\ref{quasi-geom-series}). Next consider all terms with
$j\geq3$, i.e. with higher order than $2$ in $H$. The general form of such a
term, with $k\geq j$, is
\[
\mathrm{Tr}\;M_{1}\;T_{k,j,l}=\mathrm{Tr}\;\left(  I+Q_{1}\right)
R_{1}\left(  I-R_{1}\right)  ^{a}H\Pi_{k,j,l}H\;
\]
where $a$ depends on $k$ and $l$, with $a\leq k-j$, $a\geq0$ and $\Pi_{k,j,l}$
is a matrix monomial containing $b$ factors $I-R_{1}$ and $j-2$ factors $H$
(recall that $I-R_{1}$ and $H$ do not commute). Here $b$ depends on $k$ and
$l$ and fulfills $a+b=k-j$ with $a,b\geq0$. Again we estimate, analogously to
(\ref{factors}),
\begin{align*}
\left\vert \mathrm{Tr}\;M_{1}\;T_{k,j,l}\right\vert  &  \leq\beta\lambda
^{a+1}\left\Vert H\right\Vert _{2}\left\Vert \Pi_{k,j,l}H\right\Vert _{2}\\
&  \leq\beta\lambda^{a+1}\left\vert \Pi_{k,j,l}\right\vert \left\Vert
H\right\Vert _{2}^{2}.
\end{align*}
Successive application of the inequality (\ref{norm-inequality-operator})
gives
\begin{equation}
\left\vert \Pi_{k,j,l}\right\vert \leq\lambda^{b}\left\vert H\right\vert
^{j-2}. \label{obvious-lambda-max}%
\end{equation}
As an illustration consider the simple case $\Pi_{k,j,l}=\left(
I-R_{1}\right)  H\left(  I-R_{1}\right)  $ where $k=6,j=3,a=1,b=2$ . Then
\begin{align*}
\left\vert \Pi_{k,j,l}\right\vert  &  =\left\vert \left(  I-R_{1}\right)
H\left(  I-R_{1}\right)  \right\vert \leq\lambda\left\vert H\left(
I-R_{1}\right)  \right\vert \\
&  \leq\lambda\left\vert H\right\vert \left\vert \left(  I-R_{1}\right)
\right\vert \leq\lambda^{2}\left\vert H\right\vert .
\end{align*}
Since (\ref{obvious-lambda-max}) holds generally, applying the bound
$\left\vert H\right\vert \leq\left\Vert H\right\Vert _{2}$ we obtain for
$j\geq3$
\begin{equation}
\left\vert \mathrm{Tr}\;M_{1}\;T_{k,j,l}\right\vert \leq\beta\lambda
^{k-j+1}\left\Vert H\right\Vert _{2}^{j}. \label{higher-term-bound}%
\end{equation}
From (\ref{higher-term-bound}) we obtain for the totality of terms of third
order or higher
\begin{align}
\sum_{k=3}^{\infty}\sum_{j=3}^{k}\sum_{l=1}^{\left(
\genfrac{}{}{0pt}{}{k}{j}%
\right)  }\frac{1}{k}\left\vert \mathrm{Tr}\;M_{1}\;T_{k,j,l}\right\vert  &
\leq\beta\sum_{k=3}^{\infty}\sum_{j=3}^{k}\frac{1}{k}\left(
\genfrac{}{}{0pt}{}{k}{j}%
\right)  \lambda^{k-j+1}\left\Vert H\right\Vert _{2}^{j}%
\label{abs-conv-final-0}\\
&  \leq\beta\left\Vert H\right\Vert _{2}^{2}\left\{  \sum_{k=3}^{\infty}%
\sum_{j=3}^{k}\frac{1}{k}\left(
\genfrac{}{}{0pt}{}{k}{j}%
\right)  \lambda^{k-j+1}\left\Vert H\right\Vert _{2}^{j-2}\right\}  .
\end{align}
To show that the expression in $\left\{  \cdot\right\}  $ is finite, set
$h:=k-2$, $m:=j-2$. Then the expression in $\left\{  \cdot\right\}  $ is
\begin{align*}
&  \sum_{h=1}^{\infty}\sum_{m=1}^{h}\frac{\left(  h+1\right)  !}{\left(
m+2\right)  !\left(  h-m\right)  !}\lambda^{h-m+1}\left\Vert H\right\Vert
_{2}^{m}\\
&  =\lambda\sum_{h=1}^{\infty}\sum_{m=1}^{h}\frac{\left(  h+1\right)
}{\left(  m+2\right)  \left(  m+1\right)  }\left(
\genfrac{}{}{0pt}{}{h}{m}%
\right)  \lambda^{h-m}\left\Vert H\right\Vert _{2}^{m}\\
&  \leq\frac{\lambda}{4}\sum_{h=1}^{\infty}\left(  h+1\right)  \left(
\lambda+\left\Vert H\right\Vert _{2}\right)  ^{h}.
\end{align*}
Denote $\gamma=\lambda+\left\Vert H\right\Vert _{2}$ and note that $\gamma<1$
due to $\left\Vert H\right\Vert _{2}<1-\lambda$. Using relation
(\ref{quasi-geom-series}) again, we find
\[
\sum_{h=1}^{\infty}\left(  h+1\right)  \gamma^{h}=\frac{\gamma}{\left(
1-\gamma\right)  ^{2}}+\frac{\gamma}{1-\gamma}\leq\frac{2}{\left(
1-\gamma\right)  ^{2}}.
\]
From (\ref{abs-conv-final-0}) we find for the totality of terms of third order
or higher%
\begin{equation}
\sum_{k=3}^{\infty}\sum_{j=3}^{k}\sum_{l=1}^{\left(
\genfrac{}{}{0pt}{}{k}{j}%
\right)  }\frac{1}{k}\left\vert \mathrm{Tr}\;M_{1}\;T_{k,j,l}\right\vert
\leq\frac{\beta\lambda\left\Vert H\right\Vert _{2}^{2}}{2\left(
1-\lambda-\left\Vert H\right\Vert _{2}\right)  ^{2}}. \label{abs-conv-final-1}%
\end{equation}
The argument for the terms involving $U_{k,j,l}$ is analogous, which proves
the lemma.
\end{proof}

\bigskip

\begin{lemma}
For given $\lambda$ from (\ref{R-bracket-2}) there exists $\delta>0$, not
depending on dimension $n$, such that $\left\Vert H\right\Vert <\delta$
implies
\[
\mathrm{Tr}\;\left(  I+Q_{1}\right)  \;S_{2}\left(  R_{1}||R_{2}\right)
\leq\delta^{-1}\left\Vert H\right\Vert ^{2}.
\]
Here $\delta$ can be chosen as
\[
\delta=\min\left(  \left(  1-\lambda\right)  /2,\left(  1-\lambda\right)
^{3}/8\lambda\right)  .
\]

\end{lemma}

\begin{proof}
In the series (\ref{two-series}) we can now rearrange terms; consider the
series given by all linear (in $H$) terms. This is found as
\begin{align}
&  \sum_{k=1}^{\infty}\sum_{l=1}^{k}\frac{1}{k}\mathrm{Tr}\;M_{1}%
\;T_{k,1,l}+\sum_{k=1}^{\infty}\sum_{l=1}^{k}\frac{1}{k}\mathrm{Tr}%
\;M_{2}\;U_{k,1,l}\nonumber\\
&  =\sum_{k=1}^{\infty}\mathrm{Tr}\;\left(  I+Q_{1}\right)  R_{1}\left(
I-R_{1}\right)  ^{k-1}H+\sum_{k=1}^{\infty}\mathrm{Tr}\;\left(  I+Q_{1}%
\right)  \left(  I-R_{1}\right)  R_{1}^{k-1}G\nonumber\\
&  =\mathrm{Tr}\;H\left(  I+Q_{1}\right)  R_{1}\left(  \sum_{k=1}^{\infty
}\left(  I-R_{1}\right)  ^{k-1}\right)  +\sum_{k=1}^{\infty}\mathrm{Tr}%
\;G\left(  I+Q_{1}\right)  \left(  I-R_{1}\right)  \left(  \sum_{k=1}^{\infty
}R_{1}^{k-1}\right)  . \label{lin-series}%
\end{align}
We note that
\begin{align*}
\sum_{k=1}^{\infty}\left(  I-R_{1}\right)  ^{k-1}  &  =\sum_{k=0}^{\infty
}\left(  I-R_{1}\right)  ^{k}\\
&  =\left(  I-\left(  I-R_{1}\right)  \right)  ^{-1}=R_{1}^{-1},\\
\sum_{k=1}^{\infty}R_{1}^{k-1}  &  =\left(  I-R_{1}\right)  ^{-1}.
\end{align*}
Thus, in view of $G=-H$, (\ref{lin-series}) equals%
\[
\mathrm{Tr}\;H\left(  I+Q_{1}\right)  -\mathrm{Tr}\;H\left(  I+Q_{1}\right)
=0.
\]
In the series (\ref{two-series}) there now remain only terms of quadratic and
higher order in $H$. By (\ref{quad-term-bound-2}), (\ref{abs-conv-final-1})
and the analogous bounds for terms involving $U_{k,j,l}$ with $j\geq2$,
recalling $\left\Vert H\right\Vert =\left\Vert G\right\Vert $, we have%
\[
\mathrm{Tr}\;\left(  I+Q_{1}\right)  \;S_{2}\left(  R_{1}||R_{2}\right)  \leq
\]%
\[
\beta\left\Vert H\right\Vert ^{2}\frac{\lambda}{\left(  1-\lambda\right)
^{2}}+\frac{\beta\lambda\left\Vert H\right\Vert ^{2}}{\left(  1-\lambda
-\left\Vert H\right\Vert \right)  ^{2}}\leq\frac{2\beta\lambda\left\Vert
H\right\Vert ^{2}}{\left(  1-\lambda-\left\Vert H\right\Vert \right)  ^{2}}.
\]
Set $\delta_{0}:=\left(  1-\lambda\right)  /2$; then for $\left\Vert
H\right\Vert <\delta_{0}$%
\[
\mathrm{Tr}\;\left(  I+Q_{1}\right)  \;S_{2}\left(  R_{1}||R_{2}\right)
\leq\frac{8\beta\lambda}{\left(  1-\lambda\right)  ^{2}}\left\Vert
H\right\Vert ^{2}%
\]
Now with $\delta:=\min\left(  \delta_{0},\left(  1-\lambda\right)  ^{2}%
/8\beta\lambda\right)  $ and $\beta=\left(  1-\lambda\right)  ^{-1}$ we obtain
the assertion.
\end{proof}

\bigskip

Having bounded the relative entropy $S\left(  \rho_{1}||\rho_{2}\right)  $
from (\ref{rel-entrop-formula-1}) in terms of the difference $H=R_{1}-R_{2}$,
in the next step we have to estimate $H$ in terms of the difference
$A_{1}-A_{2}$. Recalling \ref{Q-R-def}, we have
\[
H=\frac{Q_{1}}{Q_{1}+I}-\frac{Q_{2}}{Q_{2}+I}.
\]
We will estimate $H$ in terms of $D:=Q_{2}-Q_{1}$, and in view of the relation
$Q:=\left(  A-I\right)  /2$, we have $D=\left(  A_{2}-A_{1}\right)  /2$.

\begin{lemma}
Under condition (\ref{R-bracket-2}) we have
\[
\left\Vert H\right\Vert ^{2}=\left\Vert R_{1}-R_{2}\right\Vert ^{2}\leq
\frac{1}{\left(  1-\lambda\right)  ^{2}}\left\Vert A_{1}-A_{2}\right\Vert
^{2}.
\]

\end{lemma}

\begin{proof}
We have
\begin{align}
\left\Vert H\right\Vert  &  =\left\Vert \frac{Q_{1}}{Q_{1}+I}-\left(
Q_{2}+I\right)  ^{-1}\left(  Q_{1}+D\right)  \right\Vert \nonumber\\
&  \leq\left\Vert \left(  \left(  Q_{1}+I\right)  ^{-1}-\left(  Q_{2}%
+I\right)  ^{-1}\right)  Q_{1}\right\Vert +\left\Vert \left(  Q_{2}+I\right)
^{-1}D\right\Vert \label{Q-decomp}%
\end{align}
The first term in (\ref{Q-decomp}) equals
\begin{align*}
&  \left\Vert \left(  Q_{1}+I\right)  ^{-1}\left(  \left(  Q_{2}+I\right)
-\left(  Q_{1}+I\right)  \right)  \left(  Q_{2}+I\right)  ^{-1}Q_{1}%
\right\Vert \\
&  =\left\Vert \left(  Q_{1}+I\right)  ^{-1}D\left(  Q_{2}+I\right)
^{-1}Q_{1}\right\Vert \\
&  \leq\left\vert \left(  Q_{1}+I\right)  ^{-1}\right\vert \left\Vert D\left(
Q_{2}+I\right)  ^{-1}Q_{1}\right\Vert .
\end{align*}
Here $Q_{i}>0$ so that $\left(  Q_{1}+I\right)  ^{-1}<I,$ and the above is
bounded by
\begin{align}
\left\Vert D\left(  Q_{2}+I\right)  ^{-1}Q_{1}\right\Vert  &  =\left\Vert
Q_{1}\left(  Q_{2}+I\right)  ^{-1}D\right\Vert \nonumber\\
&  \leq\left\vert Q_{1}\right\vert \cdot\left\vert \left(  Q_{2}+I\right)
^{-1}\right\vert \cdot\left\Vert D\right\Vert \nonumber\\
&  \leq\frac{\lambda}{1-\lambda}\left\Vert D\right\Vert \label{Q-decomp-3}%
\end{align}
in view of (\ref{R-bracket-3}). The second term in (\ref{Q-decomp}) can be
bounded by $\left\Vert D\right\Vert $. In conjunction with (\ref{Q-decomp-3})
this gives
\begin{align*}
\left\Vert H\right\Vert ^{2}  &  \leq\left(  \frac{\lambda}{1-\lambda
}+1\right)  ^{2}\left\Vert D\right\Vert ^{2}\\
&  =\frac{1}{\left(  1-\lambda\right)  ^{2}}\left\Vert A_{1}-A_{2}\right\Vert
^{2}.
\end{align*}

\end{proof}

We can can summarize the results of this subsection as follows.

\begin{proposition}
\label{prop-summary-symbols-and-KL}Let $A_{i}$, $i=1,2$ be Hermitian $n\times
n$ symbols fulfilling for some $\mu\in\left(  0,1\right)  $
\[
\left(  1+\mu\right)  I\leq A_{i}\leq\mu^{-1}I\text{, }i=1,2.
\]
Let the Gaussian states $\rho_{i}$, $i=1,2$ be defined as in Proposition
\ref{prop-rel-entrop-form}, and let $S\left(  \rho_{1}||\rho_{2}\right)  $ be
the relative entropy. Then there exists $\delta>0$, depending on $\mu$ but not
on $n$, such that $\left\Vert A_{1}-A_{2}\right\Vert <\delta$ implies
\[
S\left(  \rho_{1}||\rho_{2}\right)  \leq\delta^{-1}\left\Vert A_{1}%
-A_{2}\right\Vert ^{2}.
\]

\end{proposition}

\begin{proof}
Given $\mu\in\left(  0,1\right)  $, we can find $\lambda\in\left(
1/2,1\right)  $ such that
\[
\frac{2}{\lambda}-1\leq\mu<\mu^{-1}\leq\frac{1+\lambda}{1-\lambda}.
\]
Then (\ref{R-bracket-4}) and hence (\ref{R-bracket-2}) is fulfilled. The
previous two lemmas then prove the claim.
\end{proof}

\subsection{\label{sec:Toeplitz-circular}Approximation of Toeplitz matrices}

\bigskip

We follow \cite{MR4319036}, 5.5 to collect some basic facts about Toeplitz and
circulant matrices. Assume $m$ is an odd numer, let $\mathbf{c}=\mathbf{c}%
_{0}=\left(  c_{0},\ldots,c_{m-1}\right)  ^{\prime}$ be a column vector of
complex elements, let $\mathbf{c}_{1}=\left(  c_{m-1},c_{0},\ldots
,c_{m-2}\right)  ^{\prime}$ be a cyclic shift, and let $\mathbf{c}_{k}$ be the
$k$-th cyclic shift such that $\mathbf{c}_{m}=\left(  c_{1},c_{2}%
,\ldots,c_{m-1},c_{0}\right)  ^{\prime}$. Then the $m\times m$ circulant
pertaining to $\mathbf{c}$ is
\begin{equation}
T_{m}=\left(
\begin{array}
[c]{ccc}%
\mathbf{c}_{0} & \ldots & \mathbf{c}_{m-1}%
\end{array}
\right)  . \label{Gamma-matr-def}%
\end{equation}
Then $\mathbf{c}$ is the representing vector and the representing polynomial
is $p\left(  z\right)  =\sum_{k=0}^{m-1}c_{k}z^{k}$; we write $T_{m}%
=T_{m}\left(  \mathbf{c}\right)  =T_{m}\left(  p\right)  $. Clearly
$T_{m}\left(  \mathbf{c}\right)  $ is a Toeplitz matrix. To describe the
spectral properties, define
\begin{equation}
\epsilon_{k}=\exp\left(  2\pi ik/m\right)  \text{, }\mathbf{u}_{k}:=\left(
1,\epsilon_{k},\epsilon_{k}^{2},\ldots,\epsilon_{k}^{m-1}\right)  ^{\prime
}m^{-1/2}\text{, }k\in\mathbb{Z} \label{epsilon-u-vectors-def}%
\end{equation}
and the discrete Fourier transform $\Phi_{d,m}:\mathbb{C}^{m}\rightarrow$
$\mathbb{C}^{m}$ by its matrix
\begin{equation}
\Phi_{d,m}=\left(  \mathbf{u}_{0},\ldots,\mathbf{u}_{m-1}\right)  .
\label{DFT-def}%
\end{equation}
Then $\Phi_{d,m}$ is unitary, and diagonalizes every circulant $T_{m}\left(
p\right)  $ in the sense that
\begin{equation}
\Phi_{d,m}^{\ast}T_{m}\left(  p\right)  \Phi_{d,m}=\mathrm{diag}\left(
p\left(  1\right)  ,p\left(  \bar{\epsilon}_{1}\right)  ,\ldots,p\left(
\bar{\epsilon}_{m-1}\right)  \right)  \label{spectral-circulant-descr}%
\end{equation}
(cf. \cite{MR4319036}, 5.5.4).

We give an alternative description of the spectral properties as follows.
Define
\begin{align}
\phi_{k}\left(  \omega\right)   &  =\exp\left(  ik\omega\right)  \text{,
}\omega\in\mathbb{R}\text{, }k\in\mathbb{Z},\label{basic-func-def}\\
\omega_{j,m}  &  =\frac{2\pi j}{m}\text{, }j\in\mathbb{Z}.
\label{Fourier-frequ-def}%
\end{align}

\begin{lemma}
\label{lem-spectrum-circulant}Assume that $m$ is odd; define $c_{-k}=c_{m-k}$,
$k=1,\ldots,\left(  m-1\right)  /2$ and a function%
\[
g_{m}\left(  \mathbf{c},\omega\right)  =\sum_{k=-\left(  m-1\right)
/2}^{\left(  m-1\right)  /2}c_{k}\phi_{-k}\left(  \omega\right)  \text{,
}\omega\in\mathbb{R}\text{.}%
\]
Then (\ref{spectral-circulant-descr}) can be written
\begin{equation}
\Phi_{d,m}^{\ast}T_{m}\left(  \mathbf{c}\right)  \Phi_{d,m}=\mathrm{diag}%
\left(  g_{m}\left(  \mathbf{c},\omega_{0,m}\right)  ,\ldots,g_{m}\left(
\mathbf{c},\omega_{m-1,m}\right)  \right)  .
\label{spectral-circulant-descr-2}%
\end{equation}
Furthermore define a unitary $m\times m$ matrix, with $\mathbf{u}_{k}$ from
(\ref{epsilon-u-vectors-def})
\begin{equation}
U_{m}=\left(  \mathbf{u}_{-(m-1)/2},\ldots,\mathbf{u}_{0},\ldots
\mathbf{u}_{(m-1)/2}\right)  . \label{special-unitary-DFT-def}%
\end{equation}
Then (\ref{spectral-circulant-descr-2}) is equivalent to
\begin{equation}
U_{m}^{\ast}T_{m}\left(  \mathbf{c}\right)  U_{m}=\mathrm{diag}\left(
g_{m}\left(  \mathbf{c},\omega_{-\left(  m-1\right)  /2,m}\right)
,\ldots,g_{m}\left(  \mathbf{c},\omega_{\left(  m-1\right)  /2,m}\right)
\right)  . \label{spectral-circulant-descr-3}%
\end{equation}

\end{lemma}

\begin{proof}
First note that the eigenvalues $p\left(  \bar{\epsilon}_{j}\right)  $ in
(\ref{spectral-circulant-descr}) can be written as
\begin{align*}
p\left(  \bar{\epsilon}_{j}\right)   &  =\sum_{k=0}^{m-1}c_{k}\bar{\epsilon
}_{j}^{k}=\sum_{k=0}^{m-1}c_{k}\exp\left(  -2\pi ijk/m\right) \\
&  =\sum_{k=0}^{m-1}c_{k}\phi_{-k}\left(  \omega_{j.m}\right)  .\text{
}j=0,\ldots,m-1.
\end{align*}
By periodicity we have
\begin{align*}
\phi_{-\left(  m-k\right)  }\left(  \omega_{j,m}\right)   &  =\exp\left(
-i\left(  m-k\right)  \omega_{j,m}\right) \\
&  =\exp\left(  ik\omega_{j,m}\right)  \exp\left(  -im\frac{2\pi j}{m}\right)
=\phi_{k}\left(  \omega_{j,m}\right)
\end{align*}
for $k=1,\ldots,m-1$. Hence
\begin{align*}
p\left(  \bar{\epsilon}_{j}\right)   &  =\sum_{k=0}^{\left(  m-1\right)
/2}c_{k}\phi_{-k}\left(  \omega_{j,m}\right)  +\sum_{k=\left(  m+1\right)
/2}^{m-1}c_{k}\phi_{-k}\left(  \omega_{j,m}\right) \\
&  =\sum_{k=0}^{\left(  m-1\right)  /2}c_{k}\phi_{-k}\left(  \omega
_{j,m}\right)  +\sum_{k=1}^{\left(  m-1\right)  /2}c_{m-k}\phi_{-\left(
m-k\right)  }\left(  \omega_{j,m}\right) \\
&  =\sum_{k=0}^{\left(  m-1\right)  /2}c_{k}\phi_{-k}\left(  \omega
_{j,m}\right)  +\sum_{k=1}^{\left(  m-1\right)  /2}c_{-k}\phi_{k}\left(
\omega_{j,m}\right)  =g_{m}\left(  \mathbf{c},\omega_{j,m}\right)  \text{,
}j=0,\ldots,m-1
\end{align*}
which implies (\ref{spectral-circulant-descr-2}). This relation is equivalent
to
\begin{align*}
T_{m}\left(  \mathbf{c}\right)   &  =\Phi_{d,m}\mathrm{diag}\left(
g_{m}\left(  \mathbf{c},\omega_{0,m}\right)  ,\ldots,g_{m}\left(
\mathbf{c},\omega_{m-1,m}\right)  \right)  \Phi_{d,m}^{\ast}\\
&  =\sum_{k=0}^{m-1}\mathbf{u}_{k}\mathbf{u}_{k}^{\ast}g_{m}\left(
\mathbf{c},\omega_{k,m}\right)  .
\end{align*}
By periodicity of the function $g_{m}$ in $\omega$ we have $g_{m}\left(
\mathbf{c},\omega_{m-k,m}\right)  =g_{m}\left(  \mathbf{c},\omega
_{-k,m}\right)  $, $k=1,\ldots,\left(  m-1\right)  /2$, and we also have
$\epsilon_{m-k}=\epsilon_{-k}$ and hence $\mathbf{u}_{m-k}=\mathbf{u}_{-k}$.
Thus we obtain
\begin{align*}
T_{m}\left(  \mathbf{c}\right)   &  =\sum_{k=0}^{m-1}\mathbf{u}_{k}%
\mathbf{u}_{k}^{\ast}g_{m}\left(  \mathbf{c},\omega_{k,m}\right)  =\sum
_{k=0}^{\left(  m-1\right)  /2}\mathbf{u}_{k}\mathbf{u}_{k}^{\ast}g_{m}\left(
\mathbf{c},\omega_{k,m}\right)  +\sum_{k=1}^{\left(  m-1\right)  /2}%
\mathbf{u}_{m-k}\mathbf{u}_{m-k}^{\ast}g_{m}\left(  \mathbf{c},\omega
_{m-k,m}\right) \\
&  =\sum_{k=-\left(  m-1\right)  /2}^{\left(  m-1\right)  /2}\mathbf{u}%
_{k}\mathbf{u}_{k}^{\ast}g_{m}\left(  \mathbf{c},\omega_{k,m}\right)  ,
\end{align*}
implying (\ref{spectral-circulant-descr-3})
\end{proof}

\bigskip

Note that the matrix $U_{m}$ is a permutation of $\Phi_{d,m}$, thus it can be
considered a version of the discrete Fourier transform.

Since we will use the circulants to approximate the Hermitian Toeplitz symbol
matrices $A_{m}\left(  a\right)  $, we will also assume that $T_{m}\left(
\mathbf{c}\right)  $ is Hermitian. From (\ref{Gamma-matr-def}) it can be seen
that in terms of $\mathbf{c}$ this means
\begin{equation}
c_{0}=\bar{c}_{0}\ \text{and }\bar{c}_{k}=c_{m-k},\ k=1,\ldots,m-1.
\label{hermite-1}%
\end{equation}
Then $c_{-k}=\bar{c}_{k}$ and consequently the function $g_{m}\left(
\mathbf{c},\omega\right)  $ is real, thus also the eigenvalues of
$T_{m}\left(  \mathbf{c}\right)  $ are real.

For a symbol matrix $A_{m}\left(  a\right)  $ defined by
(\ref{symbol-generate}) pertaining to spectral density $a$ we will define a
circular approximant $\tilde{A}_{m}\left(  a\right)  $ as
\begin{equation}
\tilde{A}_{m}\left(  a\right)  :=T_{m}\left(  \mathbf{c}\right)
\label{circulant-approx-def-3}%
\end{equation}
for a representing vector
\begin{equation}
\mathbf{c}=\left(  a_{0},a_{-1},\ldots,a_{-\left(  m-1\right)  /2},a_{\left(
m-1\right)  /2},\ldots,a_{1}\right)  ^{\prime}. \label{circulant-approx-def-2}%
\end{equation}
In view of $\bar{a}_{k}=a_{-k}$ it can be checked that (\ref{hermite-1}) is
fulfilled and thus $\tilde{A}_{m}\left(  a\right)  $ is Hermitian. One then
checks that $g_{m}\left(  \mathbf{c},\omega\right)  $ takes the form
\begin{align}
g_{m}\left(  \mathbf{c},\omega\right)   &  =\tilde{a}_{m}\left(
\omega\right)  \text{ where}\nonumber\\
\tilde{a}_{m}\left(  \omega\right)   &  =\sum_{k=-\left(  m-1\right)
/2}^{\left(  m-1\right)  /2}a_{k}\phi_{k}\left(  \omega\right)  \text{,
}\omega\in\mathbb{R}\text{.} \label{Fourier-approx-a-def}%
\end{align}
According to (\ref{spectral-circulant-descr-3}), the eigenvalues of $\tilde
{A}_{m}\left(  a\right)  $ are then $\tilde{a}_{m}\left(  \omega_{j,m}\right)
$, $j=-\left(  m-1\right)  /2,\ldots,\left(  m-1\right)  /2$. Now $\tilde
{a}_{m}$ is a Fourier series approximation to $a$; indeed it follows from
(\ref{symbol-generate}), if $a$ is square integrable on $\left(  -\pi
,\pi\right)  $, that
\[
a\left(  \omega\right)  =\sum_{k=-\infty}^{\infty}a_{k}\phi_{k}\left(
\omega\right)  \text{, }\omega\in\mathbb{R}\text{.}%
\]
For later reference we state the following simple approximation result.

\begin{lemma}
\label{lem-Fourier-series-approx}Assume $a\in\Theta_{1}\left(  \alpha
,M\right)  $ for $\alpha>1/2$. Then as $m\rightarrow\infty$
\[
\sup_{\omega\in\left(  -\pi.\pi\right)  }\left\vert a\left(  \omega\right)
-\tilde{a}_{m}\left(  \omega\right)  \right\vert =o\left(  1\right)
\]

\end{lemma}

\begin{proof}
We have
\[
\sup_{\omega\in\left(  -\pi.\pi\right)  }\left\vert a\left(  \omega\right)
-\tilde{a}_{m}\left(  \omega\right)  \right\vert ^{2}=\sup_{\omega\in\left(
-\pi.\pi\right)  }\left\vert \sum_{\left\vert k\right\vert >\left(
m-1\right)  /2}a_{k}\phi_{k}\left(  \omega\right)  \right\vert ^{2}%
\]%
\[
\leq\left(  \sum_{\left\vert k\right\vert >\left(  m-1\right)  /2}\left\vert
a_{k}\right\vert ^{2}k^{2\alpha}\right)  \left(  \sum_{\left\vert k\right\vert
>\left(  m-1\right)  /2}k^{-2\alpha}\right)  \leq M\;C_{\alpha}\left(
m-3\right)  ^{1-2\alpha}=o\left(  1\right)
\]
as $m\rightarrow\infty$, where the constant $C_{\alpha}$ depends only on
$\alpha$.
\end{proof}

We summarize the above facts about circulants as follows.

\begin{lemma}
\label{lem-spectrum-circulant-A_tilde}For a real valued function $a\in
L^{2}\left(  -\pi,\pi\right)  $ and odd $m$, consider the $m\times m$
circulant matrix $\tilde{A}_{m}\left(  a\right)  $ given by
(\ref{circulant-approx-def-3}), (\ref{circulant-approx-def-2}). Also define a
diagonal matrix
\begin{equation}
\tilde{\Lambda}_{m}\left(  a\right)  =\mathrm{diag}\left(  \tilde{a}%
_{m}\left(  \omega_{-\left(  m-1\right)  /2,m}\right)  ,\ldots,\tilde{a}%
_{m}\left(  \omega_{\left(  m-1\right)  /2,m}\right)  \right)
\label{diag-symbol-Lambda-def}%
\end{equation}
where $\tilde{a}_{m}$ is the Fourier series approximation to $a$ given in
(\ref{Fourier-approx-a-def}) and $\omega_{j,m}=2\pi j/m$, $j\in\mathbb{Z}$.
Then, for the unitary $U_{m}$ defined in (\ref{special-unitary-DFT-def}) we
have
\begin{equation}
U_{m}^{\ast}\tilde{A}_{m}\left(  a\right)  U_{m}=\tilde{\Lambda}_{m}\left(
a\right)  . \label{diagonalize-circulant-summary}%
\end{equation}

\end{lemma}

Recall that the quantum statistical experiment considered in Theorem
\ref{theor-main-1} is
\[
\mathcal{E}_{n}\left(  \Theta_{1}\left(  \alpha,M\right)  \right)  :=\left(
\mathfrak{N}_{n}\left(  0,A_{n}\left(  a\right)  \right)  ,a\in\Theta
_{1}\left(  \alpha,M\right)  \right)  .
\]
For some odd $m>n$, let $\tilde{A}_{m}\left(  a\right)  $ be the circulant
approximation (\ref{circulant-approx-def-3}) to $A_{m}\left(  a\right)  $, and
consider the $m$-mode state (or quantum time series) $\mathfrak{N}_{m}\left(
0,\tilde{A}_{m}\left(  a\right)  \right)  $ given by symbol $\tilde{A}%
_{m}\left(  a\right)  $. Furthermore consider the subsystem of the latter
given by symbol $\tilde{A}_{n,m}\left(  a\right)  $, where $\tilde{A}%
_{n,m}\left(  a\right)  $ is the upper left $n\times n$ submatrix of
$\tilde{A}_{m}\left(  a\right)  $. The following result on approximation of
symbols in Hilbert-Schmidt norm $\left\Vert A\right\Vert _{2}=\left(
\mathrm{Tr}\;A^{\ast}A\right)  ^{1/2}$ is key for an approximation of the
corresponding states.

\begin{lemma}
\label{lem-upp-inf-brack}Assume $m$ is odd, $n<m<2\left(  n-1\right)  $. Then
for $a\in W^{\alpha}(M)$, $\alpha>1/2$ (cp. (\ref{sobol-ball})) we have
\[
\left\Vert A_{n}\left(  a\right)  -\tilde{A}_{n,m}\left(  a\right)
\right\Vert _{2}^{2}\leq4\left(  m-n+1\right)  ^{1-2\alpha}M.
\]

\end{lemma}

\begin{proof}
The restriction on $m$ implies $\left(  m+1\right)  /2\leq n-1$. From the
definitions of $A_{n}\left(  a\right)  $ and $\tilde{A}_{n,m}\left(  a\right)
$ we immediately obtain
\begin{align}
&  \left\Vert A_{n}\left(  a\right)  -\tilde{A}_{n,m}\left(  a\right)
\right\Vert _{2}^{2}\nonumber\\
&  =2\sum_{k=\left(  m+1\right)  /2}^{n-1}\left(  n-k\right)  \left\vert
a_{k}-\bar{a}_{m-k}\right\vert ^{2} \label{first-upper-bound-cited-later}%
\end{align}%
\begin{equation}
\leq4\sum_{k=\left(  m+1\right)  /2}^{n-1}(n-k)\left(  \left\vert
a_{k}\right\vert ^{2}+\left\vert a_{m-k}\right\vert ^{2}\right)  .
\label{first-upper-bound}%
\end{equation}
Note that for $m>n$, the relation $\left(  m+1\right)  /2\leq k\leq n-1$
implies $k>\left(  n+1\right)  /2$ and therefore $n-k<k$, and note also
$n-k<m-k$. We obtain an upper bound for (\ref{first-upper-bound})
\begin{align*}
&  4\sum_{k=\left(  m+1\right)  /2}^{n-1}k\left\vert a_{k}\right\vert
^{2}+4\sum_{k=\left(  m+1\right)  /2}^{n-1}(m-k)\left\vert a_{m-k}\right\vert
^{2}\\
&  =4\sum_{k=\left(  m+1\right)  /2}^{n-1}k\left\vert a_{k}\right\vert
^{2}+4\sum_{k=m-n+1}^{\left(  m-1\right)  /2}k\left\vert a_{k}\right\vert
^{2}=4\sum_{k=m-n+1}^{n-1}k\left\vert a_{k}\right\vert ^{2}\\
&  \leq4(m-n+1)^{1-2\alpha}\sum_{k=m-n+1}^{n-1}k^{2\alpha}\left\vert
a_{k}\right\vert ^{2}\leq4(m-n+1)^{1-2\alpha}\left\vert a\right\vert
_{2,\alpha}^{2}%
\end{align*}
where $\left\vert \cdot\right\vert _{2,\alpha}^{2}$ is defined in
(\ref{normdef}), and $\alpha>1/2$. Now $\left\vert a\right\vert _{2,\alpha
}^{2}\leq M$ for $a\in W^{\alpha}(M)$ proves the claim.
\end{proof}

\subsection{Upper information bound via approximation of
symbols\label{subsec:approx-symbol}}

To apply Lemma \ref{lem-upp-inf-brack} on approximation of symbols $A$ to the
corresponding states $\mathfrak{N}_{n}\left(  0,A\right)  $ via Proposition
\ref{prop-summary-symbols-and-KL}, we need uniform bounds on the eigenvalues
of the symbols involved.

\begin{lemma}
\label{lem-toeplitz-EV} Suppose $a\in\Theta_{1}\left(  \alpha,M\right)  $ for
$\alpha>1/2$, $M>1$. Then there exists $C=C_{M,\alpha}>1$ such that for
$n\geq1$
\begin{equation}
\left(  1+C^{-1}\right)  \;I\leq A_{n}\left(  a\right)  \leq C\;I.
\label{firstbracket-ev}%
\end{equation}
Furthermore, there exist $C_{1}>1$ and $m_{0}$ such that for odd $m\geq m_{0}$
and all $n<m$%
\begin{equation}
\left(  1+C_{1}\right)  \;I\leq\tilde{A}_{n,m}\left(  a\right)  \leq
C_{1}\;I\text{.} \label{2ndbracket-ev}%
\end{equation}

\end{lemma}

\begin{proof}
Consider $x\in\mathbb{C}^{n}$ with $\left\Vert x\right\Vert =1$; then in view
of (\ref{symbol-generate})%
\[
\left\langle x,A_{n}\left(  a\right)  x\right\rangle =\frac{1}{2\pi}\int
_{-\pi}^{\pi}\left\vert \sum_{j=1}^{n}x_{j}\exp\left(  ij\omega\right)
\right\vert ^{2}a\left(  \omega\right)  d\omega.
\]
Applying the second inequality in (\ref{cond-boundedness-specdensity}) we
obtain
\begin{align*}
\left\langle x,A_{n}\left(  a\right)  x\right\rangle  &  \leq\frac{\mu^{-1}%
}{2\pi}\int_{-\pi}^{\pi}\left\vert \sum_{j=1}^{n}x_{j}\exp\left(
ij\omega\right)  \right\vert ^{2}d\omega\\
&  =\mu^{-1}\sum_{j=1}^{n}\left\vert x_{j}\right\vert ^{2}=\mu^{-1}.
\end{align*}
Analogously we obtain from the first inequality in
(\ref{cond-boundedness-specdensity}) $\left\langle x,A_{n}\left(  a\right)
x\right\rangle \geq\left(  1+\mu\right)  $, so that (\ref{firstbracket-ev}) is
shown. To establish (\ref{2ndbracket-ev}), note first that since $\tilde
{A}_{n,m}\left(  a\right)  $ is a central submatrix of $\tilde{A}_{m}\left(
a\right)  $, we have
\[
\lambda_{\min}\left(  \tilde{A}_{m}\left(  a\right)  \right)  \leq
\lambda_{\min}\left(  \tilde{A}_{n,m}\left(  a\right)  \right)  \text{,
}\lambda_{\max}\left(  \tilde{A}_{n,m}\left(  a\right)  \right)  \leq
\lambda_{\max}\left(  \tilde{A}_{m}\left(  a\right)  \right)
\]
so we need to deal only with $\tilde{A}_{m}\left(  a\right)  $. Lemma
\ref{lem-spectrum-circulant-A_tilde} describes the eigenvalues of this matrix
as certain function values $\tilde{a}_{m}\left(  \omega_{j,m}\right)  $. Now
according to Lemma \ref{lem-Fourier-series-approx} $\tilde{a}_{m}$
approximates $a$ uniformly if $a\in\Theta_{1}\left(  \alpha,M\right)  $ for
$\alpha>1/2$. In conjunction with (\ref{cond-boundedness-specdensity}) this
proves the second claim.
\end{proof}

\bigskip

\bigskip

In this section and the next, the parameter space for the quantum statistical
experiments to be considered will always be the set $\Theta_{1}\left(
\alpha,M\right)  $ considered in Theorem \ref{theor-main-1}, and will often be
omitted from notation.

\begin{proposition}
\label{Prop:deta-upper-bound-by-circular}Consider the experiment
$\mathcal{E}_{n}=\mathcal{E}_{n}\left(  \Theta_{1}\left(  \alpha,M\right)
\right)  $ defined in (\ref{basic-quantum-experi-def}) and define also for odd
$m$
\[
\mathcal{\tilde{E}}_{m}=\left(  \mathfrak{N}_{m}\left(  0,\tilde{A}_{m}\left(
a\right)  \right)  ,a\in\Theta_{1}\left(  \alpha,M\right)  \right)
\]
where $\tilde{A}_{m}\left(  a\right)  $ is the circulant matrix defined in
(\ref{circulant-approx-def-3}) such that $n<m<2\left(  n-1\right)  $. Assume
$m$ is chosen such that $m-n\rightarrow\infty$; then
\[
\text{ }\mathcal{E}_{n}\precsim\mathcal{\tilde{E}}_{m}\;\text{ as
}n\rightarrow\infty,
\]
i.e. $\mathcal{\tilde{E}}_{m}$ is asymptotically more informative than
$\mathcal{E}_{n}$.
\end{proposition}

\begin{proof}
Consider the submatrix $\tilde{A}_{n,m}\left(  a\right)  $ of $\tilde{A}%
_{m}\left(  a\right)  $ occurring in Lemma \ref{lem-upp-inf-brack}. If
$m-n\rightarrow\infty$ then Lemma \ref{lem-upp-inf-brack} in conjunction with
Proposition \ref{prop-summary-symbols-and-KL} and Lemma \ref{lem-toeplitz-EV}
implies existence of a constant $\delta>0$ such that for the relative entropy
$S\left(  \cdot||\cdot\right)  $
\begin{align*}
S\left(  \mathfrak{N}_{n}\left(  0,A_{n}\left(  a\right)  \right)
||\mathfrak{N}_{n}\left(  0,\tilde{A}_{n,m}\left(  a\right)  \right)  \right)
&  \leq\delta^{-1}\left\Vert A_{n}\left(  a\right)  -\tilde{A}_{n,m}\left(
a\right)  \right\Vert ^{2}\\
&  \leq\delta^{-1}4\left(  m-n+1\right)  ^{1-2\alpha}M=o\left(  1\right)
\end{align*}
since $\alpha>1/2$. By inequality
(\ref{inequality-trace-norm-relative-entropy}) we then also have
\begin{equation}
\sup_{a\in\Theta_{1}\left(  \alpha,M\right)  }\left\Vert \mathfrak{N}%
_{n}\left(  0,A_{n}\left(  a\right)  \right)  -\mathfrak{N}_{n}\left(
0,\tilde{A}_{n,m}\left(  a\right)  \right)  \right\Vert _{1}^{2}\rightarrow0.
\label{trace-norm-approx-subcircular}%
\end{equation}
Obviously there is a quantum channel which maps the $m$-mode state
$\mathfrak{N}_{m}\left(  0,\tilde{A}_{m}\left(  a\right)  \right)  $ into the
$n$-mode state $\mathfrak{N}_{n}\left(  0,\tilde{A}_{n,m}\left(  a\right)
\right)  $, as the quantum equivalent of "omitting observations", i.e. the
partial trace. Formally this channel $\alpha$ is described in terms of a map
between the respective algebras in the Appendix, Subsection
\ref{subsec:partial-trace}; we then have
\[
\mathfrak{N}_{m}\left(  0,\tilde{A}_{m}\left(  a\right)  \right)  \circ
\alpha=\mathfrak{N}_{n}\left(  0,\tilde{A}_{n,m}\left(  a\right)  \right)  .
\]
From (\ref{trace-norm-approx-subcircular}) we then obtain
\[
\sup_{a\in\Theta_{1}\left(  \alpha,M\right)  }\left\Vert \mathfrak{N}%
_{m}\left(  0,\tilde{A}_{m}\left(  a\right)  \right)  \circ\alpha
-\mathfrak{N}_{n}\left(  0,A_{n}\left(  a\right)  \right)  \right\Vert
_{1}^{2}\rightarrow0
\]
which implies the claim.
\end{proof}

\subsection{The geometric regression model\label{sec:geom-reg-model}}

The spectral decomposition of the circulant matrix $\tilde{A}_{m}\left(
a\right)  $ is described in Lemma \ref{lem-spectrum-circulant-A_tilde}. Since
\[
\mathfrak{N}_{m}\left(  0,\tilde{A}_{m}\left(  a\right)  \right)  =\frac
{2^{m}}{\det\left(  \tilde{A}_{m}\left(  a\right)  +I\right)  }\left(
\frac{\tilde{A}_{m}\left(  a\right)  -I}{\tilde{A}_{m}\left(  a\right)
+I}\right)  _{F}%
\]
by (\ref{Fock-repre-1}), we can use the property of Fock operators
(\ref{multiplying-Fock-ops}) to diagonalize the state. For the diagonal symbol
matrix $\tilde{\Lambda}_{m}\left(  a\right)  $ defined in
(\ref{diag-symbol-Lambda-def}), consider an experiment
\[
\mathcal{\tilde{E}}_{m,1}=\left(  \mathfrak{N}_{m}\left(  0,\tilde{\Lambda
}_{m}\left(  a\right)  \right)  ,a\in\Theta_{1}\left(  \alpha,M\right)
\right)  .
\]

\begin{lemma}
\label{lem-equiv-diagonal-symbol}For all odd $m\geq3$, we have statistical
equivalence%
\[
\mathcal{\tilde{E}}_{m}\sim\mathcal{\tilde{E}}_{m,1}.
\]

\end{lemma}

\begin{proof}
From (\ref{diagonalize-circulant-summary}) and (\ref{multiplying-Fock-ops}) it
follows that
\begin{align*}
\left(  U_{m}^{\ast}\right)  _{F}\mathfrak{N}_{m}\left(  0,\tilde{A}%
_{m}\left(  a\right)  \right)  \left(  U_{m}\right)  _{F}  &  =\frac{2^{m}%
}{\det\left(  \tilde{\Lambda}_{m}\left(  a\right)  +I\right)  }\left(
\frac{\tilde{\Lambda}_{m}\left(  a\right)  -I}{\tilde{\Lambda}_{m}\left(
a\right)  +I}\right)  _{F}\\
&  =\mathfrak{N}_{m}\left(  0,\tilde{\Lambda}_{m}\left(  a\right)  \right)  .
\end{align*}
Since $\left(  U_{m}\right)  _{F}$ is unitary, the above mapping of
$\mathfrak{N}_{m}\left(  0,\tilde{A}_{m}\left(  a\right)  \right)  $ to
$\mathfrak{N}_{m}\left(  0,\tilde{\Lambda}_{m}\left(  a\right)  \right)  $
represents an invertible state transition (or dual channel, cf. Subsection
\ref{subsec:States-channels-obs}). This implies the equivalence claim by
definition of $\Delta\left(  \cdot,\cdot\right)  $.
\end{proof}

\bigskip

In the experiment $\mathcal{\tilde{E}}_{m,1}$, all symbol matrices
$\Lambda_{m}\left(  a\right)  $ are commuting. The representation
(\ref{Fock-repre-1}) implies that all states in $\mathcal{E}_{m,1}$ are
commuting, hence $\mathcal{E}_{m,1}$ is equivalent (in the sense of the
$\Delta$-distance) to a classical model. To describe the latter, write the
diagonal elements of $\Lambda_{m}\left(  a\right)  $ as
\[
\lambda_{j,m}\left(  a\right)  =\tilde{a}_{m}\left(  \omega_{j-\left(
m+1\right)  /2,m}\right)  \text{, }j=1,\ldots,m
\]
and define (for odd $m$) a set of probability measures (products of geometric
distributions)
\begin{equation}
\mathcal{\tilde{F}}_{m}=\left(
{\displaystyle\bigotimes\limits_{J=1}^{m}}
\mathrm{Geo}\left(  p\left(  \lambda_{j,m}\left(  a\right)  \right)  \right)
,a\in\Theta_{1}\left(  \alpha,M\right)  \right)  .
\label{G-m-tilde-experi-def}%
\end{equation}
where $p\left(  x\right)  =\left(  x-1\right)  /\left(  x+1\right)  $.

\begin{proposition}
\label{Prop-equiv-diag-symbol-geometr}For all odd $m\geq3$, we have
statistical equivalence%
\[
\mathcal{\tilde{E}}_{m,1}\sim\mathcal{\tilde{F}}_{m}.
\]

\end{proposition}

\begin{proof}
Consider the covariance matrix of $\mathfrak{N}_{m}\left(  0,\Lambda
_{m}\left(  a\right)  \right)  $, which according to
(\ref{symbol-related-to-covmatrix}) is
\[
\Sigma=\frac{1}{2}\left(
\begin{array}
[c]{cc}%
\Lambda_{m}\left(  a\right)  & 0\\
0 & \Lambda_{m}\left(  a\right)
\end{array}
\right)  .
\]
This corresponds to a vector of canonical observables $\mathbf{R}=\left(
Q_{1},\ldots,Q_{m},P_{1},\ldots,P_{m}\right)  $. with a rearrangement as
$\mathbf{\check{R}}:=(Q_{1},P_{1},\ldots,Q_{m},P_{m})$ as in
(\ref{permuted-cov-matrix}) the covariance matrix becomes block diagonal
\[
\check{\Sigma}=\frac{1}{2}\left(
\begin{array}
[c]{ccc}%
\lambda_{1,m}\left(  a\right)  I_{2} &  & 0\\
& \ldots & \\
0 &  & \lambda_{m,m}\left(  a\right)  I_{2}%
\end{array}
\right)  .
\]
The centered $m$-mode Gaussian state is clearly the tensor product of $m$
one-mode Gaussian states with covariance matrix $\frac{1}{2}\lambda
_{j,m}\left(  a\right)  I_{2},$ $j=1,\ldots,m$. A centered Gaussian state with
covariance matrix $\frac{1}{2}\lambda I_{2}$, $\lambda>1$ has a representation
in Fock space $\mathfrak{F}(\mathbb{C})$ (according to (\ref{Fock-repre-1}))
\[
\mathfrak{N}_{1}\left(  0,\lambda\right)  =\frac{2}{\lambda+1}%
{\displaystyle\bigoplus\limits_{k\geq0}}
\left(  \frac{\lambda-1}{\lambda+1}\right)  ^{k}%
\]
and setting $p\left(  \lambda\right)  =\left(  \lambda-1\right)  /\left(
\lambda+1\right)  $, we obtain
\begin{equation}
\mathfrak{N}_{1}\left(  0,\lambda\right)  =\left(  1-p\left(  \lambda\right)
\right)
{\displaystyle\bigoplus\limits_{k\geq0}}
p\left(  \lambda\right)  ^{k}. \label{thermal-state-repre-in-Fock}%
\end{equation}
which corresponds to a one mode thermal state with covariance matrix $\frac
{1}{2}\lambda I_{2}$, $\lambda>1$ \cite{RevModPhys.84.621}\textbf{.} We obtain
that $\mathcal{E}_{m,1}$ is equivalent to
\[
\left(
{\displaystyle\bigotimes\limits_{j=1}^{m}}
\mathfrak{N}_{1}\left(  0,\lambda_{j,m}\left(  a\right)  \right)  ,a\in
\Theta_{1}\left(  \alpha,M\right)  \right)
\]
which in turn, by measuring each tensor factor in the coordinate basis, is
equivalent to observing $m$ independent r.v.'s $X_{j}$ having geometric
distributions (cf. Subsection \ref{subsec:geometr-distr})%
\begin{equation}
X_{j}\sim\mathrm{Geo}\left(  p\left(  \lambda_{j,m}\left(  a\right)  \right)
\right)  ,\;j=1,\ldots,m. \label{geom-reg-model}%
\end{equation}
This establishes the equivalence claimed .
\end{proof}

\subsection{Comparing geometric regression
models\label{subsec:comparing-geom-reg}}

\textbf{ }Having obtained an experiment $\mathcal{\tilde{F}}_{m}$ consisting
of classical probability measures, further developments will take place in
this framework. Consider the Hellinger distance $H\left(  P,Q\right)  $
between probability measures $P,Q$ on the same sample space, defined as
follows: for $\mu=P+Q$, $p=dP/d\mu$, $q=dP/d\mu$,
\[
H^{2}\left(  P,Q\right)  =\int\left(  p^{1/2}-q^{1/2}\right)  ^{2}d\mu.
\]
Note the relationship to $L_{1}$-distance $\left\Vert P-Q\right\Vert _{1}$:%
\begin{equation}
\frac{1}{2}\left\Vert P-Q\right\Vert _{1}\leq H\left(  P,Q\right)
\label{total-variation-Hellinger-inequ}%
\end{equation}
(\cite{MR2724359}, Lemma 2.3). Also, for product measures $\otimes_{j=1}%
^{n}P_{j}$ and $\otimes_{j=1}^{n}Q_{j}$ we have
\begin{equation}
H^{2}\left(  \otimes_{j=1}^{n}P_{j},\otimes_{j=1}^{n}Q_{j}\right)  \leq
2\sum_{j=1}^{n}H^{2}\left(  P_{j},Q_{j}\right)  . \label{Hellinger-sum-inequ}%
\end{equation}
as follows from Lemma 2.19 in \cite{MR812467}.

\bigskip For general $n$, consider intervals in $\left(  -\pi,\pi\right)  $ of
equal length
\begin{equation}
W_{j,n}=2\pi\left(  \frac{\left(  j-1\right)  }{n}-\frac{1}{2},\frac{j}%
{n}-\frac{1}{2}\right)  \text{, }j=1,\ldots,n \label{intervals-Wjn-def}%
\end{equation}
and for any real $f\in L_{2}(-\pi.\pi),$ let $\bar{f}_{n}$ be the $L_{2}%
$-projection onto the piecewise constant functions, i.e.
\begin{equation}
\bar{f}_{n}=\sum_{j=1}^{n}J_{j,n}(f)\mathbf{1}_{W_{j,n}}\text{, where }%
J_{j,n}(f)=\frac{n}{2\pi}\int_{W_{j,n}}f(x)dx. \label{fbar-def}%
\end{equation}
In agreement with (\ref{Fourier-approx-a-def}) define the Fourier series
approximation to $f$, for odd $n$
\begin{align*}
\tilde{f}_{n}\left(  \omega\right)   &  =\sum_{k=-\left(  n-1\right)
/2}^{\left(  n-1\right)  /2}f_{k}\phi_{k}\left(  \omega\right)  \text{,
}\omega\in\lbrack-\pi,\pi]\text{,}\\
f_{k}  &  =\frac{1}{2\pi}\int_{-\pi}^{\pi}\exp\left(  -ik\omega\right)
f\left(  \omega\right)  d\omega.
\end{align*}
Recall also the definition of the seminorm $\left\vert f\right\vert
_{2,\alpha}^{2}$ and the norm $\left\Vert f\right\Vert _{2,\alpha}^{2}$ in
(\ref{normdef}).

\begin{lemma}
\label{lem-basic-approx-in-Sobolev}For $f\in L_{2}(-\pi.\pi)$, assume
$\left\vert f\right\vert _{2,\alpha}^{2}$ is finite for given $0<\alpha<1$.
Then \newline(i) there is a constant $C_{\alpha}$ such that
\[
\left\Vert f-\bar{f}_{n}\right\Vert _{2}^{2}\leq C_{\alpha}\;n^{-2\alpha
}\;\left\vert f\right\vert _{2,\alpha}^{2}.
\]
(ii) Assume that $1/2<\alpha<1$, that $n$ is odd and let $\tilde{\omega}%
_{j,n}$ be the midpoint of $W_{j,n}$ $j=1,\ldots,n$. Then there is a constant
$C_{\alpha}$ such that
\[
\sum_{j=1}^{n}\left(  \tilde{f}_{n}(\tilde{\omega}_{j,n})-J_{j,n}(f)\right)
^{2}\leq C_{\alpha}\;n^{1-2\alpha}\;\left\Vert f\right\Vert _{2,\alpha}^{2}.
\]

\end{lemma}

\begin{proof}
(i) A version of the claim for functions $f$ defined on $\left(  0,1\right)  $
is proved in Lemma 5.3 \cite{GolNussbZ-specpaper-preprint}; a rescaling to the
interval $\left(  -\pi,\pi\right)  $ yields the present claim. Also, in
\cite{GolNussbZ-specpaper-preprint} the inequality is proved for a seminorm
$\left\vert f\right\vert _{B_{2,2}^{\alpha}}^{2}$ in place of $\left\vert
f\right\vert _{2,\alpha}^{2}$, but Lemma 5.5 in
\cite{GolNussbZ-specpaper-preprint} shows that if $\left\vert f\right\vert
_{2,\alpha}^{2}<\infty$ then $\left\vert f\right\vert _{B_{2,2}^{\alpha}}%
^{2}\leq C_{\alpha}\left\vert f\right\vert _{2,\alpha}^{2}$.\newline(ii)
Again, for an interval $\left(  0,1\right)  $ the claim is proved in Lemma 5.3
of \cite{GolNussbZ-specpaper-preprint}.
\end{proof}

\bigskip

Our next task is to compare the geometric regression experiment
$\mathcal{\tilde{F}}_{m}$ defined in (\ref{G-m-tilde-experi-def}) with the
basic one of (\ref{geometric-regression-experiment-basic}) involving the local
averages $J_{j,n}\left(  a\right)  $ from (\ref{local-averages-def}) for
$m=n$. We now write the latter as%
\[
\mathcal{F}_{n}=\left(
{\displaystyle\bigotimes\limits_{j=1}^{n}}
\mathrm{Geo}\left(  p\left(  J_{j,n}\left(  a\right)  \right)  \right)
,a\in\Theta_{1}\left(  \alpha,M\right)  \right)  .
\]

\begin{lemma}
\label{lem-equiv-asy-two-geometrics}We have asymptotic equivalence, along odd
$m\rightarrow\infty$
\[
\mathcal{\tilde{F}}_{m}\approx\mathcal{F}_{m}.
\]

\end{lemma}

\begin{proof}
In view of inequalities (\ref{total-variation-Hellinger-inequ}) and
(\ref{Hellinger-sum-inequ}) it suffices to prove for the Hellinger distance
$H\left(  \cdot,\cdot\right)  $
\[
\sum_{j=1}^{m}H^{2}\left(  \mathrm{Geo}\left(  p\left(  \lambda_{j,m}\left(
a\right)  \right)  \right)  ,\mathrm{Geo}\left(  p\left(  J_{j,m}\left(
a\right)  \right)  \right)  \right)  =o\left(  1\right)
\]
uniformly over $a\in\Theta$. Using the fact that the geometric law
\textrm{Geo}$\left(  p\right)  $ coincides with the negative binomial law
\textrm{NB}$\left(  1,p\right)  $ (Appendix, Subsection \ref{subsec:neg-binom}%
) and Lemma \ref{Lem-distance-neg-binomials} (i), we obtain
\begin{equation}
H^{2}\left(  \mathrm{Geo}\left(  p\left(  \lambda_{j,m}\left(  a\right)
\right)  \right)  ,\mathrm{Geo}\left(  p\left(  J_{j,m}\left(  a\right)
\right)  \right)  \right)  \leq\frac{\left(  \lambda_{j,m}\left(  a\right)
-J_{j,m}\left(  a\right)  \right)  ^{2}}{\left(  \lambda_{j,m}\left(
a\right)  -1\right)  \left(  J_{j,m}\left(  a\right)  -1\right)  }.
\label{Hellinger-two-geometrics}%
\end{equation}
For the numerator on the r.h.s., observe that $a\in\Theta_{1}\left(
\alpha,M\right)  $ implies $a\left(  \omega\right)  \geq1+M^{-1}$, $\omega
\in\left[  -\pi,\pi\right]  $ and hence also $J_{j,m}\left(  a\right)  -1\geq
M^{-1}$, $j=1,\ldots,m$. Furthermore for $\lambda_{j,m}\left(  a\right)
=\tilde{a}_{m}\left(  \omega_{j-\left(  m+1\right)  /2,m}\right)  $ we can use
Lemma \ref{lem-Fourier-series-approx} to show that
\[
\inf_{j=1,\ldots,m}\lambda_{j,m}\left(  a\right)  -1\geq M^{-1}\left(
1+o\left(  1\right)  \right)  .
\]
It follows that
\begin{align*}
&  \sum_{j=1}^{m}H^{2}\left(  \mathrm{Geo}\left(  p\left(  \lambda
_{j,m}\left(  a\right)  \right)  \right)  ,\mathrm{Geo}\left(  p\left(
J_{j,m}\left(  a\right)  \right)  \right)  \right) \\
&  \leq\left(  1+o\left(  1\right)  \right)  M^{-2}\sum_{j=1}^{m}\left(
\lambda_{j,m}\left(  a\right)  -J_{j,m}\left(  a\right)  \right)  ^{2}.
\end{align*}
Now observe in the setting of Lemma \ref{lem-basic-approx-in-Sobolev} (ii),
the midpoints $\tilde{\omega}_{j,m}$ of the intervals $W_{j,m}$ coincide with
$\omega_{j-\left(  m+1\right)  /2,m}$, for $j=1,\ldots,m$. Now reference to
the latter result establishes the claim.
\end{proof}

\bigskip

Our next task is to compare the basic geometric regression models
$\mathcal{F}_{n}$ for different sample sizes $n$ and $m$.

\begin{proposition}
\label{prop-reduce-additional-obs}If $m=$ $n+r_{n}$, $0\leq r_{n}=o\left(
n^{1/2}\right)  $ then we have asymptotic equivalence
\[
\mathcal{F}_{n}\approx\mathcal{F}_{m}\text{ as }n\rightarrow\infty.
\]

\end{proposition}

This will follow from Lemmas \ref{lem-first-neg-binom-equivalence} --
\ref{lem-first-neg-binom-equivalence-3} below. Abbreviate $\Theta=\Theta
_{1}\left(  \alpha,M\right)  $ and introduce an experiment
\[
\mathcal{F}_{n,m}=\left(
{\displaystyle\bigotimes\limits_{j=1}^{n}}
\mathrm{NB}^{\otimes m}\left(  m^{-1},p\left(  J_{j,n}\left(  a\right)
\right)  \right)  ,a\in\Theta\right)
\]
where $\mathrm{NB}\left(  r,p\right)  $ denotes the negative binomial
distribution (see Subsection \ref{subsec:neg-binom}) and $\mathrm{NB}^{\otimes
m}\left(  r,p\right)  $ its $m$-fold product.

\begin{lemma}
\label{lem-first-neg-binom-equivalence}For any $n,m>0$ we have equivalence
\begin{equation}
\mathcal{F}_{n}\sim\mathcal{F}_{n,m}. \label{equiv-geometric-1}%
\end{equation}

\end{lemma}

\begin{proof}
Consider a parametric model of independent r.v.'s $X_{k}\sim\mathrm{NB}\left(
m^{-1},p\right)  $, $k=1,\ldots,m$, $p\in\left(  0,1\right)  $. Then, as
argued in connection with (\ref{sum-of-neg-binomials}) below, $\sum_{k=1}%
^{m}X_{k}$ is a sufficient statistic, and $\sum_{j=1}^{n}X_{i}\sim
\mathrm{Geo}\left(  p\right)  $. Consequently
\[
\left(  \mathrm{Geo}\left(  p\right)  ,\;p\in\left(  0,1\right)  \right)
\sim\left(  \mathrm{NB}^{\otimes m}\left(  m^{-1},p\right)  ,\;p\in\left(
0,1\right)  \right)  .
\]
This equivalence via sufficiency easily extends to the experiments given by
product measures
\[
\left(
{\displaystyle\bigotimes\limits_{j=1}^{n}}
\mathrm{Geo}\left(  p_{j}\right)  ,\;\left(  p_{1},\ldots,p_{n}\right)
\in\left(  0,1\right)  ^{\times n}\right)  \sim\left(
{\displaystyle\bigotimes\limits_{j=1}^{n}}
\mathrm{NB}^{\otimes m}\left(  m^{-1},p_{j}\right)  ,\;\left(  p_{1}%
,\ldots,p_{n}\right)  \in\left(  0,1\right)  ^{\times n}\right)  .
\]
The common parameter space for $\mathcal{F}_{n},\mathcal{F}_{n,m}$ can be
construed as a subspace of the one above, which implies the claim.
\end{proof}

\bigskip

Introduce an intermediate experiment%
\[
\mathcal{F}_{m,n}^{\ast}=\left(
{\displaystyle\bigotimes\limits_{j=1}^{m}}
\mathrm{NB}^{\otimes n}\left(  m^{-1},p\left(  J_{j,m}\left(  a\right)
\right)  \right)  ,a\in\Theta\right)  .
\]

\begin{lemma}
\label{lem-first-neg-binom-equivalence-2}For $m\geq n$, we have asymptotic
total variation equivalence
\[
\mathcal{F}_{m,n}^{\ast}\simeq\mathcal{F}_{n,m}\text{ as }n\rightarrow\infty.
\]

\end{lemma}

\begin{proof}
Write the measures in $\mathcal{F}_{n,m}$ as a product of $mn$ components,
i.e. as $\otimes_{j=1}^{mn}Q_{1,j}$ where the component measures $Q_{1,j}$ are
defined as follows. For every $j=1,\ldots,mn$, let $k(1,j)$ be the unique
index $k\in\left\{  1,\ldots,n\right\}  $ such that there exists $l\in\left\{
1,\ldots,m\right\}  $ for which $j=(k-1)m+l$. Then
\[
Q_{1,j}=\mathrm{NB}\left(  m^{-1},p\left(  J_{k(1,j),n}\left(  a\right)
\right)  \right)  ,\;j=1,\ldots,mn.\text{ }%
\]
Analogously, let $k(2,j)$ be the unique index $k\in\left\{  1,\ldots
,m\right\}  $ such that there exists $l\in\left\{  1,\ldots,n\right\}  $ for
which $j=(k-1)n+l$. Then the measures in $\mathcal{F}_{m,n}^{\ast}$ can be
written $\otimes_{j=1}^{mn}Q_{2,j}$ where
\[
Q_{2,j}=\mathrm{NB}\left(  m^{-1},p\left(  J_{k(2,j),m}\left(  a\right)
\right)  \right)  .
\]
The Hellinger distance between measures in $\mathcal{F}_{n,m}$ and
$\mathcal{F}_{m,n}^{\ast}$ is, using (\ref{Hellinger-sum-inequ}) and then
Lemma \ref{Lem-distance-neg-binomials} (i)
\[
H^{2}\left(
{\displaystyle\bigotimes\limits_{j=1}^{mn}}
Q_{1,j},%
{\displaystyle\bigotimes\limits_{j=1}^{mn}}
Q_{2,j}\right)  \leq2\sum_{j=1}^{mn}H^{2}\left(  Q_{1,j},Q_{2,j}\right)
\]%
\begin{equation}
\leq\frac{2}{m}\sum_{j=1}^{mn}\frac{\left(  J_{k(1,j),n}\left(  a\right)
-J_{k(2,j),m}\left(  a\right)  \right)  ^{2}}{\left(  J_{k(1,j),n}\left(
a\right)  -1\right)  \left(  J_{k(2,j),m}\left(  a\right)  -1\right)  }.
\label{neg-binom-bound-1}%
\end{equation}
Since $a\in\Theta_{1}\left(  \alpha,M\right)  $, we have $a\left(
\omega\right)  \geq1+M^{-1}$, $\omega\in\left[  -\pi,\pi\right]  $ and hence
also
\[
\inf_{j=1,\ldots,mn}\min\left(  J_{k(1,j),n}\left(  a\right)  ,J_{k(2,j),m}%
\left(  a\right)  \right)  \geq1+M^{-1}.
\]
This implies that (\ref{neg-binom-bound-1}) can be bounded by
\begin{equation}
\leq\frac{2M^{2}}{m}\sum_{j=1}^{mn}\left(  J_{k(1,j),n}\left(  a\right)
-J_{k(2,j),m}\left(  a\right)  \right)  ^{2}. \label{neg-binom-bound-2}%
\end{equation}
The expression $J_{k(1,j),n}(a)-J_{k(2,j),m}(a)$ can be described as follows.
For any $x\in\left(  \left(  j-1\right)  /mn,j/mn\right)  $, $i=1,\ldots,mn$
we have
\begin{equation}
J_{k(1,j),n}(a)-J_{k(2,j),m}(a)=\bar{a}_{n}(x)-\bar{a}_{m}(x)
\label{averages-diiferent}%
\end{equation}
where $\bar{a}_{n}$ defined by (\ref{fbar-def}). Hence
\begin{equation}
\frac{1}{mn}\sum_{j=1}^{mn}\left(  J_{k(1,j),n}\left(  a\right)
-J_{k(2,j),m}\left(  a\right)  \right)  ^{2}=\left\Vert \bar{a}_{n}-\bar
{a}_{m}\right\Vert _{2}^{2}. \label{neg-binom-bound-3}%
\end{equation}
Now as a consequence of Lemma \ref{lem-basic-approx-in-Sobolev} (i), if
$a\in\Theta_{1}\left(  \alpha,M\right)  $ and $1/2<\alpha<1$
\[
\left\Vert a-\bar{a}_{n}\right\Vert _{2}^{2}\leq C_{\alpha}\;n^{-2\alpha
}\;\left\vert a\right\vert _{2,\alpha}^{2}\leq C_{\alpha}\;n^{-2\alpha}M.
\]
If $\alpha\geq1$ then for any $\beta\in\left(  0,1\right)  $ we have
$\left\vert a\right\vert _{2,\beta}^{2}\leq\left\vert a\right\vert _{2,\alpha
}^{2}$ and so if $a\in\Theta_{1}\left(  \alpha,M\right)  $ for $\alpha>1/2$
then there exists $\beta>1/2$ such that
\[
\left\Vert a-\bar{a}_{n}\right\Vert _{2}^{2}\leq C_{\beta}\;n^{-2\beta}M.
\]
Hence generally there exists a constant $C$ such that
\begin{align*}
n\left\Vert \bar{a}_{n}-\bar{a}_{m}\right\Vert _{2}^{2}  &  \leq2n\left\Vert
\bar{a}_{n}-a\right\Vert _{2}^{2}+2n\left\Vert \bar{a}_{m}-a\right\Vert
_{2}^{2}\\
&  \leq2CM\left(  n^{1-2\beta}+nm^{-2\beta}\right)  \leq4CMn^{1-2\beta
}=o\left(  1\right)
\end{align*}
uniformly over $a\in\Theta_{1}\left(  \alpha,M\right)  $. This relation along
with (\ref{neg-binom-bound-1})-(\ref{neg-binom-bound-3}) proves that
\[
\sup_{a\in\Theta_{1}\left(  \alpha,M\right)  }H^{2}\left(
{\displaystyle\bigotimes\limits_{j=1}^{mn}}
Q_{1,j},%
{\displaystyle\bigotimes\limits_{j=1}^{mn}}
Q_{2,j}\right)  =o\left(  1\right)  .
\]
Now (\ref{total-variation-Hellinger-inequ}) establishes the claim.
\end{proof}

\bigskip

The remaining task is to compare $\mathcal{F}_{m,n}^{\ast}$ to $\mathcal{F}%
_{m}$.

\begin{lemma}
\label{lem-first-neg-binom-equivalence-3}For $m=n+r_{n}$, $0\leq
r_{n}=o\left(  n^{-1/2}\right)  $ we have asymptotic equivalence%
\[
\mathcal{F}_{m}\approx\mathcal{F}_{m,n}^{\ast}\text{ as }n\rightarrow\infty.
\]

\end{lemma}

\begin{proof}
The sufficiency argument for the negative binomial applied in Lemma
\ref{lem-first-neg-binom-equivalence} can be used to show that%
\[
\mathcal{F}_{m,n}^{\ast}\sim\mathcal{F}_{m}^{\ast}:=\left(
{\displaystyle\bigotimes\limits_{j=1}^{m}}
\mathrm{NB}\left(  nm^{-1},p\left(  J_{j,m}\left(  a\right)  \right)  \right)
,a\in\Theta\right)  .
\]
Now it suffices to show asymptotic total variation equivalence $\mathcal{F}%
_{m}^{\ast}\simeq\mathcal{F}_{m}$. Recall that $\mathrm{Geo}\left(  p\right)
=\mathrm{NB}\left(  1,p\right)  $ and note that for the Hellinger distance we
have, according to Lemma \ref{Lem-distance-neg-binomials} (ii)
\[
H^{2}\left(  \mathrm{NB}\left(  1,p\left(  J_{j,m}\left(  a\right)  \right)
\right)  ,\mathrm{NB}\left(  nm^{-1},p\left(  J_{j,m}\left(  a\right)
\right)  \right)  \right)  \leq1-\frac{\Gamma\left(  \left(  1+nm^{-1}\right)
/2\right)  }{\Gamma^{1/2}\left(  1\right)  \Gamma^{1/2}\left(  nm^{-1}\right)
}%
\]%
\[
\leq\frac{1}{\Gamma^{1/2}\left(  nm^{-1}\right)  }\left(  \Gamma^{1/2}\left(
nm^{-1}\right)  -\Gamma\left(  \left(  1+nm^{-1}\right)  /2\right)  \right)
\]
where we used $\Gamma\left(  1\right)  =1$. Since the Gamma function is
infinitely differentiable on $(0,\infty)$ and $nm^{-1}\rightarrow1$, the first
factor above is $1+o\left(  1\right)  $. Furthermore, write $n/m=1-\delta$
where $\delta=r_{n}/m$; by a Taylor expansion we obtain%
\begin{align*}
\Gamma\left(  \left(  1+n/m\right)  /2\right)   &  =\Gamma(1-\delta
/2)=1-\Gamma^{\prime}(1)\frac{\delta}{2}+O(\delta^{2}),\\
\Gamma^{1/2}\left(  n/m\right)   &  =\Gamma^{1/2}(1-\delta)=1-\frac{1}%
{2}\Gamma^{\prime}(1)\delta+O(\delta^{2}).
\end{align*}
Consequently
\begin{align*}
\left(  \Gamma^{1/2}\left(  nm^{-1}\right)  -\Gamma\left(  \left(
1+nm^{-1}\right)  /2\right)  \right)   &  =O(\delta^{2})\\
&  =O\left(  r_{n}^{2}/m^{2}\right)  .
\end{align*}
Applying (\ref{Hellinger-sum-inequ}) we find that the squared Hellinger
distance between the respective product measures in $\mathcal{F}_{m}^{\ast}$
and $\mathcal{F}_{m}$ is of order%
\[
mO\left(  r_{n}^{2}/m^{2}\right)  \leq O\left(  r_{n}^{2}/n\right)  =o\left(
1\right)
\]
in view of the condition $r_{n}=o(n^{1/2})$. Applying
(\ref{total-variation-Hellinger-inequ}) again establishes the claim
$\mathcal{F}_{m}^{\ast}\simeq\mathcal{F}_{m}$.
\end{proof}

\bigskip

\bigskip

\begin{proof}
[Proof of Theorem \ref{theor-main-1}]Let $m=m_{n}$ be a sequence of odd
numbers such that $m>n$, $m-n=o\left(  n^{1/2}\right)  $, and assume the
parameter space for all experiments is $\Theta_{1}\left(  \alpha,M\right)  $.
Then Proposition \ref{Prop:deta-upper-bound-by-circular} implies
$\mathcal{E}_{n}\precsim\mathcal{\tilde{E}}_{m}$. Lemma
\ref{lem-equiv-diagonal-symbol} implies $\mathcal{\tilde{E}}_{m}%
\sim\mathcal{\tilde{E}}_{m}^{d}$, while Proposition
\ref{Prop-equiv-diag-symbol-geometr} implies $\mathcal{\tilde{E}}_{m}^{d}%
\sim\mathcal{\tilde{F}}_{m}$ and Lemma \ref{lem-equiv-asy-two-geometrics}
states $\mathcal{\tilde{F}}_{m}\sim\mathcal{F}_{m}$. Finally Proposition
\ref{prop-reduce-additional-obs}, by stating $\mathcal{F}_{m}\approx
\mathcal{F}_{n}$, allows to return from the (odd) increased sample size $m>n$
(or number of modes) to the original $n$. Both types of equivalence $\sim$ and
$\approx$ occurring above imply the semi-ordering $\precsim$ between sequences
of experiments having the same parameter space. The reasoning can be
summarized as
\[
\mathcal{E}_{n}\precsim\mathcal{\tilde{E}}_{m}\precsim\mathcal{\tilde{E}}%
_{m}^{d}\precsim\mathcal{\tilde{F}}_{m}\precsim\mathcal{F}_{m}\precsim
\mathcal{F}_{n}.
\]
The obvious transitivity of the relation $\precsim$ implies the claim.
\end{proof}

\subsection{Geometric regression and white
noise\label{Subsec-Geom-reg-and-white-noise}}

Consider an variant of the geometric regression model
(\ref{local-averages-def}) where the local averages $J_{j,n}\left(  a\right)
$ of the spectral density $a$ are replaced by values at points
\begin{equation}
t_{j,n}=2\pi\left(  \frac{j}{n}-\frac{1}{2}\right)  \text{, }j=1,\ldots,n.
\label{equidistant-grid-def}%
\end{equation}
Accordingly define the experiment
\begin{equation}
\mathcal{F}_{n,1}\left(  \Theta\right)  :=\left(
{\displaystyle\bigotimes\limits_{j=1}^{n}}
\mathrm{Geo}\left(  p\left(  a\left(  t_{j,n}\right)  \right)  \right)
,a\in\Theta\right)  \label{geometric-regression-experiment-basic-a}%
\end{equation}
where $p\left(  x\right)  =\left(  x-1\right)  /\left(  x+1\right)  $ for
$x>1$. To introduce an appropriate class of spectral densities $a$ with this
model, define the H\"{o}lder norm for functions on $\left[  -\pi,\pi\right]
,$ with $\alpha\in\left(  0,1\right]  $
\begin{equation}
\left\Vert f\right\Vert _{C^{\alpha}}:=\left\Vert f\right\Vert _{\infty}%
+\sup_{x\neq y}\frac{\left\vert f(x)-f(y)\right\vert }{\left\vert
x-y\right\vert ^{\alpha}} \label{holder-norm}%
\end{equation}
and the corresponding H\"{o}lder class of functions
\begin{equation}
C^{\alpha}(M):=\left\{  f:\left[  -\pi,\pi\right]  \rightarrow\mathbb{R}%
,\;\;\left\Vert f\right\Vert _{C^{\alpha}}\leq M\right\}  .
\label{Hoelder-class-def}%
\end{equation}
The periodic Sobolev norm $\left\Vert \cdot\right\Vert _{2,\alpha}$ for
functions $f$ on $\left[  -\pi,\pi\right]  $ for smoothness index $\alpha>0$
is given by (\ref{normdef}). A basic embedding theorem
(\cite{GolNussbZ-specpaper-preprint}, Lemma 5.6) gives a norm inequality, for
$\alpha\in\left(  0,1\right]  $
\[
\left\Vert f\right\Vert _{C^{\alpha}}\leq C\left\Vert f\right\Vert
_{2,\alpha+1/2}^{2}%
\]
where $C$ depends only on $\alpha$. Thus, if we consider a set of spectral
densities, in analogy to (\ref{Theta-1-functionset-def-a}) and
(\ref{Theta-1-functionset-def})
\begin{align}
\Theta_{1,c}\left(  \alpha,M\right)   &  :=C^{\alpha}(M)\cap\mathcal{L}%
_{M},\label{Big-Theta-11-c-def}\\
\mathcal{L}_{M}  &  :=\left\{  f:\left[  -\pi,\pi\right]  \rightarrow
\mathbb{R},\text{ }\;f\left(  \omega\right)  \geq1+M^{-1}\text{, }\omega
\in\left[  -\pi,\pi\right]  \right\}  \label{FM-script-lowerbound-def}%
\end{align}
then we have the inclusion, for $\alpha\in\left(  0,1\right]  $%
\begin{equation}
\Theta_{1}\left(  \alpha+1/2,M\right)  \subset\Theta_{1,c}\left(
\alpha,M^{\prime}\right)  \label{inclusion-by-embedding}%
\end{equation}
for some $M^{\prime}>0$.

\begin{lemma}
\label{Lem-Geom-avg-values-at-points}If $\Theta=\Theta_{1,c}\left(
\alpha,M\right)  $ for $\alpha\in\left(  1/2,1\right]  ,$ $M>0$ then we have
asymptotic total variation equivalence
\[
\mathcal{F}_{n}\left(  \Theta\right)  \simeq\mathcal{F}_{n,1}\left(
\Theta\right)  \text{ as }n\rightarrow\infty.
\]

\end{lemma}

\begin{proof}
As with Lemma \ref{lem-equiv-asy-two-geometrics} it suffices to prove for the
Hellinger distance $H\left(  \cdot,\cdot\right)  $
\begin{equation}
\sum_{j=1}^{n}H^{2}\left(  \mathrm{Geo}\left(  p\left(  a\left(
t_{j,n}\right)  \right)  \right)  ,\mathrm{Geo}\left(  p\left(  J_{j,n}\left(
a\right)  \right)  \right)  \right)  =o\left(  1\right)
\label{Hellinger-total}%
\end{equation}
uniformly over $a\in\Theta$. According to (\ref{Hellinger-two-geometrics}) we
have%
\[
H^{2}\left(  \mathrm{Geo}\left(  p\left(  a\left(  t_{j,n}\right)  \right)
\right)  ,\mathrm{Geo}\left(  p\left(  J_{j,m}\left(  a\right)  \right)
\right)  \right)  \leq\frac{\left(  a\left(  t_{j,n}\right)  -J_{j,m}\left(
a\right)  \right)  ^{2}}{\left(  a\left(  t_{j,n}\right)  -1\right)  \left(
J_{j,m}\left(  a\right)  -1\right)  }.
\]
Here for $a\in$ $\Theta_{1,H}\left(  \alpha,M\right)  $ we have $a\left(
t_{j,n}\right)  -1\geq M^{-1}$, $J_{j,m}\left(  a\right)  -1\geq M^{-1}$,
hence%
\[
H^{2}\left(  \mathrm{Geo}\left(  p\left(  a\left(  t_{j,n}\right)  \right)
\right)  ,\mathrm{Geo}\left(  p\left(  J_{j,m}\left(  a\right)  \right)
\right)  \right)  \leq M^{2}\left(  a\left(  t_{j,n}\right)  -J_{j,m}\left(
a\right)  \right)  ^{2}.
\]
Recalling the definition of the intervals $W_{j,n}$ in
(\ref{intervals-Wjn-def}) and (\ref{fbar-def}), we obtain
\begin{align*}
\left\vert a\left(  t_{j,n}\right)  -J_{j,n}\left(  a\right)  \right\vert  &
=\frac{n}{2\pi}\left\vert \int_{W_{j,n}}\left(  f(x)-a\left(  t_{j,n}\right)
\right)  dx\right\vert \\
&  \leq M\left(  \frac{2\pi}{n}\right)  ^{\alpha}%
\end{align*}
and hence the l.h.s. of (\ref{Hellinger-total}) is bounded by
\[
\sum_{j=1}^{n}M^{2}\left(  a\left(  t_{j,n}\right)  -J_{j,m}\left(  a\right)
\right)  ^{2}\leq M^{4}\left(  2\pi\right)  ^{2\alpha}\;n^{1-2\alpha}=o\left(
1\right)  .
\]

\end{proof}

\bigskip

At this point we use notation $\mathcal{G}_{n,1}\left(  \Theta\right)
:=\mathcal{G}_{n}\left(  \Theta\right)  $ where $\mathcal{G}_{n}\left(
\Theta\right)  $ describes the experiment given by (\ref{SDE-1}), i.e. by
\begin{equation}
dY_{\omega}=\mathrm{arc\cosh}\left(  a\left(  \omega\right)  \right)
d\omega+\left(  2\pi/n\right)  ^{1/2}dW_{\omega},\omega\in\left[  -\pi
,\pi\right]  \label{SDE-1-a}%
\end{equation}
with $a\in\Theta$.

\begin{lemma}
\label{lem-geo-reg-white-noise}For $\alpha\in\left(  1/2,1\right]  ,$ $M>0$
and $\Theta=\Theta_{1,c}\left(  \alpha,M\right)  $ we have%
\[
\Delta\left(  \mathcal{F}_{n,1}\left(  \Theta\right)  ,\mathcal{G}%
_{n,1}\left(  \Theta\right)  \right)  \rightarrow0\text{ as }n\rightarrow
\infty.
\]

\end{lemma}

\bigskip%
\begin{privatenotes}
\begin{boxedminipage}{\textwidth}%

\begin{sfblock}
The following proof needs revision.
\end{sfblock}

%

\end{boxedminipage}
\end{privatenotes}%

\begin{proof}
This follows from the results of \cite{MR1633574}. Let $\left(  Q\left(
\tau\right)  ,\tau\in T\right)  $ be a one parameter exponential family where
$\tau$ is the canonical (natural) parameter and $T=\left[  t_{1},t_{2}\right]
$ is a closed interval in $\mathbb{R}$. It is assumed in \cite{MR1633574} that
$X_{i}$ are independent observations having distributions $Q\left(  f\left(
u_{i}\right)  \right)  $ where $f$ is a function $f:\left[  0,1\right]
\rightarrow T$ and $u_{i}=i/n$, $i=1,\ldots,n$. The following regularity
condition is assumed: $T$ is in the interior of the natural parameter space of
the exponential family, and there exist $\varepsilon>0$ and constants
$C_{1},C_{2}$ \ such the Fisher information $I\left(  \tau\right)
$\ fulfills
\begin{equation}
0<C_{1}\leq I\left(  \tau\right)  \leq C_{2}<\infty,\;\tau\in\left[
t_{1}-\varepsilon,t_{2}+\varepsilon\right]  . \label{reg-cond-expon-fam}%
\end{equation}
According to Subsection \ref{subsec:geometr-distr}, the geometric
distributions have densities (with respect to counting measure $\mu$ on
$\mathbb{Z}_{+}$) which can be written as those of an exponential family of
densities in canonical form, cf. (\ref{canon-from}):%
\[
Q\left(  \tau\right)  \left(  x\right)  =\exp\left(  \tau x-V\left(
\tau\right)  \right)  \text{, }x\in\mathbb{Z}_{+}%
\]
where $\tau=\log p$ and $P\left(  X=x\right)  =\left(  1-p\right)  p^{x}$. In
our setting, $\tau$ will be parametrized according to (\ref{tau-as-func-of-a})
as $\tau=\log\left(  \left(  a-1\right)  /\left(  a+1\right)  \right)  $, so
if $a\in\left[  1+M^{-1},M\right]  $ for some $M>2$ then $\tau\in\left[
t_{1},t_{2}\right]  $ \ for some $t_{1},t_{2}$\ fulfilling $-\infty
<t_{1}<t_{2}<0$. For the Fisher information $I\left(  \tau\right)  $ we have
according to (\ref{Var-geo})
\begin{equation}
I\left(  \tau\right)  =\frac{\exp\tau}{\left(  1-\exp\tau\right)  ^{2}}
\label{Fish-info-expon-fam-proof}%
\end{equation}
such that (\ref{reg-cond-expon-fam}) is fulfilled for sufficiently small
$\varepsilon$.

In \cite{MR1633574} the function $f$ is defined on $[0,1]$ and assumed to vary
in a smoothness class $C_{1}^{\alpha}(M)$, defined as the analog of
$C^{\alpha}(M)$ from (\ref{Hoelder-class-def}) on the interval $[0,1]$. It is
easy to see that the experiment $\mathcal{G}_{n,1}\left(  \Theta\right)  $ can
be cast in this form. Indeed define functions
\begin{align*}
H\left(  a\right)   &  =\log\frac{a-1}{a+1}\text{ for }a\in\left[
1+M^{-1},M\right]  ,\\
s\left(  x\right)   &  =2\pi\left(  x-1/2\right)  .
\end{align*}
Note that $s\left(  u_{j}\right)  =s\left(  j/n\right)  =t_{j,n}$,
$j=1,\ldots,n.$ Thus in $\mathcal{G}_{n}^{\prime}\left(  \Theta\right)  $
observations $X_{j}$ are independent with distribution
\[
X_{j}\sim Q\left(  H\left(  a\left(  s\left(  j/n\right)  \right)  \right)
\right)  \text{, }j=1,\ldots,n.
\]
Setting
\[
f\left(  x\right)  =H\left(  a\left(  s\left(  x\right)  \right)  \right)  ,
\]
we see that observations$\ X_{j}$ \ are of the type considered in
\cite{MR1633574}, with points of the "regression design" $u_{j}=j/n$. The
results of \cite{MR1633574} now hold provided the function $f$ \ is in a class
$C_{1}^{\alpha}(M^{\prime})$ for some $\alpha>1/2$, $M^{\prime}>0$ and takes
values in the interval $T$. Since the function $H$ has bounded derivative on
$\left[  1+M^{-1},M\right]  $ (cf. (\ref{note-various-1})) and $s$ is linear,
the first condition can easily be checked for the given $\alpha$ \ and
\[
M^{\prime}=M\left(  2\pi\right)  ^{\alpha}\sup_{z\in\left[  1+M^{-1},M\right]
}H^{\prime}\left(  z\right)  .
\]
Also, since $H$ is strictly increasing on $\left[  1+M^{-1},M\right]  $ (cf.
(\ref{note-various-1})), the function $f$ takes values in $T=\left[
t_{1},t_{2}\right]  $ with $t_{1}=H\left(  1+M^{-1}\right)  $, $t_{2}=H\left(
M\right)  $, $t_{1}<t_{2}<0$. Thus the experiment
\[
\left(
{\displaystyle\bigotimes\limits_{j=1}^{n}}
Q\left(  f\left(  j/n\right)  \right)  \text{, }f=H\circ a\circ s\text{, }%
a\in\Theta_{1,c}\left(  \alpha,M\right)  \right)
\]
can be approximated in $\Delta$-distance by the white noise model \
\begin{equation}
dZ_{x}=G\left(  f\left(  x\right)  \right)  dx+n^{-1/2}dW_{x}\text{, }%
x\in\left[  0,1\right]  \label{SDE-experi}%
\end{equation}
where $f=H\circ a\circ s$, the function $a$ varies in $\Theta_{1,c}\left(
\alpha,M\right)  $ and $G$ is the variance stabilizing transform pertaining to
the exponential family $\left(  Q\left(  \tau\right)  ,\tau\in T\right)  $
(cf. Section 3.3 of \cite{MR1633574} or Remark 3.3 in \cite{MR1900972}). Here
$G$ is unique up to additive constants; $G$ fulfills
\[
\frac{d}{d\tau}G\left(  \tau\right)  =\sqrt{I\left(  \tau\right)  }%
\]
with $I\left(  \tau\right)  $ given by (\ref{Fish-info-expon-fam-proof}).
Finding the function $G$ is equivalent to finding the function
\[
g\left(  a\right)  :=G\left(  H\left(  a\right)  \right)  \text{, }a\in\left(
1,\infty\right)  .
\]
We have
\begin{align*}
\frac{d}{da}g\left(  a\right)   &  =G^{\prime}\left(  H\left(  a\right)
\right)  \;H^{\prime}\left(  a\right) \\
&  =\sqrt{I\left(  H\left(  a\right)  \right)  }\;H^{\prime}\left(  a\right)
.
\end{align*}
By (\ref{Var-geo}) and (\ref{note-various-2})
\[
I\left(  H\left(  a\right)  \right)  =V^{\prime\prime}\left(  H\left(
a\right)  \right)  =\frac{a^{2}-1}{4}%
\]
whereas by (\ref{note-various-1})
\[
H^{\prime}\left(  a\right)  =\frac{2}{a^{2}-1}.
\]
Thus $g$ must fulfill
\begin{equation}
\frac{d}{da}g\left(  a\right)  =\frac{\sqrt{a^{2}-1}}{2}\frac{2}{a^{2}%
-1}=\frac{1}{\sqrt{a^{2}-1}}. \label{defining-equation-varstab}%
\end{equation}
It can be checked that the function%
\[
g(x)=\mathrm{arc\cosh}\left(  x\right)  =\log\left(  x+\sqrt{x^{2}-1}\right)
,x>1
\]
fulfills (\ref{defining-equation-varstab}). From (\ref{SDE-experi}) we obtain
that the experiment given by $Z=\left\{  Z_{x},x\in\left[  0,1\right]
\right\}  $ with
\begin{equation}
dZ_{x}=g\left(  a\left(  s\left(  x\right)  \right)  \right)  dx+n^{-1/2}%
dW_{x}\text{, }x\in\left[  0,1\right]  \label{white-noise-variant-1}%
\end{equation}
and $a\in\Theta=\Theta_{1,c}\left(  \alpha,M\right)  $ is asymptotically
equivalent to $\mathcal{F}_{n,1}\left(  \Theta\right)  $. Define the
stochastic process $Y=\left\{  Y_{\omega},\omega\in\left[  -\pi,\pi\right]
\right\}  $ by $Y_{\omega}=2\pi Z_{s^{-1}\left(  \omega\right)  }$; then $Y$
satisfies
\begin{equation}
dY_{\omega}=g\left(  a\left(  \omega\right)  \right)  d\omega+\left(
2\pi/n\right)  ^{1/2}dW_{\omega}\text{, }\omega\in\left[  -\pi,\pi\right]
\label{white-noise-variant-1-a}%
\end{equation}
so that according to (\ref{SDE-1}), $Y$ has distribution $Q_{n}\left(
a\right)  $. The claim now follows from the fact that the mapping between the
processes $Y$ and $Z$ is one-to-one.
\end{proof}

\bigskip

\begin{proof}
[Proof of Theorem \ref{theor-GWN-main-1}]Consider the experiment
$\mathcal{F}_{n}\left(  \Theta\right)  $ for $\Theta=\Theta_{1}\left(
\alpha,M\right)  $ where $\alpha>1$. By relation (\ref{inclusion-by-embedding}%
) one has $\Theta_{1}\left(  \alpha,M\right)  \subset\Theta_{1,c}\left(
\alpha-1/2,M^{\prime}\right)  $. From Lemma
\ref{Lem-Geom-avg-values-at-points} it then follows that $\mathcal{F}%
_{n}\left(  \Theta\right)  \approx\mathcal{F}_{n,1}\left(  \Theta\right)  $,
and Lemma \ref{lem-geo-reg-white-noise} implies that $\mathcal{F}_{n,1}\left(
\Theta\right)  \approx\mathcal{G}_{n}\left(  \Theta\right)  $. By the
transitivity of the equivalence relation $\approx$ for sequences of
experiments, one has $\mathcal{F}_{n}\left(  \Theta\right)  \approx
\mathcal{G}_{n}\left(  \Theta\right)  $ as claimed.
\end{proof}

\bigskip

For later reference we note a localized version of the white noise model
(\ref{SDE-1}), where $a$ itself appears as the drift function rather than the
$\mathrm{arc\cosh}$-transformation, but the approximation holds in a
neighborhood of a fixed function $a_{0}\in\Theta_{1,c}\left(  \alpha,M\right)
$. This will be the analog of the localized white noise approximation
(\ref{gwn-spec-density-local}) for the classical stationary Gaussian process.
Define for some sequence $\gamma_{n}=o\left(  1\right)  $%
\begin{equation}
B\left(  a_{0},\gamma_{n}\right)  =\left\{  a:\left[  -\pi,\pi\right]
\rightarrow\mathbb{R}\text{, }\left\Vert a-a_{0}\right\Vert _{\infty}%
\leq\gamma_{n}\right\}  \label{neighborhood-uniform-B-def}%
\end{equation}
and consider restricted function sets%
\begin{equation}
\Theta_{1,c}\left(  \alpha,M\right)  \cap B\left(  a_{0},\gamma_{n}\right)  .
\label{NP-shrinking-neighborhood-def}%
\end{equation}
Furthermore let $Q_{n,2}\left(  a,a_{\left(  0\right)  }\right)  $ be the
distribution of the process $Y=\left\{  Y_{\omega},\omega\in\left[  -\pi
,\pi\right]  \right\}  $ described by
\begin{equation}
dY_{\omega}=a\left(  \omega\right)  d\omega+\left(  2\pi/n\right)
^{1/2}\left(  a_{0}^{2}\left(  \omega\right)  -1\right)  ^{1/2}dW_{\omega
}\text{, }\omega\in\left[  -\pi,\pi\right]  \label{pre-variance-stable-SDE}%
\end{equation}
and $Y_{\omega}=\int_{-\pi}^{\omega}dY_{\omega}$, and define the experiment
\begin{equation}
\mathcal{G}_{n,2}\left(  a_{0%
\acute{}%
},\Theta\right)  :=\left(  Q_{n}\left(  a,a_{0}\right)  ,\;a\in\Theta\right)
. \label{pre-variance-stable-SDE-experi}%
\end{equation}
Recall that $\mathcal{G}_{n,1}\left(  \Theta\right)  $ is defined by
(\ref{SDE-1-a}).

\begin{lemma}
\label{lem-local-white-noise} Assume $\alpha\in\left(  1/2,1\right]  $. Then
for every sequence $\gamma_{n}=o\left(  \left(  n/\log n\right)
^{-\alpha/(2\alpha+1)}\right)  $ and $\Theta_{n}=\Theta_{1,c}\left(
\alpha,M\right)  \cap B\left(  a_{0},\gamma_{n}\right)  $ one has
\[
\sup_{a_{0}\in\Theta_{1,c}\left(  \alpha,M\right)  }\Delta\left(
\mathcal{G}_{n,1}\left(  \Theta_{n}\right)  ,\mathcal{G}_{n,2}\left(
a_{0},\Theta_{n}\right)  \right)  \rightarrow0\text{ as }n\rightarrow\infty.
\]

\end{lemma}

\begin{proof}
This is essentially Theorem 3.3 in \cite{MR1633574}, specialized to the
present exponential family, i.e. the geometric distribution. The white noise
model (3.8) in \cite{MR1633574} corresponds to (\ref{pre-variance-stable-SDE}%
), and the variance-stable white noise model (3.15) in \cite{MR1633574}
corresponds to (\ref{SDE-1-a}). The models in \cite{MR1633574} are defined on
the unit interval, but the result carries over $\left[  -\pi,\pi\right]  $ in
the same way as has been noted with processes (\ref{white-noise-variant-1})
and (\ref{white-noise-variant-1-a}).
\end{proof}

\section{Lower informativity bound}

\subsection{Constructing the basic observables}

In this section we assume $n$ is an odd number. Consider the creation and
annihilation operators $\hat{A}_{j}=\frac{1}{\sqrt{2}}\left(  Q_{j}%
+iP_{j}\right)  $, $\hat{A}_{j}^{\ast}=\frac{1}{\sqrt{2}}\left(  Q_{j}%
-iP_{j}\right)  $. As a consequence of (\ref{comm-rela}), these fulfill the
commutation relations%
\begin{align}
\left[  \hat{A}_{j},\hat{A}_{j}^{\ast}\right]   &  =\mathbf{1}%
,\;\;\;\;j=1,\ldots,n\label{comm-rela-A-1}\\
\;\left[  \hat{A}_{j},\hat{A}_{k}^{\ast}\right]   &  =\left[  \hat{A}_{j}%
,\hat{A}_{k}\right]  =\left[  \hat{A}_{j}^{\ast},\hat{A}_{k}^{\ast}\right]
=0,\;\;j,k=1,\ldots,n,j\neq k, \label{comm-rela-A-2}%
\end{align}
Furthermore
\begin{equation}
\hat{A}_{j}^{\ast}\hat{A}_{j}=\frac{1}{2}\left(  Q_{j}^{2}+P_{j}%
^{2}-\mathbf{1}\right)  \label{numb-oper}%
\end{equation}
are the number operators. Thus $\hat{A}_{j}^{\ast}\hat{A}_{j}$, $\;j=1,\ldots
,n$ is a commuting set of observables; the following lemma describes the first
and second moment properties of this set.

\begin{lemma}
\label{lem-covmatrix-quasi-obs}Let $\rho=\mathfrak{N}_{n}\left(  0,A\right)  $
for a symbol matrix $A=\left(  a_{jk}\right)  _{j,k=1}^{n}$ fulfilling $A>I$
(not necessarily Toeplitz). Then we have for $\;j,k=1,\ldots,n$\newline(i)
\[
\left\langle \hat{A}_{j}^{\ast}\hat{A}_{k}\right\rangle _{\rho}=Tr\left[
\hat{A}_{j}^{\ast}\hat{A}_{k}\rho\right]  =\left\{
\begin{array}
[c]{c}%
\frac{1}{2}\left(  a_{jj}-1\right)  \text{ if }j=k\\
\frac{1}{2}a_{kj}\text{, }j\neq k
\end{array}
\right.  \text{ }%
\]
(ii)%
\[
\mathrm{Cov}_{\rho}\left(  \hat{A}_{j}^{\ast}\hat{A}_{j},\hat{A}_{k}^{\ast
}\hat{A}_{k}\right)  =\left\{
\begin{array}
[c]{c}%
\frac{1}{4}\left(  a_{jj}^{2}-1\right)  \text{ if }j=k\\
\frac{1}{4}\left\vert a_{jk}\right\vert ^{2}\text{, }j<k.
\end{array}
\right\vert \text{ }%
\]

\end{lemma}

\begin{proof}
\textbf{(i) }Consider first the case $j=k$. Then $\hat{A}_{j}^{\ast}\hat
{A}_{j}$ is the number operator of the $j$-th mode, and its distribution under
$\rho$ is the same as under the marginal state of the $j$-th mode, $\rho
_{(j)}$ say, i.e. the partial trace of $\rho$ when all other modes are traced
out. By a reasoning analogous to Subsection \ref{subsec:partial-trace}, it
follows that $\rho_{(j)}=\mathfrak{N}_{1}\left(  0,a_{jj}\right)  $, which
according to (\ref{symbol-related-to-covmatrix}) and
(\ref{symbol-related-to-covmatrix-2}) can also be described as $\varphi\left(
0,\Sigma\right)  $ for $\Sigma=\frac{1}{2}a_{jj}I_{2}$. Thus $\rho_{(j)}$ is
the thermal state with covariance matrix $\frac{1}{2}a_{jj}I_{2}$ (cp. also
(\ref{thermal-state-repre-in-Fock})), where the number operator has a
geometric distribution:
\begin{equation}
\hat{A}_{j}^{\ast}\hat{A}_{j}\sim\mathrm{Geo}\left(  p\right)  ,\mathbb{\;\;}%
p=\left(  a_{jj}-1\right)  /\left(  a_{jj}+1\right)  .\text{ }\mathbb{\;}
\label{AstarA-Geo}%
\end{equation}
The expectation is (cf. Subsection \ref{subsec:geometr-distr})%
\begin{equation}
\left\langle \hat{A}_{j}^{\ast}\hat{A}_{j}\right\rangle _{\rho}=\frac{p}%
{1-p}=\frac{a_{jj}-1}{2} \label{expec-AstarA}%
\end{equation}
which proves the claim for $j=k$. For $j\neq k$
\begin{align}
Tr\left[  \hat{A}_{j}^{\ast}\hat{A}_{k}\rho\right]   &  =\frac{1}%
{2}\left\langle \left(  Q_{j}-iP_{j}\right)  \left(  Q_{k}+iP_{k}\right)
\right\rangle _{\rho}\nonumber\\
&  =\frac{1}{2}\left(  \left\langle Q_{j}Q_{k}\right\rangle _{\rho
}+\left\langle P_{j}P_{k}\right\rangle _{\rho}+i\left\langle Q_{j}%
P_{k}\right\rangle _{\rho}-i\left\langle P_{j}Q_{k}\right\rangle _{\rho
}\right)  . \label{split-covariance}%
\end{align}
Consider the marginal state $\rho_{\left(  j,k\right)  }$ of $\rho$ where all
modes except $j$ and $k$ are traced out. Again, by a reasoning analogous to
Subsection \ref{subsec:partial-trace}, it follows that $\rho_{(j,k)}%
=\mathfrak{N}_{2}\left(  0,A_{\left(  j,k\right)  }\right)  $ where
$A_{\left(  j,k\right)  }$ is the submatrix of $A$
\[
A_{\left(  j,k\right)  }=\left(
\begin{array}
[c]{cc}%
a_{jj} & a_{jk}\\
a_{kj} & a_{kk}%
\end{array}
\right)  .
\]
According to (\ref{symbol-related-to-covmatrix}), the covariance matrix of
$\rho_{(j,k)}$ is
\begin{align}
\Sigma\left(  A_{\left(  j,k\right)  }\right)   &  =\frac{1}{2}\left(
\begin{array}
[c]{cc}%
\operatorname{Re}A_{\left(  j,k\right)  } & -\operatorname{Im}A_{\left(
j,k\right)  }\\
\operatorname{Im}A_{\left(  j,k\right)  } & \operatorname{Re}A_{\left(
j,k\right)  }%
\end{array}
\right) \nonumber\\
&  =\frac{1}{2}\left(
\begin{array}
[c]{cc}%
\begin{array}
[c]{cc}%
a_{jj} & \operatorname{Re}a_{jk}\\
\operatorname{Re}a_{jk} & a_{kk}%
\end{array}
&
\begin{array}
[c]{cc}%
0 & -\operatorname{Im}a_{jk}\\
\operatorname{Im}a_{jk} & 0
\end{array}
\\%
\begin{array}
[c]{cc}%
0 & \operatorname{Im}a_{jk}\\
-\operatorname{Im}a_{jk} & 0
\end{array}
&
\begin{array}
[c]{cc}%
a_{jj} & \operatorname{Re}a_{jk}\\
\operatorname{Re}a_{jk} & a_{kk}%
\end{array}
\end{array}
\right)  . \label{big-cov-matrix}%
\end{align}
Since this covariance matrix pertains to the vector of observables
$\mathbf{R}=\left(  Q_{j},Q_{k},P_{j},P_{k}\right)  $ in the sense that
$\mathbf{R}x\sim N\left(  0,\left\langle x,\Sigma\left(  A_{\left(
j,k\right)  }\right)  x\right\rangle \right)  $ (cp.
(\ref{observables-normal-distr}), we can directly read off the covariances:%
\begin{align*}
\left\langle Q_{j}Q_{k}\right\rangle _{\rho}  &  =\left\langle P_{j}%
P_{k}\right\rangle _{\rho}=\frac{1}{2}\operatorname{Re}a_{jk},\\
\left\langle Q_{j}P_{k}\right\rangle _{\rho}  &  =-\frac{1}{2}%
\operatorname{Im}a_{jk},\;\left\langle P_{j}Q_{k}\right\rangle _{\rho}%
=\frac{1}{2}\operatorname{Im}a_{jk}.
\end{align*}
From (\ref{split-covariance}) we obtain%
\[
Tr\left[  \hat{A}_{j}^{\ast}\hat{A}_{k}\rho\right]  =\frac{1}{2}%
\operatorname{Re}a_{jk}-\frac{1}{2}i\operatorname{Im}a_{jk}=\frac{1}{2}\bar
{a}_{jk}=\frac{1}{2}a_{kj}.
\]

\textbf{(ii) }Consider first the case $j<k$. Then in view of (\ref{numb-oper}%
)
\[
\hat{A}_{j}^{\ast}\hat{A}_{j}\hat{A}_{k}^{\ast}\hat{A}_{k}=\frac{1}{4}\left(
Q_{j}^{2}+P_{j}^{2}-\mathbf{1}\right)  \left(  Q_{k}^{2}+P_{k}^{2}%
-\mathbf{1}\right)  ,
\]
hence
\begin{align}
4\cdot\hat{A}_{j}^{\ast}\hat{A}_{j}\hat{A}_{k}^{\ast}\hat{A}_{k}  &
=Q_{j}^{2}Q_{k}^{2}+Q_{j}^{2}P_{k}^{2}-Q_{j}^{2}\nonumber\\
&  +P_{j}^{2}Q_{k}^{2}+P_{j}^{2}P_{k}^{2}-P_{j}^{2}\nonumber\\
&  -Q_{k}^{2}-P_{k}^{2}+\mathbf{1}. \label{cross-prod-1}%
\end{align}
Note that on the r.h.s. above, each summand $Q_{j}^{2}Q_{k}^{2}$, $Q_{j}%
^{2}P_{k}^{2}$ etc. contains only commuting observables, which thus have a
joint distribition. In view of (\ref{big-cov-matrix}), the joint distribution
of $Q_{j},Q_{k}$ is
\[
\left(  Q_{j},Q_{k}\right)  \sim N_{2}\left(  0,\frac{1}{2}\operatorname{Re}%
A_{(j,k)}\right)  .
\]
From formula (\ref{covform-1}) in Subsection \ref{subsec: cov} we obtain
\[
\left\langle Q_{j}^{2}Q_{k}^{2}\right\rangle _{\rho}=\frac{1}{2}\left(
\operatorname{Re}a_{jk}\right)  ^{2}+\frac{1}{4}a_{jj}a_{kk}.
\]
Similarly
\begin{align*}
\left(  Q_{j},P_{k}\right)   &  \sim N_{2}\left(  0,\frac{1}{2}\left(
\begin{array}
[c]{cc}%
a_{jj} & -\operatorname{Im}a_{jk}\\
-\operatorname{Im}a_{jk} & a_{kk}%
\end{array}
\right)  \right)  ,\\
\left\langle Q_{j}^{2}P_{k}^{2}\right\rangle _{\rho}  &  =\frac{1}{2}\left(
\operatorname{Im}a_{jk}\right)  ^{2}+\frac{1}{4}a_{jj}a_{kk},
\end{align*}%
\begin{align*}
\left\langle P_{j}^{2}Q_{k}^{2}\right\rangle _{\rho}  &  =\frac{1}{2}\left(
\operatorname{Im}a_{jk}\right)  ^{2}+\frac{1}{4}a_{jj}a_{kk},\\
\left\langle P_{j}^{2}P_{k}^{2}\right\rangle _{\rho}  &  =\frac{1}{2}\left(
\operatorname{Re}a_{jk}\right)  ^{2}+\frac{1}{4}a_{jj}a_{kk}.
\end{align*}
Furthermore
\[
\left\langle Q_{j}^{2}+P_{j}^{2}+Q_{k}^{2}+P_{k}^{2}\right\rangle _{\rho
}=a_{jj}+a_{kk}.
\]
Collecting terms in (\ref{cross-prod-1}), we obtain
\[
4\cdot\left\langle \hat{A}_{j}^{\ast}\hat{A}_{j}\hat{A}_{k}^{\ast}\hat{A}%
_{k}\right\rangle _{\rho}=\left\vert a_{jk}\right\vert ^{2}+a_{jj}%
a_{kk}-\left(  a_{jj}+a_{kk}\right)  +1.
\]
Also from (\ref{expec-AstarA})%
\[
\left\langle \hat{A}_{j}^{\ast}\hat{A}_{j}\right\rangle _{\rho}=\frac{1}%
{2}\left(  a_{jj}-1\right)  \text{, }\left\langle \hat{A}_{k}^{\ast}\hat
{A}_{k}\right\rangle _{\rho}=\frac{1}{2}\left(  a_{kk}-1\right)
\]
hence
\begin{align*}
4\cdot\mathrm{Cov}_{\rho}\left(  \hat{A}_{j}^{\ast}\hat{A}_{j},\hat{A}%
_{k}^{\ast}\hat{A}_{k}\right)   &  =4\cdot\left(  \left\langle \hat{A}%
_{j}^{\ast}\hat{A}_{j}\hat{A}_{k}^{\ast}\hat{A}_{k}\right\rangle _{\rho
}-\left\langle \hat{A}_{j}^{\ast}\hat{A}_{j}\right\rangle _{\rho}%
\cdot\left\langle \hat{A}_{k}^{\ast}\hat{A}_{k}\right\rangle _{\rho}\right) \\
&  =\left\vert a_{jk}\right\vert ^{2}+a_{jj}a_{kk}-\left(  a_{jj}%
+a_{kk}\right)  +1-\left(  a_{jj}-1\right)  \left(  a_{kk}-1\right) \\
&  =\left\vert a_{jk}\right\vert ^{2}%
\end{align*}
which proves the claim for $j<k$. For $j=k$, according to relation
(\ref{AstarA-Geo}) and the formula for the variance of the geometric
(\ref{Var-geo}) we have
\begin{align*}
\mathrm{Var}_{\rho}\left(  \hat{A}_{j}^{\ast}\hat{A}_{j}\right)   &  =\frac
{p}{\left(  1-p\right)  ^{2}}=\frac{a_{jj}-1}{a_{jj}+1}\frac{\left(
a_{jj}+1\right)  ^{2}}{4}\\
&  =\frac{1}{4}\left(  a_{jj}^{2}-1\right)  .
\end{align*}

\end{proof}

\bigskip

We note the following consequence of Lemma \ref{lem-covmatrix-quasi-obs}:
\begin{align}
\left\langle \hat{A}_{j}\hat{A}_{j}^{\ast}\right\rangle _{\rho}  &
=\left\langle \hat{A}_{j}^{\ast}\hat{A}_{j}+\mathbf{1}\right\rangle _{\rho
}=\left\langle \hat{A}_{j}^{\ast}\hat{A}_{j}\right\rangle _{\rho}+1=\frac
{1}{2}\left(  a_{jj}+1\right)  ,\label{rewrite-cov-matrix-1}\\
\left\langle \hat{A}_{j}\hat{A}_{k}^{\ast}\right\rangle _{\rho}  &
=\left\langle \hat{A}_{k}^{\ast}\hat{A}_{j}\right\rangle _{\rho}=\frac{1}%
{2}a_{jk}\text{ for }j\neq k. \label{rewrite-cov-matrix-2}%
\end{align}
Define vectors of operators%
\begin{align*}
\mathbf{\hat{A}}  &  =\left(
\begin{array}
[c]{c}%
\hat{A}_{1}\\
\ldots\\
\hat{A}_{n}%
\end{array}
\right)  ,\\
\mathbf{\hat{A}}^{\dag}  &  \mathbf{=}\left(  \hat{A}_{1}^{\ast},\ldots
,\hat{A}_{n}^{\ast}\right)  .
\end{align*}
For a matrix of operators $\mathbf{C}=\left(  C_{jk}\right)  $, introduce
notation $\left\langle \mathbf{C}\right\rangle _{\rho}=\left(  \left\langle
C_{jk}\right\rangle _{\rho}\right)  $. Then (\ref{rewrite-cov-matrix-1}),
(\ref{rewrite-cov-matrix-2}) can be written
\begin{equation}
\left\langle \mathbf{\hat{A}\hat{A}}^{\dag}\right\rangle _{\rho}=\frac{1}%
{2}\left(  A+I_{n}\right)  . \label{cov-matrix-quasi-obs}%
\end{equation}
For the special unitary $U_{n}$ from (\ref{special-unitary-DFT-def}) we set
\begin{equation}
\mathbf{\hat{B}}=U_{n}^{\ast}\mathbf{\hat{A}}\text{, }\mathbf{\hat{B}}^{\dag
}\mathbf{=\hat{A}}^{\dag}U_{n}. \label{transformed-number-ops}%
\end{equation}
It then follows that
\begin{equation}
\left\langle \mathbf{\hat{B}\mathbf{\hat{B}}^{\dag}}\right\rangle _{\rho
}=\frac{1}{2}\left(  U_{n}^{\ast}AU_{n}+I_{n}\right)  .
\label{transformed-number-ops-expec}%
\end{equation}

Since $\mathbf{\hat{B}}$ represents a discrete Fourier transform of the
creation operators, for the components of the vector $\mathbf{\hat{B}}$ we
adopt the indexing convention $\mathbf{\hat{B}=}\left(  \hat{B}_{j}\right)
_{\left\vert j\right\vert \leq\left(  n-1\right)  /2}$. This is in agreement
with the form of the unitary $U_{n}$ in (\ref{special-unitary-DFT-def}); we
then obtain for the components of the vector $\mathbf{\hat{B}}=U_{n}^{\ast
}\mathbf{\hat{A}}$
\[
\hat{B}_{j}=\mathbf{u}_{j}^{\ast}\mathbf{\hat{A}}\text{, }\left\vert
j\right\vert \leq\left(  n-1\right)  /2.
\]

\begin{lemma}
\label{lem-commut-relations-B}The set of operators $\hat{B}_{j}$, $\left\vert
j\right\vert \leq\left(  n-1\right)  /2$ fulfills commutation relations
(\ref{comm-rela-A-1}), (\ref{comm-rela-A-2}) with $\hat{A}_{j}$ replaced by
$\hat{B}_{j-\left(  n+1\right)  /2}$.
\end{lemma}

\begin{proof}
Relations (\ref{comm-rela-A-1}), (\ref{comm-rela-A-2}) can be expressed in
concise form as follows: for any $c,d\in\mathbb{C}^{n}$ and $c^{\ast
}\mathbf{\hat{A}=}\sum_{j=1}^{n}\bar{c}_{j}\hat{A}_{j}$, $\mathbf{\hat{A}%
}^{\dag}d=\sum_{j=1}^{n}d_{j}\hat{A}_{j}^{\ast}$ we have
\[
\left[  c^{\ast}\mathbf{\hat{A},\hat{A}}^{\dag}d\right]  =\left\langle
c,d\right\rangle \;\mathbf{1}.
\]
Now with definitions (\ref{transformed-number-ops}) we have indeed
\[
\left[  c^{\ast}\mathbf{\hat{B},\mathbf{\hat{B}}^{\dag}}d\right]  =\left[
c^{\ast}U_{n}^{\ast}\mathbf{\hat{A},\mathbf{\hat{A}}^{\dag}}U_{n}d\right]
=\left\langle U_{n}c,U_{n}d\right\rangle \;\mathbf{1}=\left\langle
c,d\right\rangle \;\mathbf{1.}%
\]

\end{proof}

\begin{lemma}
$\hat{B}_{j}^{\ast}\hat{B}_{j}$, $\left\vert j\right\vert \leq\left(
n-1\right)  /2$ is a commuting set of observables, fulfilling
\begin{equation}
\hat{B}_{j}^{\ast}\hat{B}_{j}=\hat{B}_{j}\hat{B}_{j}^{\ast}-\mathbf{1.}
\label{B-and-B-star}%
\end{equation}
.
\end{lemma}

\begin{proof}
The first claim follows from (\ref{comm-rela-A-2}) and the previous lemma. The
claimed equality follows from (\ref{comm-rela-A-1}) applied to $\hat{B}_{j}$,
$\hat{B}_{j}^{\ast}$.
\end{proof}

\begin{lemma}
\label{lem-covmatrix-analog-B}Assume the conditions of Lemma
\ref{lem-covmatrix-quasi-obs}. Then we have for $\left\vert j\right\vert
,\left\vert k\right\vert \leq\left(  n-1\right)  /2\;$\newline(i)
\[
\left\langle \hat{B}_{j}^{\ast}\hat{B}_{j}\right\rangle _{\rho}=\frac{1}%
{2}\left(  \mathbf{u}_{j}^{\ast}A\mathbf{u}_{j}-1\right)  ,
\]
(ii)%
\[
\mathrm{Cov}_{\rho}\left(  \hat{B}_{j}^{\ast}\hat{B}_{j},\hat{B}_{k}^{\ast
}\hat{B}_{k}\right)  =\left\{
\begin{array}
[c]{c}%
\frac{1}{4}\left(  \left(  \mathbf{u}_{j}^{\ast}A\mathbf{u}_{j}\right)
^{2}-1\right)  \text{ if }j=k\\
\frac{1}{4}\left\vert \mathbf{u}_{j}^{\ast}A\mathbf{u}_{k}\right\vert
^{2}\text{, }j<k.
\end{array}
\right\vert \text{ }%
\]

\end{lemma}

\begin{proof}
For (i), we note that (\ref{transformed-number-ops-expec}) implies%
\[
\left\langle \hat{B}_{j}\hat{B}_{j}^{\ast}\right\rangle _{\rho}=\frac{1}%
{2}\left(  \mathbf{u}_{j}^{\ast}A\mathbf{u}_{j}^{\ast}+I_{n}\right)  .
\]
so that the claim follows from (\ref{B-and-B-star}). For (ii), note that this
claim can be formulated as: if in Lemma \ref{lem-covmatrix-quasi-obs} the
$\hat{A}_{j}$ are replaced by $\hat{B}_{j}$ then the assertion (ii) holds with
the matrix $A$ replaced by $U_{n}^{\ast}AU_{n}$. Define a set of observables
$\tilde{Q}_{j},\tilde{P}_{j}$, $j=1,\ldots,n$ by
\begin{equation}
\tilde{Q}_{j-(n+1)/2}=\frac{1}{\sqrt{2}}\left(  \hat{B}_{j}+\hat{B}_{j}^{\ast
}\right)  \text{, }\tilde{P}_{j-(n+1)/2}=\frac{1}{i\sqrt{2}}\left(  \hat
{B}_{j}-\hat{B}_{j}^{\ast}\right)  . \label{analogs-P-Q-def}%
\end{equation}
These are related to $\hat{B}_{j}$ and $\hat{B}_{j}^{\ast}$ in the same way as
the original canonical observables $P_{j},Q_{j}$ are related to the creation
and annihilation operators $\hat{A}_{j}$ and $\hat{A}_{j}^{\ast}$. Due to
Lemma \ref{lem-commut-relations-B}, the set $\tilde{P}_{j},\tilde{Q}_{j}$,
$j=1,\ldots,n$ fulfills the same basic commutation relations (\ref{comm-rela}%
). Note that the proof of Lemma \ref{lem-covmatrix-quasi-obs} is based on
moment properties of the set of canonical observables $P_{j},Q_{j}$, implied
by the fact that their covariance matrix is $\Sigma\left(  A\right)  $ from
(\ref{symbol-related-to-covmatrix}). Hence it suffices to show that the
covariance matrix of $\tilde{P}_{j},\tilde{Q}_{j}$, $j=1,\ldots,n$ is
$\Sigma\left(  U_{n}^{\ast}AU_{n}\right)  $. To see this, define the vector of
observables
\[
\mathbf{\tilde{R}}:=\left(  \tilde{Q}_{1},\ldots,\tilde{Q}_{n},\tilde{P}%
_{1},\ldots,\tilde{P}_{n}\right)
\]
in analogy to the $\mathbf{R}$ occurring in (\ref{Weyl-as-function-of-P-Q}).
Then for every $x\in\mathbb{R}^{2n}$ we have to show, for $\rho=\mathfrak{N}%
_{n}\left(  0,A\right)  $
\begin{equation}
\mathrm{Tr}\;\rho\exp\left(  i\mathbf{\tilde{R}}x\right)  =\exp\left(
-\frac{1}{2}\left\langle x,\Sigma\left(  U_{n}^{\ast}AU_{n}\right)
x\right\rangle \right)  . \label{claim-anlogous-R-tilde}%
\end{equation}
Recall that in connection with (\ref{Weyl-unitaries-complex-indexed}) for
$u\in\mathbb{C}^{n}$ we set $\underline{u}:=\left(  -\operatorname{Im}%
u\right)  \oplus\operatorname{Re}u$. Setting $x=\underline{u}$ for some
$u\in\mathbb{C}^{n}$, we note that (\ref{char-function-Gaussian}) and
(\ref{char-func-gauge-invar}) imply
\[
\left\langle \underline{u},\Sigma\left(  A\right)  \underline{u}\right\rangle
=\frac{1}{2}\left\langle u,Au\right\rangle \text{, }u\in\mathbb{C}^{n}%
\]
for every symbol matrix $A$, so that (\ref{claim-anlogous-R-tilde}) is
equivalent to
\begin{equation}
\mathrm{Tr}\;\rho\exp\left(  i\mathbf{\tilde{R}}\underline{u}\right)
=\exp\left(  -\frac{1}{4}\left\langle u,U_{n}^{\ast}AU_{n}u\right\rangle
\right)  \text{, }u\in\mathbb{C}^{n}. \label{claim-anlogous-R-tilde-2}%
\end{equation}
Define
\[
\mathbf{\tilde{R}}_{Q}:=\left(  \tilde{Q}_{1},\ldots,\tilde{Q}_{n}\right)
,\;\mathbf{\tilde{R}}_{P}=\left(  \tilde{P}_{1},\ldots,\tilde{P}_{n}\right)
\]
and set $x=x_{1}\oplus x_{2}$, $x_{i}\in\mathbb{R}^{n}$, $i=1,2$. Then
\begin{align*}
\mathbf{\tilde{R}}x  &  =\mathbf{\tilde{R}}_{Q}x_{1}+\mathbf{\tilde{R}}%
_{P}x_{2}\\
&  =\frac{1}{\sqrt{2}}\left(  x_{1}^{\prime}\mathbf{\hat{B}+\mathbf{\hat{B}%
}^{\dag}}x_{1}\right)  +\frac{1}{i\sqrt{2}}\left(  x_{2}^{\prime}%
\mathbf{\hat{B}-\mathbf{\hat{B}}^{\dag}}x_{2}\right)
\end{align*}
Define $u_{x}\in\mathbb{C}^{n}$ by $u_{x}=x_{2}-ix_{1}$. Then we obtain
\begin{equation}
i\mathbf{\tilde{R}}x=2^{-1/2}\left(  u_{x}^{\ast}\mathbf{\hat{B}%
-\mathbf{\hat{B}}^{\dag}}u_{x}\right)  . \label{pre-Weyl-repres}%
\end{equation}
Analogously one shows for $\mathbf{R}$%
\[
i\mathbf{R}x=2^{-1/2}\left(  u_{x}^{\ast}\mathbf{\hat{A}-\mathbf{\hat{A}%
}^{\dag}}u_{x}\right)  ,
\]
and thus the Weyl unitaries can be written
\[
W\left(  x\right)  =\exp\left(  i\mathbf{R}x\right)  =\exp\left(
2^{-1/2}\left(  u_{x}^{\ast}\mathbf{\hat{A}-\mathbf{\hat{A}}^{\dag}}%
u_{x}\right)  \right)  .
\]
It turns out that $\underline{u_{x}}=x$, $x\in\mathbb{R}^{2n}$, and since
$V\left(  u\right)  =W\left(  \underline{u}\right)  $, the above relation can
be written
\[
V\left(  u\right)  =\exp\left(  2^{-1/2}\left(  u^{\ast}\mathbf{\hat
{A}-\mathbf{\hat{A}}^{\dag}}u\right)  \right)  ,\;u\in\mathbb{C}^{n}.
\]
Now (\ref{pre-Weyl-repres}) in connection with (\ref{transformed-number-ops})
yields%
\begin{align*}
\exp\left(  i\mathbf{\tilde{R}}x\right)   &  =\exp\left(  2^{-1/2}\left(
u_{x}^{\ast}U_{n}^{\ast}\mathbf{\hat{A}-\hat{A}}^{\dag}U_{n}u_{x}\right)
\right) \\
&  =V\left(  U_{n}u_{x}\right)  .
\end{align*}
so that (\ref{char-func-gauge-invar}) implies
\begin{align*}
\mathrm{Tr}\;\rho\exp\left(  i\mathbf{\tilde{R}}\underline{u}\right)   &
=\exp\left(  -\frac{1}{4}\left\langle U_{n}u,AU_{n}u\right\rangle \right) \\
&  =\exp\left(  -\frac{1}{4}\left\langle u,U_{n}^{\ast}AU_{n}u\right\rangle
\right)
\end{align*}
establishing (\ref{claim-anlogous-R-tilde-2}).
\end{proof}

\subsection{Unbiased covariance estimation}

Again assume that $n$ is odd$.$ We will see that in the case of a Toeplitz
symbol matrix $A$ (shift invariant time series), the set of observables
$\hat{B}_{j}^{\ast}\hat{B}_{j}$, $\left\vert j\right\vert \leq\left(
n-1\right)  /2$ allows an unbiased estimator of the coefficients
$a_{j}=a_{k,k+j}$, i.e. the analogs of the autocovariances of a classical time
series (cf. (\ref{unbiased}) below).

For the vectors $\mathbf{u}_{j}=\left(  u_{j,k}\right)  _{k=1,\ldots,n},$
$j\in\mathbb{Z}$ given by (\ref{epsilon-u-vectors-def}) for $m=n$ we note
\[
u_{j,k}=n^{-1/2}\epsilon_{j}^{k-1}=n^{-1/2}\exp\left(  2\pi ij\left(
k-1\right)  /n\right)  =n^{-1/2}\exp\left(  i\left(  k-1\right)  \omega
_{j,n}\right)
\]
for the Fourier frequencies $\omega_{j,n}$ defined in (\ref{Fourier-frequ-def}%
). Using the Toeplitz property of $A_{n}=\left(  a_{l-k}\right)
_{k=1,\ldots,n}^{l=1,\ldots,n}$ we obtain for $\left\vert j\right\vert
\leq\left(  n-1\right)  /2$
\begin{align}
\mathbf{u}_{j}^{\ast}A_{n}\mathbf{u}_{j}  &  =\sum_{k,l=1}^{n}\bar{u}%
_{j,k}\;u_{j,l}\;a_{l-k}=\sum_{k,l=1}^{n}a_{l-k}n^{-1}\exp\left(  i\left(
l-k\right)  \omega_{j,n}\right) \nonumber\\
&  =\sum_{s=-\left(  n-1\right)  }^{n-1}\frac{n-\left\vert s\right\vert }%
{n}a_{s}\exp\left(  is\omega_{j,n}\right)  =\sum_{s=-\left(  n-1\right)
}^{n-1}\left(  1-\frac{\left\vert s\right\vert }{n}\right)  a_{s}\phi
_{s}\left(  \omega_{j,n}\right)  , \label{unbiased-covest-2}%
\end{align}
$\phi_{s}$ being defined by (\ref{basic-func-def}). Define a commuting set of
observables
\begin{equation}
\Pi_{j}=2\hat{B}_{j}^{\ast}\hat{B}_{j}+\mathbf{1,}\text{ }\left\vert
j\right\vert \leq\left(  n-1\right)  /2. \label{Pi-observble-component-def}%
\end{equation}
Then from Lemma \ref{lem-covmatrix-analog-B} (i) and (\ref{unbiased-covest-2})
we obtain
\begin{equation}
\left\langle \Pi_{j}\right\rangle _{\rho}=\sum_{s=-\left(  n-1\right)  }%
^{n-1}\left(  1-\frac{\left\vert s\right\vert }{n}\right)  a_{s}\phi
_{s}\left(  \omega_{j,n}\right)  . \label{unified-3}%
\end{equation}
Recalling the series representation (\ref{symbol-generate-2}) of the spectral
density, we see that $\left\langle \Pi_{j}\right\rangle _{\rho}$ is an
approximation to the spectral density at the Fourier frequency $\omega_{j,n}$.
In particular, assuming that our quantum time series is $d$-dependent, i.e.
$a_{j}=0$ for $\left\vert j\right\vert >d$, we have for sufficiently large
$n$
\[
\left\langle \Pi_{j}\right\rangle _{\rho}=a\left(  \omega_{j,n}\right)
+O\left(  n^{-1}\right)  ,\text{ }\left\vert j\right\vert \leq\left(
n-1\right)  /2,
\]
i.e. the estimator $\Pi_{j}$ of $a\left(  \omega_{j}\right)  $ is
asymptotically unbiased of order $O\left(  n^{-1}\right)  $. Furthermore from
(\ref{unified-3}) we can obtain asymptotically unbiased estimates of the
symbol coefficients $a_{j}$ (we may informally call them the covariances).
Define vectors
\begin{equation}
\mathbf{v}_{j,n}:=n^{-1/2}\left(  \phi_{j}\left(  \omega_{k,n}\right)
\right)  _{\left\vert k\right\vert \leq\left(  n-1\right)  /2}\text{, }%
j\in\mathbb{Z}. \label{vectors-v-orthonorm}%
\end{equation}
Then $\mathbf{v}_{j,n}$, $\left\vert j\right\vert \leq\left(  n-1\right)  /2$
is an orthonormal system, thus
\begin{equation}
\mathbf{v}_{j,n}^{\ast}\mathbf{v}_{k,n}=\delta_{jk}\text{, }\left\vert
j\right\vert ,\left\vert k\right\vert \leq\left(  n-1\right)  /2.
\label{vectors-v-orthonorm-3}%
\end{equation}
Indeed set $c_{k-j,n}:=\exp\left(  i\left(  k-j\right)  \frac{2\pi}{n}\right)
$; then it can be shown that
\[
c_{k-j,n}\mathbf{v}_{j,n}^{\ast}\mathbf{v}_{k,n}=c_{k-j,n}n^{-1}%
\sum_{\left\vert s\right\vert \leq\left(  n-1\right)  /2}\exp\left(  i\left(
k-j\right)  \frac{2\pi s}{n}\right)  =\mathbf{v}_{j,n}^{\ast}\mathbf{v}_{k,n}%
\]
so that $\mathbf{v}_{j,n}^{\ast}\mathbf{v}_{k,n}$ must be zero unless $k=j$.

Define the vector of observables
\begin{equation}
\mathbf{\Pi}_{n}=\left(  \Pi_{j}\right)  _{\left\vert j\right\vert \leq\left(
n-1\right)  /2};\; \label{vec-of-observables}%
\end{equation}
then (\ref{unified-3}) can be written, for $\varrho=\mathfrak{N}_{n}\left(
0,A_{n}\right)  $
\begin{equation}
\left\langle \mathbf{\Pi}_{n}\right\rangle _{\rho}=n^{1/2}\sum_{j=-(n-1)}%
^{n-1}\left(  1-\frac{\left\vert j\right\vert }{n}\right)  a_{j}%
\mathbf{v}_{j,n}. \label{vec-of-observables-expec}%
\end{equation}
At this point, by $d$-dependency for fixed $d$ and $n$ sufficiently large, we
can assume that the above sum extends only over $\left\vert j\right\vert \leq
d\leq\left(  n-1\right)  /2$. Then, defining the estimator
\begin{equation}
\check{a}_{j,n}=\frac{n^{1/2}}{n-\left\vert j\right\vert }\mathbf{v}%
_{j,n}^{\ast}\mathbf{\Pi}_{n}\mathbf{,}\text{ for }\left\vert j\right\vert
\leq d, \label{prelim-estim-def}%
\end{equation}
we have by the orthogonality (\ref{vectors-v-orthonorm-3})
\begin{equation}
E_{\rho}\check{a}_{j,n}=\frac{n^{1/2}}{n-\left\vert j\right\vert }%
\mathbf{v}_{j}^{\ast}\left\langle \mathbf{\Pi}_{n}\right\rangle _{\rho}%
=\frac{n^{1/2}}{n-\left\vert j\right\vert }n^{1/2}\left(  1-\frac{\left\vert
j\right\vert }{n}\right)  \mathbf{v}_{j,n}^{\ast}\mathbf{v}_{j,n}a_{j}%
=a_{j}\text{.} \label{unbiased}%
\end{equation}
The estimate $\check{a}_{j.n}$ is the analog of the basic unbiased covariance
estimate in a classical time series (cf. \cite{MR3930599}, Sec 6.4).

\subsection{A preliminary estimator}

\subsubsection{Real parameters}

We will take the unbiased estimator (\ref{prelim-estim-def}) as a starting
point for constructing a preliminary estimator in the $d$-dependent case.
Since our parameter vector $\left(  a_{j}\right)  _{\left\vert j\right\vert
\leq d}$ is complex with $a_{-j}=\bar{a}_{j}$, we will transform it to a real
vector as follows: $\theta=\left(  \theta_{j}\right)  _{\left\vert
j\right\vert \leq d}$ where
\begin{equation}
\theta_{0}=a_{0}\text{, }\theta_{j}=\sqrt{2}\operatorname{Re}a_{j}\text{,
}\theta_{-j}=-\sqrt{2}\operatorname{Im}a_{j}\text{, }1\leq j\leq d.
\label{real-parameter-def}%
\end{equation}
Let us also define a set of functions on $\left[  -\pi,\pi\right]  $ as
\begin{subequations}
\label{psi-def}%
\begin{align}
\psi_{0}  &  =\phi_{j}=1\text{, }\label{psi-def-1}\\
\psi_{j}  &  =\frac{1}{\sqrt{2}}\left(  \phi_{j}+\phi_{-j}\right)  =\sqrt
{2}\cos\left(  j\cdot\right)  ,\label{psi-def-2}\\
\psi_{-j}  &  =\frac{1}{i\sqrt{2}}\left(  \phi_{j}-\phi_{-j}\right)  =\sqrt
{2}\sin\left(  j\cdot\right)  , \label{psi-def-3}%
\end{align}
for $j\in\mathbb{N}$. These functions fulfill
\end{subequations}
\begin{equation}
\frac{1}{2\pi}\int_{\left[  -\pi,\pi\right]  }\psi_{j}\left(  \omega\right)
\psi_{l}\left(  \omega\right)  d\omega=\delta_{jl},\;j,l\in\mathbb{N}.
\label{ONB-sin-cos}%
\end{equation}
Recalling (\ref{symbol-generate-2}), we can then write the spectral density as
follows:%
\begin{align}
a\left(  \omega\right)   &  =\sum_{\left\vert j\right\vert \leq d}\phi
_{j}\left(  \omega\right)  a_{j}\nonumber\\
&  =a_{0}+\sum_{1\leq j\leq d}\left(  \phi_{j}\left(  \omega\right)
+\phi_{-j}\left(  \omega\right)  \right)  \operatorname{Re}a_{j}+i\sum_{1\leq
j\leq d}\left(  \phi_{j}\left(  \omega\right)  -\phi_{-j}\left(
\omega\right)  \right)  \operatorname{Im}a_{j}\nonumber\\
&  =\sum_{\left\vert j\right\vert \leq d}\psi_{j}\left(  \omega\right)
\theta_{j}=:a_{\theta}\left(  \omega\right)  .
\label{series-repre-real-spec_density}%
\end{align}
The above defines the spectral density as a function $a_{\theta}$ of a
parameter $\theta\in\mathbb{R}^{2d+1}$. The assumption $a\in\Theta_{2}\left(
d,M\right)  $ is then equivalent to
\begin{align}
\theta &  \in\Theta_{2}^{\prime}\left(  d,M\right)  :=\left\{  \theta
:\left\Vert \theta\right\Vert ^{2}\leq M\right\}  \cap\mathcal{L}_{M}^{\prime
},\label{parametric-model}\\
\mathcal{L}_{M}^{\prime}  &  :=\left\{  \theta:\inf_{\omega\in\left[  -\pi
,\pi\right]  }a_{\theta}\left(  \omega\right)  \geq1+M^{-1}\right\}  .
\label{lowerbound-set-parametric-def}%
\end{align}
This parameter space will often be written just $\Theta_{2}^{\prime}$,
considering $d$ and $M$ fixed henceforth.

The next Lemma is an analog of Lemma \ref{lem-toeplitz-EV}.

\begin{lemma}
\label{lem-toeplitz-EV-2} Suppose $\theta\in\Theta_{2}^{\prime}\left(
d,M\right)  $ for $M>1$. Then
\begin{equation}
\left(  1+M^{-1}\right)  \;I\leq A_{n}\left(  a_{\theta}\right)  \leq\left(
2d+1\right)  ^{1/2}M^{1/2}\;I.
\end{equation}

\end{lemma}

\begin{proof}
For $\omega\in\left[  -\pi,\pi\right]  $ we have
\[
a_{\theta}\left(  \omega\right)  =\sum_{\left\vert j\right\vert \leq d}%
\psi_{j}\left(  \omega\right)  \theta_{j}\leq\left(  \sum_{\left\vert
j\right\vert \leq d}\psi_{j}^{2}\left(  \omega\right)  \right)  ^{1/2}%
\left\Vert \theta\right\Vert \leq\left(  2d+1\right)  ^{1/2}M^{1/2}.
\]
Set $C=\left(  2d+1\right)  ^{1/2}M^{1/2}$; then analogously to the proof of
Lemma \ref{lem-toeplitz-EV} for every $x\in\mathbb{C}^{n}$ with $\left\Vert
x\right\Vert =1$%
\[
\left\langle x,A_{n}\left(  a\right)  x\right\rangle \leq\frac{C}{2\pi}%
\int_{-\pi}^{\pi}\left\vert \sum_{j=1}^{n}x_{j}\exp\left(  ij\omega\right)
\right\vert ^{2}d\omega=C.
\]
Analogously we obtain from the first inequality in
(\ref{cond-boundedness-specdensity}) $\left\langle x,A_{n}\left(  a\right)
x\right\rangle \geq\left(  1+M^{-1}\right)  $.
\end{proof}

\bigskip

Define vectors, in analogy to $\mathbf{v}_{j}$ in (\ref{vectors-v-orthonorm}%
),
\begin{equation}
\mathbf{w}_{j,n}:=n^{-1/2}\left(  \psi_{j}\left(  \omega_{k}\right)  \right)
_{\left\vert k\right\vert \leq\left(  n-1\right)  /2}\text{, }\left\vert
j\right\vert \leq\left(  n-1\right)  /2. \label{vectors-w--def}%
\end{equation}
We then have $\mathbf{w}_{0}=\mathbf{v}_{0}$ and
\[
\mathbf{w}_{j,n}=\frac{1}{\sqrt{2}}\left(  \mathbf{v}_{j,n}+\mathbf{v}%
_{-j,n}\right)  \text{, }\mathbf{w}_{-j,n}=\frac{1}{i\sqrt{2}}\left(
\mathbf{v}_{j,n}-\mathbf{v}_{-j,n}\right)  \text{, }1\leq j\leq\left(
n-1\right)  /2
\]
or equivalently
\begin{equation}
\mathbf{v}_{j,n}=\frac{1}{\sqrt{2}}\left(  \mathbf{w}_{j,n}+i\mathbf{w}%
_{-j,n}\right)  ,\;\mathbf{v}_{-j,n}=\frac{1}{\sqrt{2}}\left(  \mathbf{w}%
_{j,n}-i\mathbf{w}_{-j,n}\right)  . \label{vectors-w-relation-to-v}%
\end{equation}
It follows that $\mathbf{w}_{j,n}$, $\left\vert j\right\vert \leq\left(
n-1\right)  /2$ are orthonormal; indeed they satisfy%
\begin{equation}
\mathbf{w}_{j,n}^{\prime}\mathbf{w}_{k,n}=\delta_{jk},\;\left\vert
j\right\vert \leq\left(  n-1\right)  /2. \label{vectors-w-orthonorm-2}%
\end{equation}
Since $a_{j}\mathbf{v}_{j,n}+a_{-j}\mathbf{v}_{-j,n}=\theta_{j}\mathbf{w}%
_{j,n}+\theta_{-j}\mathbf{w}_{-j,n}$ for $0\leq j\leq d$, we can rewrite
(\ref{vec-of-observables-expec}) under $d$-dependence as
\begin{equation}
E_{\rho}\mathbf{\Pi}_{n}=n^{1/2}\sum_{j=-d}^{d}\left(  1-\frac{\left\vert
j\right\vert }{n}\right)  \theta_{j}\mathbf{w}_{j,n}
\label{vec-of-observables-expec-2}%
\end{equation}
for $\rho=\mathfrak{N}_{n}\left(  0,A_{n}\left(  a_{\theta}\right)  \right)
$. Also the estimator (\ref{prelim-estim-def}) can be rewritten as
\begin{equation}
\check{\theta}_{j,n}=\frac{n^{1/2}}{n-\left\vert j\right\vert }\mathbf{w}%
_{j,n}^{\prime}\mathbf{\Pi}_{n}\mathbf{,\;}\left\vert j\right\vert \leq d.
\label{prelim-estim-def-2}%
\end{equation}
Unbiasedness then follows from (\ref{vec-of-observables-expec-2}): for
$\rho=\mathfrak{N}_{n}\left(  0,A_{n}\left(  a_{\theta}\right)  \right)  $
\begin{equation}
E_{\rho}\check{\theta}_{j,n}=\theta_{j}\text{, }\left\vert j\right\vert \leq
d\text{.} \label{unbiased-2}%
\end{equation}

\subsubsection{Partition into independent blocks\label{Subsec-partition}}

Recall that the $n$ pairs of operators $\left(  \hat{A}_{j},\hat{A}_{j}^{\ast
}\right)  $, $j=1,\ldots,n$ define the $n$ modes of the quantum Gaussian
state; we will subdivide this sequence into blocks as follows. Set%
\begin{equation}
m_{n}=2\left[  \log n/2\right]  +1,\;r_{n}=\left[  n/\left(  m_{n}+d\right)
\right]  \label{size-of-blocks-def}%
\end{equation}
so that $m_{n}$ is odd; we will write $m$ and $r$ hence forth. Consider sets
of pairs
\begin{align*}
S_{1}  &  :=\left\{  \left(  \hat{A}_{1},\hat{A}_{1}^{\ast}\right)
,\ldots,\left(  \hat{A}_{m},\hat{A}_{m}^{\ast}\right)  \right\}
,S_{2}:=\left\{  \left(  \hat{A}_{m+d+1},\hat{A}_{m+d+1}^{\ast}\right)
,\ldots,\left(  \hat{A}_{2m+d},\hat{A}_{2m+d}^{\ast}\right)  \right\}
,\ldots\\
S_{r}  &  :=\left\{  \left(  \hat{A}_{\left(  r-1\right)  \left(  m+d\right)
+1},\hat{A}_{\left(  r-1\right)  \left(  m+d\right)  +1}^{\ast}\right)
,\ldots,\left(  \hat{A}_{rm+(r-1)d},\hat{A}_{rm+(r-1)d}^{\ast}\right)
\right\}
\end{align*}
Note that operators from two different blocks $S_{j}$, $S_{h}$ are
uncorrelated: considering e.g. the last pair $\left(  \hat{A}_{m},\hat{A}%
_{m}^{\ast}\right)  $ from $S_{1}$ and the first pair $\left(  \hat{A}%
_{m+d+1},\hat{A}_{m+d+1}^{\ast}\right)  $ from $S_{2}$, we have according to
Lemma \ref{lem-covmatrix-quasi-obs} (i)
\[
\left\langle \hat{A}_{m}^{\ast}\hat{A}_{m+d+1}\right\rangle _{\rho}=\frac
{1}{2}a_{m,m+d+1}=\frac{1}{2}a_{d+1}=0
\]
in view of the $d$-dependence ($a_{h}=0$ for $\left\vert h\right\vert >d$).
Similarly, applying (\ref{symbol-related-to-covmatrix})
\begin{align*}
\left\langle \hat{A}_{m}\hat{A}_{m+d+1}\right\rangle _{\rho}  &  =\frac{1}%
{2}\left\langle \left(  Q_{m}+iP_{m}\right)  \left(  Q_{m+d+1}+iP_{m+d+1}%
\right)  \right\rangle _{\rho}\\
&  =\frac{1}{2}\left(  \left\langle Q_{m}Q_{m+d+1}\right\rangle _{\rho
}+i\left\langle Q_{m}P_{m+d+1}\right\rangle _{\rho}+i\left\langle
P_{m}Q_{m+d+1}\right\rangle _{\rho}-\left\langle P_{m}P_{m+d+1}\right\rangle
_{\rho}\right) \\
&  =\frac{1}{2}\left(  \operatorname{Re}a_{m,m+d+1}-i\operatorname{Im}%
a_{m,m+d+1}+i\operatorname{Im}a_{m,m+d+1}-\operatorname{Re}a_{m,m+d+1}\right)
=0.
\end{align*}
Intuitively, when we "omit" all pairs $\left(  \hat{A}_{j},\hat{A}_{j}^{\ast
}\right)  $ between the blocks, and also those after the last block $S_{r}$,
then, because of the $d$-dependence, the remaining blocks $S_{1},\ldots,S_{r}$
should be "independent". To make this rigorous in the quantum context, we take
a partial trace of the state $\mathfrak{N}_{n}\left(  0,A_{n}\right)  $,
tracing out all the modes corresponding to the pairs $\left(  \hat{A}_{j}%
,\hat{A}_{j}^{\ast}\right)  $ in question. What we get is a Gaussian state
with $rm$ modes and symbol matrix $I_{r}\otimes A_{\left(  m\right)  }$ (in
view of the Toeplitz form of $A_{n}$, where $A_{\left(  m\right)  }$ is the
upper central $m\times m$ submatrix of $A_{n}$, i.e. we obtain the gauge
invariant state $\mathfrak{N}_{rm}\left(  0,I_{r}\otimes A_{\left(  m\right)
}\right)  $. The details of this reasoning are given in Subsection
\ref{subsec:partial-trace}. Using characteristic functions , it is easy to
show that this state is equivalent to an $r$-fold tensor product $\left(
\mathfrak{N}_{m}\left(  0,A_{\left(  m\right)  }\right)  \right)  ^{\otimes
r}$.

Recall the basic model assumption (\ref{quantum-asy-setup-spec-density}), i.e.
$A_{n}=A_{n}\left(  a\right)  $, $n\rightarrow\infty$ for a given spectral
density $a$ (with current assumption $a=a_{\theta}$, $\theta\in\Theta
_{2}^{\prime}$, cf. (\ref{parametric-model})). It follows that $A_{\left(
m\right)  }=A_{m}\left(  a_{\theta}\right)  $, or $A_{\left(  m\right)
}=A_{m}$ for short, and we now have the parametric model of states $\left(
\mathfrak{N}_{m}\left(  0,A_{m}\left(  a_{\theta}\right)  \right)  \right)
^{\otimes r}$, $\theta\in\Theta_{2}^{\prime}$.%

\begin{privatenotes}
\begin{boxedminipage}{\textwidth}%

\begin{sfblock}
The preceding paragraph has to be reformulated in a more rigorous way, as a
lemma on the existence of a channel resulting in the state $\left(
\mathfrak{N}_{m}\left(  0,A_{m}\left(  a_{\theta}\right)  \right)  \right)
^{\otimes r}$, similar to the Fermionic paper.\texttt{ }
\end{sfblock}

%

\end{boxedminipage}
\end{privatenotes}%

For each of the $r$ component states of $\left(  \mathfrak{N}_{m}\left(
0,A_{m}\left(  a_{\theta}\right)  \right)  \right)  ^{\otimes r}$, we now form
the vector of observables $\mathbf{\Pi}_{m}$ corresponding to
(\ref{vec-of-observables}) for $n=m$, obtaining an $r$-tuple of such vectors
$\mathbf{\Pi}_{m,j}$, $j=1,\ldots,r$, and we form the average
\begin{equation}
\mathbf{\bar{\Pi}}_{n}:=r^{-1}\sum_{j=1}^{r}\mathbf{\Pi}_{m,j}.
\label{Pi-k-quer-def}%
\end{equation}
We will modify the estimator (\ref{prelim-estim-def-2}), essentially
substituting $\mathbf{\bar{\Pi}}_{n}$ for $\mathbf{\Pi}_{n}$. To write it in
vector form, consider the vectors $\mathbf{w}_{j,n}$ of (\ref{vectors-w--def})
for dimension $n=m$ and define the $m\times\left(  2d+1\right)  $ real matrix
\begin{equation}
W_{m}=\left(  \mathbf{w}_{-d,m},\ldots,\mathbf{w}_{0},\ldots,\mathbf{w}%
_{d,m}\right)  , \label{matrix-W-def}%
\end{equation}
fulfilling $W_{m}^{\prime}W_{m}=I_{2d+1}$ by (\ref{vectors-w-orthonorm-2}).
Furthermore define the diagonal $\left(  2d+1\right)  \times\left(
2d+1\right)  $ matrix%
\begin{equation}
F_{m}:=\mathrm{diag}\left(  \frac{m}{m-\left\vert j\right\vert }\right)
_{\left\vert j\right\vert \leq d}. \label{Fm-matrix-def}%
\end{equation}

\begin{definition}
The preliminary estimator of the parameter vector $\theta$ from
(\ref{real-parameter-def}) is%
\begin{equation}
\hat{\theta}_{n}:=m^{-1/2}F_{m}W_{m}^{\prime}\mathbf{\bar{\Pi}}_{n}
\label{prelim-estim-def-3}%
\end{equation}
with $\mathbf{\bar{\Pi}}_{n}$ from (\ref{Pi-k-quer-def})
\end{definition}

Since $E_{\rho}\mathbf{\Pi}_{m,j}$ coincides with $E_{\rho}\mathbf{\Pi}$ (cf.
(\ref{vec-of-observables-expec-2})) if the latter is taken at dimension $n=m$,
from (\ref{unbiased-2}) we immediately obtain unbiasedness: $E_{\rho}%
\hat{\theta}_{n}=\theta$.

Let $P_{n,\theta}$ be the joint distribution of the $\mathbb{R}^{m}$-valued
random vectors $\mathbf{\Pi}_{m,j}$, $j=1,\ldots,r$ from (\ref{Pi-k-quer-def})
under the state $\rho=\mathfrak{N}_{n}\left(  0,A_{n}\left(  a_{\theta
}\right)  \right)  $. Here $\mathbf{\bar{\Pi}}_{n}$ will function as the basic
observable for asymptotic inference about $\theta$, so that distributions of
further random variables in this section can be described in terms of
$P_{n,\theta}$ and corresponding expectations $E_{n,\theta}$.

\subsubsection{Asymptotic covariance matrix}

We have
\begin{equation}
n^{1/2}\left(  \hat{\theta}_{n}-\theta\right)  =\sum_{j=1}^{r}r^{-1}%
n^{1/2}m^{-1/2}F_{m}W_{m}^{\prime}\left(  \mathbf{\Pi}_{m,j}-E_{\rho
}\mathbf{\Pi}_{m,j}\right)  \label{sum-iid}%
\end{equation}
where it follows from (\ref{vec-of-observables-expec-2}) that
\begin{equation}
E_{n,\theta}\mathbf{\Pi}_{m,j}=m^{1/2}W_{m}F_{m}^{-1}\theta\mathbf{.}
\label{expec-Pi-m-j}%
\end{equation}
The r.h.s. of (\ref{sum-iid}) is a sum of independent, identically distributed
zero mean random vectors. In the following proof, for sequences of nonrandom
matrices $M_{1,n,}$ $M_{2,n}$ of fixed dimension as $n\rightarrow\infty$, we
write $M_{1,n}\sim M_{2,n}$ if $M_{1,n}=M_{2,n}\left(  1+o\left(  1\right)
\right)  $ elementwise. Also \textrm{Cov}$_{n,\theta}\left(  \cdot\right)  $
denotes the covariance matrix of a real random vector under $P_{n,\theta}$.

\begin{lemma}
\label{lem-covmatrix-prelim-est}Under $\rho=\mathfrak{N}_{n}\left(
0,A_{n}\left(  a_{\theta}\right)  \right)  $, $\theta\in\Theta_{2}^{\prime}$
we have%
\[
\lim_{n\rightarrow\infty}\mathrm{Cov}_{n,\theta}\left(  n^{1/2}\left(
\hat{\theta}_{n}-\theta\right)  \right)  =\Phi_{\theta}^{0}:=\left(
\Phi_{\theta,jk}^{0}\right)  _{\left\vert j\right\vert ,\left\vert
k\right\vert \leq d},\text{ }%
\]
where
\begin{equation}
\Phi_{\theta,jk}^{0}=\frac{1}{2\pi}\int_{\left(  -\pi,\pi\right)  }\left(
a_{\theta}^{2}\left(  \omega\right)  -1\right)  \psi_{j}\left(  \omega\right)
\psi_{k}\left(  \omega\right)  d\omega, \label{Gamma-covmatrix-def}%
\end{equation}
$a_{\theta}\left(  \omega\right)  $ is the spectral density depending on
$\theta\in\Theta_{2}^{\prime}$ according to
(\ref{series-repre-real-spec_density}), and functions $\psi_{h}$ are defined
by (\ref{psi-def}). The convergence is uniform over $\theta\in\Theta
_{2}^{\prime}$.
\end{lemma}

\begin{proof}
Note that in (\ref{sum-iid}) we have $r^{-1}n^{1/2}m^{-1/2}\sim r^{-1/2}$ and
$F_{m}\longrightarrow I_{2d+1}$, hence writing $\mathbf{\Pi}_{m}=\mathbf{\Pi
}_{1,m}$ we obtain
\begin{equation}
\mathrm{Cov}_{n,\theta}\left(  n^{1/2}\left(  \hat{\theta}_{n}-\theta\right)
\right)  =W_{m}^{\prime}\mathrm{Cov}_{n,\theta}\left(  \mathbf{\Pi}%
_{m}\right)  W_{m}\left(  1+o\left(  1\right)  \right)  .
\label{cov-matrix-prelimest-1}%
\end{equation}
To obtain the covariance matrix appearing on the r.h.s., consider the result
of Lemma \ref{lem-covmatrix-analog-B} for $n=m$. For a $m\times m$ matrix
$M=\left(  M_{jl}\right)  _{j,l=1}^{m}$, define the real matrix
\begin{equation}
M^{\left[  2\right]  }=\left(  \left\vert M_{jl}\right\vert ^{2}\right)
_{j,l=1}^{m}. \label{M-brackets-2-def}%
\end{equation}
.Then the result of \ref{lem-covmatrix-analog-B} (ii) can be written, with
$A_{m}=A_{m}\left(  a_{\theta}\right)  $,
\[
\mathrm{Cov}_{n,\theta}\left(  \hat{B}_{1}^{\ast}\hat{B}_{1},\ldots,\hat
{B}_{m}^{\ast}\hat{B}_{m}\right)  =\frac{1}{4}\left(  \left(  U_{m}^{\ast
}A_{m}U_{m}\right)  ^{\left[  2\right]  }-I_{m}\right)  .
\]
Now recall the definition of the observable vector $\mathbf{\Pi}_{m}$ in
(\ref{Pi-observble-component-def}), (\ref{vec-of-observables}) and identify
$\mathbf{\Pi}_{1,m}$ with $\Pi_{m}$. We obtain
\begin{equation}
\mathrm{Cov}_{n,\theta}\left(  \mathbf{\Pi}_{1,m}\right)  =\left(  U_{m}%
^{\ast}A_{m}U_{m}\right)  ^{\left[  2\right]  }-I_{m}, \label{Cov-matrix-Pi}%
\end{equation}
with $U_{m}$ from (\ref{special-unitary-DFT-def}) and $m$ from
(\ref{size-of-blocks-def}). Then (\ref{cov-matrix-prelimest-1}) can be
written
\begin{equation}
\mathrm{Cov}_{n,\theta}\left(  n^{1/2}\left(  \hat{\theta}_{n}-\theta\right)
\right)  \sim W_{m}^{\prime}\left(  U_{m}^{\ast}A_{m}U_{m}\right)  ^{\left[
2\right]  }W_{m}-I_{2d+1}. \label{cov-matrix-prelimest-2}%
\end{equation}
To treat the first term on the r.h.s. of (\ref{cov-matrix-prelimest-2}),
recall that the Hilbert-Schmidt norm $\left\Vert M\right\Vert _{2}$ of an
$m\times m$ matrix $M$ is defined as $\left\Vert M\right\Vert _{2}%
^{2}=\mathrm{Tr}\;M^{\ast}M=\sum_{j,l}\left\vert M_{jk}\right\vert ^{2}$. Note
that, under $d$-dependence, the symbol matrix $A_{m}$ is banded in the
terminology of \cite{CIT-006}. Then for $A_{m}=A_{m}\left(  a_{\theta}\right)
$ and its circulant approximation $\tilde{A}_{m}=\tilde{A}_{m}\left(
a_{\theta}\right)  $ defined in (\ref{circulant-approx-def-3}) we have for
$\theta\in\Theta_{2}^{\prime}$
\[
m^{-1}\left\Vert A_{m}-\tilde{A}_{m}\right\Vert _{2}^{2}=m^{-1}2\sum_{j=1}%
^{d}j\left\vert a_{j}\right\vert ^{2}=m^{-1}\sum_{j=-d}^{d}j\theta_{j}%
^{2}\rightarrow0\text{ as }m\rightarrow\infty;
\]
by a reasoning similar to (\ref{first-upper-bound-cited-later}) when $m=n$ and
$a_{j}=0$ for $j>d$ (or referring to Lemma 4.2 in \cite{CIT-006}). The
convergence is uniform over $\left\Vert \theta\right\Vert \leq C$, hence over
$\theta\in\Theta_{2}^{\prime}$. Let $\tilde{A}_{m}=U_{m}\tilde{\Lambda}%
_{m}U_{m}^{\ast}$ be the spectral decomposition of $\tilde{A}_{m}$; then
according to (\ref{diag-symbol-Lambda-def}) we have for sufficiently large $m$
(such that $m>2d+1$),
\begin{equation}
\tilde{\Lambda}_{m}=\Lambda_{m}:=\mathrm{diag}\left(  \;a_{\theta}\left(
\omega_{j,m}\right)  _{\left\vert j\right\vert \leq\left(  m-1\right)
/2}\right)  \label{lambda-diag-matrix-def}%
\end{equation}
where $\omega_{j,m}$ are the Fourier frequencies $\omega_{j,m}=2\pi j/m$,
$\left\vert j\right\vert \leq\left(  m-1\right)  /2$. Since $\left\Vert
M\right\Vert _{2}^{2}=\left\Vert U_{m}^{\ast}MU_{m}\right\Vert _{2}^{2}$ for
any $m\times m$ matrix $M$, we obtain
\begin{equation}
m^{-1}\left\Vert U_{m}^{\ast}A_{m}U_{m}-\Lambda_{m}\right\Vert _{2}%
^{2}\rightarrow0\text{ as }n\rightarrow\infty. \label{cov-matrix-prelimest-3}%
\end{equation}
uniformly over $\theta\in\Theta_{2}^{\prime}$. Consider the elelement with
index $\left(  j,k\right)  $ of $W_{m}^{\prime}\left(  U_{m}^{\ast}A_{m}%
U_{m}\right)  ^{\left[  2\right]  }W_{m}$; this is
\begin{align}
\mathbf{w}_{j,m}^{\prime}\left(  U_{m}^{\ast}A_{m}U_{m}\right)  ^{\left[
2\right]  }\mathbf{w}_{k,m}  &  =\mathbf{w}_{j,m}^{\prime}\Lambda_{m}%
^{2}\mathbf{w}_{k,m}+\mathbf{w}_{j,m}^{\prime}D_{m}\mathbf{w}_{k,m}\text{
where }\nonumber\\
D_{m}  &  :=\left(  U_{m}^{\ast}A_{m}U_{m}\right)  ^{\left[  2\right]
}-\Lambda_{m}^{2}. \label{D-remainder-matrix-def}%
\end{align}
Note that, since all components of $\mathbf{w}_{j,m}$ and $\mathbf{w}_{k,m}$
are bounded in modulus by $\sqrt{2}m^{-1/2}$, we have
\[
\left\vert \mathbf{w}_{j,m}^{\ast}D_{m}\mathbf{w}_{k,m}\right\vert \leq
2m^{-1}\sum_{s,t=1}^{m}\left\vert \left(  D_{m}\right)  _{st}\right\vert
=2m^{-1}\sum_{s,t=1}^{m}\left\vert \left(  U_{m}^{\ast}A_{m}U_{m}\right)
_{st}^{\left[  2\right]  }-\left(  \Lambda_{m}^{2}\right)  _{st}\right\vert .
\]
Note that for any complex $x,y$%
\begin{align*}
\left\vert \left\vert x\right\vert ^{2}-\left\vert y\right\vert ^{2}%
\right\vert  &  =\left\vert \left(  \left\vert x\right\vert -\left\vert
y\right\vert \right)  \left(  \left\vert x\right\vert +\left\vert y\right\vert
\right)  \right\vert \\
&  \leq\left\vert x-y\right\vert \;\left(  \left\vert x\right\vert +\left\vert
y\right\vert \right)  .
\end{align*}
Applying this bound to each term $\left\vert \left(  U_{m}^{\ast}A_{m}%
U_{m}\right)  _{st}^{\left[  2\right]  }-\left(  \Lambda_{m}^{2}\right)
_{st}\right\vert $, we obtain,
\begin{align*}
&  m^{-1}\sum_{s,t=1}^{m}\left\vert \left(  U_{m}^{\ast}A_{m}U_{m}\right)
_{st}^{\left[  2\right]  }-\left(  \Lambda_{m}^{2}\right)  _{st}\right\vert \\
&  \leq m^{-1}\sum_{s,t=1}^{m}\left\vert \left(  U_{m}^{\ast}A_{m}%
U_{m}-\Lambda_{m}\right)  _{st}\right\vert \left(  \left\vert \left(
U_{m}^{\ast}A_{m}U_{m}\right)  _{st}\right\vert +\left\vert \left(
\Lambda_{m}\right)  _{st}\right\vert \right)  .
\end{align*}
Applying the Cauchy-Schwartz inequality, we obtain an upper bound%
\begin{equation}
\left(  m^{-1}\left\Vert U_{m}^{\ast}A_{m}U_{m}-\Lambda_{m}\right\Vert
_{2}^{2}\right)  ^{1/2}\left(  2m^{-1}\left(  \left\Vert U_{m}^{\ast}%
A_{m}U_{m}\right\Vert ^{2}+\left\Vert \Lambda_{m}\right\Vert _{2}^{2}\right)
\right)  ^{1/2}. \label{cov-matrix-prelimest-4}%
\end{equation}
Here the first factor is $o\left(  1\right)  $ uniformly over $\theta\in
\Theta_{2}^{\prime}$ by (\ref{cov-matrix-prelimest-3}). The second factor is
bounded by the following reasoning. In view of $d$-dependence%
\[
m^{-1}\left\Vert U_{m}^{\ast}A_{m}U_{m}\right\Vert _{2}^{2}=m^{-1}\left\Vert
A_{m}\right\Vert _{2}^{2}=\sum_{\left\vert j\right\vert \leq d}\frac
{m-\left\vert j\right\vert }{m}\left\vert a_{j}\right\vert ^{2}=\sum
_{\left\vert j\right\vert \leq d}\frac{m-\left\vert j\right\vert }{m}%
\theta_{j}^{2}\leq M
\]
by $\theta\in\Theta_{2}^{\prime}$. Similarly
\begin{align*}
m^{-1}\left\Vert \Lambda_{m}\right\Vert _{2}^{2}  &  \leq\max_{-\pi\leq
\omega\leq\pi}a_{\theta}^{2}\left(  \omega\right)  =\max_{-\pi\leq\omega
\leq\pi}\left(  \sum_{\left\vert j\right\vert \leq d}\psi_{j}\left(
\omega\right)  \theta_{j}\right)  ^{2}\\
&  \leq\left(  2d+1\right)  \left\Vert \theta\right\Vert ^{2}\leq\left(
2d+1\right)  M
\end{align*}
by $\theta\in\Theta_{2}^{\prime}$. As a consequence,
(\ref{cov-matrix-prelimest-4}) is $o\left(  1\right)  $ uniformly over
$\theta\in\Theta_{2}^{\prime}$, which implies
\begin{align}
\mathbf{w}_{j,m}^{\prime}\left(  U_{m}^{\ast}A_{m}U_{m}\right)  ^{\left[
2\right]  }\mathbf{w}_{k,m}  &  =\mathbf{w}_{j,m}^{\prime}\Lambda_{m}%
^{2}\mathbf{w}_{k,m}+o\left(  1\right) \label{relation-quoted-later}\\
&  =m^{-1}\sum_{\left\vert j\right\vert \leq\left(  m-1\right)  /2}a_{\theta
}^{2}\left(  \omega_{j,m}\right)  \psi_{j}\left(  \omega_{j,m}\right)
\psi_{k}\left(  \omega_{j,m}\right)  +o\left(  1\right)  .\nonumber
\end{align}
Since the set of functions $\left\{  a_{\theta},\;\theta\in\Theta_{2}^{\prime
},\psi_{j},\;\left\vert j\right\vert \leq d\right\}  $ is uniformly bounded
and Lipschitz, we now have
\[
\mathbf{w}_{j,m}^{\prime}\left(  U_{m}^{\ast}A_{m}U_{m}\right)  ^{\left[
2\right]  }\mathbf{w}_{k,m}=\frac{1}{2\pi}\int_{\left(  -\pi,\pi\right)
}a_{\theta}^{2}\left(  \omega\right)  \psi_{j}\left(  \omega\right)  \psi
_{k}\left(  \omega\right)  d\omega+o\left(  1\right)  .
\]
uniformly over $\theta\in\Theta_{2}^{\prime}$. In view of
(\ref{cov-matrix-prelimest-2}) and (\ref{ONB-sin-cos}), the claim follows.
\end{proof}

\bigskip

\begin{lemma}
\label{lem-unif-consistency-prelim-est} Let $\gamma_{n}\rightarrow\infty$ be a
sequence such that $\gamma_{n}=o\left(  n^{1/2}\right)  $. Then for every
$\varepsilon>0$ we have
\[
\sup_{\theta\in\Theta_{2}^{\prime}}P_{n,\theta}\left(  \gamma_{n}\left\Vert
\hat{\theta}_{n}-\theta\right\Vert \geq\varepsilon\right)  \rightarrow0.
\]

\end{lemma}

\begin{proof}
We have
\[
P_{n,\theta}\left(  \gamma_{n}\left\Vert \hat{\theta}_{n}-\theta\right\Vert
\geq\varepsilon\right)  \leq\frac{\gamma_{n}^{2}}{n}\frac{E_{n,\theta
}n\left\Vert \hat{\theta}_{n}-\theta\right\Vert ^{2}}{\varepsilon^{2}}%
=\frac{\gamma_{n}^{2}}{n}\frac{\mathrm{Tr\;Cov}_{\rho}\left(  n^{1/2}\left(
\hat{\theta}_{n}-\theta\right)  \right)  }{\varepsilon^{2}}%
\]
so the claim follows from Lemma \ref{lem-covmatrix-prelim-est}.
\end{proof}

\subsection{A one-step improvement estimator}

The estimator $\hat{\theta}_{n}$ can be shown to be asymptotically normal, but
it is not optimal; indeed will turn out that the optimal covariance matrix is
not $\Phi_{\theta}^{0}$ but the inverse of the matrix%
\begin{align}
\Phi_{\theta}  &  =\left(  \Phi_{\theta,jk}\right)  _{\left\vert j\right\vert
,\left\vert k\right\vert \leq d},\label{parametric-fisher-info-def}\\
\Phi_{\theta,jk}  &  :=\frac{1}{2\pi}\int_{\left(  -\pi,\pi\right)  }\left(
a_{\theta}^{2}\left(  \omega\right)  -1\right)  ^{-1}\psi_{j}\left(
\omega\right)  \psi_{k}\left(  \omega\right)  d\omega.
\label{parametric-fisher-info-def-a}%
\end{align}
where $a_{\theta}\left(  \omega\right)  $ is the spectral density depending on
$\theta\in\Theta_{2}^{\prime}$ according to
(\ref{series-repre-real-spec_density}).

\begin{lemma}
\label{Lem-covmatrix-eigenvalues}There are constants $0<C_{1,M}<C_{2,M}$
depending only on $M$ and $d$ such that for all $\theta\in\Theta_{2}^{\prime
}.$
\[
C_{1,M}\;\leq\lambda_{\min}\left(  \Phi_{\theta}\right)  \text{,\ \ \ }%
\lambda_{\max}\left(  \Phi_{\theta}\right)  \leq C_{2,M}\text{.}%
\]

\end{lemma}

\begin{proof}
Note that in view of $a_{\theta}\left(  \omega\right)  \geq1+M^{-1}$ we have
\[
a_{\theta}^{2}\left(  \omega\right)  -1\geq\left(  1+M^{-1}\right)  ^{2}%
-1\geq\left(  1+M^{-1}\right)  -1\geq M^{-1}.
\]
Furthermore
\begin{equation}
a_{\theta}^{2}\left(  \omega\right)  =\left(  \sum_{\left\vert j\right\vert
\leq d}\psi_{j}\left(  \omega\right)  \theta_{j}\right)  ^{2}\leq\left\Vert
\theta\right\Vert ^{2}\sum_{\left\vert j\right\vert \leq d}\psi_{j}^{2}\left(
\omega\right)  =\left\Vert \theta\right\Vert ^{2}\left(  2d+1\right)  \leq
M\left(  2d+1\right)  . \label{bounds-on-athetasquared-a}%
\end{equation}
The last two displays imply%
\begin{equation}
\left(  M\left(  2d+1\right)  \right)  ^{-1}\leq\left(  a_{\theta}^{2}\left(
\omega\right)  -1\right)  ^{-1}\leq M. \label{bounds-on-athetasquared}%
\end{equation}
Now for $x=\left(  x_{j}\right)  _{\left\vert j\right\vert \leq d}%
\in\mathbb{R}^{2d+1}$ we have
\begin{align*}
x^{\prime}\Phi_{\theta}x  &  =\frac{1}{2\pi}\int_{\left(  -\pi,\pi\right)
}\left(  a_{\theta}^{2}\left(  \omega\right)  -1\right)  ^{-1}\left(
\sum_{\left\vert j\right\vert \leq d}x_{j}\psi_{j}\left(  \omega\right)
\right)  ^{2}d\omega\\
&  \leq M\frac{1}{2\pi}\int_{\left(  -\pi,\pi\right)  }\left(  \sum
_{\left\vert j\right\vert \leq d}x_{j}\psi_{j}\left(  \omega\right)  \right)
^{2}d\omega=M
\end{align*}
and the bound $x^{\prime}\Phi_{\theta}x\geq\left(  2M\left(  d+1\right)
\right)  ^{-1}$ follows analogously. Setting $C_{1,M}=\left(  M\left(
2d+1\right)  \right)  ^{-1}$, $C_{2,M}=M$ completes the proof.
\end{proof}

In order to modify the preliminary estimator $\hat{\theta}_{n}$ given by
(\ref{prelim-estim-def-3}) in a suitable way, we will need estimates of the
parameter dependent diagonal matrices
\begin{equation}
\Delta_{m,\theta}:=\mathrm{diag}\left(  \;a_{\theta}^{2}\left(  \omega
_{j,m}\right)  _{\left\vert j\right\vert \leq\left(  m-1\right)  /2}-1\right)
. \label{delta-m-theta-def}%
\end{equation}
In order to replace $\theta$ there by a suitable estimator, consider the
following lemma.

\begin{lemma}
\label{lem-compact-convex}The set $\Theta_{2}^{\prime}=\Theta_{2}^{\prime
}\left(  M,d\right)  $ given by (\ref{parametric-model}) is a compact convex
subset of $\mathbb{R}^{2d+1}$.
\end{lemma}

\begin{proof}
The set $B_{M}:=\left\{  \theta\in\mathbb{R}^{2d+1}:\left\Vert \theta
\right\Vert ^{2}\leq M,\;\right\}  $ is compact and convex for each $\omega$,
the set $\left\{  \theta\in\mathbb{R}^{2d+1}:a_{\theta}\left(  \omega\right)
\geq1+M^{-1}\right\}  $ is convex and closed, since the map $\theta\rightarrow
a_{\theta}\left(  \omega\right)  $ is linear. Since the intersection of closed
sets is closed, and convex if each set is convex, $\Theta_{2}^{\prime}\left(
M,d\right)  $ is a closed convex subset of $B_{M}$, from which the claim follows.
\end{proof}

\bigskip

\bigskip Define an estimator $\bar{\theta}_{n}$ as the projection of
$\hat{\theta}_{n}$ onto the compact convex set $\Theta_{2}^{\prime}$ and set
\[
\hat{\Delta}_{m}:=\Delta_{m,\bar{\theta}_{n}}=\mathrm{diag}\left(
\;a_{\bar{\theta}_{n}}^{2}\left(  \omega_{j,m}\right)  -1\right)  _{\left\vert
j\right\vert \leq\left(  m-1\right)  /2}.
\]
Note that do to (\ref{bounds-on-athetasquared}), $\hat{\Delta}_{m}$ is nonsingular.

\begin{definition}
\label{Def-estimator-optimal-param}The improved estimator of the parameter
vector $\theta$ from (\ref{real-parameter-def}) is%
\begin{equation}
\tilde{\theta}_{n}:=m^{-1/2}F_{m}\left(  W_{m}^{\prime}\hat{\Delta}_{m}%
^{-1}W_{m}\right)  ^{-1}W_{m}^{\prime}\hat{\Delta}_{m}^{-1}\mathbf{\bar{\Pi}%
}_{n}%
\end{equation}
with $\mathbf{\bar{\Pi}}_{n}$ from (\ref{Pi-k-quer-def}).
\end{definition}

\bigskip

Refer to Appendix \ref{subsec-uniform-in-law} for the definition of
convergence in distribution (with symbol $\Longrightarrow_{d}$) uniformly in
$\theta$, and some associated results.

\begin{theorem}
\label{Theor-asy-normal-improved-est}The estimator $\tilde{\theta}_{n}$ is
asymptotically normal
\[
n^{1/2}\left(  \tilde{\theta}_{n}-\theta\right)  \Longrightarrow_{d}%
N_{2d+1}\left(  0,\Phi_{\theta}^{-1}\right)
\]
uniformly in $\theta\in\Theta_{2}^{\prime}$.
\end{theorem}

We will begin the proof with a series of technical lemmas.

\bigskip

\begin{lemma}
\label{lem-compact-2}On the compact set $\Theta_{2}^{\prime}\in\mathbb{R}%
^{2d+1}$, the map $\theta\longrightarrow\Phi_{\theta}^{-1}$ is continuous in
Hilbert-Schmidt norm.
\end{lemma}

\begin{proof}
For $x\in\mathbb{R}^{2d+1}$ with $\left\Vert x\right\Vert =1$, we have in view
of (\ref{bounds-on-athetasquared}), setting $M_{1}=$
\begin{align*}
x^{\prime}\Phi_{\theta}x  &  =\frac{1}{2\pi}\int_{\left(  -\pi,\pi\right)
}\left(  a_{\theta}^{2}\left(  \omega\right)  -1\right)  ^{-1}\left(
\sum_{j=-d}^{d}x_{j}\psi_{j}\left(  \omega\right)  \right)  ^{2}d\omega\\
&  \leq M\frac{1}{2\pi}\int_{\left(  -\pi,\pi\right)  }\left(  \sum_{j=-d}%
^{d}x_{j}\psi_{j}\left(  \omega\right)  \right)  ^{2}d\omega=M
\end{align*}
and similarly, for $M_{1}=$ $\left(  M\left(  2d+1\right)  \right)  ^{-1}$
\[
x^{\prime}\Phi_{\theta}x\geq M_{1}.
\]
It follows that
\begin{align*}
s_{1}  &  :=\inf\left\{  \lambda_{\min}\left(  \Phi_{\theta}\right)
:\theta\in\Theta_{2}^{\prime}\right\}  \geq M_{1}>0.\\
s_{2}  &  :=\sup\left\{  \lambda_{\max}\left(  \Phi_{\theta}\right)
:\theta\in\Theta_{2}^{\prime}\right\}  \leq M.
\end{align*}
Clearly the map $\theta\longrightarrow\Phi_{\theta}$ is continuous on
$\Theta_{2}^{\prime}$. For nonsingular matrices $\Phi_{1},\Phi_{2}$ we have
\[
\Phi_{1}^{-1}-\Phi_{2}^{-1}=\Phi_{1}^{-1}\left(  \Phi_{2}-\Phi_{1}\right)
\Phi_{2}^{-1}%
\]
which for the Hilbert-Schmidt norm $\left\Vert \cdot\right\Vert _{2}$ implies,
if both $\Phi_{1},\Phi_{2}$ are positive,
\[
\left\Vert \Phi_{1}^{-1}-\Phi_{2}^{-1}\right\Vert _{2}\leq\lambda_{\max
}\left(  \Phi_{1}^{-1}\right)  \lambda_{\max}\left(  \Phi_{2}^{-1}\right)
\left\Vert \Phi_{2}-\Phi_{2}\right\Vert _{2}.
\]
Thus for , $\theta_{j}\in\Theta_{2}^{\prime}$, $j=1,2$ we have%
\[
\left\Vert \Phi_{\theta_{1}}^{-1}-\Phi_{\theta_{2}}^{-1}\right\Vert _{2}\leq
s_{1}^{-2}\left\Vert \Phi_{\theta_{1}}-\Phi_{\theta_{2}}\right\Vert _{2}%
\]
showing that the map $\theta\longrightarrow\Phi_{\theta}^{-1}$ is continuous
on $\Theta_{2}^{\prime}$.
\end{proof}

\bigskip

\bigskip

Define the function%
\[
g\left(  \theta,\omega\right)  :=\left(  a_{\theta}^{2}\left(  \omega\right)
-1\right)  ^{-1}\text{, }\theta\in\Theta_{2}^{\prime},\omega\in\left[
-\pi,\pi\right]  .
\]

\begin{lemma}
\label{Lem-Lipschitz-uniform}There exists $L>0$ depending only on $m$ and $d$
such that
\[
\sup_{\omega\in\left[  -\pi,\pi\right]  }\left\vert g\left(  \theta_{1}%
,\omega\right)  -g\left(  \theta_{2},\omega\right)  \right\vert \leq
L\left\Vert \theta_{1}-\theta_{2}\right\Vert \text{, }\theta_{1},\theta_{2}%
\in\Theta_{2}^{\prime}.
\]

\end{lemma}

\begin{proof}
We first claim that $\left\Vert \partial_{\theta}g\left(  \theta
,\omega\right)  \right\Vert ^{2}\leq2M^{2}.$Indeed for $\theta=\left(
\theta_{j}\right)  _{\left\vert j\right\vert \leq d}$ we have for any
$\left\vert j\right\vert \leq d$, recalling $a_{\theta}\left(  \omega\right)
=\sum_{\left\vert j\right\vert \leq d}\theta_{j}\psi_{j}\left(  \omega\right)
$, where we used the bounds (\ref{bounds-on-athetasquared}) and
(\ref{bounds-on-athetasquared-a}). Consequently%
\begin{align*}
\left\Vert \partial_{\theta}g\left(  \theta,\omega\right)  \right\Vert ^{2}
&  =\sum_{\left\vert j\right\vert \leq d}\left(  \partial_{\theta_{j}}g\left(
\theta,\omega\right)  \right)  ^{2}\leq\left(  2d+1\right)  M^{5}%
\sum_{\left\vert j\right\vert \leq d}\psi_{j}^{2}\left(  \omega\right) \\
&  =M^{5}(2d+1)^{2}.
\end{align*}
Noting also that $\partial_{\theta}g\left(  \theta,\omega\right)  $ is
continuous in $\theta$, the claim follows.
\end{proof}

\bigskip

\begin{proof}
[Proof of Theorem \ref{Theor-asy-normal-improved-est}]\textbf{Step 1. }Lemma
\ref{lem-compact-2} in conjunction with Lemma \ref{lem-compact-1} shows that
the mapping $\theta\rightarrow N_{2m+1}\left(  0,\Phi_{\theta}^{-1}\right)  $
is continuous in total variation norm on the compact $\Theta_{2}^{\prime}$.
According to Lemma \ref{Lem-unif-law-converg-charac} (iii), it suffices to
prove that for every sequence $\left\{  \theta_{n}\right\}  $ such that
$\theta_{n}\rightarrow\theta$ for some $\theta\in\Theta$, one has%
\[
n^{1/2}\left(  \tilde{\theta}_{n}-\theta_{n}\right)  \Longrightarrow
_{d}N_{2d+1}\left(  0,\Phi_{\theta}^{-1}\right)  \text{ under }P_{n,\theta
_{n}}.
\]
From (\ref{expec-Pi-m-j}) we obtain
\[
E_{n,\theta}\mathbf{\bar{\Pi}}_{n}=m^{1/2}W_{m}F_{m}^{-1}\theta
\]
and hence
\begin{align*}
\theta &  =m^{-1/2}F_{m}\left(  W_{m}^{\prime}\hat{\Delta}_{m}^{-1}%
W_{m}\right)  ^{-1}W_{m}^{\prime}\hat{\Delta}_{m}^{-1}\mathbf{\;}E_{n,\theta
}\mathbf{\bar{\Pi}}_{n},\\
n^{1/2}\left(  \tilde{\theta}_{n}-\theta\right)   &  =n^{1/2}\left(
mr\right)  ^{-1/2}F_{m}\left(  W_{m}^{\prime}\hat{\Delta}_{m}^{-1}%
W_{m}\right)  ^{-1}W_{m}^{\prime}\hat{\Delta}_{m}^{-1}r^{1/2}\left(
\mathbf{\bar{\Pi}}_{n}-E_{n,\theta}\mathbf{\bar{\Pi}}_{n}\right)  .
\end{align*}
Here $n^{1/2}\left(  mr\right)  ^{-1/2}=1+o\left(  1\right)  $ due to
(\ref{size-of-blocks-def}) and $F_{m}\rightarrow I_{2d+1}$ due to
(\ref{Fm-matrix-def}). Hence it suffices to prove that for all sequences
$\theta_{n}$ converging to some $\theta$
\begin{equation}
\left(  W_{m}^{\prime}\hat{\Delta}_{m}^{-1}W_{m}\right)  ^{-1}W_{m}^{\prime
}\hat{\Delta}_{m}^{-1}r^{1/2}\left(  \mathbf{\bar{\Pi}}_{n}-E_{n,\theta_{n}%
}\mathbf{\bar{\Pi}}_{n}\right)  \Longrightarrow_{d}N_{2d+1}\left(
0,\Phi_{\theta}^{-1}\right)  \text{ under }P_{n,\theta_{n}}.
\label{suffices-1-theorem}%
\end{equation}
The sequence $\left\{  \theta_{n}\right\}  \subset\Theta_{2}^{\prime}$ will be
considered fixed henceforth and $P_{n,\theta_{n}}$ is assumed to be joint
distribution of the $\mathbb{R}^{m}$-valued random vectors $\mathbf{\Pi}%
_{j,m}$, $j=1,\ldots,r$ from (\ref{Pi-k-quer-def}).

\bigskip

\textbf{Step 2. }We claim
\begin{equation}
\left(  W_{m}^{\prime}\hat{\Delta}_{m}^{-1}W_{m}\right)  ^{-1}\rightarrow
_{p}\Phi_{\theta}^{-1} \label{claim-conv-multiplier-covmatrix}%
\end{equation}
(convergence in probability of the $\left(  2d+1\right)  \times\left(
2d+1\right)  $ matrix). Note that
\[
\left\Vert W_{m}^{\prime}\hat{\Delta}_{m}^{-1}W_{m}-W_{m}^{\prime}%
\Delta_{m,\theta_{n}}^{-1}W_{m}\right\Vert _{2}^{2}\leq\left\Vert \hat{\Delta
}_{m}^{-1}-\Delta_{m,\theta_{n}}^{-1}\right\Vert _{2}^{2}%
\]%
\begin{align}
&  \leq m\sup_{\omega\in\left[  -\pi,\pi\right]  }\left(  \left(
a_{\bar{\theta}_{n}}^{2}\left(  \omega\right)  -1\right)  ^{-1}-\left(
a_{\theta_{n}}^{2}\left(  \omega\right)  -1\right)  ^{-1}\right)
^{2}\label{similar-to-deltas}\\
&  \leq m\;L^{2}\left\Vert \bar{\theta}_{n}-\theta_{n}\right\Vert ^{2}\text{
(Lemma \ref{Lem-Lipschitz-uniform})}\nonumber\\
&  \leq m\;L^{2}\left\Vert \hat{\theta}_{n}-\theta_{n}\right\Vert ^{2}\text{
(projection property of }\bar{\theta}_{n}\text{) }\nonumber\\
&  \rightarrow_{p}0 \label{conv-0-differ-multiplier}%
\end{align}
where the last claim follows from Lemma \ref{lem-unif-consistency-prelim-est})
and $m\sim\log n=o\left(  n\right)  $. Furthermore note that for each element
$(j,k)$ of $W_{m}^{\prime}\Delta_{m,\theta_{n}}^{-1}W_{m}$ we have
\begin{align*}
\mathbf{w}_{j,m}^{\prime}\Delta_{m,\theta_{n}}^{-1}\mathbf{w}_{k,m}  &
=m^{-1}\sum_{s\leq\left(  m-1\right)  /2}\left(  a_{\theta_{n}}^{2}\left(
\omega_{s,m}\right)  -1\right)  ^{-1}\psi_{j}\left(  \omega_{s,m}\right)
\psi_{k}\left(  \omega_{s,m}\right) \\
&  =\frac{1}{2\pi}\int_{\left(  -\pi,\pi\right)  }\left(  a_{\theta}%
^{2}\left(  \omega\right)  -1\right)  ^{-1}\psi_{j}\left(  \omega\right)
\psi_{k}\left(  \omega\right)  d\omega+o\left(  1\right)  .
\end{align*}
where the convergence to the integral follows from Lemma
\ref{Lem-Lipschitz-uniform} and $\theta_{n}\rightarrow\theta$. Hence by
(\ref{parametric-fisher-info-def-a})%
\[
\mathbf{w}_{j,m}^{\prime}\Delta_{m,\theta_{n}}^{-1}\mathbf{w}_{k,m}%
=\Phi_{\theta,jk}+o\left(  1\right)  .
\]
The last relation and (\ref{conv-0-differ-multiplier}) imply
(\ref{claim-conv-multiplier-covmatrix}). For (\ref{suffices-1-theorem}) it now
suffices to prove
\begin{equation}
W_{m}^{\prime}\hat{\Delta}_{m}^{-1}r^{1/2}\left(  \mathbf{\bar{\Pi}}%
_{n}-E_{n,\theta_{n}}\mathbf{\bar{\Pi}}_{n}\right)  \Longrightarrow
_{d}N_{2d+1}\left(  0,\Phi_{\theta}\right)  .\text{ }
\label{suffices-2-theorem}%
\end{equation}

\bigskip

\textbf{Step 3. }We claim
\begin{equation}
W_{m}^{\prime}\left(  \hat{\Delta}_{m}^{-1}-\Delta_{m,\theta_{n}}^{-1}\right)
r^{1/2}\left(  \mathbf{\bar{\Pi}}_{n}-E_{n,\theta_{n}}\mathbf{\bar{\Pi}}%
_{n}\right)  \rightarrow_{p}0 \label{suffices-2-theorem-a}%
\end{equation}
(convergence in probability of a $2d+1$-vector). Indeed we have
\begin{equation}
\left\Vert W_{m}^{\prime}\left(  \hat{\Delta}_{m}^{-1}-\Delta_{m,\theta_{n}%
}^{-1}\right)  r^{1/2}\left(  \mathbf{\bar{\Pi}}_{n}-E_{n,\theta_{n}%
}\mathbf{\bar{\Pi}}_{n}\right)  \right\Vert ^{2}\leq\lambda_{\max}\left(
\hat{\Delta}_{m}^{-1}-\Delta_{m,\theta_{n}}^{-1}\right)  ^{2}\left\Vert
r^{1/2}\left(  \mathbf{\bar{\Pi}}_{n}-E_{n,\theta_{n}}\mathbf{\bar{\Pi}}%
_{n}\right)  \right\Vert ^{2}. \label{conjunction-with-EV-2}%
\end{equation}
Here analogously to (\ref{similar-to-deltas})- (\ref{conv-0-differ-multiplier}%
) one obtains, in view of $m^{2}\sim\left(  \log n\right)  ^{2}=o\left(
n\right)  ,$
\begin{equation}
m^{2}\;\lambda_{\max}\left(  \hat{\Delta}_{m}^{-1}-\Delta_{m,\theta_{n}}%
^{-1}\right)  ^{2}\rightarrow_{p}0. \label{conjunction-with-EV}%
\end{equation}
Recall that $\mathbf{\bar{\Pi}}_{n}=r^{-1}\sum_{j=1}^{r}\mathbf{\Pi}_{m,j}$
(cf. (\ref{Pi-k-quer-def})) where $\mathbf{\Pi}_{m,j}$ are i.i.d. vectors;
hence
\[
\mathrm{Cov}_{n,\mathbf{\theta}_{n}}\left(  r^{1/2}\left(  \mathbf{\bar{\Pi}%
}_{n}-E_{n,\theta_{n}}\mathbf{\bar{\Pi}}_{n}\right)  \right)  =\mathrm{Cov}%
_{n,\mathbf{\theta}_{n}}\left(  \mathbf{\Pi}_{m,1}\right)
\]
and consequently
\begin{align*}
E_{n,\mathbf{\theta}_{n}}\left\Vert r^{1/2}\left(  \mathbf{\bar{\Pi}}%
_{n}-E_{n,\theta_{n}}\mathbf{\bar{\Pi}}_{n}\right)  \right\Vert ^{2}  &
=\mathrm{Tr}\;\mathrm{Cov}_{n,\mathbf{\theta}_{n}}\left(  \mathbf{\Pi}%
_{m,1}\right) \\
&  =\mathrm{Tr}\;\left(  \left(  U_{m}^{\ast}A_{m}\left(  a_{\theta_{n}%
}\right)  U_{m}\right)  ^{\left[  2\right]  }-I_{m}\right)  ,
\end{align*}
in view of (\ref{Cov-matrix-Pi}), where $A_{m}\left(  a_{\theta_{n}}\right)  $
is the $m\times m$ symbol matrix pertaining to spectral density $a_{\theta
_{n}}$ and For a $m\times m$ matrix $M$, the real matrix $M^{\left[  2\right]
}$ is defined in (\ref{M-brackets-2-def}). Hence
\[
\mathrm{Tr}\;\mathrm{Cov}_{n,\mathbf{\theta}_{n}}\left(  \mathbf{\Pi}%
_{m,1}\right)  =\sum_{\left\vert j\right\vert \leq\left(  m-1\right)
/2}\left(  \mathbf{u}_{j}^{\ast}A_{m}\left(  a_{\theta_{n}}\right)
\mathbf{u}_{j}\right)  ^{2}-m
\]
where $\mathbf{u}_{j}$ are the $m$-vectors defined in
(\ref{epsilon-u-vectors-def}), (\ref{special-unitary-DFT-def}) for the current
value of $m$. Then Lemma \ref{lem-toeplitz-EV} implies that for a constant
$C_{M}$ depending only on $M$ we have $\mathbf{u}_{j}^{\ast}A_{m}\left(
a_{\theta_{n}}\right)  \mathbf{u}_{j}\leq C_{M}$ and hence
\[
m^{-2}\;E_{n,\mathbf{\theta}_{n}}\left\Vert r^{1/2}\left(  \mathbf{\bar{\Pi}%
}_{n}-E_{n,\theta_{n}}\mathbf{\bar{\Pi}}_{n}\right)  \right\Vert ^{2}\leq
m^{-1}\left(  C_{M}^{2}-1\right)  =o\left(  1\right)  .
\]
Hence
\[
m^{-2}\;\left\Vert r^{1/2}\left(  \mathbf{\bar{\Pi}}_{n}-E_{n,\theta_{n}%
}\mathbf{\bar{\Pi}}_{n}\right)  \right\Vert ^{2}\rightarrow_{p}0
\]
which in conjunction with (\ref{conjunction-with-EV}) and
(\ref{conjunction-with-EV-2}) implies (\ref{suffices-2-theorem-a}). For
(\ref{suffices-2-theorem}) it now suffices to prove
\begin{equation}
S_{n}:=W_{m}^{\prime}\Delta_{m,\theta_{n}}^{-1}r^{1/2}\left(  \mathbf{\bar
{\Pi}}_{n}-E_{n,\theta_{n}}\mathbf{\bar{\Pi}}_{n}\right)  \Longrightarrow
_{d}N_{2d+1}\left(  0,\Phi_{\theta}\right)  .\text{ }
\label{suffices-3-theorem}%
\end{equation}

\bigskip

\textbf{Step 4. }We claim that
\begin{equation}
\lim_{n\rightarrow\infty}\mathrm{Cov}_{n,\mathbf{\theta}_{n}}\left(
S_{n}\right)  =\Phi_{\theta}. \label{claim-covmatrix-1}%
\end{equation}
Indeed, following the steps in the proof of Lemma
\ref{lem-covmatrix-prelim-est}, we obtain
\begin{align}
\mathrm{Cov}_{n,\mathbf{\theta}_{n}}\left(  T_{1,n}\right)   &  =W_{m}%
^{\prime}\Delta_{m,\theta_{n}}^{-1}\mathrm{Cov}_{n,\mathbf{\theta}_{n}}\left(
\mathbf{\Pi}_{1,m}\right)  \Delta_{m,\theta_{n}}^{-1}W_{m}\nonumber\\
&  =W_{m}^{\prime}\Delta_{m,\theta_{n}}^{-1}\left(  U_{m}^{\ast}A_{m}%
U_{m}\right)  ^{\left[  2\right]  }\Delta_{m,\theta_{n}}^{-1}W_{m}%
-W_{m}^{\prime}\Delta_{m,\theta_{n}}^{-2}W_{m}. \label{CovTn-decomp}%
\end{align}
According to (\ref{matrix-W-def}), the column vectors of the matrix
$\Delta_{m,\theta_{n}}^{-1}W_{m}$ are
\[
\mathbf{\tilde{w}}_{j,m}:=\Delta_{m,\theta_{n}}^{-1}\mathbf{w}_{j,m}=\left(
\;\left(  a_{\theta_{n}}^{2}\left(  \omega_{s,m}\right)  -1\right)
^{-1}m^{-1/2}\psi_{j}\left(  \omega_{s,m}\right)  \right)  _{\left\vert
s\right\vert \leq\left(  m-1\right)  /2}.
\]
In the proof of Lemma \ref{lem-covmatrix-prelim-est}, relation
(\ref{relation-quoted-later}) it has been shown that the element $(j,k)$ of
the matrix $W_{m}^{\prime}\left(  U_{m}^{\ast}A_{m}U_{m}\right)  ^{\left[
2\right]  }W_{m}$ satisfies%
\[
\mathbf{w}_{j,m}^{\prime}\left(  U_{m}^{\ast}A_{m}U_{m}\right)  ^{\left[
2\right]  }\mathbf{w}_{k,m}=\mathbf{w}_{j,m}^{\prime}\Lambda_{m}^{2}%
\mathbf{w}_{k,m}+o\left(  1\right)
\]
where $\Lambda_{m}=\Lambda_{m,\theta}$ is defined by
(\ref{lambda-diag-matrix-def}), with $\theta=\theta_{n}$ currently. For the
vectors $\mathbf{w}_{j,m}$ that proof only used the fact that all components
of $\mathbf{w}_{j,m}$ and $\mathbf{w}_{k,m}$ are bounded in modulus by
$\sqrt{2}m^{-1/2}$. Replacing $\mathbf{w}_{j,m}$ by $\mathbf{\tilde{w}}_{j,m}%
$, we note that all components are bounded in modulus by $M\sqrt{2}m^{-1/2}$,
due to (\ref{bounds-on-athetasquared}). Therefore we have
\[
\mathbf{\tilde{w}}_{j,m}^{\prime}\left(  U_{m}^{\ast}A_{m}U_{m}\right)
^{\left[  2\right]  }\mathbf{\tilde{w}}_{k,m}=\mathbf{\tilde{w}}_{j,m}%
^{\prime}\Lambda_{m}^{2}\mathbf{\tilde{w}}_{k,m}+o\left(  1\right)  ,
\]
hence from (\ref{CovTn-decomp}) the element $(j,k)$ of $\mathrm{Cov}%
_{n,\mathbf{\theta}_{n}}\left(  T_{1,n}\right)  $ is
\begin{align*}
\mathbf{\tilde{w}}_{j,m}^{\prime}\left(  \Lambda_{m,\theta_{n}}^{2}%
-I_{2d+1}\right)  \mathbf{\tilde{w}}_{k,m}+o\left(  1\right)   &
=\mathbf{\tilde{w}}_{j,m}^{\prime}\Delta_{m,\theta_{n}}\mathbf{\tilde{w}%
}_{k,m}+o\left(  1\right) \\
&  =\mathbf{w}_{j,m}^{\prime}\Delta_{m,\theta_{n}}^{-1}\mathbf{w}%
_{k,m}+o\left(  1\right)
\end{align*}%
\[
=m^{-1}\sum_{\left\vert j\right\vert \leq\left(  m-1\right)  /2}\left(
a_{\theta}^{2}\left(  \omega_{j,m}\right)  -1\right)  ^{-1}\psi_{j}\left(
\omega_{j,m}\right)  \psi_{k}\left(  \omega_{j,m}\right)  +o\left(  1\right)
.
\]
This expression converges to
\[
\frac{1}{2\pi}\int_{\left(  -\pi,\pi\right)  }\left(  a_{\theta}^{2}\left(
\omega\right)  -1\right)  ^{-1}\psi_{j}\left(  \omega\right)  \psi_{k}\left(
\omega\right)  d\omega,
\]
in view of Lemma \ref{Lem-Lipschitz-uniform}. The claim
(\ref{claim-covmatrix-1}) is proved.

\bigskip

\textbf{Step 5. }We use the Lindeberg-Feller Theorem to show
(\ref{suffices-3-theorem}). Consider independent random $d$-vectors
\[
X_{n,j}=W_{m}^{\prime}\Delta_{m,\theta_{n}}^{-1}\left(  \mathbf{\Pi}%
_{m,j}-E_{n,\theta_{n}}\mathbf{\Pi}_{m,j}\right)  \text{, }j=1,\ldots,r\text{
}%
\]
with $\mathbf{\Pi}_{m,j}$ from (\ref{Pi-k-quer-def}). Then $X_{n,j}$ are
identically distributed with $E_{\theta_{n}}X_{n,j}=0,$ and $\sum_{j=1}%
^{r}r^{-1/2}X_{n,j}=S_{n}$. In view of (\ref{claim-covmatrix-1}), it suffices
to establish the Lindeberg condition: for every $\varepsilon>0$%
\begin{equation}
r^{-1}\sum_{j=1}^{r}E_{n,\mathbf{\theta}_{n}}\left\Vert X_{n,j}\right\Vert
^{2}\mathbf{1}\left\{  r^{-1}\left\Vert X_{n,j}\right\Vert ^{2}>\varepsilon
\right\}  \rightarrow0 \label{Lindebg-2}%
\end{equation}
or equivalently
\begin{equation}
E_{n,\mathbf{\theta}_{n}}\left\Vert X_{n,1}\right\Vert ^{2}\mathbf{1}\left\{
\left\Vert X_{n,1}\right\Vert ^{2}>\varepsilon r\right\}  \rightarrow0.
\label{Lindebg3}%
\end{equation}
Define
\[
Y_{n}:=\mathbf{\Pi}_{m,1}-E_{n,\theta_{n}}\mathbf{\Pi}_{m,1},
\]
then in view of (\ref{bounds-on-athetasquared}) we have $\left\Vert
X_{n,1}\right\Vert \leq M\left\Vert Y_{n}\right\Vert $ and hence for
(\ref{Lindebg3}) it suffices to show
\[
E_{n,\mathbf{\theta}_{n}}\left\Vert Y_{n}\right\Vert ^{2}\mathbf{1}\left\{
\left\Vert X_{n,1}\right\Vert ^{2}>\varepsilon r\right\}  \rightarrow0.
\]
Applying the Cauchy-Schwarz and Markov inequalities, we obtain
\[
E_{n,\mathbf{\theta}_{n}}\left\Vert Y_{n}\right\Vert ^{2}\mathbf{1}\left\{
\left\Vert X_{n,1}\right\Vert ^{2}>\varepsilon r\right\}  \leq\left(
E_{n,\mathbf{\theta}_{n}}\left\Vert Y_{n}\right\Vert ^{4}\right)
^{1/2}\left(  \frac{E_{n,\mathbf{\theta}_{n}}\left\Vert X_{n,1}\right\Vert
^{2}}{\varepsilon r}\right)  ^{1/2}.
\]
Here, since $\mathrm{Cov}_{n,\mathbf{\theta}_{n}}\left(  S_{n}\right)
=\mathrm{Cov}_{n,\mathbf{\theta}_{n}}\left(  X_{n,1}\right)  $, we have
\[
E_{n,\mathbf{\theta}_{n}}\left\Vert X_{n,1}\right\Vert ^{2}=\mathrm{Tr}%
\;\mathrm{Cov}_{n,\mathbf{\theta}_{n}}\left(  S_{n}\right)  =O\left(
1\right)
\]
due to (\ref{claim-covmatrix-1}). It now suffices to show
\begin{equation}
r^{-1}E_{n,\mathbf{\theta}_{n}}\left\Vert Y_{n}\right\Vert ^{4}=o\left(
1\right)  . \label{obtain-for-Lindeberg}%
\end{equation}
Recall that according to Subsection \ref{Subsec-partition}, the random vector
$\mathbf{\Pi}_{m,1}$ has the same distribution as $\mathbf{\Pi}_{m}=\left(
\Pi_{j}\right)  _{\left\vert j\right\vert \leq\left(  m-1\right)  /2}$ given
by (\ref{vec-of-observables}) with $n$ replaced by $m$, where according to
(\ref{Pi-observble-component-def}).
\[
\Pi_{j}=2\hat{B}_{j}^{\ast}\hat{B}_{j}+\mathbf{1,}\text{ }\left\vert
j\right\vert \leq\left(  m-1\right)  /2.
\]
Hence
\begin{align*}
r^{-1}E_{n,\mathbf{\theta}_{n}}\left\Vert Y_{n}\right\Vert ^{4}  &
=r^{-1}E_{n,\mathbf{\theta}_{n}}\left(  \sum_{\left\vert j\right\vert
\leq\left(  m-1\right)  /2}\left(  \Pi_{j}-E_{n,\mathbf{\theta}_{n}}\Pi
_{j}\right)  ^{2}\right)  ^{2}\\
&  \leq\frac{m}{r}\sum_{\left\vert j\right\vert \leq\left(  m-1\right)
/2}E_{n,\mathbf{\theta}_{n}}\left(  \Pi_{j}-E_{n,\mathbf{\theta}_{n}}\Pi
_{j}\right)  ^{4}.
\end{align*}
Further note that for the observables $\tilde{Q}_{j},\tilde{P}_{j}$,
$j=1,\ldots,m$ defined in (\ref{analogs-P-Q-def}) for $n=m$, one has
\[
\hat{B}_{j}^{\ast}\hat{B}_{j}=\frac{1}{2}\left(  \tilde{Q}_{j+\left(
m+1\right)  /2}^{2}+\tilde{P}_{j+\left(  m+1\right)  /2}^{2}-\mathbf{1}%
\right)  ,\;\left\vert j\right\vert \leq\left(  m-1\right)  /2
\]
in analogy to (\ref{numb-oper}), by the argument about $\tilde{Q}_{j}%
,\tilde{P}_{j}$ used in the proof of Lemma \ref{lem-covmatrix-analog-B}. To
shorten notation, we now write $s(j):=j-\left(  m+1\right)  /2$ for
$j=1,\ldots,m$. Hence
\[
\Pi_{s(j)}=\tilde{Q}_{j}^{2}+\tilde{P}_{j}^{2}\text{, }j=1,\ldots,m
\]
and for $j=1,\ldots,m$
\[
E_{n,\mathbf{\theta}_{n}}\left(  \Pi_{s(j)}-E_{n,\mathbf{\theta}_{n}}%
\Pi_{s(j)}\right)  ^{4}=E_{n,\mathbf{\theta}_{n}}\left(  \tilde{Q}_{j}%
^{2}-E_{n,\mathbf{\theta}_{n}}\tilde{Q}_{j}^{2}+\tilde{P}_{j}^{2}%
-E_{n,\mathbf{\theta}_{n}}\tilde{P}_{j}^{2}\right)  ^{4}%
\]%
\[
\leq8\;\left(  E_{n,\mathbf{\theta}_{n}}\left(  \tilde{Q}_{j}^{2}%
-E_{n,\mathbf{\theta}_{n}}\tilde{Q}_{j}^{2}\right)  ^{4}+E_{n,\mathbf{\theta
}_{n}}\left(  \tilde{P}_{j}^{2}-E_{n,\mathbf{\theta}_{n}}\tilde{P}_{j}%
^{2}\right)  ^{4}\right)  .
\]
By (\ref{claim-anlogous-R-tilde}), $\tilde{Q}_{j}$ has a normal distribution
$\tilde{Q}_{j}\sim N\left(  0,u_{s\left(  j\right)  }^{\ast}A_{m}u_{s\left(
j\right)  }\right)  $ where $A_{m}=A_{m}\left(  a_{\theta_{n}}\right)  $.
Writing $\tilde{Q}_{j}=\left(  u_{s\left(  j\right)  }^{\ast}Au_{s\left(
j\right)  }\right)  ^{1/2}Z$ for a standard normal $Z$, we obtain%
\[
E_{\mathbf{\theta}}\left(  \tilde{Q}_{j}^{2}-E_{\mathbf{\theta}}\tilde{Q}%
_{j}^{2}\right)  ^{4}=\left(  u_{s\left(  j\right)  }^{\ast}A_{m}u_{s\left(
j\right)  }\right)  ^{4}\mu_{4}%
\]
where $\mu_{4}$ is the fourth central moment of $N\left(  0,1\right)  $.
Applying the same reasoning to $\tilde{P}_{j}\sim N\left(  0,u_{s\left(
j\right)  }^{\ast}Au_{s\left(  j\right)  }\right)  $, we obtain
\[
\frac{m}{r}\sum_{\left\vert j\right\vert \leq\left(  m-1\right)
/2}E_{n,\mathbf{\theta}_{n}}\left(  \Pi_{j}-E_{n,\mathbf{\theta}_{n}}\Pi
_{j}\right)  ^{4}\leq8\mu_{4}\frac{m^{2}}{r}\max_{\left\vert k\right\vert
\leq\left(  m-1\right)  /2}\left(  u_{k}^{\ast}A_{m}u_{k}\right)  ^{4}.
\]
To bound $u_{k}^{\ast}A_{m}u_{k}$, apply an Lemma\textit{ }%
\ref{lem-toeplitz-EV-2} to conclude that $\left(  u_{k}^{\ast}A_{m}\left(
a_{\theta}\right)  u_{k}\right)  ^{2}\leq\left(  2d+1\right)  M$, for
$\left\vert k\right\vert \leq\left(  m-1\right)  /2$ and spectral densities
$a_{\theta}\in\Theta_{2}\left(  d,M\right)  $, Since $m^{2}/r\rightarrow0$, we
obtain (\ref{obtain-for-Lindeberg}) and hence (\ref{suffices-3-theorem}).
\end{proof}

%

\begin{privatenotes}
\begin{boxedminipage}{\textwidth}%

\begin{sfblock}
This last Step 5 needs particularly careful revision. Make it clear early that
$A_{m}=$ $A_{m}\left(  a_{\theta}\right)  $. Also check in the whole proof for
doubling arguments involving $\mathrm{Cov}_{n,\mathbf{\theta}_{n}}\left(
\mathbf{\Pi}_{1,m}\right)  $.
\end{sfblock}

%

\end{boxedminipage}
\end{privatenotes}%
.

\subsection{A deficiency bound from limit distributions}

We now show how uniform asymptotic normality an estimator can be used to
establish a bound on the one sided Le Cam deficiency. The result is inspired
by the two theorems in \cite{MR618863}. For the definition of uniform
convergence in distribution, of the bounded Lipschitz norm $\left\Vert
\cdot\right\Vert _{BL}$ for functions and the bounded Lipschitz metric $\beta$
for probability measures cf. Section \ref{subsec-uniform-in-law}.

\begin{theorem}
\label{Theor-deficiency-from-law-converg}Consider a sequence of experiments
$\mathcal{P}_{n}=\left(  X_{n},\Omega_{n},P_{n,\theta},\theta\in\Theta\right)
$ where $P_{n,\theta}$ are probability measures on $\left(  X_{n},\Omega
_{n}\right)  $, and $\Theta$ is a compact subset of $\mathbb{R}^{d}$. Assume
that for a sequence of statistics $\hat{\theta}_{n}:\left(  X_{n,}\Omega
_{n}\right)  \rightarrow\left(  \mathbb{R}^{d},\mathfrak{B}^{d}\right)  $ one
has
\begin{equation}
\mathcal{L}\left(  \sqrt{n}\left(  \hat{\theta}_{n}-\theta\right)
|P_{n,\theta}\right)  \Longrightarrow_{d}N_{d}\left(  0,\Sigma_{\theta
}\right)  \text{ uniformly in }\theta\in\Theta\label{assump-1}%
\end{equation}
where the map $\theta\rightarrow\Sigma_{\theta}$ is continuous in the norm
$\left\Vert \cdot\right\Vert _{2}$ for covariance matrices and $\Sigma
_{\theta}>0,\theta\in\Theta$ . Assume each experiment $\mathcal{P}_{n}$ is
dominated by a sigma-finite measure. Then for experiments
\begin{equation}
\mathcal{Q}_{n}=\left(  \mathbb{R}^{d},\mathfrak{B}^{d},N_{d}\left(
\theta,n^{-1}\Sigma_{\theta}\right)  ,\theta\in\Theta\right)
\label{heteroskedastic-normal}%
\end{equation}
one has
\begin{equation}
\delta\left(  \mathcal{P}_{n},\mathcal{Q}_{n}\right)  \rightarrow0.
\label{claim-def-distance-by-smoothing}%
\end{equation}

\end{theorem}

\begin{proof}
Let $f$ be a measurable function on $\mathbb{R}^{d}$ with $\left\Vert
f\right\Vert _{\infty}\leq1$, set $X_{n,\theta}:=\sqrt{n}\left(  \hat{\theta
}_{n}-\theta\right)  $, and let $Y_{n}$ be a random vector on $\left(
\mathbb{R}^{d},\mathfrak{B}^{d}\right)  $ with $\mathcal{L}\left(
Y_{n}\right)  =N_{d}\left(  0,\Sigma_{\theta}\right)  $. Consider the
following Markov kernel: for $x\in\mathbb{R}^{d}$, $A\in\mathfrak{B}^{d}$ and
some $\gamma\in\left(  0,1\right)  $ set
\[
H_{\gamma}\left(  A,x\right)  =N_{d}\left(  x,\gamma^{2}I_{d}\right)  \left(
A\right)  .
\]
Set $P_{n,\theta}^{\prime}:=\mathcal{L}\left(  \sqrt{n}\left(  \hat{\theta
}_{n}-\theta\right)  |P_{n,\theta}\right)  $, then the law $H_{\gamma
}P_{n,\theta}^{\prime}$ can be described by
\begin{align}
H_{\gamma}P_{n,\theta}^{\prime}  &  =\int H_{\gamma}\left(  \cdot,x\right)
dP_{n,\theta}^{\prime}\left(  x\right) \nonumber\\
&  =\mathcal{L}\left(  X_{n,\theta}+\gamma Z|P_{n,\theta}\right)
\label{law-repre-2}%
\end{align}
where $Z$ is a standard normal $d$-vector independent of $X_{n,\theta}$.
Analogously we have
\begin{equation}
H_{\gamma}N_{d}\left(  0,\Sigma_{\theta}\right)  =\mathcal{L}\left(
Y_{n}+\gamma Z^{\prime}\right)  . \label{add-indep}%
\end{equation}
where $Z^{\prime}$ is a standard normal $d$-vector independent of $Y_{n}$. Now
for the total variation metric (cf. (\ref{TV-metric-def})) we have%
\begin{align}
&  \left\Vert H_{\gamma}P_{n,\theta}^{\prime}-N_{d}\left(  0,\Sigma_{\theta
}\right)  \right\Vert _{TV}\nonumber\\
&  \leq\left\Vert H_{\gamma}P_{n,\theta}^{\prime}-H_{\gamma}N_{d}\left(
0,\Sigma_{\theta}\right)  \right\Vert _{TV}+\left\Vert H_{\gamma}N_{d}\left(
0,\Sigma_{\theta}\right)  -N_{d}\left(  0,\Sigma_{\theta}\right)  \right\Vert
_{TV}. \label{TV-est-2}%
\end{align}
For the first term on the r.h.s. we have (cp. (\ref{TV-metric-and-L1}))
\begin{equation}
\left\Vert H_{\gamma}P_{n,\theta}^{\prime}-H_{\gamma}N_{d}\left(
0,\Sigma_{\theta}\right)  \right\Vert _{TV}=\frac{1}{2}\sup_{\left\Vert
f\right\Vert _{\infty}\leq1}\left\vert \int f\;dH_{\gamma}P_{n,\theta}%
^{\prime}-\int f\;dH_{\gamma}N_{d}\left(  0,\Sigma_{\theta}\right)
\right\vert . \label{TV-est-1}%
\end{equation}
Here
\[
\int f\;dH_{\gamma}P_{n,\theta}^{\prime}=\int g_{f}\left(  x\right)
P_{n,\theta}^{\prime}\left(  dx\right)
\]
where
\[
g_{f}\left(  x\right)  =\int f\left(  t\right)  \;dH_{\gamma}\left(
dt,x\right)  =Ef\left(  x+\gamma Z\right)
\]
and similarly
\[
\int f\;dH_{\gamma}N_{d}\left(  0,\Sigma_{\theta}\right)  =\int g_{f}\left(
x\right)  N_{d}\left(  0,\Sigma_{\theta}\right)  \left(  dx\right)  .
\]
We claim that $g_{f}\left(  x\right)  $ is a Lipschitz function. Indeed for
$h\in\mathbb{R}^{d}$
\[
\left\vert g_{f}\left(  x+h\right)  -g_{f}\left(  x\right)  \right\vert
=\left\vert Ef\left(  x+h+\gamma Z\right)  -Ef\left(  x+\gamma Z\right)
\right\vert
\]%
\[
\leq2\left\Vert N\left(  x+h,\gamma^{2}I_{d}\right)  -N\left(  x,\gamma
^{2}I_{d}\right)  \right\Vert _{TV}\text{ by (\ref{TV-metric-and-L1})}%
\]%
\[
\leq2H\left(  N\left(  x+h,\gamma^{2}I_{d}\right)  ,N\left(  h,\gamma^{2}%
I_{d}\right)  \right)  \text{ by (\ref{Lecam-inequ}).}%
\]
By a well known formula
\[
H^{2}\left(  N\left(  x+h,\gamma^{2}I_{d}\right)  ,N\left(  h,\gamma^{2}%
I_{d}\right)  \right)  =2\left(  1-\exp\left(  -\frac{1}{8\gamma^{2}%
}\left\Vert h\right\Vert ^{2}\right)  \right)  \leq\frac{\left\Vert
h\right\Vert ^{2}}{4\gamma^{2}}%
\]
so that
\[
\left\vert g_{f}\left(  x+h\right)  -g_{f}\left(  x\right)  \right\vert
\leq\frac{\left\Vert h\right\Vert }{\gamma}.
\]
It follows that for $\gamma\leq1$ the function $\gamma g_{f}/2$ satisfies
$\left\Vert f\right\Vert _{BL}\leq1$. By (\ref{TV-est-1})
\begin{align*}
&  \left\Vert H_{\gamma}P_{n,\theta}^{\prime}-H_{\gamma}N_{d}\left(
0,\Sigma_{\theta}\right)  \right\Vert _{TV}\\
&  \leq\sup_{\left\Vert f\right\Vert _{\infty}\leq1}\left\vert \int
g_{f}\left(  x\right)  P_{n,\theta}^{\prime}\left(  dx\right)  -\int
g_{f}\left(  x\right)  N_{d}\left(  0,\Sigma_{\theta}\right)  \left(
dx\right)  \right\vert \\
&  \leq2\gamma^{-1}\;\beta\left(  P_{n,\theta}^{\prime}\left(  dx\right)
,N_{d}\left(  0,\Sigma_{\theta}\right)  \right)  .
\end{align*}
By Lemma \ref{Lem-unif-law-converg-charac} and (\ref{assump-1}) one obtains
for every fixed $\gamma\in\left(  0,1\right)  $
\[
\sup_{\theta\in\Theta}\left\Vert H_{\gamma}P_{n,\theta}^{\prime}-H_{\gamma
}N_{d}\left(  0,\Sigma_{\theta}\right)  \right\Vert _{TV}\rightarrow0.
\]
Hence there is a sequence $\gamma_{n}\rightarrow0$ such that
\begin{equation}
\sup_{\theta\in\Theta}\left\Vert H_{\gamma_{n}}P_{n,\theta}^{\prime}%
-H_{\gamma_{n}}N_{d}\left(  0,\Sigma_{\theta}\right)  \right\Vert
_{TV}\rightarrow0. \label{TV-est-3}%
\end{equation}
Now consider the second term in (\ref{TV-est-2}) for $\gamma=\gamma_{n}$: in
view of (\ref{add-indep}) we have
\[
H_{\gamma}N_{d}\left(  0,\Sigma_{\theta}\right)  =N_{d}\left(  0,\Sigma
_{\theta}+\gamma^{2}I_{d}\right)
\]
and thus
\begin{align*}
&  \left\Vert H_{\gamma_{n}}N_{d}\left(  0,\Sigma_{\theta}\right)
-N_{d}\left(  0,\Sigma_{\theta}\right)  \right\Vert _{TV}\\
&  =\left\Vert N_{d}\left(  0,\Sigma_{\theta}+\gamma_{n}^{2}I_{d}\right)
-N_{d}\left(  0,\Sigma_{\theta}\right)  \right\Vert _{TV}.
\end{align*}
Since the map $\theta\rightarrow\Sigma_{\theta}$ is continuous and
$\Theta\subset\mathbb{R}^{d}$ is compact, the set $\left\{  \Sigma_{\theta
},\theta\in\Theta\right\}  $ is compact in Hilbert-Schmidt norm. Then
$\Sigma_{\theta}>0,\theta\in\Theta$ implies
\[
s_{1}:=\inf\left\{  \lambda_{\min}\left(  \Sigma_{\theta}\right)  :\theta
\in\Theta\right\}  >0,
\]
and by compactness we also have
\[
s_{2}:=\sup\left\{  \lambda_{\max}\left(  \Sigma_{\theta}\right)  :\theta
\in\Theta\right\}  <\infty.
\]
Then by (\ref{Lecam-inequ}), Lemma \ref{lem-covmatrices-cite-GNZpaper} and
$\gamma_{n}\rightarrow0$
\begin{align*}
&  \left\Vert N_{d}\left(  0,\Sigma_{\theta}+\gamma_{n}^{2}I_{d}\right)
-N_{d}\left(  0,\Sigma_{\theta}\right)  \right\Vert _{TV}.\\
&  \leq C\left\Vert \gamma_{n}^{2}I_{d}\right\Vert _{2}=Cd^{1/2}\gamma_{n}%
^{2}\rightarrow0
\end{align*}
since $d$ is fixed here. In conjunction with (\ref{TV-est-3}) and
(\ref{TV-est-2}) this implies
\begin{equation}
\sup_{\theta\in\Theta}\left\Vert H_{\gamma_{n}}P_{n,\theta}^{\prime}%
-N_{d}\left(  0,\Sigma_{\theta}\right)  \right\Vert _{TV}\rightarrow0.
\label{TV-est-4}%
\end{equation}
Consider now a one-to-one transformation of the sample space $\left(
\mathbb{R}^{d},\mathfrak{B}^{d}\right)  $ as $T_{\theta}\left(  x\right)
=n^{-1/2}x+\theta$. For any probability measure $P$ on $\left(  \mathbb{R}%
^{d},\mathfrak{B}^{d}\right)  $ consider the induced measure $\left(
T_{\theta}\circ P\right)  \left(  A\right)  =P\left(  T_{\theta}^{-1}\left(
A\right)  \right)  $, equivalently described by $T_{\theta}\circ
P=\mathcal{L}\left(  T_{\theta}\left(  X\right)  \right)  $ if $P=\mathcal{L}%
\left(  X\right)  $. Note the total variation distance then is invariant: for
any $P,Q$
\begin{equation}
\left\Vert P-Q\right\Vert _{TV}=\left\Vert T_{\theta}\circ P-T_{\theta}\circ
Q\right\Vert _{TV}. \label{invar}%
\end{equation}
Now $T_{\theta}\circ N_{d}\left(  0,\Sigma_{\theta}\right)  =N_{d}\left(
\theta,n^{-1}\Sigma_{\theta}\right)  $ and by (\ref{law-repre-2})
\[
T_{\theta}\circ H_{\gamma_{n}}P_{n,\theta}^{\prime}=T_{\theta}\circ
\mathcal{L}\left(  \sqrt{n}\left(  \theta_{n}-\theta\right)  +\gamma
_{n}Z|P_{n,\theta}\right)
\]
where $Z$ is a standard normal vector, independent of $\theta_{n}$. Thus
\begin{align*}
T_{\theta}\circ H_{\gamma_{n}}P_{n,\theta}^{\prime}  &  =\mathcal{L}\left(
T_{\theta}\left(  \sqrt{n}\left(  \hat{\theta}_{n}-\theta\right)  +\gamma
_{n}Z\right)  |P_{n,\theta}\right) \\
&  =\mathcal{L}\left(  \hat{\theta}_{n}+n^{-1/2}\gamma_{n}Z|P_{n,\theta
}\right)
\end{align*}
so that from (\ref{TV-est-4}) and (\ref{invar}) we obtain
\[
\sup_{\theta\in\Theta}\left\Vert \mathcal{L}\left(  \hat{\theta}_{n}%
+n^{-1/2}\gamma_{n}Z|P_{n,\theta}\right)  -N_{d}\left(  \theta,n^{-1}%
\Sigma_{\theta}\right)  \right\Vert _{TV}\rightarrow0.
\]
The transition from $P_{n,\theta}$ to $\mathcal{L}\left(  \hat{\theta}%
_{n}+n^{-1/2}\gamma_{n}Z|P_{n,\theta}\right)  $ represents a Markov kernel
operation, so that the claim (\ref{claim-def-distance-by-smoothing}) follows.
\end{proof}

We note that the assumption that the experiments $\mathcal{P}_{n}$ be
dominated is used only to fit the quantum version
(\ref{deficiency-quantum-def}) of the deficiency $\delta\left(  \mathcal{P}%
_{n},\mathcal{Q}_{n}\right)  $ which covers classical experiment only if these
are dominated (cf. Section \ref{subsubsec-qu-statist-experiments}, last paragraph)

\subsection{Le Cam's globalization method}

The "heteroskedastic normal experiment" (\ref{heteroskedastic-normal})
resulting from Theorem \ref{Theor-deficiency-from-law-converg} arises as a
global approximation, roughly speaking, in regular parametric models with
asymptotic normalized information matrix $\Sigma_{\theta}^{-1}$; cf.
\cite{MR0395005} and discussions in \cite{MR840521}, \cite{MR1425959}. We will
utilize this result as a tool in our quest for lower information bounds for
the quantum time series. Below we cite Le Cam's original result and then give
an application in our context.

For an experiment $\mathcal{P}=\left(  X,\Omega,P_{\theta},\theta\in
\Theta\right)  $ and a $S\subset\Theta$ we denote the "localized" experiment
by $\mathcal{P}_{S}:=$ $\left(  X,\Omega,P_{\theta},\theta\in S\right)  $. We
will frequently omit the sample spaces from notation, with the understanding
that they may be different for different experiments. All experiments are
assumed to be dominated by sigma-finite measures on their respective sample spaces.

\begin{proposition}
\label{Prop-Lecam-local}(Theorem 1 in \cite{MR0395005}) Let $\mathcal{P}%
=\left(  P_{\theta},\theta\in\Theta\right)  $ and $\mathcal{Q}=\left(
Q_{\theta},\theta\in\Theta\right)  $ be two dominated experiments indexed by
the set $\Theta$. Assume that $\Theta$ is metrized by $W$, that $0\leq a<b$
are given. Assume also that \newline(i) any subset of diameter $4b+2a$ of
$\Theta$ can be covered by no more than $C$ sets of diameter $b$,\newline(ii)
if $S\subset\Theta$ has a diameter $3b$ then the deficiency $\delta\left(
\mathcal{P}_{S},\mathcal{Q}_{S}\right)  $ does not exceed $\varepsilon_{1}%
,$\newline(iii) there is an estimator $\hat{\theta}_{n}$ available on
$\mathcal{P}$ such that $P_{\theta}\left(  W\left(  \hat{\theta}_{n}%
,\theta\right)  >a\right)  \leq\varepsilon_{2}$ for all $\theta\in\Theta
.$\newline Then
\[
\delta\left(  \mathcal{P},\mathcal{Q}\right)  \leq\varepsilon_{1}%
+\varepsilon_{2}+\frac{1}{2}\frac{a}{b}C.
\]

\end{proposition}

\bigskip

The coverage condition on $\Theta$ is well known to be related to the
dimension of $\Theta$. A set $S\subset\Theta$ has diameter $b$ if
$b=\sup_{s,t\in S}W\left(  s,t\right)  $. Since $a<b$, a stronger condition
than (i) above is: any subset of diameter $6b$ of $\Theta$ can be covered by
no more than $C$ sets of diameter $b$. If $\Theta\subset\mathbb{R}^{d}$, a
crude bound for $C$ can be given as follows.\footnote{In the paper, this lemma
will have to be replaced by a reference.}

\begin{lemma}
\label{lem-cover}Assume $\Theta\subset\mathbb{R}^{d}$ and $W\left(  \theta
_{1},\theta_{2}\right)  $ is euclidean distance. Then $C$ can be chosen as
$\left(  12d\right)  ^{d}$.
\end{lemma}

\begin{proof}
Assume $S\subset\Theta$ has diameter $6b$. Then it is contained in a ball of
radius $6b$. This ball is contained in a square of side length $12b$. The
square can be partitioned into $12^{d}$ squares of side length $b$. Each of
these squares has radius $\sqrt{d}b$. Each of these squares can be further
partitioned into $d^{d}$ smaller squares with side length $b/d$, such that the
diameter of these squares is $\sqrt{d}b/d=b/\sqrt{d}\leq b$. Then $S$ can be
covered by the totality of these smaller squares, i.e. by $\left(  12d\right)
^{d}$ sets of diameter $b$.
\end{proof}

\bigskip

We will apply Proposition \ref{Prop-Lecam-local} when $\mathcal{P}$ is an
element of the sequence $\mathcal{P}_{n}=\left(  N_{d}\left(  \theta
,n^{-1}\Sigma_{\theta}\right)  ,\theta\in\Theta\right)  $ and $\Theta$ is a
subset of $\mathbb{R}^{d}$. The claim of Lemma \ref{lem-cover} remains valid
if the euclidean metric $\left\Vert \theta_{1}-\theta_{2}\right\Vert $ is
replaced by $c\left\Vert \theta_{1}-\theta_{2}\right\Vert $ for any $c>0;$ in
particular for $W\left(  \theta_{1},\theta_{2}\right)  =\sqrt{n}\left\Vert
\theta_{1}-\theta_{2}\right\Vert $. Consider some other sequence of dominated
experiments $\mathcal{Q}_{n}=\left(  Q_{n,\theta},\theta\in\Theta\right)  $
and consider localized versions: for $\theta_{0}\in\Theta$ and $r>0$ set%
\begin{equation}
S_{n}\left(  \theta_{0},r\right)  =\left\{  \theta\in\mathbb{R}^{d}%
:n^{1/2}\left\Vert \theta-\theta_{0}\right\Vert \leq r\right\}  ,
\label{parametric-shrinking-neighborhood-def}%
\end{equation}%
\begin{align}
\mathcal{P}_{n}\left(  \theta_{0},r\right)   &  :=\left(  N_{d}\left(
\theta,n^{-1}\Sigma_{\theta}\right)  ,\;,\theta\in\Theta\cap S_{n}\left(
\theta_{0},r\right)  \right)  ,\label{local-version-1}\\
\mathcal{Q}_{n}\left(  \theta_{0},r\right)   &  :=\left(  Q_{n,\theta
},\;\theta\in\Theta\cap S_{n}\left(  \theta_{0},r\right)  \right)  .
\label{local-version-2}%
\end{align}
%

\begin{privatenotes}
\begin{boxedminipage}{\textwidth}%

\begin{sfblock}
Bring this notation in line with the notation conventions for eperiments we
have developed below, using generic $P_{n}\left(  \theta_{0},S\right)  $ for a
set $S$ etc.
\end{sfblock}

%

\end{boxedminipage}
\end{privatenotes}%
.

\begin{lemma}
\label{lem-local-to-global}Assume that the sequence $\mathcal{P}_{n}$
fulfills
\begin{equation}
s_{2}:=\sup_{\theta\in\Theta}\lambda_{\max}\left(  \Sigma_{\theta}\right)
<\infty\label{cond-lambdamax-bounded}%
\end{equation}
and for every $r>0$
\begin{equation}
\sup_{\theta_{0}\in\Theta}\delta\left(  \mathcal{P}_{n}\left(  \theta
_{0},r\right)  ,\mathcal{Q}_{n}\left(  \theta_{0},r\right)  \right)
\rightarrow0. \label{local-experi-delta}%
\end{equation}
Then
\[
\delta\left(  \mathcal{P}_{n},\mathcal{Q}_{n}\right)  \rightarrow0.
\]

\end{lemma}

\begin{proof}
First we show that in $\mathcal{P}_{n}$ an estimator $\hat{\theta}_{n}$ is
available such that for $P_{n,\theta}=N_{d}\left(  \theta,n^{-1}\Sigma
_{\theta}\right)  $
\begin{equation}
\sup_{\theta\in\Theta}P_{n,\theta}\left(  n^{1/2}\left\Vert \hat{\theta}%
_{n}-\theta\right\Vert >a\right)  \rightarrow0\text{ as }a\rightarrow
\infty\label{unif-root-n-consist}%
\end{equation}
($\hat{\theta}_{n}$ is uniformly $\sqrt{n}$-consistent). Indeed let
$\hat{\theta}_{n}$ be the identity map on $\left(  \mathbb{R}^{d}%
,\mathfrak{B}^{d}\right)  $, i.e. a random $d$-vector such that $\mathcal{L}%
\left(  \hat{\theta}_{n}|P_{n,\theta}\right)  =N_{d}\left(  \theta
,n^{-1}\Sigma_{\theta}\right)  $. Then
\[
\mathcal{L}\left(  n^{1/2}\left(  \hat{\theta}_{n}-\theta\right)
|P_{n,\theta}\right)  =N_{d}\left(  0,\Sigma_{\theta}\right)  ,
\]
hence for a standard normal $d$-vector $Z$
\[
\sup_{\theta\in\Theta}P_{n,\theta}\left(  n^{1/2}\left\Vert \hat{\theta}%
_{n}-\theta\right\Vert >a\right)  =\sup_{\theta\in\Theta}P\left(  \left\Vert
\Sigma_{\theta}^{1/2}Z\right\Vert >a\right)
\]%
\[
\leq\sup_{\theta\in\Theta}P\left(  \lambda_{\max}^{1/2}\left(  \Sigma_{\theta
}\right)  \left\Vert Z\right\Vert >a\right)  \leq P\left(  s_{2}%
^{1/2}\left\Vert Z\right\Vert >a\right)  \rightarrow0\text{ as }%
a\rightarrow\infty
\]
so (\ref{unif-root-n-consist}) is shown. Now (\ref{local-experi-delta})
implies that there is a sequence $r_{n}\rightarrow\infty$ such that
\[
\sup_{\theta_{0}\in\Theta}\delta\left(  \mathcal{P}_{n}\left(  \theta
_{0},r_{n}\right)  ,\mathcal{Q}_{n}\left(  \theta_{0},r_{n}\right)  \right)
\rightarrow0.
\]
Let $\varepsilon>0$ and choose $n_{1}$ such that for $n\geq n_{1}$
\[
\sup_{\theta_{0}\in\Theta}\delta\left(  \mathcal{P}_{n}\left(  \theta
_{0},r_{n}\right)  ,\mathcal{P}_{n}\left(  \theta_{0},r_{n}\right)  \right)
\leq\varepsilon/3.
\]
Set $b_{n}=2r_{n}/3$; then the diameter of $S_{n}\left(  \theta_{0}%
,r_{n}\right)  $ is $3b_{n}$. Then choose $n_{2}\geq n_{1}$ such that for
$n\geq n_{2}$ and $a_{n}=$ $\sqrt{b_{n}}$%
\[
\sup_{\theta\in\Theta}P_{n,\theta}\left(  n^{1/2}\left\Vert \hat{\theta}%
_{n}-\theta\right\Vert >a_{n}\right)  \leq\varepsilon/3.
\]
Finally choose $n_{3}\geq n_{2}$ such that for $n\geq n_{3}$ and the constant
$C$ described in Lemma \ref{lem-cover}%
\[
\frac{1}{2}\frac{a_{n}}{b_{n}}C=\frac{C}{2\sqrt{b_{n}}}\leq\varepsilon/3.
\]
By Proposition \ref{Prop-Lecam-local}, for $n\geq n_{3}$ we then have
$\delta\left(  \mathcal{P}_{n},\mathcal{Q}_{n}\right)  \leq\varepsilon.$
\end{proof}

%

\begin{privatenotes}
\begin{boxedminipage}{\textwidth}%

\begin{sfblock}
Explain why we impose domination in Le Cam's result, even though it is not
required in \cite{MR0395005}: in order that $\delta$ can be given in terms of
quantum channels, see Appendix, Subsection "classical channels". In our
definition of $\delta$ on p. 2, we also imposed Polish experiments, so
according to Proposition 9.2 in our N96 paper, all transitions are given by
Markov kernels. Maybe impose it also here? And maybe we should add something
about $\delta$ in the appendix "classical channels".
\end{sfblock}

%

\end{boxedminipage}
\end{privatenotes}%
.

\subsection{Proof of the lower informativity
bound\label{subsec-proof-lower-info-bound}}

Consider again the set $\Theta_{2}^{\prime}=\Theta_{2}^{\prime}\left(
M,d\right)  $ given by (\ref{parametric-model}) and let $\theta_{0}\in
\Theta_{2}^{\prime}$ be a fixed parameter point therein. Recall that the
distribution $Q_{n,2}\left(  a,a_{0}\right)  $ was described by
(\ref{pre-variance-stable-SDE}); with a slight abuse of notation, we write
$Q_{n,2}\left(  \theta,\theta_{0}\right)  $ for this distribution when
$a=a_{\theta}$ and $a_{0}=a_{\theta_{0}}$, so that $Q_{n,2}\left(
\theta,\theta_{0}\right)  $ is described by
\begin{equation}
dY_{\omega}=a_{\theta}\left(  \omega\right)  d\omega+\left(  2\pi/n\right)
^{1/2}\left(  a_{\theta_{0}}^{2}-1\right)  ^{1/2}dW_{\omega}\text{, }\omega
\in\left[  -\pi,\pi\right]  . \label{gwn-local-1}%
\end{equation}
In a similar way, for a subset $S\subset\mathbb{R}^{2d+1}$ we now write
$\mathcal{G}_{n,2}\left(  \theta_{0},S\right)  $ for $\mathcal{G}_{n,2}\left(
a_{0},\Theta\right)  $ if $a_{0}=a_{\theta_{0}}$ and $\Theta=\left\{
a_{\theta},\theta\in S\right\}  $, so that
\begin{align}
\mathcal{G}_{n,2}\left(  \theta_{0},S\right)   &  :=\left(  Q_{n,2}\left(
\theta,\theta_{0}\right)  ,\;\theta\in S\right)  ,\label{gwn-local-1-experi}\\
\mathcal{G}_{n,3}\left(  \theta_{0},S\right)   &  :=\left(  N_{2d+1}\left(
\theta,n^{-1}\Phi_{\theta_{0}}^{-1}\right)  ,\theta\in S\right)
\label{normal-local-experi}%
\end{align}
with $\Phi_{\theta_{0}}$ given by (\ref{parametric-fisher-info-def}).

\begin{lemma}
\label{lem-exact-equiv-GWN}For any $S\subset\mathbb{R}^{2d+1}$, $\theta_{0}%
\in\Theta_{2}^{\prime}$ and each $n$, we have
\[
\Delta\left(  \mathcal{G}_{n,2}\left(  \theta_{0},S\right)  ,\mathcal{G}%
_{n,3}\left(  \theta_{0},S\right)  \right)  =0.
\]

\end{lemma}

\begin{proof}
For $\theta=\left(  \theta_{j}\right)  _{\left\vert j\right\vert \leq d}$ we
have according to (\ref{series-repre-real-spec_density})%
\[
a_{\theta}\left(  \omega\right)  =\sum_{\left\vert j\right\vert \leq d}%
\psi_{j}\left(  \omega\right)  \theta_{j}.
\]
Define a vector of functions $\Psi:=\left(  \psi_{j}\right)  _{\left\vert
j\right\vert \leq d}$ and write $a_{\theta}\left(  \omega\right)
=\theta^{\prime}\Psi\left(  \omega\right)  $. For the likelihood ratio in the
model (\ref{gwn-local-1}) we have
\[
\frac{dQ_{n,2}\left(  \theta,\theta_{0}\right)  }{dQ_{n,2}\left(  0,\theta
_{0}\right)  }\left(  Y\right)  =\exp\left(  \frac{n}{2\pi}\int_{\left[
-\pi,\pi\right]  }a_{\theta}\left(  a_{\theta_{0}}^{2}-1\right)
^{-1}dY_{\omega}-\frac{n}{4\pi}\int_{\left[  -\pi,\pi\right]  }a_{\theta}%
^{2}\left(  a_{\theta_{0}}^{2}-1\right)  ^{-1}d\omega\right)  .
\]
Here we can write
\[
\int_{\left[  -\pi,\pi\right]  }a_{\theta}\left(  a_{\theta_{0}}^{2}-1\right)
^{-1}dY_{\omega}=\theta^{\prime}\int_{\left[  -\pi,\pi\right]  }\Psi\left(
\omega\right)  \left(  a_{\theta_{0}}^{2}\left(  \omega\right)  -1\right)
^{-1}dY_{\omega}.
\]
By the Neyman factorization criterion, the random $2d+1$-vector%
\[
T\left(  Y\right)  =\frac{1}{2\pi}\int_{\left[  -\pi,\pi\right]  }\Psi\left(
a_{\theta_{0}}^{2}-1\right)  ^{-1}dY_{\omega}%
\]
is a sufficient statistic. Then the distributions of $T\left(  Y\right)  $
under $Q_{n,2}\left(  \theta,\theta_{0}\right)  $ for $\theta\in S$ form an
equivalent experiment. Clearly these distributions are $2d+1$-variate normal.
We have
\[
E_{n,\theta}T=\frac{1}{2\pi}\int_{\left[  -\pi,\pi\right]  }\Psi\left(
a_{\theta_{0}}^{2}-1\right)  ^{-1}\Psi^{\prime}\theta\;d\omega.
\]
In view of (\ref{parametric-fisher-info-def}), we have
\[
E_{n,\theta}T=\Phi_{\theta_{0}}\theta.
\]
To find the covariance matrix, observe that for $T\left(  Y\right)  =\left(
T_{j}\left(  Y\right)  \right)  _{\left\vert j\right\vert \leq d}$ we have
\begin{align*}
2\pi\left(  T_{j}\left(  Y\right)  -E_{n,\theta}T_{j}\left(  Y\right)
\right)   &  =\int_{\left[  -\pi,\pi\right]  }\psi_{j}\left(  a_{\theta_{0}%
}^{2}-1\right)  ^{-1}\left(  2\pi/n\right)  ^{1/2}\left(  a_{\theta_{0}}%
^{2}-1\right)  ^{1/2}dW_{\omega}\\
&  =\left(  2\pi/n\right)  ^{1/2}\int_{\left[  -\pi,\pi\right]  }\psi
_{j}\left(  a_{\theta_{0}}^{2}-1\right)  ^{-1/2}dW_{\omega}.
\end{align*}
Consequently
\begin{align*}
&  \mathrm{Cov}_{n,\theta}\left(  T_{j}\left(  Y\right)  ,T_{k}\left(
Y\right)  \right) \\
&  =\frac{1}{2\pi n}\int_{\left[  -\pi,\pi\right]  }\psi_{j}\left(
a_{\vartheta_{0}}^{2}-1\right)  ^{-1}\psi_{k}d\omega\\
&  =n^{-1}\Phi_{\theta_{0},jk}%
\end{align*}
by (\ref{parametric-fisher-info-def-a}). Hence
\[
\mathcal{L}\left(  T\left(  Y\right)  |Q_{n,2}\left(  \theta,\theta
_{0}\right)  \right)  =N_{2d+1}\left(  \Phi_{\theta_{0}}\theta,n^{-1}%
\Phi_{\theta_{0}}\right)
\]
and the respective experiment with $\theta\in S$ is equivalent to
$\mathcal{G}_{n,2}\left(  \theta_{0},S\right)  $. Define
\begin{equation}
\tilde{T}\left(  Y\right)  :=\Phi_{\theta_{0}}^{-1}T\left(  Y\right)  ;
\label{one-to-one-transf}%
\end{equation}
then
\begin{equation}
\mathcal{L}\left(  \tilde{T}\left(  Y\right)  |Q_{n,2}\left(  \theta
,\theta_{0}\right)  \right)  =N_{2d+1}\left(  \theta,n^{-1}\Phi_{\theta_{0}%
}^{-1}\right)  . \label{describes}%
\end{equation}
Since (\ref{one-to-one-transf}) is a one-to-one transformation of the data,
giving an equivalent experiment, and (\ref{describes}) with $\theta\in S$
describes $\mathcal{G}_{n,3}\left(  \theta_{0},S\right)  $, the claim is proved.
\end{proof}

\bigskip

Recall that the distribution $Q_{n}\left(  a\right)  $ was described by
(\ref{SDE-1}); we now write $Q_{n,1}\left(  \theta\right)  $ for this
distribution when $a=a_{\theta}$ so that $Q_{n,1}\left(  \theta\right)  $ is
described by
\[
dY_{\omega}=\mathrm{arc\cosh}\left(  a_{\theta}\left(  \omega\right)  \right)
d\omega+\left(  2\pi/n\right)  ^{1/2}dW_{\omega},\omega\in\left[  -\pi
,\pi\right]  .
\]
Analogously to (\ref{gwn-local-1-experi}), (\ref{normal-local-experi}) for
$S\subset\Theta_{2}^{\prime}\left(  d,M\right)  $, define experiments%
\begin{align}
\mathcal{G}_{n,1}\left(  S\right)   &  =\left(  Q_{n,1}\left(  \theta\right)
,\;\theta\in S\right)  ,\label{gwn-global-experi-def}\\
\mathcal{G}_{n,4}\left(  S\right)   &  :=\left(  N_{2d+1}\left(  \theta
,n^{-1}\Phi_{\theta}^{-1}\right)  ,\theta\in S\right)  .
\label{heteroskedastic-normal-def-a}%
\end{align}
Recall the definitions of parameter sets $\Theta_{1,c}\left(  1,M\right)  $ in
(\ref{Big-Theta-11-c-def}), of neighborhoods $S_{n}\left(  \theta
_{0},r\right)  $ for $r>0$ and $\theta_{0}\in\Theta_{2}^{\prime}\left(
d,M\right)  $ in (\ref{parametric-shrinking-neighborhood-def}) and of
neighborhoods $B\left(  a_{\theta_{0}},\gamma_{n}\right)  $ for $\gamma
_{n}=o\left(  1\right)  $ in (\ref{neighborhood-uniform-B-def}).

\begin{lemma}
\label{lem-inclusions-parspace}(i) For all $M>0,$ there exists $M^{\prime}>0$
such that $\left\{  a_{\theta},\;\theta\in\Theta_{2}^{\prime}\left(
d,M\right)  \right\}  \subset\Theta_{1,c}\left(  1,M^{\prime}\right)  .$
\newline(ii) For all $M>0$ and $r>0$, there exists $M^{\prime}>0$ and a
sequence $\gamma_{n}=O\left(  n^{-1/2}\right)  $ such that for all $\theta
_{0}\in\Theta_{2}^{\prime}\left(  d,M\right)  $
\[
\left\{  a_{\theta},\;\theta\in\Theta_{2}^{\prime}\left(  d,M\right)  \cap
S_{n}\left(  \theta_{0},r\right)  \right\}  \subset\Theta_{1,c}\left(
1,M^{\prime}\right)  \cap B\left(  a_{\theta_{0}},\gamma_{n}\right)  .
\]

\end{lemma}

\begin{proof}
(i) Recall the definition of $\Theta_{1,c}\left(  1,M^{\prime}\right)  $ in
(\ref{Big-Theta-11-c-def}). If $\theta\in\Theta_{2}^{\prime}\left(
d,M\right)  $ then
\begin{align}
a_{\theta}\left(  \omega\right)   &  =\sum_{\left\vert j\right\vert \leq
d}\theta_{j}\psi_{j}\left(  \omega\right)  ,\nonumber\\
\left\vert a_{\theta}\left(  \omega\right)  \right\vert  &  \leq\left(
2d+1\right)  ^{1/2}\left\Vert \theta\right\Vert \leq\left(  2d+1\right)
^{1/2}M^{1/2},\label{analogous-to-trivial}\\
\left\vert a_{\theta}^{\prime}\left(  \omega\right)  \right\vert  &
\leq\left(  2d\right)  ^{1/2}\left\Vert \theta\right\Vert \leq\left(
2d\right)  ^{1/2}M^{1/2},\nonumber
\end{align}
hence for $\alpha=1$
\[
\left\Vert a_{\theta}\right\Vert _{C^{\alpha}}\leq\left\Vert a_{\theta
}\right\Vert _{\infty}+\left\Vert a_{\theta}^{\prime}\right\Vert _{\infty}%
\leq2\left(  2d+1\right)  ^{1/2}M^{1/2}.
\]
If $\theta\in\Theta_{2}^{\prime}\left(  d,M\right)  $ then we also have
$\inf_{\omega\in\left[  -\pi,\pi\right]  }a_{\theta}\left(  \omega\right)
\geq1+M^{-1}$, so by choosing $M^{\prime}=\max\left(  2\left(  2d+1\right)
^{1/2}M^{1/2},M\right)  $ we have $\left\Vert a_{\theta}\right\Vert
_{C^{\alpha}}\leq M^{\prime}$ and $a_{\theta}\in\mathcal{L}_{M^{\prime}}$,
i.e. $a_{\theta}\in\Theta_{1,c}\left(  1,M^{\prime}\right)  $.

(ii) If $\theta\in S_{n}\left(  \theta_{0},r\right)  $ then then we have
analogously to (\ref{analogous-to-trivial})
\[
\left\Vert a_{\theta}-a_{\theta_{0}}\right\Vert _{\infty}\leq\left(
2d+1\right)  ^{1/2}\left\Vert \theta-\theta_{0}\right\Vert \leq\left(
2d+1\right)  ^{1/2}n^{-1/2}r=:\gamma_{n}%
\]
and $\gamma_{n}=O\left(  n^{-1/2}\right)  $. In conjunction with (i) the claim
is proved.
\end{proof}

\bigskip

\begin{lemma}
\label{lem-local-equiv-gwn-global-to-hetero}For any $M>0,r>0$ and $\Theta
_{n}\left(  \theta_{0}\right)  :=\Theta_{2}^{\prime}\left(  d,M\right)  \cap
S_{n}\left(  \theta_{0},r\right)  $ we have
\[
\sup_{\theta_{0}\in\Theta_{2}^{\prime}}\Delta\left(  \mathcal{G}_{n,1}\left(
\Theta_{n}\left(  \theta_{0}\right)  \right)  ,\mathcal{G}_{n,4}\left(
\Theta_{n}\left(  \theta_{0}\right)  \right)  \right)  \rightarrow0.
\]

\end{lemma}

\begin{proof}
Consider the experiment $\mathcal{G}_{n,2}\left(  \theta_{0},\Theta
_{n}\right)  $ defined by (\ref{gwn-local-1-experi}).
\[
\mathcal{G}_{n,2}\left(  \theta_{0},\Theta_{n}\right)  :=\left(
Q_{n,2}\left(  \theta,\theta_{0}\right)  ,\;\theta\in\Theta_{n}\right)
\]
with $Q_{n,2}\left(  \theta,\theta_{0}\right)  $ given by (\ref{gwn-local-1}).
We claim that Lemma \ref{lem-local-white-noise} implies that
\begin{equation}
\sup_{\theta_{0}\in\Theta_{2}^{\prime}}\Delta\left(  \mathcal{G}_{n,1}\left(
\Theta_{n}\right)  ,\mathcal{G}_{n,2}\left(  \theta_{0},\Theta_{n}\right)
\right)  \rightarrow0. \label{equiv-local-parametric-1}%
\end{equation}
Indeed it can be seen that $\mathcal{G}_{n,1}\left(  \Theta_{n}\right)  $, as
a set of probability measures, can be considered a subset of $\mathcal{G}%
_{n,1}\left(  \tilde{\Theta}_{n}\right)  $ as defined in (\ref{SDE-1-a}) for
$\tilde{\Theta}_{n}=\Theta_{1,c}\left(  \alpha,M^{\prime}\right)  \cap
B\left(  a_{0},\gamma_{n}\right)  $ for a certain sequence $\gamma_{n},$ a
certain $M^{\prime}>0$ and $\alpha=1$, upon setting $a=a_{\theta}$ and
$a_{0}=a_{\theta_{0}}$. Indeed, even though $\Theta_{n}\subset\mathbb{R}%
^{2d+1}$ and $\tilde{\Theta}_{n}$ is a set of functions, we can compare
$\mathcal{G}_{n,1}\left(  \Theta_{n}\right)  $ and $\mathcal{G}_{n,1}\left(
\tilde{\Theta}_{n}\right)  $ as sets of probability measures on the same
sample space. With that understanding, the claim $\mathcal{G}_{n,1}\left(
\Theta_{n}\right)  \subset\mathcal{G}_{n,1}\left(  \tilde{\Theta}_{n}\right)
$, i.e.
\[
\mathcal{G}_{n,1}\left(  \Theta_{2}^{\prime}\left(  d,M\right)  \cap
S_{n}\left(  \theta_{0},r\right)  \right)  \subset\mathcal{G}_{n,1}\left(
\Theta_{1,c}\left(  1,M^{\prime}\right)  \cap B\left(  a_{\theta_{0}}%
,\gamma_{n}\right)  \right)  \text{, for all }\theta_{0}\in\Theta_{2}^{\prime
}\left(  d,M\right)
\]
follows from Lemma \ref{lem-inclusions-parspace} for a certain $M^{\prime}>0$
and a sequence $\gamma_{n}=O\left(  n^{-1/2}\right)  $. Analogously we obtain,
for the same $M^{\prime}$ and $\gamma_{n}$
\[
\mathcal{G}_{n,2}\left(  \theta_{0},\Theta_{2}^{\prime}\left(  d,M\right)
\cap S_{n}\left(  \theta_{0},r\right)  \right)  \subset\mathcal{G}%
_{n,2}\left(  a_{\theta_{0}},\Theta_{1,c}\left(  1,M^{\prime}\right)  \cap
B\left(  a_{\theta_{0}},\gamma_{n}\right)  \right)  \text{, for all }%
\theta_{0}\in\Theta_{2}^{\prime}\left(  d,M\right)
\]
where the experiment on the r.h.s. is defined in
(\ref{pre-variance-stable-SDE-experi}). Since $\gamma_{n}$ fulfills the
condition $\gamma_{n}=o\left(  \left(  n/\log n\right)  ^{-\alpha/(2\alpha
+1)}\right)  $ for $\alpha=1$, Lemma \ref{lem-local-white-noise} indeed
implies (\ref{equiv-local-parametric-1}). Now Lemma \ref{lem-exact-equiv-GWN}
implies
\begin{equation}
\sup_{\theta_{0}\in\Theta_{2}^{\prime}}\Delta\left(  \mathcal{G}_{n,2}\left(
\theta_{0},\Theta_{n}\right)  ,\mathcal{G}_{n,3}\left(  \theta_{0},\Theta
_{n}\right)  \right)  =0. \label{equiv-local-parametric-2}%
\end{equation}
We now claim
\begin{equation}
\sup_{\theta_{0}\in\Theta_{2}^{\prime}}\Delta\left(  \mathcal{G}_{n,3}\left(
\theta_{0},\Theta_{n}\right)  ,\mathcal{G}_{n,4}\left(  \Theta_{n}\right)
\right)  \rightarrow0. \label{equiv-local-parametric-3}%
\end{equation}
For that consider the total variation distance, for $\theta,\theta_{0}%
\in\Theta_{2}^{\prime}\left(  d,M\right)  $
\begin{align*}
&  \left\Vert N_{2d+1}\left(  \theta,n^{-1}\Phi_{\theta}^{-1}\right)
-N_{2d+1}\left(  \theta,n^{-1}\Phi_{\theta_{0}}^{-1}\right)  \right\Vert
_{TV}\\
&  =\left\Vert N_{2d+1}\left(  0,\Phi_{\theta}^{-1}\right)  -N_{2d+1}\left(
0,\Phi_{\theta_{0}}^{-1}\right)  \right\Vert _{TV}%
\end{align*}
where the equality is obtained by applying the one-to-one map $x\rightarrow
n^{1/2}\left(  x-\theta\right)  $. By (\ref{Lecam-inequ}) the above is
upperbounded by the Hellinger distance
\[
H\left(  N_{2d+1}\left(  0,\Phi_{\theta}^{-1}\right)  ,N_{2d+1}\left(
0,\Phi_{\theta_{0}}^{-1}\right)  \right)  .
\]
Now by Lemma \ref{lem-covmatrices-cite-GNZpaper} and Lemma
\ref{Lem-covmatrix-eigenvalues} the above is upperbounded by
\[
C\;\left\Vert \Phi_{\theta}^{-1}-\Phi_{\theta_{0}}^{-1}\right\Vert _{2}%
\]
where $\left\Vert \cdot\right\Vert _{2}$ denotes Hilbert-Schmidt norm for
matrices and $C$ only depends on $M$ and $d$. By Lemma \ref{lem-compact-2},
the mapping $\theta\longrightarrow\Phi_{\theta}^{-1}$ is continuous in
Hilbert-Schmidt norm on the compact set $\Theta_{2}^{\prime}\left(
d,M\right)  \in\mathbb{R}^{2d+1}$, and thus uniformly continuous
(\cite{MR120319}, 3.16.5). Hence
\begin{align*}
&  \sup_{\theta_{0}\in\Theta_{2}^{\prime}}\Delta\left(  \mathcal{G}%
_{n,3}\left(  \theta_{0},\Theta_{n}\right)  ,\mathcal{G}_{n,4}\left(
\Theta_{n}\right)  \right) \\
&  \leq\sup_{\theta_{0}\in\Theta_{2}^{\prime}}\sup_{\theta\in\Theta_{n}\left(
\theta_{0}\right)  }\left\Vert N_{2d+1}\left(  \theta,n^{-1}\Phi_{\theta}%
^{-1}\right)  -N_{2d+1}\left(  \theta,n^{-1}\Phi_{\theta_{0}}^{-1}\right)
\right\Vert _{TV}\\
&  \leq\sup_{\theta,\theta_{0}\in\Theta_{2}^{\prime},\left\Vert \theta
-\theta_{0}\right\Vert \leq n^{-1/2}r}C\;\left\Vert \Phi_{\theta}^{-1}%
-\Phi_{\theta_{0}}^{-1}\right\Vert _{2}\rightarrow0
\end{align*}
confirming (\ref{equiv-local-parametric-3}). Now relations
(\ref{equiv-local-parametric-1}) -(\ref{equiv-local-parametric-3}) establish
the claim.
\end{proof}

\bigskip

\begin{lemma}
\label{lem-semiorder-gwn-hetero}We have
\[
\delta\left(  \mathcal{G}_{n,4}\left(  \Theta_{2}^{\prime}\right)
,\mathcal{G}_{n,1}\left(  \Theta_{2}^{\prime}\right)  \right)  \rightarrow0.
\]

\end{lemma}

\begin{proof}
Apply Lemma \ref{lem-local-to-global} with $\mathcal{P}_{n}=\mathcal{G}%
_{n,4}\left(  \Theta_{2}^{\prime}\right)  $, $\mathcal{Q}_{n}=\mathcal{G}%
_{n,1}\left(  \Theta_{2}^{\prime}\right)  $ and
\begin{align*}
\mathcal{P}_{n}\left(  \theta_{0},r\right)   &  =\mathcal{G}_{n,4}\left(
\Theta_{2}^{\prime}\left(  d,M\right)  \cap S_{n}\left(  \theta_{0},r\right)
\right)  ,\\
\mathcal{Q}_{n}\left(  \theta_{0},r\right)   &  =\mathcal{G}_{n,1}\left(
\Theta_{2}^{\prime}\left(  d,M\right)  \cap S_{n}\left(  \theta_{0},r\right)
\right)  .
\end{align*}
Then condition (\ref{local-experi-delta}) is guaranteed by Lemma
\ref{lem-local-equiv-gwn-global-to-hetero}, while condition
(\ref{cond-lambdamax-bounded}) is guaranteed by Lemma
\ref{Lem-covmatrix-eigenvalues}.
\end{proof}

\bigskip

\begin{proof}
[Proof of Theorem \ref{theor-main-2}]Identify the experiment $\mathcal{E}%
_{n}\left(  \Theta_{2}\left(  d,M\right)  \right)  $ of
(\ref{basic-quantum-experi-def}) with a set of states indexed by $\theta
\in\Theta_{2}^{\prime}\left(  d,M\right)  $, i.e. with
\[
\mathcal{E}_{n,1}\left(  \Theta_{2}^{\prime}\right)  :=\left(  \mathfrak{N}%
_{n}\left(  0,A_{n}\left(  a_{\theta}\right)  \right)  ,\theta\in\Theta
_{2}^{\prime}\right)  .
\]
In the same way, we can identify $\mathcal{G}_{n}\left(  \Theta_{2}\left(
d,M\right)  \right)  $ in the Theorem with $\mathcal{G}_{n,1}\left(
\Theta_{2}^{\prime}\right)  $ defined in (\ref{gwn-global-experi-def}). Then
the claim is
\begin{equation}
\mathcal{G}_{n,1}\left(  \Theta_{2}^{\prime}\right)  \precsim\mathcal{E}%
_{n,1}\left(  \Theta_{2}^{\prime}\right)  . \label{claim-semiorder-lower-1}%
\end{equation}
Consider the observable $\mathbf{\bar{\Pi}}_{n}$ defined in
(\ref{Pi-k-quer-def}) and the experiment formed by its distributions under the
state $\mathfrak{N}_{n}\left(  0,A_{n}\left(  a_{\theta}\right)  \right)  $,
i.e.%
\[
\mathcal{E}_{n,2}\left(  \Theta_{2}^{\prime}\right)  :=\left(  \mathcal{L}%
\left(  \mathbf{\bar{\Pi}}_{n}|\theta\right)  ,\;\theta\in\Theta_{2}^{\prime
}\right)  .
\]
Since $\mathbf{\bar{\Pi}}_{n}$ is based on a measurement of the state, the map
from $\mathfrak{N}_{n}\left(  0,A_{n}\left(  a_{\theta}\right)  \right)  $ to
$\mathcal{L}\left(  \mathbf{\bar{\Pi}}_{n}|\theta\right)  $ is given by a dual
observation channel (i.e. a state transition, cf. Section
\ref{subsubsec:measmt-obs-channel}), hence%
\begin{equation}
\mathcal{E}_{n,2}\left(  \Theta_{2}^{\prime}\right)  \preceq\mathcal{E}%
_{n,1}\left(  \Theta_{2}^{\prime}\right)  . \label{semiorder-lower-1}%
\end{equation}
Consider now the estimator $\tilde{\theta}_{n}$ according to Definition
\ref{Def-estimator-optimal-param}, which is a function of $\mathbf{\bar{\Pi}%
}_{n}$. According to Theorem \ref{Theor-asy-normal-improved-est},
$\tilde{\theta}_{n}$ is asymptotically normal
\[
n^{1/2}\left(  \tilde{\theta}_{n}-\theta\right)  \Longrightarrow_{d}%
N_{2d+1}\left(  0,\Phi_{\theta}^{-1}\right)
\]
uniformly in $\theta\in\Theta_{2}^{\prime}$, so condition (\ref{assump-1}) of
Theorem \ref{Theor-deficiency-from-law-converg} is fulfilled. Furthermore
$\Theta_{2}^{\prime}$ is compact according to Lemma \ref{lem-compact-convex},
the map $\theta\rightarrow\Phi_{\theta}^{-1}$ is continuous in norm
$\left\Vert \cdot\right\Vert _{2}$ according to Lemma \ref{lem-compact-2}, and
$\Phi_{\theta}^{-1}>0,\;\theta\in\Theta_{2}^{\prime}$ holds according to Lemm
\ref{Lem-covmatrix-eigenvalues}. Then, with $\mathcal{G}_{n,4}\left(
S\right)  $ defined by (\ref{heteroskedastic-normal-def-a}), Theorem
\ref{Theor-deficiency-from-law-converg} gives
\[
\delta\left(  \mathcal{E}_{n,2}\left(  \Theta_{2}^{\prime}\right)
,\mathcal{G}_{n,4}\left(  \Theta_{2}^{\prime}\right)  \right)  \rightarrow0,
\]
or in semiordering notation
\begin{equation}
\mathcal{G}_{n,4}\left(  \Theta_{2}^{\prime}\right)  \precsim\mathcal{E}%
_{n,2}\left(  \Theta_{2}^{\prime}\right)  . \label{semiorder-lower-2}%
\end{equation}
Now Lemma \ref{lem-semiorder-gwn-hetero} states%
\begin{equation}
\mathcal{G}_{n,1}\left(  \Theta_{2}^{\prime}\right)  \precsim\mathcal{G}%
_{n,4}\left(  \Theta_{2}^{\prime}\right)  . \label{semiorder-lower-3}%
\end{equation}
Relations (\ref{semiorder-lower-1}), (\ref{semiorder-lower-2}) and
(\ref{semiorder-lower-3}) establish the claim (\ref{claim-semiorder-lower-1}).
\end{proof}

\bigskip%

\begin{privatenotes}
\begin{boxedminipage}{\textwidth}%

\begin{sfblock}
\textbf{ATTENTION:}\ Notation with (\ref{semiorder-lower-1}): here the index
$k$ in $\mathcal{E}_{n,k}\left(  \Theta_{2}^{\prime}\right)  $ has to be
chosen later, depending on how many indices $j$ in $\mathcal{E}_{n,j}\left(
\Theta_{2}\right)  $ we will need for the notations in the Section on upper
info. bound\texttt{.}
\end{sfblock}

%

\end{boxedminipage}
\end{privatenotes}%
.%

\begin{privatenotes}%

\section{Optimal estimation}

\subsection{Parametric estimation}

Under the assumption of $d$-dependence, i.e. that the spectral density $a$
varies in a set $\Theta_{2}\left(  d,M\right)  $ given by
(\ref{d-dependence-param-set}), consider estimation of the real parameter
vector $\theta\in\mathbb{R}^{2d+1}$ which then parametrizes $a=a_{\theta}$
according to (\ref{series-repre-real-spec_density}). Then $\theta$ is a
function of the complex autocovariances $a_{j}$, $\left\vert j\right\vert \leq
d$ and varies in the set $\Theta_{2}^{\prime}\left(  d,M\right)  $ given by
(\ref{parametric-model}). In Theorem \ref{Theor-asy-normal-improved-est} it
has been shown that estimator $\tilde{\theta}_{n}$ of $\theta$ is
asymptotically normal with limiting covariance matrix $\Phi_{\theta}^{-1}$,
\[
n^{1/2}\left(  \tilde{\theta}_{n}-\theta\right)  \Longrightarrow_{d}%
N_{2d+1}\left(  0,\Phi_{\theta}^{-1}\right)
\]
uniformly in $\theta\in\Theta_{2}^{\prime}$.

\subsection{Nonparametric estimation}

\bigskip We assume the conditions of Theorem \ref{theor-main-1}, i.e. the
spectral density $a$ is in a set $\Theta_{1}\left(  \alpha,M\right)  $ given
by (\ref{Theta-1-functionset-def}) for some $\alpha>1/2,$ $M>1$. Then $a$ has
representation (\ref{symbol-generate-2}) in terms of the coefficients $a_{j}$,
the analogs of the autocovariances. For purposes of estimation, we will revert
to the real coefficients $\theta_{j}$ given by (\ref{real-parameter-def})
without the restriction $\left\vert j\right\vert \leq d$, i.e. we will set
\[
\theta_{0}=a_{0},\theta_{j}=\sqrt{2}\operatorname{Re}a_{j},\theta_{-j}%
=-\sqrt{2}\operatorname{Im}a_{j},\;j\in\mathbb{N}%
\]
recalling $\bar{a}_{j}=a_{-j}$. Then analogously to
(\ref{series-repre-real-spec_density}), writing $\theta=\left(  \theta
_{j}\right)  _{j\in\mathbb{Z}}\in\ell_{2}$ and $a=a_{\theta}$, we have
\begin{equation}
a_{\theta}\left(  \omega\right)  =\sum_{j=-\infty}^{\infty}\psi_{j}\left(
\omega\right)  \theta_{j}\text{, } \label{Fourier-repre-real-2}%
\end{equation}
where the functions $\psi_{j}$ on $\left[  -\pi,\pi\right]  $ are defined in
(\ref{psi-def}). The assumption $a_{\theta}\in\Theta_{1}\left(  \alpha
,M\right)  $ is then equivalent to $\theta\in\Theta_{1}^{\prime}\left(
\alpha,M\right)  $ where%
\begin{equation}
\Theta_{1}^{\prime}\left(  \alpha,M\right)  :=\left\{  \theta:\text{
}\left\vert \theta_{0}\right\vert ^{2}+\sum_{j=-\infty}^{\infty}j^{2\alpha
}\left\vert \theta_{j}\right\vert ^{2}\leq M\right\}  \cap\mathcal{L}%
_{M}^{\prime} \label{Theta-1-prime-first-def}%
\end{equation}
where $\mathcal{L}_{M}^{\prime}$ is defined in \textbf{(}%
\ref{lowerbound-set-parametric-def}\textbf{), }assuming $\theta\in\ell_{2}$.
Unbiased estimators of the coefficients $\theta_{j}$ have already been found
in (\ref{prelim-estim-def-2}); replacing the fixed dimension $d$ there by a
bandwidth parameter $d_{n}\rightarrow\infty$, $d_{n}=o\left(  n\right)  $ and
writing the estimators as $\hat{\theta}_{n,j}$ now, we obtain
\[
\hat{\theta}_{j,n}=\frac{n^{1/2}}{n-\left\vert j\right\vert }\mathbf{w}%
_{j,n}^{\prime}\mathbf{\Pi}_{n}\mathbf{,\;}\left\vert j\right\vert \leq d_{n}%
\]
where the $\mathbb{R}^{n}$-valued observable $\mathbf{\Pi}_{n}$ is given by
(\ref{vec-of-observables}), (\ref{Pi-observble-component-def}).

\textbf{ATTENTION:\ the next statement is false; we need assumption }%
$\theta_{j}=0$\textbf{, }$\left\vert j\right\vert >d_{n}.$\textbf{ Modify to
"asy. unbiased" }

Unbiasedness then follows as in (\ref{unbiased-2}): for $\rho=\mathfrak{N}%
_{n}\left(  0,A_{n}\left(  a_{\theta}\right)  \right)  $
\begin{equation}
E_{\rho}\hat{\theta}_{j,n}=\theta_{j}\text{, }\left\vert j\right\vert \leq
d_{n}\text{.} \label{unbiased-3}%
\end{equation}
Following the standard approach of truncated orthogonal series estimation, we
define the nonparametric estimator of the spectral density as
\[
\hat{a}_{n}\left(  \omega\right)  =\sum_{\left\vert j\right\vert \leq d_{n}%
}\hat{\theta}_{j,n}\text{. }%
\]

\begin{lemma}
Assume the bandwidth parameter $d_{n}$ is chosen such that $d_{n}\leq
n^{1/2}/2$. Then for $\rho=\mathfrak{N}_{n}\left(  0,A_{n}\left(  a_{\theta
}\right)  \right)  $, $\theta\in\Theta_{1}^{\prime}\left(  \alpha,M\right)  $
and $\left\vert j\right\vert \leq d_{n}$ we have
\[
\mathrm{Var}_{\rho}\left(  \hat{\theta}_{j,n}\right)  \leq C_{\alpha
,M}\;n^{-1}%
\]
where $C_{\alpha,M}$ depends on $\alpha$ and $M$ but not on $j$ and $n$.
\end{lemma}

\begin{proof}
We have
\[
\mathrm{Var}_{\rho}\left(  \hat{\theta}_{j,n}\right)  =\frac{n}{\left(
n-\left\vert j\right\vert \right)  ^{2}}\mathbf{w}_{j,n}^{\prime}%
\mathrm{Cov}_{\rho}\left(  \mathbf{\Pi}_{n}\right)  \mathbf{w}_{j,n}%
\]
where the covariance matrix $\mathrm{Cov}_{\rho}\left(  \mathbf{\Pi}%
_{n}\right)  $ of $\mathbf{\Pi}_{n}$ is described in Lemma
\ref{lem-covmatrix-analog-B} (ii). We can write $\mathrm{Cov}_{\rho}\left(
\mathbf{\Pi}_{n}\right)  $ in closed form, analogously to the reasoning in the
proof of Lemma \ref{lem-covmatrix-prelim-est}. Write the symbol matrix $A_{n}$
as a function of $a_{\theta}$ as in (\ref{quantum-asy-setup-spec-density}),
and for a $n\times n$ matrix $M=\left(  M_{jl}\right)  _{j,l=1}^{n}$, define
the real matrix
\[
M^{\left[  2\right]  }=\left(  \left\vert M_{jl}\right\vert ^{2}\right)
_{j,l=1}^{n}.
\]
Then the result of Lemma \ref{lem-covmatrix-analog-B} (ii) can be written
\[
\mathrm{Cov}_{\rho}\left(  \hat{B}_{-(n-1)/2}^{\ast}\hat{B}_{-(n-1)/2}%
,\ldots,\hat{B}_{(n-1)/2}^{\ast}\hat{B}_{(n-1)/2}\right)  =\frac{1}{4}\left(
\left(  U_{n}^{\ast}A_{n}\left(  a_{\theta}\right)  U_{n}\right)  ^{\left[
2\right]  }-I_{n}\right)
\]
and hence for $\mathbf{\Pi}_{n}$ in (\ref{Pi-observble-component-def}),
(\ref{vec-of-observables}) we obtain
\[
\mathrm{Cov}_{n,\theta}\left(  \mathbf{\Pi}_{n}\right)  =\left(  U_{n}^{\ast
}A_{n}\left(  a_{\theta}\right)  U_{n}\right)  ^{\left[  2\right]  }-I_{n},
\]
with $U_{n}$ from (\ref{special-unitary-DFT-def}) for $m=n$. Recall the
definition of the real vectors $\mathbf{w}_{j,n}$ in (\ref{vectors-w--def}):
\[
\mathbf{w}_{j,n}=\left(  w_{j,k,n}\right)  _{\left\vert k\right\vert
\leq\left(  n-1\right)  /2}\text{ for }w_{j,k,n}=n^{-1/2}\psi_{j}\left(
\omega_{k,n}\right)
\]
where by (\ref{vectors-w-orthonorm-2}) we have $\left\Vert \mathbf{w}%
_{j,n}\right\Vert ^{2}=1$. Now
\begin{align*}
\mathbf{w}_{j,n}^{\prime}\mathrm{Cov}_{\rho}\left(  \mathbf{\Pi}_{n}\right)
\mathbf{w}_{j,n}  &  \leq\mathbf{w}_{j,n}^{\prime}\left(  U_{n}^{\ast}%
A_{n}\left(  a_{\theta}\right)  U_{n}\right)  ^{\left[  2\right]  }%
\mathbf{w}_{j,n}\\
&  =\mathbf{w}_{j,n}^{\prime}\left(  \left(  \left\vert \mathbf{u}_{k}^{\ast
}A\left(  a_{\theta}\right)  \mathbf{u}_{l}\right\vert ^{2}\right)
_{\left\vert k\right\vert ,\left\vert l\right\vert \leq\left(  n-1\right)
/2}\right)  \mathbf{w}_{j,n}\\
&  =\mathbf{w}_{j,n}^{\prime}\left(  \left(  \mathbf{u}_{k}^{\ast}A\left(
a_{\theta}\right)  \mathbf{u}_{l}\mathbf{u}_{l}^{\ast}A\left(  a_{\theta
}\right)  \mathbf{u}_{k}\right)  _{\left\vert k\right\vert ,\left\vert
l\right\vert \leq\left(  n-1\right)  /2}\right)  \mathbf{w}_{j,n}\\
&  =\sum_{\left\vert k\right\vert ,\left\vert l\right\vert \leq\left(
n-1\right)  /2}w_{j,k,n}w_{j,l,n}\mathbf{u}_{k}^{\ast}A\left(  a_{\theta
}\right)  \mathbf{u}_{l}\mathbf{u}_{l}^{\ast}A\left(  a_{\theta}\right)
\mathbf{u}_{k}%
\end{align*}%
\begin{align}
&  =\sum_{\left\vert k\right\vert ,\left\vert l\right\vert \leq\left(
n-1\right)  /2}w_{j,k,n}w_{j,l,n}\mathrm{Tr}\;\mathbf{u}_{k}\mathbf{u}%
_{k}^{\ast}A\left(  a_{\theta}\right)  \mathbf{u}_{l}\mathbf{u}_{l}^{\ast
}A\left(  a_{\theta}\right) \nonumber\\
&  =\mathrm{Tr}\;\left(  \sum_{\left\vert k\right\vert \leq\left(  n-1\right)
/2}w_{j,k,n}\mathbf{u}_{k}\mathbf{u}_{k}^{\ast}\right)  A\left(  a_{\theta
}\right)  \left(  \sum_{\left\vert l\right\vert \leq\left(  n-1\right)
/2}w_{j,l,n}\mathbf{u}_{l}\mathbf{u}_{l}^{\ast}\right)  A\left(  a_{\theta
}\right)  . \label{Bigtrace-1}%
\end{align}
Define a Hermitian matrix
\[
V_{j,n}:=\sum_{\left\vert k\right\vert \leq\left(  n-1\right)  /2}%
w_{j,k,n}\mathbf{u}_{k}\mathbf{u}_{k}^{\ast};
\]
then (\ref{Bigtrace-1}) equals
\begin{align*}
\mathrm{Tr}\;V_{j,n}A\left(  a_{\theta}\right)  V_{j,n}A\left(  a_{\theta
}\right)   &  =\mathrm{Tr}\;A^{1/2}\left(  a_{\theta}\right)  V_{j,n}A\left(
a_{\theta}\right)  V_{j,n}A^{1/2}\left(  a_{\theta}\right) \\
&  \leq\lambda_{\max}\left(  A\left(  a_{\theta}\right)  \right)
\;\mathrm{Tr}\;A^{1/2}\left(  a_{\theta}\right)  V_{j,n}^{2}A^{1/2}\left(
a_{\theta}\right) \\
&  =\lambda_{\max}\left(  A\left(  a_{\theta}\right)  \right)  \;\mathrm{Tr}%
\;V_{j,n}A\left(  a_{\theta}\right)  V_{j,n}\\
&  \leq\lambda_{\max}^{2}\left(  A\left(  a_{\theta}\right)  \right)
\;\mathrm{Tr}\;V_{j,n}^{2}.
\end{align*}
Here
\[
\mathrm{Tr}\;V_{j,n}^{2}=\mathrm{Tr}\;\sum_{\left\vert k\right\vert
\leq\left(  n-1\right)  /2}w_{j,k,n}^{2}\mathbf{u}_{k}\mathbf{u}_{k}^{\ast
}=\sum_{\left\vert k\right\vert \leq\left(  n-1\right)  /2}w_{j,k,n}^{2}=1
\]
since the $\mathbf{u}_{k}$, $\left\vert k\right\vert \leq\left(  n-1\right)
/2$ are an orthonormal system and $\left\Vert \mathbf{w}_{j,n}\right\Vert
^{2}=1.$ Furthermore by Lemma \ref{lem-toeplitz-EV} we have $\lambda_{\max
}\left(  A\left(  a_{\theta}\right)  \right)  \leq C_{1}$ for a $C_{1}$
depending on $M$ and $\alpha$; summarizing the reasoning so far, we have
\begin{align*}
\mathrm{Var}_{\rho}\left(  \hat{\theta}_{j,n}\right)   &  \leq%
\acute{}%
\frac{n}{\left(  n-\left\vert j\right\vert \right)  ^{2}}C_{1}^{2}=\frac{1}%
{n}\frac{1}{\left(  1-\left\vert j\right\vert n^{-1}\right)  }C_{1}^{2}\\
&  \leq\frac{1}{n}\frac{1}{\left(  1-\frac{1}{2}n^{-1/2}\right)  }C_{1}%
^{2}\leq\frac{1}{n}2C_{1}^{2}.
\end{align*}
Setting $C_{\alpha,M}=2C_{1}^{2}$ we obtain the claim.
\end{proof}

\bigskip

\textbf{(Attainment of optimal rate here)}

\bigskip

For the lower asymptotic risk bound, we assume that $\alpha\geq1$ is integer.
We will also modify the parameter sets $\Theta_{1}\left(  \alpha,M\right)  $
in the following way. Consider the seminorm $\left\vert \cdot\right\vert
_{2,\alpha}^{2}$and norm $\left\Vert \cdot\right\Vert _{2,\alpha}^{2}$ on
functions $a$ defined in (\ref{normdef})
\[
\left\vert a\right\vert _{2,\alpha}^{2}:=\sum_{k=-\infty}^{\infty}\left\vert
k\right\vert ^{2\alpha}\left\vert a_{k}\right\vert ^{2},\;\left\Vert
a\right\Vert _{2,\alpha}^{2}:=a_{0}^{2}+\left\vert a\right\vert _{2,\alpha
}^{2}%
\]
where $a_{k}$ are the Fourier coefficients (\ref{symbol-generate}). Define for
$\varepsilon,K>0$
\begin{equation}
\Theta_{1}\left(  \alpha,\varepsilon,K\right)  :=\left\{  a:\left\Vert
a\right\Vert _{2,\alpha}^{2}\leq K^{2},\inf_{\omega\in\left[  -\pi,\pi\right]
}a\left(  \omega\right)  \geq1+\varepsilon\right\}  .
\label{Theta-1-second-def}%
\end{equation}
This set is nonempty if $1+\varepsilon\leq K$: any constant function $a\left(
\omega\right)  =c\in\left[  1+\varepsilon,K\right]  $ fulfills
\begin{align*}
\left\Vert a\right\Vert _{2,\alpha}^{2}  &  =a_{0}^{2}=\left(  \frac{1}{2\pi
}\int_{\left[  -\pi,\pi\right]  }a\left(  \omega\right)  d\omega\right)
^{2}=c^{2}\leq K^{2},\\
\inf_{\omega\in\left[  -\pi,\pi\right]  }a\left(  \omega\right)   &
=c\geq1+\varepsilon.
\end{align*}
Henceforth we will assume $a\in\Theta_{1}\left(  \alpha,\varepsilon,K\right)
$ where $1+\varepsilon<K$.

Note that given $\varepsilon$ and $K$, there always exists $M$ such that
$\Theta_{1}\left(  \alpha,\varepsilon,K\right)  \subset\Theta_{1}\left(
\alpha,M\right)  $. Therefore the previous results on equivalence and
informativity bounds can be applied to experiments where $\Theta_{1}\left(
\alpha,\varepsilon,K\right)  $ is the parameter set, upon considering it a
subset of some $\Theta_{1}\left(  \alpha,M\right)  $.

\bigskip

\textbf{Quantum nonparametric estimators. }We identify a square integrable
spectral density $a$ with its set of real Fourier coefficients $\theta=\left(
\theta_{j}\right)  _{j\in\mathbb{Z}}\in\ell^{2}$ via
(\ref{Fourier-repre-real-2}). Define in analogy to
(\ref{Theta-1-prime-first-def}) and (\ref{Theta-1-second-def})%
\[
\Theta_{1}^{\prime}\left(  \alpha,\varepsilon,K\right)  :=\left\{
\theta:\text{ }\left\vert \theta_{0}\right\vert ^{2}+\sum_{j=-\infty}^{\infty
}j^{2\alpha}\left\vert \theta_{j}\right\vert ^{2}\leq K^{2},\inf_{\omega
\in\left[  -\pi,\pi\right]  }a_{\theta}\left(  \omega\right)  \geq
1+\varepsilon\right\}  .
\]
We consider estimators of $\theta$ in the quantum statistical experiment
\[
\mathcal{E}_{n,1}\left(  \Theta_{1}\left(  \alpha,\varepsilon,K\right)
\right)  =\left(  \mathcal{L}\left(  \mathfrak{F}\left(  \mathbb{C}%
^{n}\right)  \right)  ,\mathfrak{N}_{n}\left(  0,A_{n}\left(  a_{\theta
}\right)  \right)  ,\theta\in\Theta_{1}\left(  \alpha,\varepsilon,K\right)
\right)  .
\]
Let $\mathcal{B}_{\ell^{2}}$ be the Borel sigma-algebra of $\ell^{2}$ and let
$\mu$ be a sigma-finite measure on $\mathcal{B}_{\ell_{2}}$; an estimator of
$\theta$ based on a quantum measurement of a state $\rho=\mathfrak{N}%
_{n}\left(  0,A_{n}\left(  a_{\theta}\right)  \right)  $ will be given by a
POVM\ $\tau$ on $\left(  \ell^{2},\mathcal{B}_{\ell^{2}}\right)  $, i.e. a
mapping
\begin{equation}
\tau:\mathcal{B}_{\ell^{2}}\rightarrow\mathcal{L}\left(  \mathfrak{F}\left(
\mathbb{C}^{n}\right)  \right)  \label{POVM-basic}%
\end{equation}
with properties given in Appendix \ref{subsubsec:measmt-obs-channel}. If
$\rho$ is a state on the Fock space $\mathfrak{F}\left(  \mathbb{C}%
^{n}\right)  $, applying the POVM $\tau$ results in a probability measure on
$\left(  \ell^{2},\mathcal{B}_{\ell^{2}}\right)  $ given by%
\begin{equation}
Q_{\rho,\tau}\left(  B\right)  =\mathrm{Tr}\;\rho\tau\left(  B\right)  \text{,
}B\in\mathcal{B}_{\ell^{2}}. \label{distr-observable-1}%
\end{equation}
The POVM\ $\tau$ gives a state transition
\begin{equation}
T_{\tau}:\mathcal{L}^{1}\left(  \mathfrak{F}\left(  \mathbb{C}^{n}\right)
\right)  \rightarrow L^{1}\left(  \mu\right)
\label{state-transition-lower-bound-1}%
\end{equation}
(or dual channel, or TP-CP map, cf. Appendix
\ref{subsub-Append-state-transitions}), where $\mu$ can be chosen as a
probability measure on $\left(  \ell_{2},\mathcal{B}_{\ell_{2}}\right)
$\footnote{\texttt{This statement is in [BGN18], Lemma 2.1, without proof, see
also the related statement in Kahn, Guta [KG09], p. 605 top. We will have to
incorporate it in the Appendix statement, where POVM are treated (at present
there are only PVMs mentioned). 'For now, we will write a sketch of one part
of the proof in the "Discussion of technical details". }}. Thus $\tau$ gives
an $\ell^{2}$-valued observable (cp. Appendix
\ref{subsubsec:measmt-obs-channel}), i.e. an $\ell^{2}$-valued-random variable
having distribution $Q_{\rho,\tau}$; the random variable will be written
$\hat{\theta}_{n,\tau}$. From (\ref{distr-observable-1}) we obtain
\begin{equation}
Q_{\rho,\tau}\left(  B\right)  =P\left(  \hat{\theta}_{\tau}\in B|\rho\right)
=\int_{B}T_{\tau}\left(  \rho\right)  d\mu. \label{distr-observable-2}%
\end{equation}
Let $\left(  Y,\Omega_{Y}\right)  $ be a measurable space and $H$ be a Markov
kernel $H:\Omega_{Y}\times\ell^{2}\rightarrow\left[  0,1\right]  $. Then $H$
gives a state transition $L^{1}\left(  \mu\right)  \rightarrow L^{1}\left(
\nu\right)  $ where $\nu$ is a probability measure on $\Omega_{Y}$. Let
$HT_{\tau}$ be the composed state transition; in particular, if $Y$ is a
subset of $\ell^{2}$ and $H$ is given by a deterministic mapping $g:\ell
^{2}\rightarrow Y$ such that $H\left(  B,x\right)  =\mathbf{1}_{B}\left(
g\left(  x\right)  \right)  $, then $HT_{\tau}$ corresponds to the observable
$g\left(  \hat{\theta}_{n,\tau}\right)  $ where $\hat{\theta}_{n,\tau}$ has
distribution $Q_{\rho,\tau}$.

Let $\left\Vert \cdot\right\Vert _{2}^{2}$ denote squared norm in
$L^{2}\left(  -\pi,\pi\right)  $. The following lower asymptotic risk bound
for estimation of the spectral density $a_{\theta}$ can be established.

\begin{theorem}
Assume $\alpha\geq1$ integer and $1+\varepsilon<K$. For $\theta\in\Theta
_{1}^{\prime}\left(  \alpha,\varepsilon,K\right)  $, introduce shorthand
notation $\rho_{n}\left(  \theta\right)  =\mathfrak{N}_{n}\left(
0,A_{n}\left(  a_{\theta}\right)  \right)  $. Then%
\[
\liminf_{n}\inf_{\tau}\sup_{\theta\in\Theta_{1}^{\prime}\left(  \alpha
,\varepsilon,K\right)  }n^{2\alpha/\left(  2\alpha+1\right)  }\int\left\Vert
a_{u}-a_{\theta}\right\Vert _{2}^{2}\;\mathrm{Tr}\rho_{n}\left(
\theta\right)  \tau\left(  du\right)  >0
\]
where $a_{\theta},$ $\theta\in\ell^{2}$ is given by
(\ref{Fourier-repre-real-2}) and for given $n$, the infimum extends over all
POVM $\tau$ in (\ref{POVM-basic}).
\end{theorem}

\begin{proof}
Let $\hat{\theta}_{n}$ be the random variable having distribution $Q_{\rho
_{n}\left(  \theta\right)  ,\tau}$ given by (\ref{distr-observable-1}). By
Markov's inequality we have for any $C>0$
\begin{equation}
n^{2\alpha/\left(  2\alpha+1\right)  }\int\left\Vert a_{u}-a_{\theta
}\right\Vert _{2}^{2}Q_{\rho_{n}\left(  \theta\right)  ,\tau}\left(
du\right)  \geq C\;Q_{\rho_{n}\left(  \theta\right)  ,\tau}\left(
n^{2\alpha/\left(  2\alpha+1\right)  }\left\Vert a_{\hat{\theta}_{n}%
}-a_{\theta}\right\Vert _{2}^{2}\geq C\right)  . \label{Markov-lower-bound-0}%
\end{equation}
We will use the method of "multiple hypotheses" for establishing the lower
bound, inspired by Sec. 2.6 in \cite{MR2724359}. Let $c\in\left(
1+\varepsilon,K\right)  $ and consider the constant function $a_{(0)}\left(
\omega\right)  =c,$ $\omega\in\left[  -\pi,\pi\right]  $. To define a set of
perturbations of $a_{(0)}$, let $\varphi\left(  x\right)  $ be a $\alpha
$-times differentiable function on $x\in\mathbb{R}$ with support contained in
$\left(  -1,1\right)  ,$ satisfying
\begin{equation}
\int_{\left[  -1,1\right]  }\varphi\left(  x\right)  dx=0.
\label{integral-zero}%
\end{equation}
Let $m$ be integer and let $\Delta_{1},\ldots,\Delta_{m}$ be a partition of
$\left[  -\pi,\pi\right]  $ into consecutive nonintersecting intervals of
length $2h$ where $h=\pi/m$. Let $x_{j}$ be the midpoint of the interval
$\Delta_{j}$, $j=1,\ldots,m$. Denote by $\Lambda_{m}$ a set of $N$-dimensional
binary vectors
\[
\Lambda_{m}=\left\{  \lambda:\lambda=\left(  \lambda_{1},\ldots,\lambda
_{m}\right)  ,\lambda_{j}\in\left\{  0,1\right\}  ,j=1,\ldots,m\right\}  .
\]
Define a set of functions, indexed by $\lambda\mathbf{\in}\Lambda_{m}$, by
\begin{equation}
a\left(  \omega,\lambda\right)  =c+\sum_{j=1}^{m}L\lambda_{j}h^{\alpha}%
\varphi\left(  \frac{\omega-x_{j}}{h}\right)  .
\label{special-functions-a-lambda-def}%
\end{equation}
for some $L>0$. We note that
\begin{equation}
\sup_{\omega\in\left[  -\pi,\pi\right]  }\left\vert a\left(  \omega
,\lambda\right)  -a_{(0)}\left(  \omega\right)  \right\vert \leq Lh^{\alpha
}\left\Vert \varphi\right\Vert _{\infty} \label{shrinking-neighb-unif-norm}%
\end{equation}
so that for sufficiently large $m$, we have
\begin{equation}
\inf_{\omega\in\left[  -\pi,\pi\right]  }a\left(  \omega,\lambda\right)
\geq1+\varepsilon,\;\lambda\mathbf{\in}\Lambda_{m}.
\label{lower-bound-test-functions}%
\end{equation}
For $\alpha$-times differentiable functions $f$ on $\left[  -\pi,\pi\right]  $
fulfilling periodic boundary conditions%
\[
f^{\left(  k\right)  }\left(  -\pi\right)  =f^{\left(  k\right)  }\left(
\pi\right)  \text{, }k=0,\ldots,\alpha-1
\]
the equality
\begin{equation}
\left\vert f\right\vert _{2,\alpha}^{2}=\frac{1}{2\pi}\int_{\left[  -\pi
,\pi\right]  }\left(  f^{\left(  \alpha\right)  }\left(  \omega\right)
\right)  ^{2}d\omega\label{norm-equal-basic}%
\end{equation}
is well known (cp. Lemma A.3 in \cite{MR2724359}). Now $a\left(
\omega,\lambda\right)  $ fulfills the boundary conditions, and
\begin{align*}
a\left(  \omega,\lambda\right)  ^{\left(  \alpha\right)  }  &  =\sum_{j=1}%
^{N}L\lambda_{j}\varphi^{\left(  \alpha\right)  }\left(  \frac{\omega-x_{j}%
}{h}\right)  ,\\
\int_{\left[  -\pi,\pi\right]  }\left(  a\left(  \omega,\lambda\right)
^{\left(  \alpha\right)  }\right)  ^{2}d\omega &  =L\sum_{j=1}^{m}\lambda
_{j}\int_{\Delta_{j}}\left(  \varphi^{\left(  \alpha\right)  }\left(
\frac{\omega-x_{j}}{h}\right)  \right)  ^{2}d\omega\\
&  \leq Lmh\left\Vert \varphi^{\left(  \alpha\right)  }\right\Vert _{2}%
^{2}=L\pi\left\Vert \varphi^{\left(  \alpha\right)  }\right\Vert _{2}^{2}.
\end{align*}
Hence using (\ref{norm-equal-basic})
\begin{equation}
\left\vert a\left(  \cdot,\lambda\right)  \right\vert _{2,\alpha}^{2}\leq
\frac{L}{2}\left\Vert \varphi^{\left(  \alpha\right)  }\right\Vert _{2}^{2}.
\label{seminorm-bound}%
\end{equation}
To bound the full norm $\left\Vert a\left(  \cdot,\lambda\right)  \right\Vert
_{2,\alpha}^{2}$, we also have to bound the Fourier coefficient of order $0$
of $a\left(  \cdot,\lambda\right)  $. We obtain
\begin{align*}
&  \frac{1}{2\pi}\int_{\left[  -\pi,\pi\right]  }a\left(  \omega
,\lambda\right)  d\omega\\
&  =\frac{1}{2\pi}\int_{\left[  -\pi,\pi\right]  }\left(  c+\sum_{j=1}%
^{m}\lambda_{j}h^{\alpha}\varphi\left(  \frac{\omega-x_{j}}{h}\right)
\right)  d\omega=c
\end{align*}
where we have used (\ref{integral-zero}) in the last equality. In conjunction
with (\ref{seminorm-bound}) this gives
\[
\left\Vert a\left(  \cdot,\lambda\right)  \right\Vert _{2,\alpha}^{2}\leq
c^{2}+\frac{L}{2}\left\Vert \varphi^{\left(  \alpha\right)  }\right\Vert
_{2}^{2}.
\]
Recall $c<K$; hence choosing $L$ sufficiently small guarantees $\left\Vert
a\left(  \cdot,\lambda\right)  \right\Vert _{2,\alpha}^{2}\leq K^{2}$. In view
of (\ref{lower-bound-test-functions}), for sufficiently large $m,$ we have
\begin{equation}
\Theta_{0,n}:=\left\{  a\left(  \cdot,\lambda\right)  ,\lambda\mathbf{\in
}\Lambda_{m}\right\}  \subset\Theta_{1}\left(  \alpha,\varepsilon,K\right)  .
\label{inclusion-0}%
\end{equation}
Now consider the H\"{o}lder norm $\left\Vert \cdot\right\Vert _{C^{\beta}}$
given by (\ref{holder-norm}) for $\beta=1$. For the functions $a\left(
\cdot,\lambda\right)  $ we have
\begin{align*}
\left\Vert a\left(  \cdot,\lambda\right)  \right\Vert _{C^{1}}  &
\leq\left\Vert a\left(  \cdot,\lambda\right)  \right\Vert _{\infty}+\left\Vert
a^{\prime}\left(  \cdot,\lambda\right)  \right\Vert _{\infty}\\
&  \leq c+Lh^{\alpha}\left\Vert \varphi\right\Vert _{\infty}+Lh^{\alpha
-1}\left\Vert \varphi^{\prime}\right\Vert _{\infty}\leq c+L\pi^{\alpha}\left(
\left\Vert \varphi\right\Vert _{\infty}+\left\Vert \varphi^{\prime}\right\Vert
_{\infty}\right)  .
\end{align*}
Since also (\ref{lower-bound-test-functions}) holds, we have for any $L>0$,
for sufficiently large $m$ and some $M>0$%
\begin{equation}
\Theta_{0,n}\subset\Theta_{1,c}\left(  1,M\right)  \label{inclusion-1}%
\end{equation}
for the set $\Theta_{1,c}\left(  \beta,M\right)  $ given by
(\ref{Big-Theta-11-c-def}) for $\beta=1$. Suppose we choose
\begin{equation}
m=\left[  c_{0}n^{1/\left(  2\alpha+1\right)  }\right]  +1, \label{choice-m}%
\end{equation}
so that $h\sim c_{0}^{-1}n^{-1/\left(  2\alpha+1\right)  }$. Then we have by
(\ref{shrinking-neighb-unif-norm}) for all $\lambda\mathbf{\in}\Lambda_{m}$
\[
\left\Vert a\left(  \cdot,\lambda\right)  -a_{(0)}\right\Vert _{\infty}\leq
h^{\alpha}\left\Vert \varphi\right\Vert _{\infty}=O\left(  n^{-\alpha/\left(
2\alpha+1\right)  }\right)  .
\]
Setting $\gamma_{n}=n^{-\alpha/\left(  2\alpha+1\right)  }$ and recalling
$a_{(0)}\left(  \omega\right)  =c$, we also have
\begin{equation}
\Theta_{0,n}\subset\Theta_{1,c}\left(  1,M\right)  \cap B\left(
a_{(0)},\gamma_{n}\right)  \label{inclusion-2}%
\end{equation}
where the set $B\left(  a_{(0)},\gamma_{n}\right)  $ is defined in
(\ref{NP-shrinking-neighborhood-def}). Since $\gamma_{n}$ fulfills the
condition\newline$\gamma_{n}=o\left(  \left(  n/\log n\right)  ^{-1/3)}%
\right)  $ if $\alpha\geq1$, we can apply Lemma \ref{lem-local-white-noise};
this will be invoked below.

Recall that our nonparametric estimates of the spectral density $a$ are
originally obtained in the quantum experiment $\mathcal{E}_{n}\left(
\Theta_{1}\left(  \alpha,\varepsilon,K\right)  \right)  $ of
(\ref{basic-quantum-experi-def}), $\Theta_{1}\left(  \alpha,\varepsilon
,K\right)  $ being a subset of some $\Theta_{1}\left(  \alpha,M\right)  $. By
Theorem \ref{theor-main-1} and (\ref{inclusion-0}) we have
\begin{equation}
\mathcal{E}_{n}\left(  \Theta_{0,n}\right)  \preccurlyeq\mathcal{F}_{n}\left(
\Theta_{0,n}\right)  . \label{semi-order-chain-1}%
\end{equation}
In view of (\ref{inclusion-1}), we can apply Lemmas
\ref{Lem-Geom-avg-values-at-points} and \ref{lem-geo-reg-white-noise} to
conclude
\begin{equation}
\mathcal{F}_{n}\left(  \Theta_{0,n}\right)  \simeq\mathcal{F}_{n}^{\prime
}\left(  \Theta_{0,n}\right)  \approx\mathcal{G}_{n,1}\left(  \Theta
_{0,n}\right)  \label{semi-order-chain-2}%
\end{equation}
By (\ref{inclusion-2}) and Lemma \ref{lem-local-white-noise}, we obtain
\begin{equation}
\mathcal{G}_{n,1}\left(  \Theta_{0,n}\right)  \approx\mathcal{G}_{n,2}\left(
\Theta_{0,n}\right)  \label{semi-order-chain-3}%
\end{equation}
where $\mathcal{G}_{n,2}\left(  \Theta_{0,n}\right)  $ is the experiment given
by observations (\ref{pre-variance-stable-SDE}) with $a\in\Theta_{0,n}$, which
for $a_{(0)}\left(  \omega\right)  =c$ is
\begin{equation}
dY_{\omega}=a\left(  \omega\right)  d\omega+\left(  2\pi/n\right)
^{1/2}\left(  c^{2}-1\right)  ^{1/2}dW_{\omega}\text{, }\omega\in\left[
-\pi,\pi\right]  . \label{pre-variance-stable-SDE-2}%
\end{equation}
Set $\sigma^{2}=c^{2}-1$ and consider the equidistant grid in $\left[
-\pi,\pi\right]  $ given by points $t_{j,n}$, $j=1,\ldots,n$ defined in
(\ref{equidistant-grid-def}). Consider the experiment $\mathcal{G}%
_{n,3}\left(  \Theta_{0,n}\right)  $ given by observations
\begin{equation}
Y_{j,n}=a\left(  t_{j,n}\right)  +\sigma\xi_{j},\;j=1,\ldots,n
\label{NP-normal-reg-model}%
\end{equation}
where $\xi_{j}$ are independent standard normal and $a\in\Theta_{0,n}$. By the
result of \cite{MR1425958} for parameter space $C^{1}(M)$ given by
(\ref{Hoelder-class-def}) (which contains $\Theta_{1,c}\left(  1,M\right)  $),
we have in view of (\ref{inclusion-1})
\begin{equation}
\mathcal{G}_{n,2}\left(  \Theta_{0,n}\right)  \approx\mathcal{G}_{n,3}\left(
\Theta_{0,n}\right)  . \label{semi-order-chain-4}%
\end{equation}
Since the relation "$\preccurlyeq$" as a semi-order between sequences of
experiments is transitive, we obtain from (\ref{semi-order-chain-1}%
)-(\ref{semi-order-chain-3}) and (\ref{semi-order-chain-4})%
\begin{equation}
\mathcal{E}_{n}\left(  \Theta_{0,n}\right)  \preccurlyeq\mathcal{G}%
_{n,3}\left(  \Theta_{0,n}\right)  . \label{semi-order-chain-5}%
\end{equation}
As a classical statistical experiment, $\mathcal{G}_{n,3}\left(  \Theta
_{0,n}\right)  $ is to be written%
\[
\left(  \mathbb{R}^{n},\mathcal{B}^{n},\left(  N_{n}\left(  a^{\left(
n\right)  },\sigma^{2}I_{n}\right)  ,a\in\Theta_{0,n}\right)  \right)
\]
where $a^{\left(  n\right)  }:=\left(  a\left(  t_{j,n}\right)  \right)
_{j=1}^{n}$ for a function $a$ on $\left[  -\pi,\pi\right]  $. To write it as
a quantum experiment, let $\nu^{n}$ be the normal $N_{n}\left(  0,\sigma
^{2}I_{n}\right)  $-distribution in $\mathbb{R}^{n}$; then $\mathcal{G}%
_{n,3}\left(  \Theta_{0,n}\right)  $ is the set of states on the von Neumann
algebra $L^{\infty}\left(  \nu^{n}\right)  $ given by the densities $\nu
^{n}\left(  a\right)  :=dN_{n}\left(  a^{\left(  n\right)  },\sigma^{2}%
I_{n}\right)  /d\nu^{n},a\in\Theta_{0,n}$, a subset of $L^{1}\left(  \nu
^{n}\right)  $. By (\ref{semi-order-chain-5}) there is a TP-CP map (dual
channel) $V_{n}:L^{1}\left(  \nu^{n}\right)  \rightarrow\mathcal{L}^{1}\left(
\mathfrak{F}\left(  \mathbb{C}^{n}\right)  \right)  $ such that
\[
\sup_{a\in\Theta_{0,n}}\left\Vert \mathfrak{N}_{n}\left(  0,A_{n}\left(
a\right)  \right)  -V_{n}\left(  \nu^{n}\left(  a\right)  \right)  \right\Vert
_{1}\rightarrow0.
\]
If $T_{\tau}$ is the dual channel given by
(\ref{state-transition-lower-bound-2}) and $T_{\tau}V_{n}$ the composed map,
it follows
\[
\sup_{a\in\Theta_{0,n}}\left\Vert T_{\tau}\left(  \mathfrak{N}_{n}\left(
0,A_{n}\left(  a\right)  \right)  \right)  -T_{\tau}V_{n}\left(  \nu
^{n}\left(  a\right)  \right)  \right\Vert _{1}\rightarrow0.
\]
Equivalently, in terms of the parameter $\theta$, we obtain
\[
\sup_{\theta\in\Theta_{0,n}^{\prime}}\left\Vert T_{\tau}\left(  \mathfrak{N}%
_{n}\left(  0,A_{n}\left(  a_{\theta}\right)  \right)  \right)  -T_{\tau}%
V_{n}\left(  \nu^{n}\left(  a_{\theta}\right)  \right)  \right\Vert
_{1}\rightarrow0
\]
where, with $\Theta_{0,n}$ from (\ref{inclusion-0}), we define%
\[
\Theta_{0,n}^{\prime}:=\left\{  \theta\in\ell^{2}:a_{\theta}\in\Theta
_{0,n}\right\}  .
\]
In (\ref{Markov-lower-bound-a}) the probability measure $Q_{\rho_{n}\left(
\theta\right)  ,\tau}$ corresponding to the state $T_{\tau}\left(  \rho
_{n}\left(  \theta\right)  \right)  $, $\rho_{n}\left(  \theta\right)
=\mathfrak{N}_{n}\left(  0,A_{n}\left(  a_{\theta}\right)  \right)  $ can
therefore be replaced, up to an $o\left(  1\right)  $ term, by a probability
measure $Q_{1,n,\theta,\tau}\left(  \cdot\right)  $, say, corresponding to the
state $T_{\tau}V_{n}\left(  \nu^{n}\left(  a_{\theta}\right)  \right)  $.
Hence from (\ref{Markov-lower-bound-a}) we obtain,
\begin{align}
&  \inf_{\tau}\sup_{\theta\in\Theta_{1}^{\prime}\left(  \alpha,\varepsilon
,K\right)  }n^{2\alpha/\left(  2\alpha+1\right)  }\int\left\Vert
a_{u}-a_{\theta}\right\Vert _{2}^{2}Q_{\rho_{n}\left(  \theta\right)  ,\tau
}\left(  du\right) \nonumber\\
&  \geq\inf_{\tau}\sup_{\theta\in\Theta_{0,n}^{\prime}}C\;Q_{1,n,\theta,\tau
}\left(  \left\Vert \hat{a}_{n,\tau}-a_{\theta}\right\Vert _{2}^{2}\geq
n^{-2\alpha/\left(  2\alpha+1\right)  }2C\right)  +o\left(  1\right)
\label{Markov-lower-bound-2}%
\end{align}
where $\hat{a}_{n,\tau}$ is the estimator of $a$ corresponding via
(\ref{a-hat-n-p-def}) to a state transition $T_{\tau}V_{n}$ with POVM\ $\tau$
as in (\ref{POVM-mod}). Here $T_{\tau}V_{n}$ is a state transition%
\[
T_{\tau}V_{n}:L^{1}\left(  \nu^{n}\right)  \rightarrow L^{1}\left(
\mu\right)
\]
where $\mu$ is the probability measure on $\left(  \mathbb{R}^{2n+1}%
,\mathcal{B}^{2n+1}\right)  $ from (\ref{state-transition-lower-bound-2}),
related to the POVM\ $\tau$. Since $T_{\tau}V_{n}$ is a TP-CP map, it is
linear, positive and $L^{1}$-norm-preserving, hence it is a transition in the
sense of Le Cam. The measures $\nu^{n}$ and $\mu$ are both given on Polish
spaces with their respective Borel sigma-algebra ($\mathbb{R}^{n}$ and
$\mathbb{R}^{2n+1}$, resp.), hence by Proposition 9.2. in \cite{MR1425959} the
transition $T_{\tau}V_{n}$ is given by a Markov kernel
\begin{equation}
\kappa:\mathcal{B}^{2n+1}\times\mathbb{R}^{n}\rightarrow\left[  0,1\right]  .
\label{Markov-kernel-np-1}%
\end{equation}
Alternatively, that claim follows from Remark 55.6(3) in \cite{MR812467}
\textbf{(CHECK\ THAT\ AGAIN!). }Hence we obtain a lower bound to
(\ref{Markov-lower-bound-2}) by replacing $\inf_{\tau}$ by an infimum over all
Markov kernels $\kappa$ as in (\ref{Markov-kernel-np-1}). Accordingly, let
$Q_{2,n,\theta}$ be the probability measure on $\left(  \mathbb{R}%
^{n},\mathcal{B}^{n}\right)  $
\begin{equation}
Q_{2,n,\theta}=N_{n}\left(  a_{\theta}^{\left(  n\right)  },\sigma^{2}%
I_{n}\right)  , \label{law-Q-2-n-theta-def}%
\end{equation}
with $E_{2,n,\theta}$ being the corresponding expectation, and for
$u\in\mathbb{R}^{2n+1}$ denote the function on $\left[  -\pi,\pi\right]  $
\[
\hat{a}\left(  u\right)  :=\sum_{\left\vert j\right\vert \leq n}u_{j}\psi_{j}%
\]
where $\psi_{j}$ are the basis functions (\ref{psi-def}). Then we obtain
\begin{align}
&  \inf_{\tau}\sup_{\theta\in\Theta_{1}^{\prime}\left(  \alpha,\varepsilon
,K\right)  }n^{2\alpha/\left(  2\alpha+1\right)  }E_{n,\rho_{n}\left(
\theta\right)  ,\tau}\left\Vert \hat{a}_{n}-a_{\theta}\right\Vert _{2}%
^{2}\nonumber\\
&  \geq\inf_{\kappa}\sup_{\theta\in\Theta_{0,n}^{\prime}}C\;E_{2,n,\theta}%
\int\mathbf{1}\left(  \left\Vert \hat{a}\left(  u\right)  -a_{\theta
}\right\Vert _{2}^{2}\geq n^{-2\alpha/\left(  2\alpha+1\right)  }2C\right)
\kappa\left(  du,Y^{\left(  n\right)  }\right)  +o\left(  1\right)
\label{pre-lower-rate-bound-2}%
\end{align}
where $Y^{(n)}$ are the data in the regression model
(\ref{NP-normal-reg-model}). We will now apply the method used in Theorem 2.5
of \cite{MR2724359} to obtain a lower risk bound for $L^{2}$-error in multipe
hypothesis testing, appropriately modified to allow for randomized decisions.
Let $L^{2}=L^{2}\left(  -\pi,\pi\right)  $ and $\mathcal{B}_{L^{2}}$ be the
corresponding Borel sigma-algebra; for probability measures $P,Q$ with $P\ll
Q$ the relative .entropy (Kullback information) is
\[
S\left(  P||Q\right)  =\int\log\frac{dP}{dQ}dP.
\]
The method involves the following.

\textbf{(a)} A set of "hypotheses" $\Theta_{00}=\left\{  f_{j},j=0,\ldots
,N\right\}  \subset L^{2}$ fulfilling%
\[
\left\Vert f_{j}-f_{k}\right\Vert _{2}\geq2s>0,0\leq j<k\leq N
\]
\newline\textbf{(b) }a set of probability measures $Q_{f}$ on a measurable
space $\left(  X,\Omega_{X}\right)  $, indexed by $f\in\Theta_{00}$ such that
with $Q_{j}:=Q_{f_{j}}$ and $0<\gamma<1/8$ one has
\begin{align}
Q_{j}  &  \ll Q_{0},\;j=1,\ldots,N,\label{cond-Tsyb-2}\\
\frac{1}{N}\sum_{j=1}^{N}S\left(  Q_{j}||Q_{0}\right)   &  \leq\gamma\log N.
\label{cond-Tsyb-3}%
\end{align}
Then for any nonrandomized estimator $\hat{f}:\left(  X,\Omega_{X}\right)
\rightarrow\left(  L^{2},\mathcal{B}_{L^{2}}\right)  $ one has
\begin{equation}
\max_{f\in\Theta_{00}}Q_{f}\left(  \left\Vert \hat{f}-f\right\Vert
_{2}>s\right)  \geq\frac{\sqrt{N}}{1+\sqrt{N}}\left(  1-2\gamma-\sqrt
{\frac{2\gamma}{\log N}}\right)  . \label{result-Tsyb-4}%
\end{equation}
The nonasymptotic result (\ref{result-Tsyb-4}) holds in generality under the
setting \textbf{(a)} and \textbf{(b)}; it is not tied to any particular form
of the distributions $Q_{j}$. It can be extended to randomized estimators in
the following way. Let $\left(  Y,\Omega_{Y}\right)  $ be a measurable space
and $\kappa$ be a Markov kernel $\kappa:\Omega_{Y}\times w\rightarrow\left[
0,1\right]  $ for $w\in X$. Consider extensions $\tilde{Q}_{j}=\kappa\times
Q_{j}$ of the probability measures $Q_{j}$ onto $\Omega_{Y}\times\Omega_{X}$
in the following way: for any p.m. $Q$ on $\left(  X,\Omega_{X}\right)  $ we
set
\[
\left(  \kappa\times Q\right)  \left(  B_{1}\times B_{2}\right)  =\int_{B_{2}%
}\kappa\left(  B_{1},w\right)  Q\left(  dw\right)  \text{, }B_{1}\in\Omega
_{Y}\text{, }B_{2}\in\Omega_{X}\text{.}%
\]
Assume $Q_{j}\ll Q_{0}$, $\;j=1,\ldots,N$ and let $q_{j}=dQ_{j}/dQ_{0}$,
$\;j=1,\ldots,N$; then it is easy to see that $\tilde{Q}_{j}\ll\tilde{Q}_{0}$
and
\[
\frac{d\tilde{Q}_{j}}{d\tilde{Q}_{0}}\left(  y,w\right)  =q_{j}\left(
w\right)  \text{, }j=1,\ldots,N\text{, a.s. }\left[  \tilde{Q}_{0}\right]
\text{.}%
\]
Consequently
\[
S\left(  \tilde{Q}_{j}||\tilde{Q}_{0}\right)  =\int\log\frac{d\tilde{Q}_{j}%
}{d\tilde{Q}_{0}}d\tilde{Q}_{j}=\int\log q_{j}d\tilde{Q}_{j}=\int\log
q_{j}dQ_{j}=S\left(  Q_{j}||Q_{0}\right)  .
\]
In the setting \textbf{(a)}, \textbf{(b)} we now set $Q_{j}:=\kappa\times
Q_{f_{j}},$ and consider all nonrandomized estimators $\tilde{f}:\left(
Y\times X,\Omega_{Y}\times\Omega_{X}\right)  \rightarrow\left(  L^{2}%
,\mathcal{B}_{L^{2}}\right)  $. If $\kappa$ is such that $\left(  Y,\Omega
_{Y}\right)  =\left(  L^{2},\mathcal{B}_{L^{2}}\right)  $ and $\tilde
{f}\left(  y,w\right)  =y$ then $\tilde{f}$ corresponds to a randomized
estimator using Markov $\kappa$ and original data $w\in X$. Applying result
(\ref{result-Tsyb-4}) in this setting, we can now claim that it holds for any
randomized estimator given by a Markov kernel $\kappa:\mathcal{B}_{L^{2}%
}\times\mathbb{R}^{n}\rightarrow\left[  0,1\right]  $ with input data in
$\left(  X,\Omega\right)  $.

In the regression model (\ref{NP-normal-reg-model}) this can be applied in the
following way. To find a set of functions $\Theta_{00,n}$, note first that for
the functions $a\left(  \cdot,\lambda\right)  $ given by
(\ref{special-functions-a-lambda-def}) we have for $\lambda,\lambda^{\prime
}\in\Lambda_{m}$%
\begin{align}
\left\Vert a\left(  \cdot,\lambda\right)  -a\left(  \cdot,\lambda^{\prime
}\right)  \right\Vert _{2}^{2}  &  =\sum_{j=1}^{m}\left(  \lambda_{j}%
-\lambda_{j}^{\prime}\right)  ^{2}\int_{\Delta_{j}}L^{2}h^{2\alpha}\varphi
^{2}\left(  \frac{\omega-x_{j}}{h}\right)  d\omega\nonumber\\
&  =L^{2}h^{2\alpha+1}\left\Vert \varphi\right\Vert _{2}^{2}\sum_{j=1}%
^{m}\left(  \lambda_{j}-\lambda_{j}^{\prime}\right)  ^{2}\nonumber\\
&  =L^{2}h^{2\alpha+1}\left\Vert \varphi\right\Vert _{2}^{2}\rho\left(
\lambda,\lambda^{\prime}\right)  \label{Var-Gil-1}%
\end{align}
where
\[
\rho\left(  \lambda,\lambda^{\prime}\right)  =\sum_{j=1}^{m}\mathbf{1}\left(
\lambda_{j}\neq\lambda_{j}^{\prime}\right)
\]
is the Hammming distance between $\lambda,\lambda^{\prime}\in\Lambda_{m}$. Now
the Varshamov-Gilbert bound (Lemma 2.9 in \cite{MR2724359}) provides a subset
$\left\{  \lambda^{\left(  0\right)  },\ldots,\lambda^{\left(  N\right)
}\right\}  $ of $\Lambda_{m}$ such that $\lambda^{\left(  0\right)  }=\left(
0,\ldots,0\right)  $ and
\begin{align}
\rho\left(  \lambda^{\left(  j\right)  },\lambda^{\left(  k\right)  }\right)
&  \geq\frac{m}{8}\text{, }0\leq j<k\leq N,\label{Var-Gil-2}\\
N  &  \geq2^{m/8}. \label{Var-Gil-3}%
\end{align}
Define a set of functions on $\left[  -\pi,\pi\right]  $%
\begin{align*}
a_{j,n}  &  =a\left(  \cdot,\lambda^{\left(  j\right)  }\right)
,j=0,\ldots,N,\\
\Theta_{00,n}  &  =\left\{  a_{j,n},j=0,\ldots,N\right\}  .
\end{align*}
Then with a choice of $m$ as in (\ref{choice-m}) we obtain from () and (9%
\begin{align*}
\left\Vert a_{j,n}-a_{k,n}\right\Vert _{2}^{2}  &  \geq L^{2}h^{2\alpha
+1}\left\Vert \varphi\right\Vert _{2}^{2}\frac{m}{8}=\frac{L^{2}\pi
^{2\alpha+1}}{8}\left\Vert \varphi\right\Vert _{2}^{2}m^{-2\alpha}\\
&  \geq\frac{L^{2}\pi^{2\alpha+1}}{16}\left\Vert \varphi\right\Vert _{2}%
^{2}c_{0}^{2\alpha}n^{-2\alpha/\left(  2\alpha+1\right)  }%
\end{align*}
for sufficiently large $n$. Hence for a constant $C_{1}=L\pi^{\alpha
+1/2}\left\Vert \varphi\right\Vert _{2}/8$ and
\begin{equation}
s_{n}:=C_{L,\varphi}c_{0}^{\alpha}n^{-\alpha/\left(  2\alpha+1\right)  },
\label{s-n-rate-def}%
\end{equation}
Condition \textbf{(a)} is fulfilled with $f_{j}=a_{j,n}$ and
\[
\left\Vert a_{j,n}-a_{k,n}\right\Vert _{2}\geq2s_{n}\text{. }%
\]
Regarding Condition \textbf{(b),} we set $Q_{j}=N_{n}\left(  a_{j,n}^{\left(
n\right)  },\sigma^{2}I_{n}\right)  $, $j=0,\ldots,N$, then evidently
(\ref{cond-Tsyb-2}) is fulfilled. We note the well known general result, for
$d_{i}\in\mathbb{R}^{n}$, $i=0,1$
\[
S\left(  N_{n}\left(  d_{1},\sigma^{2}I_{n}\right)  ||N_{n}\left(
d_{0},\sigma^{2}I_{n}\right)  \right)  =\frac{1}{2\sigma^{2}}\left\Vert
d_{1}-d_{0}\right\Vert ^{2}.
\]
In view of $Q_{0}=N_{n}\left(  c\mathbf{1}_{n},\sigma^{2}I_{n}\right)  $,
$\mathbf{1}_{n}=\left(  1,\ldots,1\right)  ^{\prime}$ we obtain with
$\sigma^{2}=c^{2}-1$
\begin{align*}
S\left(  Q_{j}||Q_{0}\right)   &  =\frac{1}{2\sigma^{2}}\left\Vert
a_{j,n}^{\left(  n\right)  }-c\mathbf{1}_{n}\right\Vert ^{2}=\frac{1}%
{2\sigma^{2}}\sum_{k=1}^{n}\left(  a_{j,n}\left(  t_{k,n}\right)  -c\right)
^{2}\\
&  \leq\frac{1}{2\sigma^{2}}L^{2}h^{2\alpha}\left\Vert \varphi\right\Vert
_{\infty}n=\frac{\pi^{2\alpha}}{2\left(  c^{2}-1\right)  }L^{2}m^{-2\alpha
}\left\Vert \varphi\right\Vert _{\infty}n\\
&  \leq C_{2}c_{0}^{-\left(  2\alpha+1\right)  }m\text{ for }C_{2}%
=L^{2}\left\Vert \varphi\right\Vert _{\infty}\pi^{2\alpha}/2\left(
c^{2}-1\right)  .
\end{align*}
By (\ref{Var-Gil-3}) we have $m\leq8\log N/\log2$. Now select some $\gamma
\in\left(  0,1/8\right)  $ and then choose
\[
c_{0}=\left(  \frac{8C_{2}}{\gamma\log2}\right)  ^{1/\left(  2\alpha+1\right)
}.
\]
We then obtain
\[
S\left(  Q_{j}||Q_{0}\right)  \leq\gamma\log N
\]
and thus condition \textbf{(b)}. It is then to be noted that $N\rightarrow
\infty$ by (\ref{Var-Gil-3}), so that the r.h.s. of (\ref{result-Tsyb-4}) is
$\left(  1-2\gamma\right)  \left(  1+o\left(  1\right)  \right)  $. Moreover,
the choice of $C>0$ in (\ref{Markov-lower-bound-0}) was arbitrary; with a
choice $2C>C_{L,\varphi}c_{0}^{\alpha}$ (cp. (\ref{s-n-rate-def})) we obtain
in (\ref{pre-lower-rate-bound-2}), with (\ref{result-Tsyb-4})
\begin{align*}
&  \inf_{\kappa}\sup_{\theta\in\Theta_{0,n}^{\prime}}\;E_{2,n,\theta}%
\int\mathbf{1}\left(  \left\Vert \hat{a}\left(  u\right)  -a_{\theta
}\right\Vert _{2}^{2}\geq n^{-2\alpha/\left(  2\alpha+1\right)  }2C\right)
\kappa\left(  du,Y^{\left(  n\right)  }\right) \\
&  \geq\left(  1-2\gamma\right)  \left(  1+o\left(  1\right)  \right)  \text{
as }n\rightarrow\infty
\end{align*}
where the infimum extends over all Markov kernels (\ref{Markov-kernel-np-1}).
Since $\gamma\in\left(  0,1/8\right)  $, with (\ref{Markov-lower-bound-2}) the
result is proved.
\end{proof}

%

\end{privatenotes}%

\appendix

\section{Appendix}

\subsection{States, channels, observables\label{subsec:States-channels-obs}}

\subsubsection{\textit{Von Neumann algebras}}

\textit{ }Let $\mathcal{A}$ be a von Neumann algebra of bounded linear
operators on a complex Hilbert space $\mathcal{H}$ (\cite{MR1721402}, \S 46).
$\mathcal{H}$ will be assumed separable in the sequel. The two examples we
will consider are (i) the set $\mathcal{L}(\mathcal{H})$ of bounded linear
operators on $\mathcal{H}$ (\cite{MR1070713}, IX.7.2) ), (ii) the set of
functions $L^{\infty}\left(  \mu\right)  $ on a $\sigma$-finite measure space
$\left(  X,\Omega,\mu\right)  $, construed as linear operators on
$\mathcal{H}=L^{2}\left(  \mu\right)  $ by pointwise multiplication
(\cite{MR1070713}, IX.7.2 for both cases). In the former case, $\mathcal{H}$
will always be a symmetric Fock space $\mathfrak{F}\left(  \mathbb{C}%
^{n}\right)  $, which his separable (\cite{MR3012668}, 19.3, cf. also Lemma
\ref{Lem-spec-decompos-Fock-op} below). In the latter case, the measurable
space $\left(  X,\Omega\right)  $ will be a Polish space with the respective
Borel $\sigma$-algebra, so that $L^{2}\left(  \mu\right)  $ is separable
(\cite{MR3098996}, 3.4.5).%

\begin{privatenotes}
\begin{boxedminipage}{\textwidth}%

\begin{sfblock}
In the latter reference (Cohn), the theorem 3.4.5 only assumes that a
sigma-algebra $\Omega$ is countably generated. But this is true for the Borel
$\sigma$-algebra of a Polish space; the near-trivial proof can be found in
discussions on the Internet. Let $X_{0}$ be a countable dense subset of $x$
and consider all open metric balls $\mathcal{M}=\left\{  B\left(
x,1/n\right)  ,x\in X_{0},n\in\mathbb{N}\right\}  $ ($x$ the center of the
ball, $1/n$ the radius). Then $\mathcal{M}$ is a countable collection of sets,
and every open set in $X$ is a countable union of sets from $\mathcal{M}$.
Hence the Borel $\sigma$-algebra of $X$ is generated by $\mathcal{M}$, hence
countably generated.
\end{sfblock}

%

\end{boxedminipage}
\end{privatenotes}%

\subsubsection{\textit{The predual}}

For every von Neumann algebra $\mathcal{A}$ there is a Banach space
$\mathcal{A}_{\ast}$ such that $\mathcal{A}$ is the dual Banach space of
$\mathcal{A}_{\ast}$ (\cite{MR1490835}, 1.1.2). $\mathcal{A}_{\ast}$ is unique
up to an isometric isomorphism (\cite{MR1490835} 1.13.3, \cite{MR1741419}
VI.6.9, Corollary 1). $\mathcal{A}_{\ast}$ is called the predual of
$\mathcal{A}$; the pertaining duality is
\begin{equation}
\left\langle a,\tau\right\rangle =a\left(  \tau\right)  \text{, }%
a\in\mathcal{A}\text{, }\tau\in\mathcal{A}_{\ast}. \label{duality-predual}%
\end{equation}
The norm on $\mathcal{A}_{\ast}$, written $\left\Vert \cdot\right\Vert _{1}$
here, is derived from the norm of the dual Banach space $\mathcal{A}^{\ast}$ (
\cite{MR887100}, 2.4.18), i. e.
\begin{equation}
\left\Vert \tau\right\Vert _{1}:=\sup_{\left\Vert a\right\Vert \leq
1}\left\vert \left\langle a,\tau\right\rangle \right\vert \text{, }\tau
\in\mathcal{A}_{\ast}. \label{norm-predual}%
\end{equation}
On the other hand, since $\mathcal{A}$ is the dual of $\mathcal{A}_{\ast}$,
the norm of $\mathcal{A}$ fulfills
\[
\left\Vert a\right\Vert :=\sup_{\left\Vert \tau\right\Vert \leq1}\left\vert
\left\langle a,\tau\right\rangle \right\vert \text{.}%
\]
In case (i), if $\mathcal{A=L}(\mathcal{H})$ then $\mathcal{A}_{\ast
}\mathcal{=L}^{1}(\mathcal{H}),$ the Banach space of trace class operators $R$
on $\mathcal{H}$ with norm $\left\Vert R\right\Vert _{1}=\mathrm{Tr\,}\left(
R^{\ast}R\right)  ^{1/2}$, and (\ref{duality-predual}), (\ref{norm-predual})
take the form
\begin{align}
\left\langle a,R\right\rangle  &  =\mathrm{Tr\,}aR\text{, }a\in\mathcal{L}%
(\mathcal{H})\text{, }R\in\mathcal{L}^{1}(\mathcal{H}%
),\label{duality-predual-1}\\
\left\Vert R\right\Vert _{1}  &  =\mathrm{Tr\,}\left(  R^{\ast}R\right)
^{1/2}, \label{norm-predual-1}%
\end{align}
(\cite{MR1741419} VI.6, \cite{MR3468018} 2.1.6). In case (ii), if
$\mathcal{A=}L^{\infty}\left(  \mu\right)  $ then $\mathcal{A}_{\ast
}\mathcal{=}L^{1}\left(  \mu\right)  $, and (\ref{duality-predual}),
(\ref{norm-predual}) are given by
\begin{align}
\left\langle a,f\right\rangle  &  =\int afd\mu\text{, }a\in L^{\infty}\left(
\mu\right)  \text{, }f\in L^{1}\left(  \mu\right)  , \label{duality-predual-2}%
\\
\left\Vert f\right\Vert _{1}  &  =\int\left\vert f\right\vert d\mu,
\end{align}
(\cite{MR1490835}, 1.13.3, \cite{MR1741419}, VI.6.8, \cite{MR887100}, 2.4.17,
\cite{MR3468018}, 2.1.12)

\subsubsection{\textit{States }}

(\cite{MR887100},\textit{ }\cite{MR3468018}, sec. 2.2). An element $a$ of
$\mathcal{A}$ is positive ($a\geq0$) if $a$ is self-adjoint and $\left\langle
x|ax\right\rangle \geq0$ for every $x\in\mathcal{H}_{A}$ (\cite{MR1070713},
VIII, \S 3). A linear functional $\tau:$ $\mathcal{A}\rightarrow\mathbb{C}$ is
said to be positive if $\tau(a)\geq0$ for all $a\geq0$. Such functionals are
continuous (bounded) on $\mathcal{A}$ (\cite{MR887100} , 2.3.11). A state on
$\mathcal{A}$ is a positive element of $\mathcal{A}_{\ast}$ which takes value
$1$ on the unit of $\mathcal{A}$. In case (i), by (\ref{duality-predual-1})
$\tau$ is given by a positive element $\rho_{\tau}$ of $\mathcal{L}%
^{1}(\mathcal{H})$ with $\mathrm{Tr\,}\rho=1$ (a density operator) such that
$\tau\left(  A\right)  =\mathrm{Tr\,}\rho A$. In case (ii), by by
(\ref{duality-predual-2}) $\tau$ is given by a positive function $f_{\tau}$ in
$L^{1}\left(  \mu\right)  $ with $\int f_{\tau}d\mu=1$ (a probability density
function) such that $\tau\left(  \phi\right)  =\int\phi f_{\tau}d\mu$.

\subsubsection{\textit{Normal maps\label{subsubsec-normal}}}

For the strong and weak operator topologies on $\mathcal{A}$ (SOT, WOT) cf.
\cite{MR1721402}, \S 8; for the weak* topology cf. \cite{MR1721402}, \S 20 or
its equivalent definition as the $\sigma$-weak topology in \cite{MR887100},
2.4.2. For two von Neumann algebras $\mathcal{A},\mathcal{B}$, a linear map
$\alpha:\mathcal{A}\rightarrow\mathcal{B}$ is positive if $\alpha\left(
a\right)  \geq0$ for every $a\geq0.$Such maps are bounded \cite{MR1721402},
33.4. A positive linear map $\alpha:\mathcal{A}\rightarrow\mathcal{B}$ is said
to be normal if for every increasing net $\left\{  a_{\gamma}\right\}  $ such
that $a_{\gamma}\rightarrow a$ (SOT) one has $\alpha\left(  a_{\gamma}\right)
\rightarrow\alpha\left(  a\right)  $ (SOT) (\cite{MR1721402}, 46.1). If the
respective Hilbert spaces $\mathcal{H}_{A}$, $\mathcal{H}_{B}$ are separable
then the SOT is metrizable on bounded subsets (\cite{MR1070713} , IX.1.3) and
hence nets can be replaced by sequences. A positive linear map $\alpha$ is
normal if and only if it is weak* continuous (\cite{MR1721402}, 46.5). It is
clear that compositions of bounded positive normal maps are normal. Consider
the special case of $\mathcal{B}=\mathbb{C}$, when $\alpha$ is a positive
linear functional on $\mathcal{A}$. The predual of $\mathcal{A}$ can be taken
as the Banach space generated by all normal linear forms on $\mathcal{A}$
(\cite{MR1741419}, VI.6.9, or \cite{MR887100}, 2.4.18, 2.4.21). Thus states on
$\mathcal{A}$ can also be described as positive normal linear forms on
$\mathcal{A}$ which take value $1$ on the unit of $\mathcal{A}$ (cf. also
\cite{MR1721402} 46.4 or \cite{MR3468018}, 2.1.7).

\subsubsection{\textit{Complete positivity}}

Let $\mathcal{A}$, $\mathcal{B}$ be a von Neumann algebras of operators on
respective Hilbert spaces $\mathcal{H}_{A},$ $\mathcal{H}_{B}$. The algebra
$M_{n}\left(  \mathcal{A}\right)  $ of all $n\times n$ matrices with entries
from $\mathcal{A}$ acting on the $n$-fold direct sum $\mathcal{H}_{A}%
^{(n)}:=\mathcal{H}_{A}\oplus\ldots\oplus\mathcal{H}_{A}$ is a von Neumann
algebra, with norm derived from its being a subalgebra of $\mathcal{L}\left(
\mathcal{H}_{A}^{(n)}\right)  $ (\cite{MR1721402}, \S 34, \S 44). An element
$a=\left(  a_{ij}\right)  _{i,j=1}^{n}\in$ $M_{n}\left(  \mathcal{A}\right)  $
is called positive if the associated linear operator on $\mathcal{H}_{A}%
^{(n)}$ is positive, i.e $a$ is self-adjoint and $\left\langle
x|ax\right\rangle \geq0$ for every $x\in\mathcal{H}_{A}^{(n)}$. For a linear
map $\alpha:\mathcal{A}\rightarrow\mathcal{B}$, define an associated map
$\alpha_{n}:M_{n}\left(  \mathcal{A}\right)  \rightarrow M_{n}\left(
\mathcal{B}\right)  $ by $\alpha_{n}\left(  a\right)  =\left(  \alpha\left(
a_{ij}\right)  \right)  _{i,j=1}^{n}$. The map $\alpha$ is completely positive
if for every $n\geq1$, the map $\alpha_{n}$ is positive\textit{ (}%
\cite{MR3468018}, sec. 5.4). Compositions of completely positive maps are
completely positive \textbf{(}\cite{MR3468018}, 5.4.9\textbf{). }If either
$\mathcal{A}$ or $\mathcal{B}$ are commutative then every positive linear map
is completely positive (\cite{MR3468018}, 5.4.6).

This concept can be developed in parallel for the preduals $\mathcal{A}_{\ast
}$, $\mathcal{B}_{\ast}$. The predual of $M_{n}\left(  \mathcal{A}\right)  $
is the Banach space $M_{n}\left(  \mathcal{A}\right)  _{\ast}$ of $n\times n$
matrices with entries from $\mathcal{A}_{\ast}$, acting on $M_{n}\left(
\mathcal{A}\right)  $ according to
\[
\left\langle a,\tau\right\rangle =\sum_{i,j=1}^{n}\left\langle a_{ij}%
,\tau_{ij}\right\rangle \text{, }a\in M_{n}\left(  \mathcal{A}\right)  \text{,
}\tau\in M_{n}\left(  \mathcal{A}\right)  _{\ast}\text{ }%
\]
where $a=\left(  a_{ij}\right)  _{i,j=1}^{n}$, $\tau=\left(  \tau_{ij}\right)
_{i,j=1}^{n}$. The norm of $M_{n}\left(  \mathcal{A}\right)  _{\ast}$ is
\[
\left\Vert \tau\right\Vert _{1}=\sup_{a\in M_{n}\left(  \mathcal{A}\right)
,\left\Vert a\right\Vert =1}\left\vert \left\langle a,\tau\right\rangle
\right\vert \text{, }\tau\in M_{n}\left(  \mathcal{A}\right)  _{\ast}.
\]
An element $\tau\in$ $M_{n}\left(  \mathcal{A}\right)  _{\ast}$ is positive if
$\left\langle a,\tau\right\rangle \geq0$ for every $a\geq0$, $a\in
M_{n}\left(  \mathcal{A}\right)  $. Let $\mathbf{1}$\textbf{ }be the unit
of\textbf{ }$\mathcal{A}$ and let $\mathbf{1}_{n}$ be the unit of
$M_{n}\left(  \mathcal{A}\right)  $, i.e. the diagonal matrix with diagonal
entries all $\mathbf{1}$. Let $\tau\in$ $M_{n}\left(  \mathcal{A}\right)
_{\ast}$, $\tau\geq0$; then%
\[
\left\Vert \tau\right\Vert _{1}=\left\langle \mathbf{1}_{n},\tau\right\rangle
=\sum_{i=1}^{n}\left\langle \mathbf{1},\tau_{ii}\right\rangle =\sum_{i=1}%
^{n}\left\Vert \tau_{ii}\right\Vert _{1}.
\]
For a linear map $T:\mathcal{A}_{\ast}\rightarrow\mathcal{B}_{\ast}$, define
an associated map $T_{n}:M_{n}\left(  \mathcal{A}\right)  _{\ast}\rightarrow
M_{n}\left(  \mathcal{B}\right)  _{\ast}$ by $T_{n}\left(  a\right)  =\left(
T\left(  a_{ij}\right)  \right)  _{i,j=1}^{n}$. The map $T$ is completely
positive if for every $n\geq1$, the map $T_{n}$ is positive.

\subsubsection{\textit{ Channels}}

\textit{ }(\cite{MR1230389}, chap. 8)\textit{. }Consider a linear map
$\alpha:\mathcal{A}\rightarrow\mathcal{B}$. The mapping $\alpha$ is unital if
it maps the unit of $\mathcal{A}$ into the unit of $\mathcal{B}$. A quantum
channel is a linear, completely positive, unital and normal map $\alpha
:\mathcal{A}\rightarrow\mathcal{B}$. Here boundedness of $\alpha$ follows from
positivity (\cite{MR1721402}, 33.4). Compositions of channels are channels
again. Channels have the Kraus representation
\[
\alpha\left(  a\right)  =\sum_{j=1}^{\infty}V_{j}^{\ast}aV_{j}\text{, }%
a\in\mathcal{A}%
\]
where $\left\{  V_{j}\right\}  _{j\geq1}$ is a sequence of bounded linear
operators $V_{j}:\mathcal{H}_{B}\rightarrow$ $\mathcal{H}_{A}$ such that
$\sum_{j=1}^{\infty}V_{j}^{\ast}V_{j}=\mathbf{1}$, and the sums are convergent
in SOT (\cite{MR3012668}, 29.8, \cite{MR3468018}, 5.4.16 ). An important
special case with $\mathcal{A}=\mathcal{B=L}(\mathcal{H})$ is
\begin{equation}
\alpha\left(  a\right)  =U^{\ast}aU\text{, }a\in\mathcal{A}
\label{channel-unitary}%
\end{equation}
where $U$ is a unitary operator on $\mathcal{H}_{A}$.

\subsubsection{\textit{State transitions (TP-CP
maps)\label{subsub-Append-state-transitions}}}

Since a state is a channel $\tau:\mathcal{A}\rightarrow\mathbb{C}$, it follows
that a composition of a state $\tau$ on $\mathcal{A}$ with a channel
$\alpha:\mathcal{B}\rightarrow\mathcal{A}$ gives a state $\tau\circ\alpha$ on
$\mathcal{B}$. This mapping of states extends to a linear map of the preduals
$T:\mathcal{A}_{\ast}\rightarrow\mathcal{B}_{\ast}$; the map $T$ is called the
dual channel of $\alpha$. Since $\alpha$ is completely positive, it can be
shown that $T$ is completely positive (CP), and since $\alpha$ is unital, it
follows that $T$ is norm preserving on positives:
\begin{equation}
\left\Vert T\left(  \sigma\right)  \right\Vert _{1}=\left\Vert \sigma
\right\Vert _{1},\sigma\geq0\text{, }\sigma\in\mathcal{A}_{\ast}.
\label{norm-preserving}%
\end{equation}
In the case $\mathcal{A=L}(\mathcal{H}_{A})$, $\mathcal{B=L}(\mathcal{H}_{B})$
the latter property can be written $\mathrm{Tr\,}T\left(  \rho\right)
=\mathrm{Tr\,}\rho$ for $\rho\geq0$, $\rho\in\mathcal{L}^{1}(\mathcal{H}_{A}%
)$, thus $T$ is trace preserving (TP) on positives. In this context a dual
channel $T$ is often called a TP-CP map; more generally a CP linear map
$T:\mathcal{A}_{\ast}\rightarrow\mathcal{B}_{\ast}$ fulfilling
(\ref{norm-preserving}) will be called a state transition. State transitions
have the contraction property:%
\begin{equation}
\left\Vert T\left(  \sigma_{1}\right)  -T\left(  \sigma_{2}\right)
\right\Vert _{1}\leq\left\Vert \sigma_{1}-\sigma_{2}\right\Vert _{1}%
,\sigma_{i}\geq0\text{, }\sigma_{i}\in\mathcal{A}_{\ast}\text{, }i=1,2.
\label{contraction-property}%
\end{equation}
The pair $\left(  \alpha,T\right)  $ is said to be a dual pair if
\begin{equation}
\left\langle \alpha\left(  b\right)  ,\omega\right\rangle =\left\langle
b,T\left(  \omega\right)  \right\rangle \text{, }b\in\mathcal{B}\text{,
}\omega\in\mathcal{A}_{\ast}. \label{dual-channels-1}%
\end{equation}
The above construction shows that for every channel $\alpha:\mathcal{B}%
\rightarrow\mathcal{A}$ there exists a state transition $T:\mathcal{A}_{\ast
}\rightarrow\mathcal{B}_{\ast}$ such that $\left(  \alpha,T\right)  $ is a
dual pair. The converse can also be shown: for every state transition
$T:\mathcal{A}_{\ast}\rightarrow\mathcal{B}_{\ast}$ there exists a channel
$\alpha:\mathcal{B}\rightarrow\mathcal{A}$ such that $\left(  \alpha,T\right)
$ is a dual pair.

In the case $\mathcal{A=L}(\mathcal{H}_{A})$, $\mathcal{B=L}(\mathcal{H}_{B}%
)$, the duality (\ref{dual-channels-1}) for a given channel $\alpha$ and a
state transition (TP-CP map) $T$ writes as
\begin{equation}
\mathrm{Tr\,}\alpha\left(  b\right)  R=\mathrm{Tr\,}b\ T\left(  R\right)
,\text{ }b\in\mathcal{\mathcal{L}(\mathcal{H}_{\mathcal{B}})}\text{, }%
R\in\mathcal{L}^{1}(\mathcal{\mathcal{H}_{\mathcal{A}}}).\text{ }
\label{dual-channels-2}%
\end{equation}
In the case described in (\ref{channel-unitary}) where $\mathcal{A}%
=\mathcal{B=L}(\mathcal{H})$ one has
\[
T\left(  R\right)  =URU^{\ast}\text{, }R\in\mathcal{L}^{1}%
(\mathcal{\mathcal{H}_{\mathcal{A}}}).
\]
Consider now the case $\mathcal{A=}L^{\infty}\left(  \mu\right)  $,
$\mathcal{B=}L^{\infty}\left(  \nu\right)  $ where $\mu,\nu$ are a
sigma-finite measures on measurable spaces $\left(  X,\Omega_{X}\right)  $,
$\left(  Y,\Omega_{Y}\right)  $ respectively. Then a dual pair $\left(
\alpha,T\right)  $ fulfills
\begin{equation}
\int\alpha\left(  g\right)  fd\mu=\int gT\left(  f\right)  d\nu\text{, }g\in
L^{\infty}\left(  \nu\right)  \text{, }f\in L^{1}\left(  \mu\right)  .
\label{dual-channels-3}%
\end{equation}
This duality is described in Theorems 24.4 and 24.5 of \cite{MR812467}. Only
real function spaces and maps between them are considered, but then the
duality (\ref{dual-channels-3}) extends to the complex spaces and
corresponding maps. The equivalent terminology for a channel $\alpha
:L^{\infty}\left(  \nu\right)  \rightarrow L^{\infty}\left(  \mu\right)  $
there is \textit{Markov operator} (a linear, positive, unital and normal map)
and for a state transition $T:L^{\infty}\left(  \nu\right)  \rightarrow
L^{\infty}\left(  \mu\right)  $ it is \textit{stochastic operator} (a linear,
positive and $\left\Vert \cdot\right\Vert _{1}$-norm preserving map on positives).

Assume that $\Omega_{Y}$ is the Borel sigma-algebra of a Polish space $Y$ and
$\nu$ is a measure on $\left(  Y,\Omega_{Y}\right)  $. Then for every state
transition $T:L^{1}\left(  \mu\right)  \rightarrow L^{1}\left(  \nu\right)  $
there is a Markov kernel $K\left(  B,x\right)  $, $B\in\Omega_{Y}$, $x\in X$
such that
\begin{equation}
\int_{B}T\left(  f\right)  d\nu=\int K\left(  B,\cdot\right)  fd\mu\text{,
}B\in\Omega_{Y}\text{, }f\in L^{1}\left(  \mu\right)  \text{, }f\geq0.
\label{Markov-kernel-transition-exist-b}%
\end{equation}

holds (\cite{MR812467}, Remark 55.6(3), \cite{MR1425959}, Proposition 9.2).

\bigskip%
\begin{privatenotes}
\begin{boxedminipage}{\textwidth}%

\begin{sfblock}
1) The contraction property of state transitions has been shown to hold for
various special cases in our "informal appendix". We should add there the
(easy) proof for the general case of a map $T:\mathcal{A}_{\ast}%
\rightarrow\mathcal{B}_{\ast}$ between preduals, using \S 54 of the Conway
operator book.

2) References for the case of bounded / trace class operators
(\ref{dual-channels-2}) should be added.

3) We need to first indicate that every Markov kernel generates a state
transition. For that, our previous condition "$\kappa_{x}\ll\nu$ for $\mu
$-almost all $x$" seems to exclude deterministic estimators, as then
$\kappa_{x}$ is a point mass in $\hat{f}\left(  x\right)  $ if $\hat{f}$ is
the estimator. \textbf{Fundamentally flawed? Our text was:}

Consider a Markov kernel $K\left(  B,x\right)  $, $B\in\Omega_{Y}$, $x\in X$
such that the probability measure $\kappa_{x}\left(  \cdot\right)  =K\left(
\cdot,x\right)  $ satisfies $\kappa_{x}\ll\nu$ for $\mu$-almost all $x$. Then
$K$ defines a state transition (stochastic operator) $T:L^{1}\left(
\mu\right)  \rightarrow L^{1}\left(  \nu\right)  $ via
\begin{equation}
\int_{B}T\left(  f\right)  d\nu=\int K\left(  B,\cdot\right)  fd\mu\text{,
}B\in\Omega_{Y}\text{, }f\in L^{1}\left(  \mu\right)  \text{, }f\geq0.
\label{Markov-kernel-transition-exist-a}%
\end{equation}
Assume that $\Omega_{Y}$ is the Borel sigma-algebra of a Polish space $Y$ and
$\nu$ is a measure on $\left(  Y,\Omega_{Y}\right)  $. Then, conversely, for
every state transition $T:L^{1}\left(  \mu\right)  \rightarrow L^{1}\left(
\nu\right)  $ there is a Markov kernel $K$ as above such that
(\ref{Markov-kernel-transition-exist-a}) holds (\cite{MR812467}, Remark
55.6(3), \cite{MR1425959}, Proposition 9.2).

Check whether Strasser's Remark 55.6(3) indeed shows our claim above (he works
with bilinear forms).

4) Also check whether it is more elegant to first define the Markov operator
via $K$. Previous version we had:

Consider a channel (Markov operator) $\alpha:L^{\infty}\left(  \nu\right)
\rightarrow L^{\infty}\left(  \mu\right)  $ where $\nu$ is a Borel measure on
a Polish space $\left(  Y,\Omega_{Y}\right)  $. Then there is a Markov kernel
$K\left(  B,x\right)  $, $B\in\Omega_{Y}$, $x\in X$ such that $\alpha\left(
g\right)  \in L^{\infty}\left(  \mu\right)  $ is given by $\alpha\left(
g\right)  \left(  x\right)  =\int g\left(  y\right)  K\left(  dy,x\right)  $.
\end{sfblock}

Consider a state transition (stochastic operator) $T:L^{1}\left(  \mu\right)
\rightarrow L^{1}\left(  \nu\right)  $ where $\nu$ is a Borel measure on a
Polish space $\left(  Y,\Omega_{Y}\right)  $. Then there is a kernel function
$k:Y\times X\rightarrow\mathbb{R}_{+},$ measurable wrt $\Omega_{Y}\times
\Omega_{X}$ and fulfilling $\int k\left(  \cdot,x\right)  d\nu=1$ for all
$x\in X$ such that $T\left(  f\right)  \left(  y\right)  =\int k\left(
y,x\right)  f\left(  x\right)  d\nu\left(  x\right)  $.

Then there is a Markov kernel $K\left(  B,x\right)  $, $B\in\Omega_{Y}$, $x\in
X$ such that for $f\in L^{1}\left(  \mu\right)  $, one has $T\left(  f\right)
=\frac{d}{d\nu}\int K\left(  \cdot,x\right)  f\left(  x\right)  d\mu\left(
x\right)  $.%

\end{boxedminipage}
\end{privatenotes}%

\subsubsection{\textit{*-Homomorphisms}}

\textit{ }A bounded linear map $\alpha:\mathcal{B}\rightarrow\mathcal{A}$ is
called a *-homomorphism if for any $a,b\in\mathcal{B}$
\begin{align*}
\alpha\left(  ab\right)   &  =\alpha\left(  a\right)  \alpha\left(  b\right)
\\
\alpha\left(  a^{\ast}\right)   &  =\alpha\left(  a\right)  ^{\ast}%
\end{align*}
\textit{(}\cite{MR3468018}, 1.5.3). Such maps are completely positive
(\cite{MR3468018}, 5.4.2) and $\sigma$-weakly continuous (\cite{MR887100},
2.4.23), hence normal. Thus they are quantum channels; in our application,
$\mathcal{B}$ will represent a "smaller" quantum system compared to
$\mathcal{A}$, in the sense that $\mathcal{A=B}\otimes\mathcal{C}$ for a von
Neumann algebra $\mathcal{C}$ of linear operators on $\mathcal{H}_{C}$.
Setting $\alpha:\mathcal{B}\rightarrow\mathcal{A}$ as $\alpha\left(  b\right)
=b\otimes\mathbf{1}$ where $\mathbf{1}$ is the unit of $\mathcal{C}$, we
obtain a *-homomorphism. The corresponding state transition operates by
restricting a state $\rho$ on $\mathcal{A}$ to the subalgebra $\mathcal{B}%
\otimes\mathbf{1}$, isomorphic to $\mathcal{B}$ (the partial trace).

\subsubsection{\textit{Measurements and observation
channels\label{subsubsec:measmt-obs-channel}}}

\textit{ } A channel $\alpha:\mathcal{A}\rightarrow\mathcal{B}$ is said to be
an observation channel if $\mathcal{A}$ is commutative (\cite{MR1230389}, chap
8). Here we focus on the case where $\mathcal{A}$ is given by $L^{\infty
}\left(  \mu\right)  $ pertaining to a measurable space $\left(  X,\Omega
,\mu\right)  $ and $\mathcal{B=\mathcal{L}(\mathcal{H}_{\mathcal{B}})}$.
Observation channels arise from a positive operator valued measure (POVM) in
the following way. A POVM\ on $\left(  X,\Omega\right)  $ is a mapping
$M:\Omega\rightarrow\mathcal{\mathcal{L}(\mathcal{H}_{\mathcal{B}})}$ with
properties (i) $M\left(  A\right)  \geq0$, $A\in\Omega$ (hence $M\left(
A\right)  $ is self-adjoint), (ii) $M\left(  X\right)  =1$, (iii) if $\left\{
A_{j}\right\}  _{j=1}^{\infty}$ are pairwise disjoint set from $\Omega$ then
\[
M\left(
{\textstyle\bigcup\nolimits_{j=1}^{\infty}}
A_{j}\right)  =%
{\textstyle\sum\nolimits_{j=1}^{\infty}}
M\left(  A_{j}\right)
\]
where the r.h.s. is an SOT convergent sum. Then for any state $\rho
\in\mathcal{\mathcal{L}}^{1}\mathcal{(\mathcal{H}_{\mathcal{B}})}$,
\begin{equation}
\nu_{\rho}\left(  A\right)  =\mathrm{Tr\,}\rho M\left(  A\right)  \text{,
}A\in\Omega\label{prob-distrib-via-POVM}%
\end{equation}
is a probability measure on $\Omega$. This defines a state transition $T$ for
a certain measure $\nu_{0}$ on $\left(  X,\Omega\right)  $ in the following
way. Suppose that $\rho_{0}\in\mathcal{\mathcal{L}}^{1}\mathcal{(\mathcal{H}%
_{\mathcal{B}})}$ is a faithful state on $\mathcal{\mathcal{L}(\mathcal{H}%
_{\mathcal{B}})}$, i.e. $\rho_{0}>0$, and set $\nu_{0}=\nu_{\rho_{0}}$. Note
that such a $\rho_{0}$ exists if and only if $\mathcal{H}$ is separable
(\cite{MR887100}, 2.5.5). Then $\nu_{\rho}\ll\nu_{0}$ and
\begin{equation}
T\left(  \rho\right)  =\frac{d\nu_{\rho}}{d\nu_{0}} \label{transition by-POVM}%
\end{equation}
defines a transition $T:$ $\mathcal{\mathcal{L}}^{1}\mathcal{(\mathcal{H}%
_{\mathcal{B}})\rightarrow}L^{1}\left(  \nu_{0}\right)  $. Then the dual
$\alpha_{T}:L^{\infty}\left(  \nu_{0}\right)  \rightarrow\mathcal{\mathcal{L}%
(\mathcal{H}_{\mathcal{B}})}$ is an observation channel, satisfying for any
state $\rho\in\mathcal{\mathcal{L}}^{1}\mathcal{(\mathcal{H}_{\mathcal{B}})}$
\begin{equation}
\mathrm{Tr\,}\rho M\left(  A\right)  =\int_{A}T\left(  \rho\right)  d\nu
_{0}=\mathrm{Tr\,}\rho\alpha_{T}\left(  \mathbf{1}_{A}\right)  \text{, }%
A\in\Omega. \label{duality-POVM}%
\end{equation}
This in conjunction with (\ref{dual-channels-1}) shows that $M\left(
A\right)  =\alpha_{T}\left(  \mathbf{1}_{A}\right)  $, where $A\in\Omega$ and
$\mathbf{1}_{A}\in L^{\infty}\left(  \mu\right)  $ is the indicator function.

Conversely, let $\alpha:L^{\infty}\left(  \mu\right)  \rightarrow
\mathcal{\mathcal{L}(\mathcal{H}_{\mathcal{B}})}$ be an observation channel
for a sigma-finite $\mu$ on $\left(  X,\Omega\right)  $ and let $T_{\alpha
}:\mathcal{\mathcal{L}}^{1}\mathcal{(\mathcal{H}_{\mathcal{B}})\rightarrow
}L^{1}\left(  \mu\right)  $ be the dual channel (transition). Then there is a
POVM\ $M$ on $\left(  X,\Omega\right)  $ such that (\ref{duality-POVM}) holds
for $T=T_{\alpha}$ and any state $\rho\in\mathcal{\mathcal{L}}^{1}%
\mathcal{(\mathcal{H}_{\mathcal{B}})}$, and it follows that $M\left(
A\right)  =\alpha\left(  \mathbf{1}_{A}\right)  $, $A\in\Omega$ .

If $M\left(  A\right)  $, $A\in\Omega$ are projections then $M$ is called a
projection valued measure (PVM) or spectral measure.%

\begin{privatenotes}
\begin{boxedminipage}{\textwidth}%

\begin{sfblock}
In the Conway operator book, 9.1, the additional property $M\left(  A_{1}\cap
A_{2}\right)  =M\left(  A_{1}\right)  M\left(  A_{2}\right)  $ is required in
addition to (i)-(iii) and $M\left(  A\right)  $ being projections, for $M$ to
be a spectral measure; it is also required in Lax, 31.3, Theorem 9.
Parthasarathy on p. 23 pretends that the property follows from (i)-(iii) and
$M\left(  A\right)  $ being projections. We believe Partha is right, see
remark in the "informal appendix", Subsection "POVM\ and state transitions".
\end{sfblock}

%

\end{boxedminipage}
\end{privatenotes}%

\subsubsection{\textit{Real and vector valued observables}}

\textit{ }Consider a self-adjoint operator $S$ on $\mathcal{\mathcal{H}}$ ,
possibly unbounded and densely defined. By the spectral theorem there is a PVM
$M$ on $\left(  \mathbb{R},\mathfrak{B}_{\mathbb{R}}\right)  $ ($\mathfrak{B}%
_{\mathbb{R}}$ being the Borel $\sigma$-algebra) such that
\[
Sx=\int_{\mathbb{R}}tdM(t)x
\]
for all $x$ in the domain of $S$, i.e. all $x\in\mathcal{\mathcal{H}}$
satisfying $\int t^{2}d\left\langle x,M(t)x\right\rangle <\infty$, with
$\left\langle \cdot,\cdot\right\rangle $ being the inner product of
$\mathcal{\mathcal{H}}$ (\cite{MR1892228}, 32.1). The operator $S$ is bounded
if and only if $M$ is concentrated on a bounded set in $\mathbb{R}$. Consider
the state transition $T_{M}:\mathcal{L}^{1}\left(  \mathcal{H}\right)
\rightarrow L^{1}\left(  \nu_{0}\right)  $ given by the PVM $M$ according to
(\ref{transition by-POVM}); its dual $\alpha_{M}:L^{\infty}\left(  \nu
_{0}\right)  \rightarrow\mathcal{\mathcal{L}(\mathcal{H})}$ is an observation
channel. For a given state $\rho\in\mathcal{L}^{1}(\mathcal{H})$, application
of $T_{M}$ produces a probability density $T_{M}\left(  \rho\right)  \in$
$L^{1}\left(  \nu_{0}\right)  $. If $M$ is absolutely continuous w.r.t.
Lebesgue measure $\lambda$ (i.e. $S$ has absolutely continuous spectrum,
\cite{MR1892228}, 31.4) then the measure $\nu_{0}$ in
(\ref{transition by-POVM}) is absolutely continuous, and setting $p_{0}%
=d\nu_{0}/d\lambda$ for Lebesgue measure $\lambda$ on $\mathbb{R}$, for a
given state $\rho\in\mathcal{L}^{1}(\mathcal{H})$, the transition
$\rho\rightarrow T_{M}\left(  \rho\right)  p_{0}\in L^{1}\left(
\lambda\right)  $ produces a Lebesgue density on $\left(  \mathbb{R}%
,\mathfrak{B}_{\mathbb{R}}\right)  $. If $M$ is concentrated on a discrete set
$D\subset\mathbb{R}$ (i.e. $S$ has point spectrum), $\kappa$ is counting
measure on $D$ and $p_{0}=d\nu_{0}/d\kappa$, then analogously $T_{M}\left(
\rho\right)  p_{0}$ is a density w.r.t. counting measure on $D$, i.e. gives a
discrete distribution.

The random variable having distribution given by the $\nu_{0}$-density
$T_{M}(\rho)$ is commonly identified in notation with the operator $S$. If $S$
is bounded then, with $D$ being the support of $M$, applying the basic duality
(\ref{dual-channels-1}) with an function $g\left(  t\right)  =t\mathbf{1}%
_{D}\left(  t\right)  $, $t\in\mathbb{R}$, such that $g\in L^{\infty}\left(
\nu_{0}\right)  $
\begin{equation}
E_{\rho}S=\int_{D}tT_{M}(\rho)\left(  t\right)  \nu_{0}\left(  dt\right)
=\left\langle g,T_{M}(\rho)\right\rangle =\left\langle \alpha_{M}%
(g),\rho\right\rangle . \label{trace-rule-pre}%
\end{equation}
From (\ref{duality-POVM}) which holds for all $A\in\Omega$ it can be seen that
in the case of a spectral measure $M$, the channel $\alpha_{M}\left(
f\right)  $ for $f\in L^{\infty}\left(  \nu_{0}\right)  $ acts as
\[
\alpha_{M}\left(  f\right)  =\int f\left(  t\right)  dM(t)
\]
so that from (\ref{trace-rule-pre}) we obtain
\begin{equation}
E_{\rho}S=\mathrm{Tr\,}\left(  \int_{D}tdM(t)\right)  \rho=\mathrm{Tr\,}%
S\rho\label{trace-rule-3}%
\end{equation}
giving the basic trace rule for expectation of bounded observables. If the
operator $S$ is unbounded but the density $T_{M}(\rho)$ has an expectation
then the trace rule $E_{\rho}S=\mathrm{Tr\,}S\rho$ extends from
(\ref{trace-rule-3}) through an approximation of $S$ by bounded operators
$\int_{B}tdM(t)$ for bounded $B\subset\mathbb{R}$.

Let $S_{i},$ $i=1,2$ be self-adjoint operators on $\mathcal{H}$ with
respective spectral measures $M_{i}$, $i=1,2$ on $\left(  \mathbb{R}%
,\mathfrak{B}_{\mathbb{R}}\right)  $. The operators $S_{i}$ commute
($S_{1}S_{2}=S_{2}S_{1}$) if and only if the respective spectral measures
commute, i.e. if $M_{1}(A_{1})M_{2}(A_{2})=M_{2}(A_{2})M_{1}(A_{1})$ for all
Borel sets $A_{i}\in\mathfrak{B}_{\mathbb{R}}$ (\cite{MR1721402}, Theorem
10.2). Then all operators $M_{1}(A_{1})M_{2}(A_{2})$ are projections in
$\mathcal{H}$, and setting for cylinder sets $A_{1}\times A_{2}\subset
\mathbb{R}^{2}$
\[
M\left(  A_{1}\times A_{2}\right)  :=M_{1}(A_{1})M_{2}(A_{2}),
\]
by extension to $\mathfrak{B}_{\mathbb{R}^{2}}$ one defines a PVM\ $M$ on
$\left(  \mathbb{R}^{2},\mathfrak{B}_{\mathbb{R}^{2}}\right)  $
(\cite{MR3012668}, 10.9). For a given state $\rho\in\mathcal{L}^{1}%
(\mathcal{H})$, the commuting operators $S_{i}$ give a bivariate probability
distribution $\nu_{\rho}$ on $\left(  \mathbb{R}^{2},\mathfrak{B}%
_{\mathbb{R}^{2}}\right)  $ by
\begin{equation}
\nu_{\rho}\left(  A_{1}\times A_{2}\right)  =\mathrm{Tr\,}\rho M(A_{1}\times
A_{2})\text{, }A_{i}\in\mathfrak{B}_{\mathbb{R}},i=1,2
\label{prob-distrib-via-POVM-2}%
\end{equation}
in accordance with (\ref{prob-distrib-via-POVM}). Its marginal distributions
are those given by the operators $S_{i}$. Therefore, if self-adjoint operators
are to be identified in notation with the corresponding random variables, then
(\ref{prob-distrib-via-POVM-2}) describes a bivariate random variable $\left(
S_{1},S_{2}\right)  $.

Consider now the Fock space $\mathfrak{F}\left(  \mathcal{H}\right)  $ where
$\mathcal{H}$ is a direct sum $\mathcal{H=H}_{1}\oplus\mathcal{H}_{2}$. In
this case $\mathfrak{F}\left(  \mathcal{H}\right)  $ is unitarily isomorphic
to $\mathfrak{F}\left(  \mathcal{H}_{1}\right)  \otimes\mathfrak{F}\left(
\mathcal{H}_{2}\right)  $ (\cite{MR3012668}, 19.6). Suppose that $\tilde
{S}_{i}$ are self-adjoint operators on $\mathfrak{F}\left(  \mathcal{H}%
_{i}\right)  $ with respective spectral measures $M_{i}$, $i=1,2$, and let
$\mathbf{1}_{i}$ be the unit operators on $\mathfrak{F}\left(  \mathcal{H}%
_{i}\right)  $. Then $S_{1}:=\tilde{S}_{1}\otimes\mathbf{1}_{2}$,
$S_{2}:=\mathbf{1}_{1}\otimes\tilde{S}_{2}$ commute on $\mathcal{F}\left(
\mathcal{H}\right)  $ and thus generate a bivariate random variable, with
marginal distributions those generated by $\tilde{S}_{i}$. If $M_{i}$ are the
respective spectral measures for $\tilde{S}_{i}$ then the PVM $M$ on $\left(
\mathbb{R}^{2},\mathfrak{B}_{\mathbb{R}^{2}}\right)  $ generating the joint
distribution is
\[
M(A_{1}\times A_{2})=M_{1}(A_{1})\otimes M_{2}(A_{2})\text{, }A_{i}%
\in\mathfrak{B}_{\mathbb{R}},i=1,2.
\]
In this paper, $\mathcal{H=}\mathbb{C}^{n}$ such that $\mathfrak{F}\left(
\mathcal{H}\right)  $ is identified with $\mathfrak{F}\left(  \mathbb{C}%
\right)  ^{\otimes n}$. Let $\tilde{Q},\tilde{P}$ be the pair of canonical
observables in $\mathfrak{F}\left(  \mathbb{C}\right)  $ and let $Q_{i},P_{i}$
be their extension to the whole of $\mathfrak{F}\left(  \mathbb{C}\right)
^{\otimes n}$ such that
\begin{equation}
Q_{i}=\mathbf{1}^{\otimes\left(  i-1\right)  }\otimes\tilde{Q}\otimes
\mathbf{1}^{\otimes\left(  n-i\right)  } \label{extension-to-tens-prod}%
\end{equation}
where $\mathbf{1}$ is the unit operator on $\mathfrak{F}\left(  \mathbb{C}%
\right)  $, and analogously for $P_{i}$. Then any subset of $\left\{
Q_{i},P_{i}\text{, }i=1,\ldots,n\right\}  $ which does not contain a pair
$\left\{  Q_{j},P_{j}\right\}  $ is a commuting set, and under a state
$\mathfrak{N}_{n}\left(  0,A\right)  $ (cf.
\ref{symbol-related-to-covmatrix-2}) the corresponding joint distribution is
Gaussian. Let
\[
\tilde{N}=\frac{1}{2}\left(  \tilde{Q}^{2}+\tilde{P}^{2}-\mathbf{1}\right)
\]
be the number operator on $\mathfrak{F}\left(  \mathbb{C}\right)  $ and let
$N_{i}$ be its extension to $\mathfrak{F}\left(  \mathbb{C}\right)  ^{\otimes
n}$ in analogy to (\ref{extension-to-tens-prod}), for $i=1,\ldots,n$. Then
$\left\{  N_{i}\text{, }i=1,\ldots,n\right\}  $ is a commuting set, and under
a state $\mathfrak{N}_{n}\left(  0,A\right)  $ the corresponding joint
distribution is discrete (concentrated on $\mathbb{Z}_{+}^{n}$) with Geometric marginals.

\subsubsection{\textit{Quantum statistical
experiments\label{subsubsec-qu-statist-experiments}}}%

\begin{privatenotes}
\begin{boxedminipage}{\textwidth}%

\begin{sfblock}
The motivation for this section is the following. In Le Cam \cite{MR856411}, a
transition between experiments (families of probability measures)
$\mathcal{E}_{i}=\left(  P_{i,\theta},\theta\in\Theta\right)  $ on $\left(
X_{i},\Omega_{i}\right)  $, $i=1,2$ has been defined as a linear, positive and
norm preserving (on positives) map $T$ between the two respective L-spaces.
The L-space $\mathcal{E}_{i}$ is defined as the linear space spanned by
$\left(  P_{i,\theta},\theta\in\Theta\right)  $ (I skip some technicalities
here); if $\mathcal{E}_{i}$ is a dominated family, the L-space can be
equivalently defined as in \cite{MR812467}, 24.6: if $Q_{i}$ is a probability
measure which is equivalent to the whole family, i.e. $P_{i,\theta}\ll Q_{i}$
and
\begin{equation}
\sup_{\theta\in\Theta}P_{i,\theta}\left(  B\right)  =0\text{ implies }%
Q_{i}\left(  B\right)  =0\text{, }B\in\Omega_{i} \label{domination-reverse}%
\end{equation}
(which is often written as an equivalence $Q_{i}\sim\mathcal{E}_{i}$) then the
L-space is $L^{1}\left(  Q_{i}\right)  $. The dual of $L^{1}\left(
Q_{i}\right)  $ is then $L^{\infty}\left(  Q_{i}\right)  $ which is called the
M-space of $\mathcal{E}_{i}$ (\cite{MR812467}, 24.8).

The authors which we have followed in the definition of the quantum $\Delta
$-distance (mainly \cite{MR2506764}) considered quantum experiments
$\mathcal{E}_{i}=\left(  \mathcal{A}_{i},\tau_{i,\theta},\theta\in
\Theta\right)  $, $i=1,2$ with $\mathcal{A}_{i}$ von Neumann algebras and
defined a transition between the experiments as a norm preserving completely
positive map $T:\mathcal{E}_{1}\rightarrow\mathcal{E}_{2}$ (TP-CP map, where
TP stands for "trace preserving", a special case). This does not conform to Le
Cam's original definition for the classical case: if $\mathcal{E}_{i}=\left(
L^{\infty}\left(  \mu_{i}\right)  ,p_{i,\theta},\theta\in\Theta\right)  $,
$i=1,2$, where $p_{i,\theta}$ are $\mu_{i}$-densities, then the Kahn-Guta
definition allows the measures $\mu_{i}$ to be NOT equivalent to the whole
family $P_{i,\theta}$ generated by the densities $p_{i,\theta}$. Of course we
have $P_{i,\theta}\ll\mu_{i}$ but not necessarily $\mu_{i}\sim\left(
P_{i,\theta},\theta\in\Theta\right)  $ (i. e. (\ref{domination-reverse}) with
$Q_{i}=\mu_{i}$ may not hold). This means that the space $L^{1}\left(  \mu
_{i}\right)  $ which contains $p_{i,\theta}$ is not the L-space of the
experiment but may be larger. For instance, if $\mu_{i}$ is a measure on
$\left(  X,\Omega\right)  $ then all $p_{i,\theta}$ may be concentrated on a
subset $B\subset X$ but $\mu_{i}$ may have support on the whole of $X$. Then
any prob. measure $Q_{i}\sim\left(  P_{i,\theta},\theta\in\Theta\right)  $ is
concentrated on $B$ as well and $Q_{i}\ll\mu_{i}$ but $\mu_{i}\ll Q_{i}$ does
not hold, so that $L^{1}\left(  Q_{i}\right)  \subset L^{1}\left(  \mu
_{i}\right)  $ and for the M-spaces we have $L^{\infty}\left(  \mu_{i}\right)
\subset L^{\infty}\left(  Q_{i}\right)  $. Then the Kahn-Guta definition of a
transition requires more than Le Cam's: there must be a positive norm
preserving linear map $T:L^{1}\left(  \mu_{1}\right)  \rightarrow L^{1}\left(
\mu_{2}\right)  $ whereas Le Cam requires it only for the narrower spaces
$L^{1}\left(  Q_{i}\right)  $. If there is a Kahn-Guta transition
$T:L^{1}\left(  \mu_{1}\right)  \rightarrow L^{1}\left(  \mu_{2}\right)  $
then by restricting it we also can get one in the sense of Le Cam, but the
converse is unclear (it's unlikely, but I haven't checked).

Below we introduce a restriction on quantum experiments $\mathcal{E}%
_{i}=\left(  \mathcal{A}_{i},\tau_{i,\theta},\theta\in\Theta\right)  $ (being
in "reduced form") which generalizes the condition that in the classical
experiment $\mathcal{E}=\left(  L^{\infty}\left(  \mu\right)  ,p_{\theta
},\theta\in\Theta\right)  $, the space $L^{1}\left(  \mu\right)  $ is the
L-space of $\mathcal{E}$ (and hence $L^{\infty}\left(  \mu\right)  $ the
M-space). So then our definition of the quantum $\Delta$-distance is a strict
generalization of Le Cam's. In the paper \cite{MR2506764} this condition is
fulfilled and so Kahn, Guta don't run into trouble, but it was not
acknowledged there.

As a by-product, we obtain the possibility of simplifying notation, by being
able to omit the von Neumann algebra $\mathcal{A}_{i}$ (or $L^{\infty}\left(
\mu\right)  $, resp.): it's either $\mathcal{L}\left(  \mathcal{H}\right)  $,
the full algebra of bounded operators on the Fock space, or $L^{\infty}\left(
Q\right)  $ where $Q$ is a prob. measure with $Q\sim\left(  P_{\theta}%
,\theta\in\Theta\right)  $.
\end{sfblock}

%

\end{boxedminipage}
\end{privatenotes}%

A \textit{quantum statistical experiment} is a family of normal states
$\mathcal{E}=\left(  \mathcal{A},\tau_{\theta},\theta\in\Theta\right)  $ on a
von Neumann algebra $\mathcal{A}$. The experiment $\mathcal{E}$ is said to be
\textit{dominated} if there exists a normal state
\begin{equation}
\omega=\sum_{n=1}^{\infty}\lambda_{n}\tau_{n} \label{dominating_state}%
\end{equation}
with $\tau_{n}\in\mathcal{E}$, $\lambda_{n}\geq0$, $\sum_{n=1}^{\infty}%
\lambda_{n}=1$ such that
\begin{equation}
\mathrm{supp\,}\tau_{\theta}\leq\mathrm{supp\,}\omega\text{ for all }\theta
\in\Theta\label{dominating-state-2}%
\end{equation}
where $\mathrm{supp\,}\omega$ is the support projection of $\omega$. If the
von Neumann algebra $\mathcal{A}$ admits a faithful normal state then every
experiment $\mathcal{E}$ on $\mathcal{A}$ is dominated (\cite{MR2207329},
Lemma 2), and for $\mathcal{A=L}\left(  \mathcal{H}\right)  $, $\mathcal{H}$
separable this is the case. An experiment $\mathcal{E}=\left(  \mathcal{A}%
,\tau_{\theta},\theta\in\Theta\right)  $ is said to be in \textit{reduced
form} if it is dominated and any dominating state $\omega$ fulfilling
(\ref{dominating_state}) and (\ref{dominating-state-2}) is faithful.
$\mathcal{E}$ is said to be \textit{homogeneous} if \textrm{supp}%
$\mathrm{\,}\tau_{\theta_{1}}\leq$\textrm{supp}$\mathrm{\,}\tau_{\theta_{2}}$
for all $\theta_{1},\theta_{2}\in\Theta$. If every $\tau_{\theta},\theta
\in\Theta$ is faithful ($\mathrm{supp\,}\tau_{\theta}=\boldsymbol{1}$) then
$\mathcal{E}$ is homogeneous and in reduced form.

We note that the Fock space $\mathfrak{F}\left(  \mathbb{C}^{n}\right)  $ is
separable since the exponential vectors $x_{F}$, $x\in\mathbb{C}^{n}$ (cf.
(\ref{exponential-vectors})) are dense in $\mathfrak{F}\left(  \mathbb{C}%
^{n}\right)  $. For a separable Hilbert space $\mathcal{H}$, a state on the
von Neumann algebra $\mathcal{L}\left(  \mathcal{H}\right)  $ is faithful if
the density operator is strictly positive. The Gaussian states $\mathfrak{N}%
\left(  0,A\right)  $ on $\mathcal{L}\left(  \mathfrak{F}\left(
\mathbb{C}^{n}\right)  \right)  $ have density operator (\ref{Fock-repre-1});
Lemma \ref{Lem-spec-decompos-Fock-op} below then shows that if the Hermitian
$n\times n$ matrix $A-I$ is strictly positive then $\mathfrak{N}\left(
0,A\right)  $ is faithful. In Theorem \ref{theor-main-1} we consider the
quantum experiment%
\[
\mathcal{E}_{n}\left(  \Theta\right)  =\left(  \mathcal{L}\left(
\mathfrak{F}\left(  \mathbb{C}^{n}\right)  \right)  ,\mathfrak{N}\left(
0,A_{n}\left(  a\right)  \right)  ,\,a\in\Theta\right)
\]
for $\Theta=\Theta_{1}\left(  \alpha,M\right)  $ given by
(\ref{Theta-1-functionset-def}), (\ref{lowerbound-set-def}). Here
(\ref{lowerbound-set-def}) and Lemma \ref{lem-toeplitz-EV} guarantee that
$A_{n}\left(  a\right)  -I>0$ for $a\in\Theta$, hence $\mathcal{E}_{n}\left(
\Theta\right)  $ is homogeneous and in reduced form. The latter also applies
to all Gaussian quantum experiments $\mathcal{E}_{n}\left(  \Theta\right)  $
occurring in this paper with modified $\Theta$. When $\mathcal{E}=\left(
\mathcal{A},\tau_{\theta},\theta\in\Theta\right)  $ is such that
$\mathcal{A=L}\left(  \mathfrak{F}\left(  \mathbb{C}^{n}\right)  \right)  $
and $\mathcal{E}$ is in reduced form, we will omit $\mathcal{A}$ from notation
and simply write $\mathcal{E}$ as a family of density operators $\tau_{\theta
}\in\mathcal{L}^{1}\left(  \mathfrak{F}\left(  \mathbb{C}^{n}\right)  \right)
$.

Consider now the commutative case where $\mathcal{A}=L^{\infty}\left(
\mu\right)  $ on a $\sigma$-finite measure space $\left(  X,\Omega,\mu\right)
$, construed as an algebra of linear operators acting on $\mathcal{H}%
=L^{2}\left(  \mu\right)  $ by pointwise multiplication. Here every
$\tau_{\theta},\theta\in\Theta$ can be identified with a probability density
$p_{\theta}\in\mathcal{A}_{\ast}=L^{1}\left(  \mu\right)  $, and for $\phi
\in\mathcal{A}$ we have (cp. (\ref{duality-predual-2}))
\[
\tau_{\theta}\left(  \phi\right)  =\int\phi p_{\theta}d\mu.
\]
The set of probability measures $\mathcal{P}=\left(  P_{\theta}:dP_{\theta
}/d\mu=p_{\theta},\theta\in\Theta\right)  $ is then dominated by the measure
$\mu$ ($P_{\theta}\ll\mu$, $\theta\in\Theta$). By the Halmos-Savage Theorem
(\cite{MR812467}, 20.3) there exists a probability measure%
\begin{equation}
Q=\sum_{n=1}^{\infty}\lambda_{n}P_{n} \label{equiv-measure-1}%
\end{equation}
with $P_{n}\in\mathcal{P}$, $\lambda_{n}\geq0$, $\sum_{n=1}^{\infty}%
\lambda_{n}=1$ such that%
\begin{equation}
P_{\theta}\ll Q\text{, }\theta\in\Theta. \label{equiv-measure-2}%
\end{equation}
Then (\ref{equiv-measure-1}) and (\ref{equiv-measure-2}) imply that for every
$B\in\Omega$%
\[
Q\left(  B\right)  =0\Longleftrightarrow P_{\theta}\left(  B\right)  =0\text{
for all }\theta\in\Theta.
\]
The latter relation is also written $\mathcal{P}\sim Q$. Set $q=dQ/d\mu$; then
(\ref{equiv-measure-1}) can be written as
\begin{equation}
q=\sum_{n=1}^{\infty}\lambda_{n}p_{n} \label{dom-state-3}%
\end{equation}
where $p_{n}=dP_{n}/d\mu$. For a function $f\in L^{1}\left(  \mu\right)  $,
let $\mathrm{supp\,}f$ be the support projection in $L^{\infty}\left(
\mu\right)  $: if $f_{0}$ is a function in the $\mu$-equivalence class $f$
then $\mathrm{supp\,}f$ is the $\mu$-equivalence class of \textbf{
}$\mathbf{1}\left\{  f_{0}\left(  x\right)  \neq0\right\}  $ (\cite{MR1721402}%
, 54.5). Then (\ref{equiv-measure-2}) is equivalent to
\begin{equation}
\mathrm{supp\,}p_{\theta}\leq\mathrm{supp\,}q\text{ for all }\theta\in
\Theta\label{dom-state-4}%
\end{equation}
so that (\ref{dom-state-3}), (\ref{dom-state-4}) are the versions of
(\ref{dominating_state}), (\ref{dominating-state-2}) for the quantum
experiment $\mathcal{E}=\left(  L^{\infty}\left(  \mu\right)  ,p_{\theta
},\theta\in\Theta\right)  $.

\bigskip Consider now an arbitrary family of probability measures
$\mathcal{P}=\left(  P_{\theta},\theta\in\Theta\right)  $ on $\left(
X,\Omega\right)  $ dominated by sigma-finite measure $\mu$, i.e. a dominated
classical statistical experiment. The above reasoning shows that there exists
a probability measure $Q$ of form (\ref{equiv-measure-1}) with $\mathcal{P}%
\sim Q$. Then $L^{\infty}\left(  Q\right)  $, $L^{1}\left(  Q\right)  $ are an
M-space and an L-space of $\mathcal{P}$, respectively (\cite{MR812467}, 24.6,
24.8). The choice of $Q$ is not unique, but all $L^{\infty}\left(  Q\right)  $
are isometrically isomorphic Banach spaces, and the same holds for
$L^{1}\left(  Q\right)  $. Moreover all the $L^{\infty}\left(  Q\right)  $
with $\mathcal{P}\sim Q$ are isomorphic as von Neumann algebras. Thus
$\mathcal{P}$ can be identified in a canonical way with a quantum experiment
$\mathcal{E}_{\mathcal{P},Q}=\left(  L^{\infty}\left(  Q\right)  ,dP_{\theta
}/dQ,\theta\in\Theta\right)  $. Here $\mathcal{E}_{\mathcal{P},\mathcal{Q}}$
is in reduced form since $1=dQ/dQ\in L^{1}\left(  Q\right)  $ is a faithful
state on $L^{\infty}\left(  Q\right)  $. The condition that a quantum
experiment $\mathcal{E}=\left(  \mathcal{A},\tau_{\theta},\theta\in
\Theta\right)  $ be in reduced form thus generalizes the condition that if a
classical dominated family $\mathcal{P}$ is represented as $\left(  L^{\infty
}\left(  \mu\right)  ,dP_{\theta}/d\mu,\,\theta\in\Theta\right)  $, the space
$L^{\infty}\left(  \mu\right)  $ is an $M$-space of $\mathcal{P}$.

We note that for different $Q$, all quantum experiments $\mathcal{E}%
_{\mathcal{P},Q}$ are statistically equivalent in the sense of the quantum Le
Cam distance (\ref{Q-Delta-distance}). All classical experiments occurring in
this paper are dominated, and the simplifying notation $\mathcal{P}=\left(
P_{\theta},\theta\in\Theta\right)  $ will be used to denote one of the
(statistically equivalent) quantum experiments $\mathcal{E}_{\mathcal{P},Q}$.%

\begin{privatenotes}%

\subsubsection{\textit{Gaussian states}}

In Section \ref{Sec: Intro and main results} $n$-mode Gaussian states were
introduced by their density operators $\rho$ on the Fock space $\mathfrak{F}%
\left(  \mathbb{C}^{n}\right)  $, via the characteristic function
(\ref{char-function-general}) involving the Weyl unitaries $W\left(  x\right)
$, $x\in\mathbb{C}^{n}$. For the present framework of quantum statistical
experiments and channels, we need to exhibit the pertaining von Neumann
algebra. For that, note that for any Hilbert space $\mathcal{H}$, the smallest
von Neumann algebra containing the Weyl unitaries $W\left(  x\right)  $,
$x\in\mathcal{H}$ is $\mathcal{L}(\mathfrak{F}\left(  \mathcal{H}\right)  )$,
the set of bounded operators on $\mathfrak{F}\left(  \mathcal{H}\right)  $
(\cite{MR3012668}, 22.10). This will be different from the smallest $C^{\ast}%
$-algebra containing the Weyl unitaries on $\mathfrak{F}\left(  \mathcal{H}%
\right)  $, commonly denoted by $CCR\left(  \mathcal{H}\right)  $. The
CCR-algebras are pervasively used in the literature in connection with
Gaussian states \cite{MR1230389}, \cite{MR2207329}, \cite{MR2346393},
\cite{MR2510896}; we will use an analogous notation, setting
\[
CCR_{W}\left(  \mathcal{H}\right)  :=\mathcal{L}(\mathfrak{F}\left(
\mathcal{H}\right)  ).
\]
Being $C^{\ast}$-algebras as well, the $CCR_{W}$-algebras inherit all relevant
properties of $CCR\left(  \mathcal{H}\right)  $. For instance, let
$\mathcal{H}_{A}$ a finite dimensional complex Hilbert space which is a direct
sum $\mathcal{H}_{A}=\mathcal{H}_{B}\oplus\mathcal{H}_{C}$.

\texttt{Attention: the paper by H\"{o}rmann \cite{MR4224167} seems to give the
solution for our problem: }$CCR\left(  \mathcal{H}\right)  $\texttt{ is not
von Neumann, but if }$\mathcal{H}$ \texttt{is fidi then the smallest vN
algebra containig it is }$\mathcal{B}\left(  \mathcal{H}\right)  $,
\texttt{i.e. the bounded operators. Hence avoid the "new" notation }%
$CCR_{W}\left(  \mathcal{H}\right)  $\texttt{, just write }$\mathcal{B}\left(
\mathcal{H}\right)  $.\texttt{ }

\texttt{iOther remark: a paper by Mosonyi (Fermionic) or the preceding
Diercks, Fannes et al. it is claimed that for fidi }$H$\texttt{, }$CAR\left(
\mathcal{H}\right)  $\texttt{ is already the set of bounded operators on
}$\mathfrak{F}\left(  \mathcal{H}\right)  ,$\texttt{ i.e a von Neumann
algebra. Does this hold also for }$CCR\left(  \mathcal{H}\right)  $\texttt{
?(No, probably not, since }$CAR\left(  \mathcal{H}\right)  $ \texttt{is fidi
if }$\mathcal{H}$\texttt{ is). Then we wouldn't need the distinction between
}$CCR_{W}\left(  \mathcal{H}\right)  $\texttt{ and }$CCR\left(  \mathcal{H}%
\right)  \mathtt{,}$\texttt{but unlikely.}

\bigskip%

\end{privatenotes}%

\subsection{Further facts about Gaussian states}

\subsubsection{Partial trace\label{subsec:partial-trace}}

In \cite{MR2589320}, for the treatment of a classical stationary Gaussian time
series $X_{1},\ldots,X_{n}$, an essential step of reasoning has been to
consider a series where some observations are omitted, say $X_{m+1}%
,\ldots,X_{n}$, $m<n$ and make the obvious claim that the reduced series is
"less informative" than the original. In the framework of Le Cam theory, this
means that there exists a transition (Markov kernel) mapping the law
$\mathcal{L}\left(  X_{1},\ldots,X_{n}\right)  $ into its marginal law
$\mathcal{L}\left(  X_{1},\ldots,X_{m}\right)  $. For a Gaussian, zero mean
time series, we then know that $X_{1},\ldots,X_{m}$ is again Gaussian
centered, and the covariance matrix is just the pertaining submatrix. We now
set out to describe the analog of this reasoning for a quantum Gaussian time series.

We will consider centered gauge invariant Gaussian states $\mathfrak{N}%
_{n}\left(  0,A\right)  $ with $n\times n$ symbol matrix $A$, given by
characteristic function (\ref{char-function-Gaussian}). Assume for some $m<n$
we consider $\mathfrak{N}_{m}\left(  0,A_{(m)}\right)  $ where $A_{(m)}$ is
the upper $m\times m$ central submatrix of $A$. Is there a quantum channel,
mapping $\mathfrak{N}_{n}\left(  0,A\right)  $ into $\mathfrak{N}_{m}\left(
0,A_{(m)}\right)  $ for all (permissible) symbol matrices $A$ ?

Let again $\mathcal{H}_{A}$ be a finite dimensional complex Hilbert space
which is a direct sum $\mathcal{H}_{A}=\mathcal{H}_{B}\oplus\mathcal{H}_{C}$.
The Fock space $\mathfrak{F}\left(  \mathcal{H}_{A}\right)  $ is unitarily
isomorphic to $\mathfrak{F}\left(  \mathcal{H}_{B}\right)  \otimes
\mathfrak{F}\left(  \mathcal{H}_{C}\right)  $ (\cite{MR3012668}, 19.6) and the
respective Weyl operators $W\left(  \cdot\right)  $ satisfy
\begin{equation}
W\left(  u_{1}\oplus u_{2}\right)  =W\left(  u_{1}\right)  \otimes W\left(
u_{2}\right)  \text{ for }u_{1}\in\mathcal{H}_{B},u_{2}\in\mathcal{H}_{C}.
\label{functorial-1}%
\end{equation}
(\cite{MR3012668}, 20.21).

\bigskip%
\begin{privatenotes}
\begin{boxedminipage}{\textwidth}%

\begin{sfblock}
Here we are using notation $CCR_{W}\left(  \mathcal{H}_{A}\right)  ,$
introduced in the above suppressed subsection "Gaussian States", which means
just $\mathcal{L}(\mathfrak{F}\left(  \mathcal{H}\right)  )$, the set of
bounded operators. Eliminate notation $CCR_{W}\left(  \mathcal{H}_{A}\right)
$ also below.
\end{sfblock}

%

\end{boxedminipage}
\end{privatenotes}%

This means that
\begin{equation}
CCR_{W}\left(  \mathcal{H}_{A}\right)  \simeq CCR_{W}\left(  \mathcal{H}%
_{B}\right)  \otimes CCR_{W}\left(  \mathcal{H}_{C}\right)
\label{functorial-2}%
\end{equation}
in the sense of a W*-isomorphism (\cite{MR1741419}, VI.6.9). In view of
(\ref{functorial-1}), (\ref{functorial-2}), we can describe the quantum
channel realizing the restriction of a Gaussian state on a system $A$ to a
subsystem $B$\textit{: }it is $\alpha:CCR\left(  \mathcal{H}_{B}\right)
\rightarrow CCR\left(  \mathcal{H}_{A}\right)  $ given by
\begin{equation}
\alpha\left(  W\left(  u\right)  \right)  =W\left(  u\right)  \otimes
\mathbf{1=}W\left(  u\oplus0\right)  \text{, }u\in\mathcal{H}_{B}.
\label{alpha-map-homo}%
\end{equation}
It remains to show that for $\mathcal{H}_{A}=\mathbb{C}^{n}$, $\mathcal{H}%
_{B}=\mathbb{C}^{m}$ we have%
\begin{equation}
\mathfrak{N}_{n}\left(  0,A\right)  \circ\alpha=\mathfrak{N}_{m}\left(
0,A_{(m)}\right)  \label{claim-channel-omit-observ}%
\end{equation}
for all Hermitian $A\geq I$. To this end we compute the characteristic
function (\ref{char-function-Gaussian}).

Let $W\left(  u\right)  \in CCR\left(  \mathcal{H}_{B}\right)  $ be a Weyl
unitary with $u\in\mathbb{C}^{m}$; then for $\rho=\mathfrak{N}_{n}\left(
0,A\right)  $ according to (\ref{char-function-general}) we have
\begin{align*}
\hat{W}\left[  \mathfrak{\rho}\circ\alpha\right]  \left(  u\right)   &
:=\left(  \mathfrak{\rho}\circ\alpha\right)  \left(  W\left(  u\right)
\right)  =\rho\left(  \alpha\left(  W\left(  u\right)  \right)  \right) \\
&  =\rho\left(  W\left(  u\oplus0\right)  \right)  =\exp\left(  -\frac{1}%
{4}\left\langle \left(  u\oplus0\right)  ,A\left(  u\oplus0\right)
\right\rangle \right) \\
&  =\exp\left(  -\frac{1}{4}\left\langle u,A_{(m)}u\right\rangle \right)
\end{align*}
which confirms (\ref{claim-channel-omit-observ}).

\subsubsection{The density operator under gauge
invariance\label{subsec:density-op-GIV-Gaussians}}

\begin{lemma}
Consider the gauge invariant centered $n$-mode Gausian state $\mathfrak{N}%
_{n}\left(  0,A\right)  $ with symbol $A$, where $A$ is a complex Hermitian
$n\times n$ matrix fulfilling $A\geq I$. Its density operator on the symmetric
Fock space $\mathfrak{F}\left(  \mathbb{C}^{n}\right)  $ is
\[
\rho_{A}=\frac{2^{n}}{\det\left(  I+A\right)  }\left(  \frac{A-I}{A+I}\right)
_{F}.
\]

\end{lemma}

%

\begin{privatenotes}
\begin{boxedminipage}{\textwidth}%

\begin{sfblock}
(Proof needs revision, mainly notations)
\end{sfblock}

%

\end{boxedminipage}
\end{privatenotes}%

\begin{proof}
Write $\mathcal{H}=\mathbb{C}^{n}$ and $H=\mathbb{R}^{2n}$. For any
$u\in\mathcal{H}$, consider the exponential vector $u_{F}=\oplus_{k=0}%
^{\infty}\left(  k!\right)  ^{-1/2}u^{\otimes n}$. Then we have
\[
\left\langle u_{F},v_{F}\right\rangle =\exp\left\langle u,v\right\rangle .
\]
Define the coherent vector $\psi\left(  u\right)  :=u_{F}\exp\left(
-\left\Vert u\right\Vert ^{2}/2\right)  $. For any $x=x_{1}\oplus x_{2}$,
$x_{i}\in\mathbb{R}^{n}$ set $\mathrm{c}\left(  x\right)  =x_{1}+ix_{2}$. We
claim that the coherent vectors $\pi^{-n/2}\psi\left(  c\left(  x\right)
\right)  $, $x\in H$ form a resolution of the identity, i.e.
\begin{equation}
\frac{1}{\pi^{n}}\int_{H}\left\vert \left(  \psi\left(  \mathrm{c}\left(
x\right)  \right)  \right)  \right\rangle \left\langle \psi\left(
\mathrm{c}\left(  x\right)  \right)  \right\vert dx=I \label{overcomplete}%
\end{equation}
where $I$ is the unit operator on $\mathfrak{F}\left(  \mathbb{C}^{n}\right)
$ and the integral converges in a weak sense in $\mathfrak{F}\left(
\mathbb{C}^{n}\right)  $. For a proof, denote $\mu$ the l.h.s. above and note
that for every unit vector $\psi\left(  y\right)  $, $y\in\mathbb{R}^{2n}$
\begin{align*}
\left\langle \psi\left(  \mathrm{c}\left(  y\right)  \right)  \right\vert
\mu\left\vert \left(  \psi\left(  \mathrm{c}\left(  y\right)  \right)
\right)  \right\rangle  &  =\frac{1}{\pi^{n}}\int_{H}\exp\left(
2\operatorname{Re}\left\langle \mathrm{c}\left(  x\right)  ,\mathrm{c}\left(
y\right)  \right\rangle \right)  dx\exp\left(  -\left\Vert y\right\Vert
^{2}\right) \\
&  =\frac{1}{\pi^{n}}\int_{H}\exp\left(  2\left(  x,y\right)  -\left\Vert
x\right\Vert ^{2}\right)  dx\exp\left(  -\left\Vert y\right\Vert ^{2}\right)
\\
&  =\frac{2^{n}}{\left(  2\pi\right)  ^{n}}\int_{H}\exp\left(  -\left\Vert
x-y\right\Vert ^{2}\right)  dx\\
&  =\frac{1}{\left(  2\pi\right)  ^{n}\sigma^{2n}}\int_{H}\exp\left(
-\frac{1}{2\sigma^{2}}\left\Vert x-y\right\Vert ^{2}\right)  dx
\end{align*}
for $\sigma^{2}=1/2$. The above expression is the integral of the density of
the $N_{2n}\left(  y,\sigma^{2}I_{2n}\right)  $ law, which is $1$. Since the
unit vectors $\psi\left(  \mathrm{c}\left(  y\right)  \right)  $ are dense in
$\mathfrak{F}\left(  \mathbb{C}^{n}\right)  $, (\ref{overcomplete}) is proved.

Since $c:\mathbb{R}^{2n}\rightarrow\mathbb{C}^{n}$ is an isometry, for every
unitary $U$ there is an orthogonal matrix $O_{U}$ such that $U\mathrm{c}%
\left(  \mu\right)  =\mathrm{c}\left(  O_{U}\mu\right)  $ .Let $W\left(
v\right)  $, $v\in\mathbb{C}^{n}$ be an element of the Weyl algebra on
$\mathfrak{F}\left(  \mathbb{C}^{n}\right)  $, acting on exponential vectors
as
\[
W\left(  v\right)  u_{F}=\left(  v+u\right)  _{F}\exp\left(  -\left\langle
v,u\right\rangle -\left\Vert v\right\Vert ^{2}/2\right)  .
\]
The state $\rho_{A}$ is centered Gaussian gauge invariant if its
characteristic function is%
\begin{equation}
\phi\left(  t\right)  =\mathrm{tr\;}W\left(  \mathrm{c}\left(  t\right)
\right)  \rho_{A}=\exp\left(  -\frac{1}{2}\operatorname{Re}\left\langle
A\mathrm{c}\left(  t\right)  ,\mathrm{c}\left(  t\right)  \right\rangle
\right)  \text{, }t\in\mathbb{R}^{2n}. \label{claimed-form-char-func}%
\end{equation}
Setting $R:=\left(  A-I\right)  /\left(  A+I\right)  $, we then have
\[
A=\left(  I+R\right)  /\left(  I-R\right)  \text{, }\frac{I+A}{2}=1/\left(
I-R\right)
\]
and
\[
\mathrm{tr\;}R_{F}=\frac{1}{\det\left(  I-R\right)  }=\det\left(  \frac
{I+A}{2}\right)
\]
(see \cite{MR2510896}, Appendix for the last relation). It follows that
\[
\rho_{A}=\det\left(  I-R\right)  R_{F}\text{. }%
\]
If $u_{F}$, $u\in\mathcal{H}$ is an exponential vector then%
\[
R_{F}u_{F}=\left(  Ru\right)  _{F}.
\]
To find the characteristic function of $\rho_{A}$, note that
\begin{align*}
\phi\left(  t\right)   &  =\frac{1}{\pi^{n}}\int_{\mathbb{R}^{2n}%
}\mathrm{tr\;}W\left(  \mathrm{c}\left(  t\right)  \right)  \left\vert \left(
\psi\left(  \mathrm{c}\left(  x\right)  \right)  \right)  \right\rangle
\left\langle \psi\left(  \mathrm{c}\left(  x\right)  \right)  \right\vert
\rho_{A}dx\\
&  =\frac{1}{\pi^{n}}\int_{\mathbb{R}^{2n}}\mathrm{tr\;}W\left(
\mathrm{c}\left(  t\right)  \right)  \left\vert \mathrm{c}\left(  x\right)
_{F}\right\rangle \left\langle \mathrm{c}\left(  x\right)  _{F}\right\vert
R_{F}\exp\left(  -\left\Vert x\right\Vert ^{2}\right)  dx
\end{align*}%
\begin{align*}
&  =\frac{\det\left(  I-R\right)  }{\pi^{n}}\int_{\mathbb{R}^{2n}%
}\mathrm{tr\;}\left\vert \left(  \mathrm{c}\left(  x\right)  +\mathrm{c}%
\left(  t\right)  \right)  _{F}\right\rangle \left\langle \left(
R\mathrm{c}\left(  x\right)  \right)  _{F}\right\vert \exp\left(  -\left\Vert
x\right\Vert ^{2}-\left\langle \mathrm{c}\left(  t\right)  ,\mathrm{c}\left(
x\right)  \right\rangle \right)  dx\exp\left(  -\left\Vert t\right\Vert
^{2}/2\right) \\
&  =\frac{\det\left(  I-R\right)  }{\pi^{n}}\int_{\mathbb{R}^{2n}}\exp\left(
\left\langle R\mathrm{c}\left(  x\right)  ,\mathrm{c}\left(  x\right)
+\mathrm{c}\left(  t\right)  \right\rangle -\left\Vert x\right\Vert
^{2}-\left\langle \mathrm{c}\left(  t\right)  ,\mathrm{c}\left(  x\right)
\right\rangle \right)  dx\exp\left(  -\left\Vert t\right\Vert ^{2}/2\right)
\end{align*}
Let $R=UDU^{\ast}$ where $D=\mathrm{Diag}(r_{1},\ldots,r_{n})$ is real
diagonal and $U$ is unitary in $\mathbb{C}^{n}$. Let $O$ be orthogonal in
$\mathbb{R}^{2n}$ such that $R\mathrm{c}\left(  x\right)  =\mathrm{c}\left(
Ox\right)  $. By a change of variable $U\mathrm{c}\left(  x\right)
=\mathrm{c}\left(  y\right)  $, or equivalently $x=Oy$ , setting
$y=\oplus_{j=1}^{n}y_{j}$, $y_{j}\in\mathbb{R}^{2}$ and $t=Os$, $s=\oplus
_{j=1}^{n}s_{j}$, $s_{j}\in\mathbb{R}^{2}$ accordingly, we obtain%
\[
\phi\left(  t\right)  =%
{\displaystyle\prod\limits_{j=1}^{n}}
\frac{\left(  1-r_{j}\right)  }{\pi^{n}}\int_{\mathbb{R}^{2}}\exp\left(
\left\langle r_{j}\mathrm{c}\left(  y_{j}\right)  ,\mathrm{c}\left(
y_{j}\right)  +\mathrm{c}\left(  s_{j}\right)  \right\rangle -\left\Vert
y_{j}\right\Vert ^{2}-\left\langle \mathrm{c}\left(  s_{j}\right)
,\mathrm{c}\left(  y_{j}\right)  \right\rangle \right)  dy_{j}\exp\left(
-\left\Vert s_{j}\right\Vert ^{2}/2\right)
\]
We will compute each of the factors above, $\phi_{j}\left(  t\right)  $ say,
omitting the index $j$ for the variables. Then each of the factors can be
understood as pertaining to the case $n=1$, where $R=r=r_{j}$ and $A=a=\left(
1+r_{j}\right)  /\left(  1-r_{j}\right)  $. Then $r=\left(  a-1\right)
/\left(  a+1\right)  $, and
\[
\phi_{j}\left(  t\right)  =\frac{\left(  1-r\right)  }{\pi^{n}}\int
_{\mathbb{R}^{2}}\exp\left(  -\left(  1-r\right)  \left\Vert y\right\Vert
^{2}-\left(  1-r\right)  \left\langle \mathrm{c}\left(  y\right)
,\mathrm{c}\left(  s\right)  \right\rangle +2i\operatorname{Im}\left\langle
\mathrm{c}\left(  y\right)  ,\mathrm{c}\left(  s\right)  \right\rangle
-\left\Vert s\right\Vert ^{2}/2\right)  dy.
\]
Note that for $y=y_{1}\oplus y_{2}$, $y_{i}\in\mathbb{R}$ we have
\[
\left\langle \mathrm{c}\left(  y\right)  ,\mathrm{c}\left(  s\right)
\right\rangle =\left(  y,s\right)  +i\left(  y,Js\right)
\]
where $J$ is the operator in $\mathbb{R}$ satisfying
\[
J\left(  y_{1}\oplus y_{2}\right)  =y_{2}\oplus-y_{1}.
\]
Note that $\left(  y,Jy\right)  =0$. Now
\begin{align*}
&  -\left(  1-r\right)  \left\langle \mathrm{c}\left(  y\right)
,\mathrm{c}\left(  s\right)  \right\rangle +2i\left\langle \operatorname{Im}%
\mathrm{c}\left(  y\right)  ,\mathrm{c}\left(  y\right)  \right\rangle \\
&  =-\left(  1-r\right)  \operatorname{Re}\left\langle \mathrm{c}\left(
y\right)  ,\mathrm{c}\left(  s\right)  \right\rangle +i\left(  1+r\right)
\left\langle \operatorname{Im}\mathrm{c}\left(  y\right)  ,\mathrm{c}\left(
s\right)  \right\rangle \\
&  =-\left(  1-r\right)  \left(  y,s\right)  +i\left(  1+r\right)  \left(
y,Js\right)  .
\end{align*}
This gives%
\begin{align*}
\phi_{j}\left(  t\right)   &  =\frac{\left(  1-r\right)  }{\pi}\int_{H}%
\exp\left(  -\left(  1-r\right)  \left\Vert y\right\Vert ^{2}-\left(
1-r\right)  \left(  y,s\right)  +i\left(  1+r\right)  \left(  y,Js\right)
-\left\Vert s\right\Vert ^{2}/2\right)  dy\\
&  =\frac{\left(  1-r\right)  }{\pi^{n}}\int_{H}\exp\left(  -\left(
1-r\right)  \left\Vert y+s/2\right\Vert ^{2}+i\left(  1+r\right)  \left(
y,Js\right)  \right)  dy\cdot\exp\left(  -\left(  1+r\right)  \left\Vert
s\right\Vert ^{2}/4\right) \\
&  =\frac{2\left(  1-r\right)  }{2\pi}\int_{H}\exp\left(  -\frac{2\left(
1-r\right)  }{2}\left\Vert y+s/2\right\Vert ^{2}+i\left(  1+r\right)  \left(
y,Js\right)  \right)  dy\cdot\exp\left(  -\left(  1+r\right)  \left\Vert
t\right\Vert ^{2}/4\right)
\end{align*}
The expression before the second exponential factor is the characteristic
function of the $N_{2}\left(  -t/2,1/2\left(  1-r\right)  \right)  $ law at
position $w=\left(  1+r\right)  Js\in\mathbb{R}^{2}$, which is
\begin{align*}
&  \exp\left(  i\left(  -s/2,w\right)  -\frac{1}{2\cdot2\left(  1-r\right)
}\left\Vert w\right\Vert ^{2}\right) \\
&  =\exp\left(  -\left(  1+r\right)  i\left(  s,Js\right)  -\frac{\left(
1+r\right)  ^{2}}{2\cdot2\left(  1-r\right)  }\left\Vert s\right\Vert
^{2}\right) \\
&  =\exp\left(  -a\left(  1+r\right)  \left\Vert s\right\Vert ^{2}/4\right)  .
\end{align*}
Hence
\[
\phi_{j}\left(  t\right)  =\exp\left(  -\left(  a+1\right)  \left(
1+r\right)  \left\Vert s\right\Vert ^{2}/4\right)  .
\]
Since $\left(  1+r\right)  =2a/\left(  a+1\right)  $, we obtain%
\[
\phi_{j}\left(  t\right)  =\exp\left(  -a\left\Vert s\right\Vert
^{2}/2\right)
\]
Hence, setting $a_{j}=\left(  1+r_{j}\right)  /\left(  1-r_{j}\right)  $, we
obtain
\begin{equation}
\phi\left(  t\right)  =%
{\displaystyle\prod\limits_{j=1}^{n}}
\phi_{j}\left(  t\right)  =\exp\left(  -\sum_{j=1}^{n}a_{j}\left\Vert
s_{j}\right\Vert ^{2}/2\right)  . \label{pre-claim}%
\end{equation}
Here $\left\Vert s_{j}\right\Vert ^{2}=\left\vert e_{j}^{\prime}%
\mathrm{c}\left(  O^{\prime}t\right)  \right\vert ^{2}$ where $e_{j}$ is the
$j$-th standard unit vector in $\mathbb{C}^{n}$. But then%
\[
e_{j}^{\prime}\mathrm{c}\left(  O^{\prime}t\right)  =e_{j}^{\prime}U^{\ast
}\mathrm{c}\left(  t\right)  =\mathbf{u}_{j}^{\ast}\mathrm{c}\left(  t\right)
\]
where is an $\mathbf{u}_{j}^{\ast}$ eigenvector of $A$ pertaining to
eigenvalue $\alpha_{j}$. Then
\begin{align*}
a_{j}\sum_{j=1}^{n}\left\Vert s_{j}\right\Vert ^{2}  &  =a_{j}\sum_{j=1}%
^{n}\left\vert \mathbf{u}_{j}^{\ast}\mathrm{c}\left(  t\right)  \right\vert
^{2}\\
&  =\left\langle A\mathrm{c}\left(  t\right)  ,\mathrm{c}\left(  t\right)
\right\rangle =\operatorname{Re}\left\langle A\mathrm{c}\left(  t\right)
,\mathrm{c}\left(  t\right)  \right\rangle
\end{align*}
such that (\ref{pre-claim}) yields the claimed form of $\phi\left(  t\right)
$.
\end{proof}

\subsubsection{Some facts on Fock operators\label{sec-spec-repre-Fock-oper}}

The following technical result for finite dimensional $B$ allows to relate the
spectral decompositions of $B_{F}$ and $B$ (cp. (A1), (A3) of \cite{MR2510896}%
). Define the multiindex set $D\left(  m\right)  :=\left\{  \mathbf{m}%
\in\mathbb{Z}_{+}^{d}:m_{1}+\ldots+m_{d}=m\right\}  $ and for any
$\mathbf{m}\in D\left(  m\right)  $ let $\Pi\left(  \mathbf{m},d\right)  $ be
the set of partitions of $m$ objects into $d$ distinct groups, each of size
$m_{j}$, $j=1,\ldots,d$. It is well known that
\[
\mathrm{card}\left(  \Pi\left(  \mathbf{m},d\right)  \right)  =d_{\mathbf{m}%
}:=\left(
\genfrac{}{}{0pt}{}{m}{m_{1}\ldots m_{d}}%
\right)  =\frac{m!}{m_{1}!\ldots m_{d}!}.
\]
For each $\nu\in\Pi\left(  \mathbf{m},d\right)  $ and $j\in\left\{
1,\ldots,m\right\}  $, let $\nu\left(  j\right)  \in\left\{  1,\ldots
,d\right\}  $ be the index of the group to which the $j$-th object has been assigned.

\begin{lemma}
\label{Lem-spec-decompos-Fock-op}Let $B$ be Hermitian on $\mathcal{H}%
=\mathbb{C}^{d}$ with spectral decomposition $B=\sum_{k=1}^{d}\lambda
_{k}\left\vert e_{k}\right\rangle \left\langle e_{k}\right\vert $. Then the
spectral decomposition of $\vee^{m}B$ is
\begin{equation}
\vee^{m}B=\sum_{\mathbf{m}\in D\left(  m\right)  }\lambda_{\mathbf{m}%
}\left\vert e_{\mathbf{m}}\right\rangle \left\langle e_{\mathbf{m}}\right\vert
\label{symm-basis-0}%
\end{equation}
where
\begin{align}
\lambda_{\mathbf{m}}  &  :=\lambda_{1}^{m_{1}}\ldots\lambda_{d}^{m_{d}}\text{,
}\label{symm-basis}\\
e_{\mathbf{m}}  &  =\frac{1}{\sqrt{d_{\mathbf{m}}}}\sum_{\nu\in\Pi\left(
\mathbf{m},d\right)  }e_{\nu\left(  1\right)  }\otimes\ldots\otimes
e_{\nu\left(  m\right)  }. \label{symm-basis-2}%
\end{align}

\end{lemma}

\begin{proof}
For $\mathcal{H}=\mathbb{C}^{d}$, consider the symmetrization operator in
$\mathcal{H}^{\otimes m}$: let $\mathbf{k}\in\left[  1,d\right]  ^{\times m}$
be a multiindex and let
\[
\tilde{e}_{\mathbf{k}}:=e_{k(1)}\otimes\ldots\otimes e_{k(m)}%
\]
be an orthonormal basis \ of $\mathcal{H}^{\otimes m}$; then, if $U_{\sigma
}^{(m)}$, $\sigma\in S_{m}$ denotes the standard unitary representation of the
symmetric group $S_{m}$ on $\mathcal{H}^{\otimes m}$,
\[
\Pi_{m}\tilde{e}_{\mathbf{k}}:=\frac{1}{m!}\sum_{\sigma\in S_{m}}U_{\sigma
}^{(m)}\left(  e_{k(1)}\otimes\ldots\otimes e_{k(m)}\right)
\]
is the symmetrization operator in $\mathcal{H}^{\otimes m}$, where every
$U_{\sigma}^{(m)}\left(  e_{k(1)}\otimes\ldots\otimes e_{k(m)}\right)  $ gives
just a permutation of the tensor components. It is a projection, and the space
$\vee^{m}\mathcal{H}$ is the eigenspace. Note that for $\mathbf{k}_{1}$,
$\mathbf{k}_{2}\in$ $\left[  1,d\right]  ^{\times m}$ we have $\Pi_{m}%
\tilde{e}_{\mathbf{k}_{1}}=\Pi_{m}\tilde{e}_{\mathbf{k}_{2}}$ if and only if
there exists a multiindex $\mathbf{m}\in\mathbb{Z}_{+}^{d},m_{1}+\ldots
+m_{d}=m$ such that both $\tilde{e}_{\mathbf{k}_{j}}$ are permutations of
$e_{1}^{\otimes m_{1}}\otimes\ldots\otimes e_{d}^{\otimes m_{d}}$, in other
words there exist permutations $\sigma_{1},\sigma_{2}\in S_{m}$ such that
\[
U_{\sigma_{j}}^{(m)}\left(  e_{1}^{\otimes m_{1}}\otimes\ldots\otimes
e_{d}^{\otimes m_{d}}\right)  =\tilde{e}_{\mathbf{k}_{j}}\text{, }j=1,2.
\]
If $\Pi_{m}\tilde{e}_{\mathbf{k}_{1}}\neq\Pi_{m}\tilde{e}_{\mathbf{k}_{2}}$
then the images are orthogonal, i.e. $\left\langle \Pi_{m}\tilde
{e}_{\mathbf{k}_{1}},\Pi_{m}\tilde{e}_{\mathbf{k}_{2}}\right\rangle =0$. This
implies that the set
\begin{align*}
&  \left\{  f_{\mathbf{m}},\mathbf{m}\in\mathbb{Z}_{+}^{d},m_{1}+\ldots
+m_{d}=m\right\}  \text{ where }\\
f_{\mathbf{m}}  &  :=\Pi_{m}\left(  e_{1}^{\otimes m_{1}}\otimes\ldots\otimes
e_{d}^{\otimes m_{d}}\right)  =\frac{1}{m!}\sum_{\sigma\in S_{m}}U_{\sigma
}^{(m)}\left(  e_{1}^{\otimes m_{1}}\otimes\ldots\otimes e_{d}^{\otimes m_{d}%
}\right)
\end{align*}
is an orthogonal (not yet orthonormal) basis of $\vee^{m}\mathcal{H}$. For the
normalization, note that the set $\left\{  U_{\sigma}^{(m)}\left(
e_{1}^{\otimes m_{1}}\otimes\ldots\otimes e_{d}^{\otimes m_{d}}\right)
,\sigma\in S_{m}\right\}  $ has $m!$ elements, but only $d_{\mathbf{m}%
}=\left(
\genfrac{}{}{0pt}{}{m}{m_{1}\ldots m_{d}}%
\right)  =\frac{m!}{m_{1}!\ldots m_{d}!}$ different elements, each with
multiplicity $m!/d_{\mathbf{m}}$. The different elements can be described as
\begin{equation}
\hat{e}_{\nu}:=e_{\nu\left(  1\right)  }\otimes\ldots\otimes e_{\nu\left(
m\right)  }\text{, }\nu\in\Pi\left(  \mathbf{m},d\right)  ;
\label{notation-permut-vectors}%
\end{equation}
they are orthogonal to each other. Hence
\[
f_{\mathbf{m}}=\frac{1}{d_{\mathbf{m}}}\sum_{\nu\in\Pi\left(  \mathbf{m}%
,d\right)  }\hat{e}_{\nu},
\]%
\[
\left\Vert f_{\mathbf{m}}\right\Vert ^{2}=\left(  \frac{1}{d_{\mathbf{m}}%
}\right)  ^{2}d_{\mathbf{m}}=\frac{1}{d_{\mathbf{m}}},
\]
which implies that the vectors
\begin{equation}
e_{\mathbf{m}}:=f_{\mathbf{m}}/\left\Vert f_{\mathbf{m}}\right\Vert =\frac
{1}{\sqrt{d_{\mathbf{m}}}}\sum_{\nu\in\Pi\left(  \mathbf{m},d\right)  }\hat
{e}_{\nu} \label{basis-vecs-descrip-1}%
\end{equation}
are an orthonormal basis of $\vee^{m}\mathcal{H}$. To see that they are an
eigenbasis of $\vee^{m}B$ for eigenvalues $\lambda_{\mathbf{m}}$, note that
each $\hat{e}_{\nu}$ is an eigenvector of $B^{\otimes m}$ for eigenvalue
$\lambda_{\mathbf{m}}$, hence $e_{\mathbf{m}}$ is also an eigenvector for
$\lambda_{\mathbf{m}}$. Since the $e_{\mathbf{m}}$ are an orthonormal basis of
$\vee^{m}\mathcal{H}$, they are an eigenbasis of $\vee^{m}B$.
\end{proof}

\bigskip

\begin{lemma}
\label{Lem-pre-log-formula} Let $A,B$ be Hermitian operators on $\mathcal{H=}%
\mathbb{C}^{d}$ such that $0<A<I$, and let $\Gamma\left(  B\right)
:=\oplus_{m=0}^{\infty}\Gamma_{m}\left(  B\right)  $, where $\Gamma_{m}\left(
B\right)  $ is the restriction of $\sum_{k=1}^{m}I^{\otimes\left(  k-1\right)
}\otimes B\otimes I^{\otimes\left(  m-k\right)  }$ onto $\vee^{m}\mathcal{H}$,
with $\Gamma_{0}\left(  B\right)  =0$. Then
\begin{equation}
\mathrm{Tr}\;A_{F}\Gamma\left(  B\right)  =\frac{1}{\det\left(  I-A\right)
}\mathrm{Tr}\;\frac{A}{I-A}B. \label{pre-log-formula}%
\end{equation}

\end{lemma}

\begin{proof}
We have
\[
\mathrm{Tr}\;A_{F}\Gamma\left(  B\right)  =\sum_{m=0}^{\infty}\mathrm{Tr}%
\;\left(  \vee^{m}A\right)  \Gamma_{m}\left(  B\right)  ,
\]%
\begin{align*}
\mathrm{Tr}\;\left(  \vee^{m}A\right)  \Gamma_{m}\left(  B\right)   &
=\sum_{\mathbf{m\in}D\left(  m\right)  }\left\langle e_{\mathbf{m}}|\left(
\vee^{m}A\right)  \Gamma_{m}\left(  B\right)  |e_{\mathbf{m}}\right\rangle \\
&  =\sum_{\mathbf{m\in}D\left(  m\right)  }\lambda_{\mathbf{m}}\left\langle
e_{\mathbf{m}}|\Gamma_{m}\left(  B\right)  |e_{\mathbf{m}}\right\rangle .
\end{align*}
Set $\Gamma_{m,j}\left(  B\right)  =I^{\otimes\left(  j-1\right)  }\otimes
B\otimes I^{\otimes\left(  m-j\right)  }$ and let $\check{\Gamma}_{m,j}\left(
B\right)  $ be the restriction to $\vee^{m}\mathcal{H}$ for $\mathcal{H}%
=\mathbb{C}^{d}$. We have
\begin{equation}
\left\langle e_{\mathbf{m}}|\Gamma_{m}\left(  B\right)  |e_{\mathbf{m}%
}\right\rangle =\sum_{j=1}^{m}\left\langle e_{\mathbf{m}}|\check{\Gamma}%
_{m,j}\left(  B\right)  |e_{\mathbf{m}}\right\rangle =\sum_{j=1}%
^{m}\left\langle e_{\mathbf{m}}|\Gamma_{m,j}\left(  B\right)  |e_{\mathbf{m}%
}\right\rangle \label{restric-dissolve}%
\end{equation}
since $e_{\mathbf{m}}\in\vee^{m}\mathcal{H}$. Furthermore, using
(\ref{basis-vecs-descrip-1})
\[
\left\langle e_{\mathbf{m}}|\Gamma_{m,j}\left(  B\right)  |e_{\mathbf{m}%
}\right\rangle =\frac{1}{d_{\mathbf{m}}}\sum_{\nu,\mu\in\Pi\left(
\mathbf{m},d\right)  }\left\langle e_{\nu}|\Gamma_{m,j}\left(  B\right)
|e_{\mu}\right\rangle .
\]
We note that any term $\left\langle e_{\nu}|\Gamma_{m,j}\left(  B\right)
|e_{\mu}\right\rangle $ must be zero unless $\nu=\mu$. Indeed
\begin{equation}
\left\langle e_{\nu}|\Gamma_{m,j}\left(  B\right)  |e_{\mu}\right\rangle
=\left(
{\displaystyle\prod\limits_{k=1}^{j-1}}
\left\langle e_{\nu\left(  k\right)  }|e_{\mu\left(  k\right)  }\right\rangle
\right)  \left\langle e_{\nu\left(  j\right)  }|B|e_{\mu\left(  j\right)
}\right\rangle \left(
{\displaystyle\prod\limits_{k=j+1}^{m}}
\left\langle e_{\nu\left(  k\right)  }|e_{\mu\left(  k\right)  }\right\rangle
\right)  . \label{factor-decompos-prel-log-formula}%
\end{equation}
For two partitions $\nu\neq\mu$, there must be at least two indices
$k\in\left\{  1,\ldots,m\right\}  $ such that $\nu\left(  k\right)  \neq
\mu\left(  k\right)  $. Indeed if there is no such index then $\nu=\mu$, and
if there is only one such index then this contradicts the assumption that both
$\nu$ and $\mu$ are in $\Pi\left(  \mathbf{m},d\right)  $ (i.e. the $l$-th
group has a given number of elements $m_{l}$, $l=1,\ldots,d$). This implies
that on the r.h.s. of (\ref{factor-decompos-prel-log-formula}), either the
first or the third factor (or both) are zero, unless $\nu=\mu$. Hence
\begin{align*}
\left\langle e_{\mathbf{m}}|\Gamma_{m,j}\left(  B\right)  |e_{\mathbf{m}%
}\right\rangle  &  =\frac{1}{d_{\mathbf{m}}}\sum_{\nu\in\Pi\left(
\mathbf{m},d\right)  }\left\langle e_{\nu}|\Gamma_{m,j}\left(  B\right)
|e_{\nu}\right\rangle \\
&  =\frac{1}{d_{\mathbf{m}}}\sum_{\nu\in\Pi\left(  \mathbf{m},d\right)
}\left\langle e_{\nu\left(  j\right)  }|B|e_{\nu\left(  j\right)
}\right\rangle
\end{align*}
and with (\ref{restric-dissolve})%
\begin{align*}
\left\langle e_{\mathbf{m}}|\Gamma_{m}\left(  B\right)  |e_{\mathbf{m}%
}\right\rangle  &  =\frac{1}{d_{\mathbf{m}}}\sum_{\nu\in\Pi\left(
\mathbf{m},d\right)  }\sum_{j=1}^{m}\left\langle e_{\nu\left(  j\right)
}|B|e_{\nu\left(  j\right)  }\right\rangle \\
&  =\frac{1}{d_{\mathbf{m}}}\sum_{\nu\in\Pi\left(  \mathbf{m},d\right)  }%
\sum_{k=1}^{d}m_{k}\left\langle e_{k}|B|e_{k}\right\rangle \\
&  =\sum_{k=1}^{d}m_{k}\left\langle e_{k}|B|e_{k}\right\rangle .
\end{align*}
Hence
\begin{align*}
\mathrm{Tr}\;\left(  \vee^{m}A\right)  \Gamma_{m}\left(  B\right)   &
=\sum_{\mathbf{m\in}D\left(  m\right)  }\lambda_{\mathbf{m}}\left(  \sum
_{k=1}^{d}m_{k}\right) \\
&  =\sum_{k=1}^{d}\left\langle e_{k}|B|e_{k}\right\rangle \sum_{\mathbf{m\in
}D\left(  m\right)  }\left(
{\displaystyle\prod\limits_{j=1,\ldots,d,j\neq k}}
\lambda_{j}^{m_{j}}\right)  m_{k}\lambda_{k}^{m_{k}}\text{.}%
\end{align*}
By summing over $m\geq0$, we obtain
\[
\mathrm{Tr}\;A_{F}\Gamma\left(  B\right)  =\sum_{k=1}^{d}\left\langle
e_{k}|B|e_{k}\right\rangle \left(  \sum_{m=0}^{\infty}m\lambda_{k}^{m}\right)
%
{\displaystyle\prod\limits_{j=1,\ldots,d,j\neq k}}
\left(  \sum_{m=0}^{\infty}\lambda_{j}^{m}\right)
\]
Using the elementary relation, for $0\leq x<1$
\begin{equation}
\sum_{m=0}^{\infty}mx^{m}=\frac{x}{\left(  1-x\right)  ^{2}}
\label{quasi-geom-series}%
\end{equation}
we obtain
\begin{align*}
\mathrm{Tr}\;A_{F}\Gamma\left(  B\right)   &  =\sum_{k=1}^{d}\left\langle
e_{k}|B|e_{k}\right\rangle \left(  \frac{\lambda_{k}}{\left(  1-\lambda
_{k}\right)  ^{2}}\right)
{\displaystyle\prod\limits_{j=1,\ldots,d,j\neq k}}
\frac{1}{1-\lambda_{j}}\\
&  =\left(
{\displaystyle\prod\limits_{j=1,\ldots,d}}
\frac{1}{1-\lambda_{j}}\right)  \sum_{k=1}^{d}\frac{\lambda_{k}}{1-\lambda
_{k}}\left\langle e_{k}|B|e_{k}\right\rangle \\
&  =\frac{1}{\det\left(  I-A\right)  }\mathrm{Tr}\;\frac{A}{I-A}B.
\end{align*}

\end{proof}

To compute the relative entropy of Gaussian states, we need the logarithm of a
Fock operator. This can be found with the help of the spectral decomposition
of Lemma \ref{Lem-spec-decompos-Fock-op}.

\begin{lemma}
\label{Lem-log-Gamma-m}Let $B$ be Hermitian on $\mathcal{H}=\mathbb{C}^{d}$
with spectral decomposition $B=\sum_{k=1}^{d}\lambda_{k}\left\vert
e_{k}\right\rangle \left\langle e_{k}\right\vert $. Then
\[
\log\vee^{m}B=\Gamma_{m}\left(  \log B\right)
\]
where $\Gamma_{m}\left(  \cdot\right)  $ has been defined in Lemma
\ref{Lem-pre-log-formula}.
\end{lemma}

\begin{proof}
From Lemma \ref{Lem-spec-decompos-Fock-op} we obtain, if $B=$
\begin{align*}
\log\vee^{m}B  &  =\sum_{\mathbf{m}\in D\left(  m\right)  }\left(  \log
\lambda_{\mathbf{m}}\right)  \left\vert e_{\mathbf{m}}\right\rangle
\left\langle e_{\mathbf{m}}\right\vert \\
&  =\sum_{\mathbf{m}\in D\left(  m\right)  }\left(  \sum_{j=1}^{d}m_{j}%
\log\lambda_{j}\right)  \left\vert e_{\mathbf{m}}\right\rangle \left\langle
e_{\mathbf{m}}\right\vert .
\end{align*}
It now suffices to show that each $e_{\mathbf{m}}$ is an eigenvector of
$\Gamma_{m}\left(  \log B\right)  $ for eigenvalue $\sum_{j=1}^{d}m_{j}%
\log\lambda_{j}$:%
\[
\Gamma_{m}\left(  \log B\right)  e_{\mathbf{m}}=\left(  \sum_{j=1}^{d}%
m_{j}\log\lambda_{j}\right)  e_{\mathbf{m}}.
\]
Equivalently we can show that $\sum_{\nu\in\Pi\left(  \mathbf{m},d\right)
}\hat{e}_{\nu}$ is an eigenvector for the same eigenvalue, where $\hat{e}%
_{\nu}$, $\nu\in\Pi\left(  \mathbf{m},d\right)  $ have been defined in
(\ref{notation-permut-vectors}). Write
\[
\Gamma_{m}\left(  \log B\right)  =\sum_{k=1}^{m}\check{\Gamma}_{m,j}\left(
\log B\right)
\]
where $\check{\Gamma}_{m,j}\left(  \log B\right)  $ is the rectriction to
$\vee^{m}\mathcal{H}$ of
\[
\Gamma_{m,j}\left(  \log B\right)  =I^{\otimes\left(  j-1\right)  }\otimes\log
B\otimes I^{\otimes\left(  m-j\right)  }\text{, }j=1,\ldots,m.
\]
Note that $\sum_{\nu\in\Pi\left(  \mathbf{m},d\right)  }\hat{e}_{\nu}$ is an
element of $\vee^{m}\mathcal{H}$ while the $\hat{e}_{\nu}$ generally are not.
But it suffices to show that for all $\nu\in\Pi\left(  \mathbf{m},d\right)  $
\begin{equation}
\left(  \sum_{j=1}^{m}\Gamma_{m,j}\left(  \log B\right)  \right)  \hat{e}%
_{\nu}=\left(  \sum_{j=1}^{d}m_{j}\lambda_{j}\right)  \hat{e}_{\nu}.
\label{eigenvec-non-symm}%
\end{equation}
Consider the particular $\nu\in\Pi\left(  \mathbf{m},d\right)  $ for which
\begin{equation}
\hat{e}_{\nu}=e_{1}^{\otimes m_{1}}\otimes\ldots\otimes e_{d}^{\otimes m_{d}}.
\label{special-e-nu-hat}%
\end{equation}
In this case we have
\begin{align*}
\Gamma_{m,j}\left(  \log B\right)  \hat{e}_{\nu}  &  =\lambda_{1}\hat{e}_{\nu
}\text{, }j=1,\ldots,m_{1},\\
\Gamma_{m,j}\left(  \log B\right)  \hat{e}_{\nu}  &  =\lambda_{2}\hat{e}_{\nu
}\text{, }j=m_{1}+1,\ldots,m_{1}+m_{2},\\
&  \ldots\\
\Gamma_{m,j}\left(  \log B\right)  \hat{e}_{\nu}  &  =\lambda_{d}\hat{e}_{\nu
}\text{, }j=\sum_{j=1}^{d-1}m_{j}+1,\ldots,m.
\end{align*}
This implies (\ref{eigenvec-non-symm}) for $\hat{e}_{\nu}$ given by
(\ref{eigenvec-non-symm}). Since all other $\hat{e}_{\nu}$, $\nu\in\Pi\left(
\mathbf{m},d\right)  $ arise from a permutation of the tensor factors, they
also fulfill (\ref{eigenvec-non-symm}).
\end{proof}

\subsection{Uniform convergence in distribution\label{subsec-uniform-in-law}}%

\begin{privatenotes}
\begin{boxedminipage}{\textwidth}%

\begin{sfblock}
This whole appendix subsection was revised 1/25/26.
\end{sfblock}

%

\end{boxedminipage}
\end{privatenotes}%

Let us define uniform convergence in distribution, following \cite{MR620321},
Appendix I. Consider a sample space $\left(  \mathbb{R}^{d},\mathfrak{B}%
^{d}\right)  $ where $\mathfrak{B}^{d}$ is the Borel sigma-algebra;
convergence in distribution of a sequence of probability measures $Q_{n}$ to
some $Q$ is written $Q_{n}\Longrightarrow_{d}Q$. Assume on $\left(
\mathbb{R}^{d},\mathfrak{B}^{d}\right)  $ there is a sequence of families of
probability measures $\mathcal{P}_{n}=\left(  P_{n,\theta}\text{, }\theta
\in\Theta\right)  $, $n\in\mathbb{N}$ where $\Theta$ is an arbitrary set. The
family $\mathcal{P}_{n}$ is said to uniformly converge in distribution to a
family $\mathcal{P}=\left(  P_{\theta}\text{, }\theta\in\Theta\right)  $ if
for every bounded continuous function $g$ on $\mathbb{R}^{d}$ we have
\begin{equation}
\int_{\mathbb{R}^{d}}gdP_{n,\theta}\rightarrow\int_{\mathbb{R}^{d}}%
gdP_{\theta} \label{unif-conv-law}%
\end{equation}
uniformly in $\theta$.

Consider the bounded Lipschitz norm for real valued functions $f$ on
$\mathbb{R}^{d}$
\begin{equation}
\left\Vert f\right\Vert _{BL}:=\left\Vert f\right\Vert _{\infty}+\sup_{x\neq
y}\frac{\left\vert f\left(  x\right)  -f\left(  y\right)  \right\vert
}{\left\vert x-y\right\vert }\text{, }\left\Vert f\right\Vert _{\infty}%
:=\sup_{x}\left\vert f\left(  x\right)  \right\vert \label{BL-def}%
\end{equation}
and the bounded Lipschitz metric for probability measures $P,Q$ on
$\mathbb{R}^{d}$
\begin{equation}
\beta\left(  P,Q\right)  :=\sup\left\{  \left\vert \int f\left(  dP-dQ\right)
\right\vert :\left\Vert f\right\Vert _{BL}\leq1\right\}  .
\label{BL-metric-for-probs-def}%
\end{equation}
It is well known that $P_{n}\Longrightarrow_{d}Q$ if and only if $\beta\left(
P_{n},Q\right)  \rightarrow0$ (\cite{MR0982264}, Theorem 11.3.3). Also
consider the total variation metric:%
\begin{equation}
\left\Vert P-Q\right\Vert _{TV}=\sup_{A\in\mathfrak{B}^{d}}\left\vert P\left(
A\right)  -Q\left(  A\right)  \right\vert . \label{TV-metric-def}%
\end{equation}
Recall that for $\nu=P+Q$ and $p=dP/d\nu$, $q=dQ/d\nu$ one has
\begin{subequations}
\label{TV-metric-and-L1}%
\begin{align}
\left\Vert P-Q\right\Vert _{TV}  &  =\frac{1}{2}\int\left\vert p-q\right\vert
d\mu=\frac{1}{2}\left\Vert P-Q\right\Vert _{1}\text{ where }%
\label{TV-metric-and-L1-b}\\
\left\Vert P-Q\right\Vert _{1}  &  =\sup\left\{  \left\vert \int f\left(
dP-dQ\right)  \right\vert :\left\Vert f\right\Vert _{\infty}\leq1\text{,
}f\text{ measurable}\right\}  . \label{TV-metric-and-L1-a}%
\end{align}
Also consider the Hellinger metric
\end{subequations}
\[
H\left(  P,Q\right)  =\left(  \int\left(  p^{1/2}-q^{1/2}\right)  ^{2}%
d\nu\right)  ^{1/2}.
\]
See \cite{MR2724359}, Sec. 2.4 for relations between these distances. In
particular, by Le Cam's inequality (\cite{MR2724359}, Lemma 2.3), one has
\begin{equation}
\left\Vert P-Q\right\Vert _{TV}\leq H\left(  P,Q\right)  . \label{Lecam-inequ}%
\end{equation}

\begin{lemma}
\label{Lem-unif-law-converg-charac}Assume $\Theta$ is a compact metric space
with metric $\mu$ and the mapping $\theta\rightarrow P_{\theta}$, $\theta
\in\Theta$ is continuous in total variation metric. Then the following
statements are equivalent:\newline(i) $\mathcal{P}_{n}$ uniformly converges in
distribution to $\mathcal{P}$\newline(ii) $\sup_{\theta}\beta\left(
P_{n,\theta},P_{\theta}\right)  \rightarrow0$ \newline(iii) For every sequence
$\left\{  \theta_{n}\right\}  $ such that $\theta_{n}\rightarrow\theta$ for
some $\theta\in\Theta$, one has $P_{n,\theta_{n}}\Longrightarrow_{d}P_{\theta
}$.
\end{lemma}

\begin{proof}
(i) $\Longrightarrow$(ii). Assume that (ii) does not hold. Then there is a
subsequence $\mathcal{N}_{1}\subset\mathbb{N}$ such that $P_{\theta_{n}}$
converges in total variation to some $P_{\theta}$ along $n\in\mathcal{N}_{1}$,
but for some $\delta>0$
\[
\beta\left(  P_{n,\theta_{n}},P_{\theta_{n}}\right)  >\delta,\text{ }%
n\in\mathcal{N}_{1}\text{.}%
\]
In view of (i), we have for every bounded continuous $g$
\[
\left\vert \int_{\mathbb{R}^{d}}gdP_{n,\theta_{n}}-\int_{\mathbb{R}^{d}%
}gdP_{\theta_{n}}\right\vert \rightarrow0,\text{ }n\in\mathcal{N}_{1}\text{.}%
\]
Recall that the total variation metric satisfies (\ref{TV-metric-and-L1});
hence
\[
\beta\left(  P,Q\right)  \leq2\left\Vert P-Q\right\Vert _{TV}%
\]
and for every bounded continuous $g$%
\[
\left\vert \int_{\mathbb{R}^{d}}gdP-\int_{\mathbb{R}^{d}}gdQ\right\vert
\leq\frac{2\left\Vert P-Q\right\Vert _{TV}}{\left\Vert g\right\Vert _{\infty}%
}.
\]
Now $\left\Vert P_{\theta_{n}}-P_{\theta}\right\Vert _{TV}\rightarrow0$,
$n\in\mathcal{N}_{1}$ implies%
\begin{align}
\beta\left(  P_{n,\theta_{n}},P_{\theta}\right)   &  \geq\beta\left(
P_{n,\theta_{n}},P_{\theta_{n}}\right)  -\beta\left(  P_{\theta_{n}}%
,P_{\theta}\right) \nonumber\\
&  \geq\delta/2,\text{ }n\in\mathcal{N}_{1}\text{, }n\text{ sufficiently large
} \label{contra-5}%
\end{align}
and for every bounded continuous $g$%
\begin{align*}
&  \left\vert \int_{\mathbb{R}^{d}}gdP_{n,\theta_{n}}-\int_{\mathbb{R}^{d}%
}gdP_{\theta}\right\vert \\
&  \leq\left\vert \int_{\mathbb{R}^{d}}gdP_{n,\theta_{n}}-\int_{\mathbb{R}%
^{d}}gdP_{\theta_{n}}\right\vert +\frac{2\left\Vert P_{\theta_{n}}-P_{\theta
}\right\Vert _{TV}}{\left\Vert g\right\Vert _{\infty}}\rightarrow0,\text{
}n\in\mathcal{N}_{1}.
\end{align*}
The latter relation means $P_{n,\theta_{n}}\Longrightarrow_{d}P_{\theta}$
along $n\in\mathcal{N}_{1}$, hence $\beta\left(  P_{n,\theta_{n}},P_{\theta
}\right)  \rightarrow0$, which contradicts (\ref{contra-5}).

(ii)$\Longrightarrow$(iii). Let $\left\{  \theta_{n}\right\}  $ be a sequence
with $\left\Vert P_{\theta_{n}}-P_{\theta}\right\Vert _{TV}\rightarrow0$.
Then
\begin{align*}
\beta\left(  P_{n,\theta_{n}},P_{\theta}\right)   &  \leq\beta\left(
P_{n,\theta_{n}},P_{\theta_{n}}\right)  +\beta\left(  P_{\theta_{n}}%
,P_{\theta}\right) \\
&  \leq\beta\left(  P_{n,\theta_{n}},P_{\theta_{n}}\right)  +2\left\Vert
P_{\theta_{n}}-P_{\theta}\right\Vert _{TV}\rightarrow0.
\end{align*}
Hence $\beta\left(  P_{n,\theta_{n}},P_{\theta}\right)  \rightarrow0$,
implying $P_{n,\theta_{n}}\overset{d}{\implies}P_{\theta}$.

(iii)$\Longrightarrow$(i). Assume (i) does not hold. Then there is a
subsequence $\mathcal{N}_{1}\subset\mathbb{N}$, a sequence $\left\{
\theta_{n},n\in\mathcal{N}_{1}\right\}  \subset\Theta$, a bounded continuous
$g$ and a $\delta>0$ such that
\[
\left\vert \int_{\mathbb{R}^{d}}gdP_{n,\theta_{n}}-\int_{\mathbb{R}^{d}%
}gdP_{\theta_{n}}\right\vert \geq\delta\text{, }n\in\mathcal{N}_{1}.
\]
Then there is a further subsequence $\mathcal{N}_{2}\subset\mathcal{N}_{1}$
such that for some $\theta\in\Theta$ one has $\theta_{n}\rightarrow\theta$
along $\mathcal{N}_{2}$ and hence
\[
\left\Vert P_{\theta_{n}}-P_{\theta}\right\Vert _{TV}\rightarrow
0,n\in\mathcal{N}_{2}.
\]
This implies
\begin{equation}
\left\vert \int_{\mathbb{R}^{d}}gdP_{n,\theta_{n}}-\int_{\mathbb{R}^{d}%
}gdP_{\theta}\right\vert \geq\delta/2\text{, }n\in\mathcal{N}_{2}\text{,
}n\text{ sufficiently large. } \label{contra-6}%
\end{equation}
Define a sequence $\theta_{n}^{\ast}$, $n\in\mathbb{N}$ by $\theta_{n}^{\ast
}=\theta_{n}$ for $n\in\mathcal{N}_{2}$, $\theta_{n}^{\ast}=\theta$ for
$n\notin\mathcal{N}_{2}$. Then $\theta_{n}^{\ast}\rightarrow\theta$ and by
(iii) we have $P_{n,\theta_{n}^{\ast}}\overset{d}{\implies}P_{\theta}$, which
contradicts (\ref{contra-6}).
\end{proof}

\bigskip

In the context of the CLT, consider a set $\mathcal{S}$ of family of $d\times
d$ nonsingular covariance matrices and the Hilbert-Schmidt norm $\left\Vert
\Sigma\right\Vert _{2}=\left(  \mathrm{Tr}\;\Sigma^{2}\right)  ^{1/2}$.

\begin{lemma}
\label{lem-covmatrices-cite-GNZpaper}(Lemma 2.1 of \cite{MR2589320}) Suppose
the set $\mathcal{S}$ satisfies
\[
s_{1}:=\inf_{\Sigma\in\mathcal{S}}\lambda_{\min}\left(  \Sigma\right)
>0,\;\;s_{2}:=\sup_{\Sigma\in\mathcal{S}}\lambda_{\max}\left(  \Sigma\right)
<\infty.
\]
Then there exists $C>0$ depending on $s_{1},s_{2}$ but not on $d$ such that
for all $\Sigma_{1},\Sigma_{2}\in\mathcal{S}$
\begin{equation}
H\left(  N_{d}\left(  0,\Sigma_{1}\right)  ,N_{d}\left(  0,\Sigma_{2}\right)
\right)  \leq C\;\left\Vert \Sigma_{1}-\Sigma_{2}\right\Vert _{2}.
\label{conti-2}%
\end{equation}

\end{lemma}

\bigskip

\begin{lemma}
\label{lem-compact-1}Consider a set of normal distributions $\mathcal{P}%
=\left(  N_{d}\left(  0,\Sigma\right)  ,\Sigma\in\mathcal{S}\right)  $ where
$\mathcal{S}$ is compact in Hilbert-Schmidt metric and satisfies%
\begin{equation}
\inf_{\Sigma\in\mathcal{S}}\lambda_{\min}\left(  \Sigma\right)  >0.
\label{cond-min-eigenvalue-nonzero}%
\end{equation}
Then the mapping $\Sigma\rightarrow N_{d}\left(  0,\Sigma\right)  $ is
continuous on $\mathcal{S}$ in total variation metric.
\end{lemma}

\begin{proof}
By (\ref{Lecam-inequ}) we obtain for $\Sigma_{1},\Sigma_{2}\in\mathcal{S}$
\begin{equation}
\left\Vert N_{d}\left(  0,\Sigma_{1}\right)  -N_{d}\left(  0,\Sigma
_{2}\right)  \right\Vert _{TV}\leq H\left(  N_{d}\left(  0,\Sigma_{1}\right)
,N_{d}\left(  0,\Sigma_{2}\right)  \right)  . \label{conti-1}%
\end{equation}
Note that compactness of $\mathcal{S}$ implies%
\[
s_{2}=\sup_{\Sigma\in\mathcal{S}}\lambda_{\max}\left(  \Sigma\right)  \leq
\sup_{\Sigma\in\mathcal{S}}\left(  \mathrm{Tr\;}\Sigma^{2}\right)  ^{1/2}%
=\sup_{\Sigma\in\mathcal{S}}\left\Vert \Sigma\right\Vert _{2}<\infty
\]
so that (\ref{cond-min-eigenvalue-nonzero}) and Lemma
\ref{lem-covmatrices-cite-GNZpaper} imply the claim.
\end{proof}

\subsection{Geometric distribution\label{subsec:geometr-distr}}

Let $X$ be a r.v. with geometric law $\mathrm{Geo}\left(  p\right)  $ for
parameter $p\in\left(  0,1\right)  $, given by
\[
P\left(  X=k\right)  =\mathrm{Geo}\left(  p\right)  \left(  k\right)  =\left(
1-p\right)  p^{k}\text{, }k=0,1,\ldots
\]
As is well known, for a sequence of i.i.d. Bernoulli r.v's with success
probability $q=1-p,$ the r.v. $X$ is the number of failures before the first
success occurs ($X=0$ if success occurs in the first trial). Since
$\mathrm{Geo}\left(  p\right)  $, $p\in\left(  0,1\right)  $ forms an
exponential family, we refer to section 2.1. of \cite{MR1633574} for some
basic properties of that family. Setting $x=k$, the probabilities can be
written
\begin{align}
\left(  1-p\right)  p^{x}  &  =\exp\left(  x\log p+\log\left(  1-p\right)
\right) \nonumber\\
&  =\exp\left(  x\tau-V\left(  \tau\right)  \right)  =:q\left(  x,\tau\right)
\label{canon-from}%
\end{align}
for $\tau=\log p$, and
\[
V\left(  \tau\right)  =-\log\left(  1-p\right)  =-\log\left(  1-\exp
\tau\right)  .
\]
Thus (\ref{canon-from}) is the canonical form of the exponential family, and
$\tau\in\left(  -\infty,0\right)  $ the relevant parameter. The moments are,
noting that $\ V^{\prime\prime}\left(  \tau\right)  $ \ is also the Fisher
information $I\left(  \tau\right)  $,
\begin{align}
E_{\tau}X  &  =V^{\prime}\left(  \tau\right)  =\frac{\exp\tau}{1-\exp\tau
}=\frac{p}{1-p},\label{Expec-geo}\\
\mathrm{Var}\left(  X\right)   &  =V^{\prime\prime}\left(  \tau\right)
=\frac{\exp\tau\left(  1-\exp\tau\right)  +\left(  \exp\tau\right)  ^{2}%
}{\left(  1-\exp\tau\right)  ^{2}}\nonumber\\
&  =\frac{\exp\tau}{\left(  1-\exp\tau\right)  ^{2}}=I\left(  \tau\right)
=\frac{p}{\left(  1-p\right)  ^{2}}. \label{Var-geo}%
\end{align}
In connection with the representation of the thermal state $\mathfrak{N}%
_{1}\left(  0,a\right)  $, $a>1$ (cf. (\ref{thermal-state}) and
(\ref{Fock-repre-1})) we are interested in yet another parametrization of
$\mathrm{Geo}\left(  p\right)  $: setting $\mathbb{\;}p=\left(  a-1\right)
/\left(  a+1\right)  $, we obtain
\[
a=\frac{2}{1-p}-1=\frac{1+p}{1-p}.
\]
The canonical parameter $\tau$ then can be expressed as
\begin{equation}
\tau=\tau\left(  a\right)  =\log\left(  \left(  a-1\right)  /\left(
a+1\right)  \right)  . \label{tau-as-func-of-a}%
\end{equation}
We note
\begin{align}
\tau^{\prime}\left(  a\right)   &  =\frac{2}{a^{2}-1}\text{, }\tau
^{\prime\prime}\left(  a\right)  =-\frac{4a}{\left(  a^{2}-1\right)  ^{2}%
},\label{note-various-1}\\
V^{\prime}\left(  \tau\left(  a\right)  \right)   &  =\frac{a-1}{2}%
\text{,\ }V^{\prime\prime}\left(  \tau\left(  a\right)  \right)  =\frac
{a^{2}-1}{4}. \label{note-various-2}%
\end{align}
The fourth central moment of $X$ is, for $r=1-p$ \cite{Geometr-overview}
\[
E\left(  X-EX\right)  ^{4}=\frac{\left(  1-r\right)  \left(  r^{2}%
-9r+9\right)  }{r^{4}}\leq\frac{10}{\left(  1-p\right)  ^{4}}.
\]
In terms of parameter $a$ this bound is
\begin{equation}
E\left(  X-EX\right)  ^{4}\leq\frac{5}{8}\left(  a+1\right)  ^{4}.
\label{geometr-4th-central-moment}%
\end{equation}
In accordance with (\ref{canon-from}) and (\ref{tau-as-func-of-a}) the
geometric probability function, now parametrized by $a$, is
\begin{equation}
q\left(  x,\tau\left(  a\right)  \right)  =\exp\left(  x\tau\left(  a\right)
-V\left(  \tau\left(  a\right)  \right)  \right)  \text{, }x=0,1,\ldots
\label{geometric-pf-alternative}%
\end{equation}
Then the score function in this parametrization is
\begin{align}
s\left(  x,a\right)   &  :=\frac{\partial}{\partial a}\log q\left(
x,\tau\left(  a\right)  \right)  =\left(  x-V^{\prime}\left(  \tau\left(
a\right)  \right)  \right)  \tau^{\prime}\left(  a\right)
\label{score-func-def}\\
&  =\left(  x-\frac{a-1}{2}\right)  \frac{2}{a^{2}-1},\nonumber
\end{align}
and Fisher information is
\begin{align}
J\left(  a\right)   &  :=E_{a}s^{2}\left(  X,a\right)  =\frac{4}{\left(
a^{2}-1\right)  ^{2}}\mathrm{Var}_{a}\left(  X\right)
\label{Fish-info-geometric}\\
&  =\frac{4}{\left(  a^{2}-1\right)  ^{2}}\cdot\frac{a^{2}-1}{4}=\frac
{1}{a^{2}-1}. \label{Fish-info-geometric-a}%
\end{align}

\begin{lemma}
\label{lem-geometr-various-bounds}(i) If $1+C_{1}^{-1}\leq a\leq C_{1}$ for
some $C_{1}>0$ then for some $C_{2}$
\[
E_{a}s^{4}\left(  X,a\right)  \leq C_{2}.
\]
\newline(ii) We have
\begin{equation}
\frac{\partial^{2}}{\partial a^{2}}q^{1/2}\left(  x,\tau\left(  a\right)
\right)  =q^{1/2}\left(  x,\tau\left(  a\right)  \right)  \cdot\rho\left(
x,a\right)  \label{claim-decompose-2nd-deriv}%
\end{equation}
where $\rho\left(  x,a\right)  $ has the property: $1+C_{1}^{-1}\leq a\leq
C_{1}$ implies that for some $C_{3}$
\begin{equation}
\sum_{x=0}^{\infty}\left(  \frac{\partial^{2}}{\partial a^{2}}q^{1/2}\left(
x,\tau\left(  a\right)  \right)  \right)  ^{2}=E_{a}\rho^{2}\left(
X,a\right)  \leq C_{3}. \label{claim-squared-expec}%
\end{equation}

\end{lemma}

\begin{proof}
(i) By (\ref{score-func-def}) and (\ref{geometr-4th-central-moment})
\[
E_{a}s^{4}\left(  X,a\right)  =E_{a}\left(  X-E_{a}X\right)  ^{4}\frac
{2}{a^{2}-1}\leq\frac{5\cdot2^{4}}{8}\frac{\left(  a+1\right)  ^{4}}{a^{2}%
-1}=\frac{10}{\left(  a-1\right)  ^{4}}\leq C_{2}.
\]

(ii) We have
\begin{align*}
\frac{\partial}{\partial a}q^{1/2}\left(  x,\tau\left(  a\right)  \right)   &
=\frac{1}{2}q^{1/2}\left(  x,\tau\left(  a\right)  \right)  \frac{\partial
}{\partial a}\log q\left(  x,\tau\left(  a\right)  \right) \\
&  =\frac{1}{2}q^{1/2}\left(  x,\tau\left(  a\right)  \right)  \cdot\left(
x-V^{\prime}\left(  \tau\left(  a\right)  \right)  \right)  \tau^{\prime
}\left(  a\right)
\end{align*}
and hence
\begin{align*}
\frac{\partial^{2}}{\partial a^{2}}q^{1/2}\left(  x,\tau\left(  a\right)
\right)   &  =\frac{1}{4}q^{1/2}\left(  x,\tau\left(  a\right)  \right)
\cdot\left(  x-V^{\prime}\left(  \tau\left(  a\right)  \right)  \right)
^{2}\left(  \tau^{\prime}\left(  a\right)  \right)  ^{2}\\
&  -\frac{1}{2}q^{1/2}\left(  x,\tau\left(  a\right)  \right)  \cdot
V^{\prime\prime}\left(  \tau\left(  a\right)  \right)  \left(  \tau^{\prime
}\left(  a\right)  \right)  ^{2}\\
&  +\frac{1}{2}q^{1/2}\left(  x,\tau\left(  a\right)  \right)  \cdot\left(
x-V^{\prime}\left(  \tau\left(  a\right)  \right)  \right)  \tau^{\prime
\prime}\left(  a\right)  .
\end{align*}
Hence the l.h.s. of (\ref{claim-squared-expec}) is bounded by
\begin{align*}
&  E_{a}\left(  X-V^{\prime}\left(  \tau\left(  a\right)  \right)  \right)
^{4}\left(  \tau^{\prime}\left(  a\right)  \right)  ^{4}+\left(
V^{\prime\prime}\left(  \tau\left(  a\right)  \right)  \right)  ^{2}\left(
\tau^{\prime}\left(  a\right)  \right)  ^{4}\\
&  +E_{a}\left(  X-V^{\prime}\left(  \tau\left(  a\right)  \right)  \right)
^{2}\left(  \tau^{\prime\prime}\left(  a\right)  \right)  ^{2}.
\end{align*}
By (\ref{note-various-1}), (\ref{note-various-2}), the terms $V^{\prime\prime
}\left(  \tau\left(  a\right)  \right)  $, $\tau^{\prime}\left(  a\right)  $
and $\tau^{\prime\prime}\left(  a\right)  $ are all bounded when $1+C_{1}%
^{-1}\leq a\leq C_{1}$. It now suffices to prove that
\[
E_{a}\left(  X-V^{\prime}\left(  \tau\left(  a\right)  \right)  \right)
^{4}+E_{a}\left(  x-V^{\prime}\left(  \tau\left(  a\right)  \right)  \right)
^{2}%
\]
is bounded. The first term above is the fourth central moment of $X$ which is
bounded by (\ref{geometr-4th-central-moment}). The second term is the variance
of $X$, which is $V^{\prime\prime}\left(  \tau\left(  a\right)  \right)  $ by
(\ref{Var-geo}) and (\ref{note-various-2}) and thus bounded as well.
\end{proof}

\bigskip

\textbf{Estimation of parameter }$a$\textbf{.} From (\ref{Expec-geo}) and
(\ref{note-various-2}) we obtain
\begin{equation}
E_{\theta}\left(  2X+1\right)  =a. \label{expec-simple-est}%
\end{equation}
Setting $\hat{a}=2X+1$, we thus obtain an unbiased estimator of $a$ based on
one observation $X$. We also have by (\ref{Var-geo}) and (\ref{note-various-2}%
)
\begin{align}
\mathrm{Var}_{\tau}\left(  \hat{a}\right)   &  =4\mathrm{Var}_{\tau}\left(
X\right)  =4\frac{p}{\left(  1-p\right)  ^{2}}\nonumber\\
&  =a^{2}-1 \label{comput-p-expression}%
\end{align}
which is the inverse Fisher information $1/J\left(  a\right)  $ from
(\ref{Fish-info-geometric}). Hence $\hat{a}$ is best unbiased estimator of
$a$. If $\bar{X}_{n}$ is the mean of $n$ i.i.d. observations with law
$\mathrm{Geo}\left(  p\right)  $ then $\hat{a}_{n}=2\bar{X}_{n}+1$ is best
unbiased estimator of $a$, with variance $1/nJ\left(  a\right)  $.

\bigskip

\textbf{Asymptotically equivalent family. }The local approximating Gaussian
shift model (according to LAN theory) is
\begin{equation}
Y=a+n^{-1/2}\sqrt{\left(  a_{0}^{2}-1\right)  }\xi\label{LAN-geom}%
\end{equation}
where $\xi\sim N\left(  0,1\right)  $ and $a_{0}$ is the center of the
parametric neighborhood in $a$. Using (\ref{arccosh-def}), the variance-stable
form (cf. Section 3.3 of \cite{MR1633574}) is
\begin{equation}
Y=\mathrm{arc\cosh}\left(  a\right)  +n^{-1/2}\xi. \label{var-stable-form}%
\end{equation}
We can check this claim in the following way: setting
\[
g(a)=\mathrm{arc\cosh}\left(  a\right)  =\log\left(  a+\sqrt{a^{2}-1}\right)
,a>1
\]
we obtain by a computation
\begin{align*}
g^{\prime}\left(  a\right)   &  =\frac{1}{\left(  a^{2}-1\right)  ^{1/2}},\\
\left(  g^{\prime}\left(  a\right)  \right)  ^{2}  &  =\frac{1}{a^{2}%
-1}=J\left(  a\right)  .
\end{align*}
This means that at $n=1$ the geometric law $\mathrm{Geo}\left[  \left(
a-1\right)  /\left(  a+1\right)  \right]  $ and the Gaussian model $N\left(
g\left(  a\right)  ,1\right)  $ have the same Fisher information, which
implies that the model (\ref{var-stable-form}) is locally asymptotically
equivalent to the model of $n$ i.i.d. geometrics.

\subsection{Negative binomial distribution\label{subsec:neg-binom}}

The negative binomial distribution \textrm{NB}$\left(  r,p\right)  $ has
probability function, for $r>0$ and $p\in\left(  0,1\right)  $
\[
\mathrm{NB}\left(  r,p\right)  \left(  k\right)  =P\left(  X=k\right)
=\frac{\Gamma\left(  k+r\right)  }{k!\Gamma\left(  r\right)  }(1-p)^{r}p^{k}.
\]
For $r=1$ the geometric distribition \textrm{Geo}$\left(  p\right)  $ is
obtained. Setting $x=k$, the probabilities can be written
\begin{equation}
\frac{\Gamma\left(  x+r\right)  }{x!\Gamma\left(  r\right)  }\left(
1-p\right)  ^{r}p^{x}=\exp\left(  x\log p\right)  h_{r}\left(  x\right)
\left(  1-p\right)  ^{r} \label{neg-binomial-as-expon-family}%
\end{equation}
for $h_{r}\left(  x\right)  =$ $\Gamma\left(  x+r\right)  /x!\Gamma\left(
r\right)  $. This shows that for fixed $r$, is $\mathrm{NB}\left(  r,p\right)
$ is an exponential family in the parameter $p$ (with natural parameter
$\tau=\log p\in\left(  -\infty,0\right)  $). Expectation and variance are%
\[
EX=\frac{rp}{1-p},\;\mathrm{Var}\left(  X\right)  =\frac{rp}{\left(
1-p\right)  ^{2}}%
\]
and the characteristic function is
\begin{equation}
\phi\left(  t\right)  =\left(  \frac{1-p}{1-p\exp\left(  it\right)  }\right)
^{r}\text{, }t\in\mathbb{R}\text{.} \label{char-func-NB}%
\end{equation}
The distribution can be represented as a Gamma-Poisson mixture: if
\textrm{Gam}$\left(  s,r\right)  $ is the Gamma distribution with scale
parameter $s$ and shape parameter $r$, having density
\[
f_{s,r}\left(  x\right)  =\frac{x^{r-1}s^{r}}{\Gamma\left(  r\right)  }%
\exp\left(  -xs\right)  \text{, }x\geq0,
\]
and \textrm{Po}$\left(  \lambda\right)  \left(  k\right)  =\exp\left(
-\lambda\right)  \lambda^{k}/k!$ is the Poisson probability function then%
\begin{equation}
\mathrm{NB}\left(  r,p\right)  \left(  k\right)  =\int_{0}^{\infty}%
\mathrm{Po}\left(  \lambda\right)  \left(  k\right)  f_{s,r}\left(
\lambda\right)  d\lambda\text{ for }s=\left(  1-p\right)  /p.
\label{Gamma-Poisson-mixture-NB}%
\end{equation}
Relation (\ref{char-func-NB}) implies that $\mathrm{NB}\left(  r,p\right)  $
is infinitely divisible; equivalently , if $X_{1},\ldots,X_{n}$ are i.i.d.
$\mathrm{NB}\left(  r,p\right)  $ then
\begin{equation}
\sum_{j=1}^{n}X_{i}\sim\mathrm{NB}\left(  nr,p\right)  .
\label{sum-of-neg-binomials}%
\end{equation}
Moreover, if $X_{1},\ldots,X_{n}$ follow a parametric model as i.i.d.
$\mathrm{NB}\left(  r,p\right)  $, $p\in\left(  0,1\right)  $, then by the
exponential family representation (\ref{neg-binomial-as-expon-family}),
$\sum_{j=1}^{n}X_{i}$ is a sufficient statistic.

\begin{lemma}
\label{Lem-distance-neg-binomials}(i) Let $a_{1},a_{2}>0$ and $p_{j}=\left(
a_{j}-1\right)  /\left(  a_{j}+1\right)  $, $j=1,2$. Then for any $r>0$ we
have
\[
H^{2}\left(  \mathrm{NB}\left(  r,p_{1}\right)  ,\mathrm{NB}\left(
r,p_{2}\right)  \right)  \leq\frac{r\left(  a_{1}-a_{2}\right)  ^{2}}{\left(
a_{1}-1\right)  \left(  a_{2}-1\right)  }.
\]
(ii) Let $r_{1},r_{2}>0$. Then for any $p\in\left(  0,1\right)  $ we have
\[
H^{2}\left(  \mathrm{NB}\left(  r_{1},p\right)  ,\mathrm{NB}\left(
r_{2},p\right)  \right)  \leq1-\frac{\Gamma\left(  \left(  r_{1}+r_{2}\right)
/2\right)  }{\Gamma^{1/2}\left(  r_{1}\right)  \Gamma^{1/2}\left(
r_{2}\right)  }.
\]

\end{lemma}

\begin{proof}
(i) The mixture (\ref{Gamma-Poisson-mixture-NB}) represents the operation of a
stochastic kernel on $\mathrm{Gam}\left(  s,r\right)  $. Then it suffices to
prove, for $s_{j}=\left(  1-p_{j}\right)  /p_{j}=2/\left(  a_{j}-1\right)  $,
$j=1,2,$ that
\[
H^{2}\left(  \mathrm{Gam}\left(  s_{1},r\right)  ,\mathrm{Gam}\left(
s_{2},r\right)  \right)  \leq\frac{r}{\left(  a_{1}-1\right)  \left(
a_{2}-1\right)  }\left(  a_{1}-a_{2}\right)  ^{2}%
\]
(cf. \cite{MR2421720}, Problem 1.72). The squared Hellinger distance can be
bounded by the Kullback-Leiber relative entropy $K\left(  \cdot.\cdot\right)
$ (cf. \cite{MR2724359} ):%
\begin{align*}
H^{2}\left(  \mathrm{Gam}\left(  s_{1},r\right)  ,\mathrm{Gam}\left(
s_{2},r\right)  \right)   &  \leq K\left(  \mathrm{Gam}\left(  s_{1},r\right)
,\mathrm{Gam}\left(  s_{2},r\right)  \right)  =\int_{0}^{\infty}f_{s_{1}%
,r}\left(  x\right)  \log\frac{f_{s_{1},r}\left(  x\right)  }{f_{s_{2}%
,r}\left(  x\right)  }dx\\
&  =\int_{0}^{\infty}\frac{x^{r-1}s_{1}^{r}}{\Gamma\left(  r\right)  }\left[
\left(  -x\left(  s_{1}-s_{2}\right)  \right)  +\log\left(  \frac{s_{1}}%
{s_{2}}\right)  ^{r}\right]  dx\\
&  =-\left(  s_{1}-s_{2}\right)  \int_{0}^{\infty}\frac{x^{r}s_{1}^{r}}%
{\Gamma\left(  r\right)  }dx+r\log\frac{s_{1}}{s_{2}}\\
&  =-\frac{\left(  s_{1}-s_{2}\right)  \Gamma\left(  r+1\right)  }{s_{1}%
\Gamma\left(  r\right)  }\int_{0}^{\infty}\frac{x^{r}s_{1}^{r+1}}%
{\Gamma\left(  r+1\right)  }dx+r\log\frac{s_{1}}{s_{2}}\\
&  =-r\left(  1-\frac{s_{2}}{s_{1}}\right)  -r\log\frac{s_{2}}{s_{1}}.
\end{align*}
The well-known inequality
\[
\log x\geq\frac{x-1}{x}\text{, }x>0
\]
applied for $x=s_{2}/s_{1}=\left(  a_{1}-1\right)  /\left(  a_{2}-1\right)  $
implies%
\begin{align*}
H^{2}\left(  f_{s_{1},r},f_{s_{2},r}\right)   &  \leq-r\left(  \left(
1-x\right)  +\frac{x-1}{x}\right)  =r\frac{\left(  x-1\right)  ^{2}}{x}\\
&  =r\frac{\left(  a_{1}-a_{2}\right)  ^{2}}{\left(  a_{1}-1\right)  \left(
a_{2}-1\right)  }.
\end{align*}
(ii) Again it suffices to prove the bound for the respective Gamma laws, i.e.
for $s=\left(  1-p\right)  /p$
\begin{align*}
H^{2}\left(  \mathrm{Gam}\left(  s,r_{1}\right)  ,\mathrm{Gam}\left(
s,r_{2}\right)  \right)   &  =1-\int_{0}^{\infty}f_{s_{1},r}^{1/2}\left(
x\right)  f_{s_{2},r}^{1/2}\left(  x\right)  dx\\
&  =1-\frac{1}{\Gamma^{1/2}\left(  r_{1}\right)  \Gamma^{1/2}\left(
r_{2}\right)  }\int_{0}^{\infty}x^{\left(  r_{1}+r_{2}\right)  /2-1}s^{\left(
r_{1}+r_{2}\right)  /2}\exp\left(  -xs\right)  dx\\
&  =1-\frac{\Gamma\left(  \left(  r_{1}+r_{2}\right)  /2\right)  }%
{\Gamma^{1/2}\left(  r_{1}\right)  \Gamma^{1/2}\left(  r_{2}\right)  }.
\end{align*}

\end{proof}

\subsection{A covariance formula for Gaussians\label{subsec: cov}}

Let $\left(  X,Y\right)  $ have a bivariate normal distribution
\[
\left(
\begin{array}
[c]{c}%
X\\
Y
\end{array}
\right)  \sim N_{2}\left(  0,\Sigma\right)  \text{ where }\Sigma=\left(
\begin{array}
[c]{cc}%
\sigma_{x}^{2} & \sigma_{xy}\\
\sigma_{xy} & \sigma_{y}^{2}%
\end{array}
\right)  .
\]
Then
\begin{align}
EX^{2}Y^{2}  &  =2\sigma_{xy}^{2}+\sigma_{x}^{2}\sigma_{y}^{2}%
,\label{covform-1}\\
\mathrm{Cov}\left(  X^{2},Y^{2}\right)   &  =2\sigma_{xy}^{2}.
\label{covform-2}%
\end{align}

\begin{proof}
Consider the well-known regression representation
\[
Y=\beta X+\eta\text{ where }\eta\sim N\left(  0,\sigma_{y}^{2}-\frac
{\sigma_{xy}^{2}}{\sigma_{x}^{2}}\right)  \text{, }\beta=\frac{\sigma_{xy}%
}{\sigma_{x}^{2}}%
\]
and $\eta$ is independent of $X$. Then
\begin{align*}
EX^{2}Y^{2}  &  =EX^{2}\left(  \beta X+\eta\right)  ^{2}=EX^{2}\left(
\beta^{2}X^{2}+2\beta X\eta+\eta^{2}\right) \\
&  =\beta^{2}EX^{4}+EX^{2}E\eta^{2}=\frac{\sigma_{xy}^{2}}{\sigma_{x}^{4}%
}3\sigma_{x}^{4}+\sigma_{x}^{2}\left(  \sigma_{y}^{2}-\frac{\sigma_{xy}^{2}%
}{\sigma_{x}^{2}}\right) \\
&  =3\sigma_{xy}^{2}+\sigma_{x}^{2}\sigma_{y}^{2}-\sigma_{xy}^{2}=2\sigma
_{xy}^{2}+\sigma_{x}^{2}\sigma_{y}^{2}.
\end{align*}
This proves (\ref{covform-1}). Then (\ref{covform-2}) follows immediately by
\[
\mathrm{Cov}\left(  X^{2},Y^{2}\right)  =EX^{2}Y^{2}-EX^{2}EY^{2}=EX^{2}%
Y^{2}-\sigma_{x}^{2}\sigma_{y}^{2}.
\]

\end{proof}

\bigskip

\bigskip
\bibliographystyle{alpha}
\bibliography{bib-GIQTS-v7}
%

\begin{privatenotes}%

\section{Discussion of technical details (not part of the paper)}

\subsection{More on the subsection "Quantum Le Cam distance"}

\textit{We follow \cite{MR2346393} for defining the quantum analog of the
}$\Delta$\textit{-distance (\ref{Delta-distance}). So far the Gaussian states
}$\mathfrak{N}_{n}\left(  0,A\right)  $ \textit{ have been defined on the
}$C^{\ast}$\textit{-algebra }$CCR\left(  \mathbb{R}^{2n},\Delta\right)
$\textit{, using the Schr\"{o}dinger representation, but it is technically
convenient to define them as states on a certain von Neumann algebra. To this
end consider the Fock representation of }$CCR\left(  \mathbb{R}^{2n}%
,\Delta\right)  $\textit{ (\cite{MR1441540}, 5.2.8) where the Weyl unitaries
}$V\left(  x\right)  $\textit{, }$x\in C^{n}$\textit{ of
(\ref{Weyl-unitaries-complex-indexed}) act as operators on the symmetric Fock
space }$\mathfrak{F}\left(  \mathbb{C}^{n}\right)  $\textit{ (see XX below).
It can be shown that in this representation, any centered Gaussian state
}$\varphi$\textit{ has a density operator }$\rho_{\varphi}$\textit{ on
}$\mathfrak{F}\left(  \mathbb{C}^{n}\right)  $\textit{ (a positive operator
with unit trace) such that }$\varphi\left(  V\left(  x\right)  \right)
=Tr\;\rho_{\varphi}V\left(  x\right)  $\textit{, }$x\in\mathbb{C}^{n}%
$\textit{. Then it is straightforward to extend }$\varphi$\textit{ to the von
Neumann algebra generated by }$\left\{  V\left(  x\right)  ,x\in\mathbb{C}%
^{n}\right\}  $\textit{. For details of this argument, cf. Section
\ref{subsec:symmetric-Fock-space} below.}

In \cite{MR1441540}, p.24, a \textbf{regular }state $\omega$ over a
CCR-algebra $\mathfrak{A}(\mathfrak{h})$ is defined, where $\mathfrak{h}$ is a
complex Hilbert space. Then it is stated that $\omega$ is regular if and only
if $t\in\mathbb{R\mapsto\omega}\left(  V\left(  tf\right)  \right)
\in\mathbb{C}$ is continuous for all $f\in\mathfrak{h}$. (We write $V$ here
for the Weyl unitaries, written $W$ in the book). Now obviously, for
$\mathfrak{h=}\mathbb{C}^{n}$, if $\omega$ is a Gaussian state $\omega
=\varphi\left(  0,\Sigma\right)  $ then for every $f\in\mathbb{C}^{n}$ the
expression
\[
\mathbb{\omega}\left(  V\left(  tf\right)  \right)  =\exp\left(  -\frac{1}%
{2}\left\langle \underline{tf},\Sigma\underline{tf}\right\rangle \right)
=\exp\left(  -\frac{t^{2}}{2}\left\langle \underline{f},\Sigma\underline
{f}\right\rangle \right)
\]

is continuous in $t$. Thus $\varphi\left(  0,\Sigma\right)  $ is regular.

Consider an irreducible sub-C*-algebra $\mathfrak{A}$ of $\mathcal{L}\left(
\mathfrak{H}\right)  $ (here $\mathcal{L}\left(  \mathfrak{H}\right)  $ is the
C*-algebra of bounded operators on a Hilbert space $\mathfrak{H}$). In section
2.6.2. of \cite{MR887100} and \cite{MR1441540}, p. 26 a state $\omega$ on
$\mathfrak{A}$ is called \textbf{normal} if it is determined by a density
operator $\rho$ in the form
\[
\omega\left(  A\right)  =\mathrm{Tr\;}\rho A,\;A\in\mathfrak{A.}%
\]
Let's check on irreducibility. According to Definition 2.3.7, p. 47 of
\cite{MR887100}, a representation $\left(  \mathfrak{H},\pi\right)  $ of a
C*-algebra $\mathfrak{A}$ is \textbf{irreducible} if the set of bounded
operators $\pi\left(  \mathfrak{A}\right)  $ is irreducible on $\mathfrak{H}$.
A set $\mathfrak{M}$ of operators on $\mathfrak{H}$ is irreducible if the only
closed subspaces of $\mathfrak{H}$ which are invariant under the action of
$\mathfrak{M}$ are the trivial subspaces $\left\{  0\right\}  $ and
$\mathfrak{H}$.

Holevo does not use the term C*-algebra but speaks of the "CCR", by which he
means the set of Weyl unitaries. A representation of the CCR then is also
given by $\left(  \mathfrak{H},\pi\right)  $ where if $V(x)$ are the elements
of the CCR then $\pi\left(  V(x)\right)  $ act as operators on the Hilbert
space $\mathfrak{H}$. On p. 98 of \cite{MR2797301} it is argued that the
Sch\"{o}dinger representation of the CCR is irreducible. On p.88 of
\cite{MR2797301} irreducibility of a representation of the CCR\ is defined as
irreducibility (as defined above) of the set of operators $\mathfrak{M}%
=\left\{  \pi\left(  W(t)\right)  ,t\in\mathbb{R}^{2n}\right\}  $. We take
this as equivalent to the irreducibility of the whole C*-algebra generated by
this set, i.e. the representation of $CCR\left(  \mathbb{R}^{2n}%
,\Delta\right)  $\textit{,} (\textit{almost obvious,} \textit{but} \textit{not
checked)}. Also. in \cite{MR2797301}, Theorem 3.4.1, p. 98 the Stone-von
Neumann uniqueness theorem is stated as follows: any two irreducible
representations of the CCR are unitarily equivalent: if $\left(
\mathfrak{H}_{i},\pi_{i}\right)  $, $i=1,2$ are irreducible representations
then there is a linear isometric map $U$ of $\mathfrak{H}_{2}$ onto
$\mathfrak{H}_{1}$ such that
\begin{equation}
\pi_{2}\left(  W(t)\right)  =U^{\ast}\pi_{1}\left(  W(t)\right)  U\text{,
}t\in\mathbb{R}^{2n}. \label{equivalence-isometric-CCR}%
\end{equation}

Another irreducible representation of the CCR is the Fock representation:
\cite{MR1441540}, Proposition 5.2.4 (3), p. 13. There the Weyl operators are
defined via creation and annihilation operators, but they can be also be
defined by their action on exponential vectors, see \cite{MR2510896}, p.
032105-5, or \cite{MR3012668}, (20.2), p. 135. A discussion of the resulting
equivalence map $U$ for multimode Schr\"{o}dinger and Fock representation on
$\mathfrak{F}\left(  \mathbb{C}^{n}\right)  $ (\ref{equivalence-isometric-CCR}%
) can be found in \cite{MR1441540}, Example 5.2.16, p. 35; for the one mode
Schr\"{o}dinger and Fock representation on $\mathfrak{F}\left(  \mathbb{C}%
\right)  $ see also Exercise 20.20, p. 149 in \cite{MR3012668}.

According to Corollary 5.2.15 in \cite{MR1441540} (another variant of the
Stone-von Neumann uniqueness theorem), each regular state $\omega$ over
$\mathfrak{A}(\mathfrak{h})$ is normal with respect to the Fock
representation. Hence $\omega$ has a density operator $\rho_{\omega}$ on the
Fock space $\mathfrak{F}\left(  \mathbb{C}^{n}\right)  $; in the notation of
(\ref{equivalence-isometric-CCR}), if $\left(  \mathfrak{H}_{2},\pi
_{2}\right)  $ is the Fock representation
\[
\omega\left(  \pi_{2}\left(  W(t)\right)  \right)  =\mathrm{Tr\;}\rho_{\omega
}\pi_{2}\left(  W(t)\right)  ,\;\text{ }t\in\mathbb{R}^{2n}\mathfrak{.}%
\]
But then it follows
\begin{align*}
\omega\left(  \pi_{2}\left(  W(t)\right)  \right)   &  =\mathrm{Tr\;}%
\rho_{\omega}U^{\ast}\pi_{1}\left(  W(t)\right)  U\\
&  =\mathrm{Tr\;}U\rho_{\omega}U^{\ast}\pi_{1}\left(  W(t)\right)  \text{,
}\;\text{ }t\in\mathbb{R}^{2n}%
\end{align*}
and $U\rho_{\omega}U^{\ast}$ is a density operator on $\mathfrak{H}_{1}$. If
$\left(  \mathfrak{H}_{1},\pi_{1}\right)  $ is the multimode Schr\"{o}dinger
representation (with $\mathfrak{H}_{1}=L_{2}\left(  \mathbb{R}^{n}\right)  $)
then we have shown that any centered Gaussian state on $CCR\left(
\mathbb{R}^{2n},\Delta\right)  $ has a density operator on $L_{2}\left(
\mathbb{R}^{n}\right)  $. For gauge invariant states, the explicit form on
$\mathfrak{F}\left(  \mathbb{C}^{n}\right)  $ is described in this paper
(cited by Mosonyi \cite{MR2510896} ; we prove it for lack of reference).

\bigskip

\subsection{More on Fock operators (regarding subsec
\ref{subsec:symmetric-Fock-space} and \ref{sec:geom-reg-model})}

Let $y_{j}\in\vee^{m}\mathcal{H}$; then the Fock space $\mathfrak{F}\left(
\mathcal{H}\right)  $ consists of those $y=\left(  y_{j}\right)  _{j\geq0}$
where $\sum_{j\geq0}\left\Vert y_{j}\right\Vert ^{2}<\infty$. For
$B\in\mathcal{B}\left(  \mathcal{H}\right)  $ let
\[
\left\vert B\right\vert =\lambda_{\max}\left(  B^{\ast}B\right)
\]
be the operator norm. The Fock operator $B_{F}$ need not be bounded even if
$B$ is bounded: we have
\[
\left\vert \vee^{m}B\right\vert \leq\left\vert B^{\otimes m}\right\vert
=\left\vert B\right\vert ^{m}%
\]
(where the inequality is actually an equality) and hence for $y\in
\mathfrak{F}\left(  \mathcal{H}\right)  $
\[
\left\Vert B_{F}y\right\Vert ^{2}=\sum_{j\geq0}\left\Vert \left(  \vee
^{m}B\right)  y_{j}\right\Vert ^{2}\leq\sum_{j\geq0}\left\vert B\right\vert
^{m}\left\Vert y_{j}\right\Vert ^{2}%
\]
which need not be finite even if $\sum_{j\geq0}\left\Vert y_{j}\right\Vert
^{2}<\infty$. However if $\left\vert B\right\vert \leq1$ then the above is
finite for every $y\in\mathfrak{F}\left(  \mathcal{H}\right)  $, hence $B_{F}$
is bounded. When $A\in\mathcal{B}\left(  \mathcal{H}\right)  $ also
fulf\'{\i}lls $\left\vert A\right\vert \leq1$ then $A_{F}$, $B_{F}$ and
$\left(  AB\right)  _{F}$ are bounded and the relation
(\ref{multiplying-Fock-ops})%
\[
\left(  AB\right)  _{F}=A_{F}B_{F}%
\]
holds everywhere on $\mathfrak{F}\left(  \mathcal{H}\right)  $.

To apply this to the diagonalization of density operators in Fock space,
recall that it has been shown in subsection
\ref{subsec:density-op-GIV-Gaussians} that
\[
\mathfrak{N}_{n}\left(  0,A\right)  =\frac{2^{n}}{\det\left(  I+A\right)
}\left(  \frac{A-I}{A+I}\right)  _{F}%
\]
is a density operator, hence bounded; we also see that
\[
\left\vert \frac{A-I}{A+I}\right\vert \leq1.
\]
Furthermore, for any unitary $U\in\mathcal{B}\left(  \mathcal{H}\right)  $ we
have $\left\vert U\right\vert =1$, hence $U_{F}$ is bounded as well. Then, if
$UAU^{\ast}=\Lambda$ we have as a consequence of (\ref{multiplying-Fock-ops})%
\[
U_{F}\mathfrak{N}_{n}\left(  0,A\right)  U_{F}^{\ast}=\frac{2^{n}}{\det\left(
I+\Lambda\right)  }\left(  \frac{\Lambda-I}{\Lambda+I}\right)  _{F}.
\]

\subsection{On quantum relative entropy}

The main technical point we had to consider here is whether the expression
\begin{equation}
S\left(  \rho||\sigma\right)  =\mathrm{Tr}\;\rho\left(  \log\rho-\log
\sigma\right)  \label{basic-rela-entrop-formula}%
\end{equation}
for our faithful Gaussian states is in fact a special case of the definition
of $S\left(  \rho||\sigma\right)  $ for normal states on a von Neumann
algebra. The latter is given in two different, equivalent ways in Ohya, Petz
\cite{MR1230389}, Chap 5 and Bratteli, Robinson II \cite{MR1441540}, Def.
6.2.29. Both references discuss the special case of finite dimensional density
matrices, and Ohya, Petz also discuss the commutative special case (the
Kullback-Leibler distance between probability densities). However neither
reference clearly addresses the special case of density operators on a Hilbert
space \footnote{even though Bratteli, Robinson give the above formula for
density operators on p. 267, but this concerns the special case of density
operators on a "quantum spin system", and we did not check whether this is
general enough}. We need to be sure that the case of density operators on an
infinite dimensional Hilbert space is covered by the general von Neumann
setting, because the quantum Pinsker inequality
(\ref{inequality-trace-norm-relative-entropy}) is derived in Ohya, Petz
\cite{MR1230389} in the general abstract setting. \textit{At this point it
seemed too technical to us} to deal with the abstract definition of $S\left(
\rho||\sigma\right)  $ and to derive the formula
(\ref{basic-rela-entrop-formula}) for density operators on an infinite
dimensional Hilbert space.

In the literature which derives relative entropy expressions for Gaussian
states, the formula (\ref{basic-rela-entrop-formula}) seems to be taken for
granted, and we could not find any reference to the von Neumann algebra
setting there. Fortunately, the book of Petz (\cite{MR2363070}, sec 3.4) gives
an essential clue as to the validity of (\ref{basic-rela-entrop-formula}).
This is done in the setting of general Hilbert space $\mathcal{H}$, where the
dimensionality is not clearly spelled out. But there $S\left(  \rho
||\sigma\right)  $ is defined by means of a \textit{relative modular
operator}, a term which also figures in the abstract definition in Bratteli,
Robinson. This is defined as follows.

The set of bounded operators $B\left(  \mathcal{H}\right)  $ becomes a Hilbert
space when the Hilbert-Schmidt inner product
\begin{equation}
\left\langle A,B\right\rangle =\mathrm{Tr\;}A^{\ast}B \label{indicate-HS}%
\end{equation}
is considered. (By the way, this shows that at the moment $\mathcal{H}$ is
assumed finite dimensional, because in general not all bounded operators are
Hilbert-Schmidt). On the Hilbert space $B\left(  \mathcal{H}\right)  $, one
can define an operator $\Delta\left(  \rho_{2}/\rho_{1}\right)  =\Delta$ as
\[
\Delta a=\rho_{2}a\rho_{1}^{-1}\text{, }a\in B\left(  \mathcal{H}\right)  .
\]
assuming $\rho_{1}$ invertible. This is the relative modular operator; it is
the product of two commuting operators $\Delta=LR$, where
\[
La=\rho_{2}a\text{ and }Ra=a\rho_{1}^{-1}\text{, }a\in B\left(  \mathcal{H}%
\right)  .
\]
It is then claimed that since $\log\Delta=\log L+\log R$, we have
\begin{equation}
\mathrm{Tr}\;\rho_{1}\left(  \log\rho_{1}-\log\rho_{2}\right)  =-\left\langle
\rho_{1}^{1/2},\left(  \log\Delta\right)  \rho_{1}^{1/2}\right\rangle .
\label{Araki-rela-entropy}%
\end{equation}
It is then stated that the above r.h.s. is Araki's definition of the relative
entropy in a general von Neumann algebra.

We can ascertain (\ref{Araki-rela-entropy}) for the finite dimensional case in
a slightly different framework. For an $d\times d$ matrix $A$, write the
vectorizing (or stacking) operation as $vec\left(  A\right)  =\overrightarrow
{A}$, which gives the vector of all the column vectors stacked one below the
other. Then one has
\begin{align*}
\left\langle A,B\right\rangle  &  =\overrightarrow{A}^{\ast}\overrightarrow
{B},\\
\overrightarrow{La}  &  =\overrightarrow{\rho_{2}a}=\left(  I\otimes\rho
_{2}\right)  \overrightarrow{a}\text{, }\overrightarrow{Ra}=\overrightarrow
{a\rho_{1}^{-1}}=\left(  \rho_{1}^{-1}\otimes I\right)  \overrightarrow{a}.
\end{align*}
The Hilbert space $B\left(  \mathcal{H}\right)  $ for $\mathcal{H}%
=\mathbb{C}^{d}$ of matrices is thus "vectorized", i.e. represented as
$\mathbb{C}^{d^{2}}$, the operators $L$ and $R$ become $L=\left(  I\otimes
\rho_{2}\right)  $ and $R=\left(  \rho_{1}^{-1}\otimes I\right)  $. These
clearly commute, and the relative modular operator is
\[
\Delta=LR=\rho_{1}^{-1}\otimes\rho_{2}.
\]
It is also clear that
\[
\log\Delta=\log L+\log R
\]
Then we can write the r.h.s. of (\ref{Araki-rela-entropy}) as
\begin{align*}
-\left\langle \rho_{1}^{1/2},\left(  \log\Delta\right)  \rho_{1}%
^{1/2}\right\rangle  &  =-\overrightarrow{\rho_{1}^{1/2}}^{\ast}\left(
\log\Delta\right)  \overrightarrow{\rho_{1}^{1/2}}\\
&  =-\overrightarrow{\rho_{1}^{1/2}}^{\ast}\left(  \log\left(  I\otimes
\rho_{2}\right)  +\log\left(  \rho_{1}^{-1}\otimes I\right)  \right)
\overrightarrow{\rho_{1}^{1/2}}\\
&  =-\overrightarrow{\rho_{1}^{1/2}}^{\ast}\left(  \left(  I\otimes\log
\rho_{2}\right)  -\left(  \log\rho_{1}\otimes I\right)  \right)
\overrightarrow{\rho_{1}^{1/2}}\\
&  =-\overrightarrow{\rho_{1}^{1/2}}^{\ast}\left(  \overrightarrow{\left(
\log\rho_{2}\right)  \rho_{1}^{1/2}}-\overrightarrow{\rho_{1}^{1/2}\left(
\log\rho_{1}\right)  }\right) \\
&  =\mathrm{Tr\;}\rho_{1}\log\rho_{1}-\mathrm{Tr\;}\rho_{1}\log\rho_{2}.
\end{align*}
Thus (\ref{Araki-rela-entropy}) is shown for $\mathcal{H}=\mathbb{C}^{d}$. If
$\mathcal{H}=l_{2}$ then operators are infinite matrices, and the "stacking"
or "vectorizing" operation has to be replaced by identification with the
Hilbert space of Hilbert-Schmidt operators, as indicated in Petz
\cite{MR2363070} with (\ref{indicate-HS}). Note that $\rho_{j}^{1/2}$ are
indeed Hilbert-Schmidt operators. There is no doubt that formula
(\ref{basic-rela-entrop-formula}) then holds as well.

What remains to verify is that the relative modular operator $\Delta$ defined
here coincides with the abstract definition of Araki (cf. Bratteli, Robinson
\cite{MR1441540}) or with the "spatial derivative" $\Delta$ used in Ohya, Petz
\cite{MR1230389}. This appears to be a technical exercise which we can leave
to a later stage of writing.

\subsection{Markov kernel randomization of a binary experiment}

Consider probability measures $Q_{i}$, $i=0,1$ on a measurable space $\left(
X,\Omega_{X}\right)  $. Let $\left(  Y,\Omega_{Y}\right)  $ be a measurable
space and $\kappa$ be a Markov kernel $\kappa:\Omega_{Y}\times w\rightarrow
\left[  0,1\right]  $ for $w\in X$. Consider extensions $\tilde{Q}_{j}%
=\kappa\times Q_{j}$ of the probability measures $Q_{j}$ onto $\Omega
_{Y}\times\Omega_{X}$ in the following way: for any p.m. $Q$ on $\left(
X,\Omega_{X}\right)  $ we set
\begin{equation}
\tilde{Q}_{j}\left(  B_{1}\times B_{2}\right)  =\left(  \kappa\times
Q_{j}\right)  \left(  B_{1}\times B_{2}\right)  =\int_{B_{2}}\kappa\left(
B_{1},w\right)  Q_{j}\left(  dw\right)  \text{, }B_{1}\in\Omega_{Y}\text{,
}B_{2}\in\Omega_{X}\text{.} \label{extensions-def}%
\end{equation}
Assume $Q_{j}\ll Q_{0}$, $\;j=1,\ldots,N$ and let $q=dQ_{1}/dQ_{0}$,
$\;j=1,\ldots,N$; then it is easy to see that $\tilde{Q}_{j}\ll\tilde{Q}_{0}$
and
\[
\frac{d\tilde{Q}_{1}}{d\tilde{Q}_{0}}\left(  y,w\right)  =q\left(  w\right)
\text{, }j=1,\ldots,N\text{, a.s. }\left[  \tilde{Q}_{0}\right]  \text{.}%
\]
To prove that, define the measure $Q_{1}^{\ast}$ on $\Omega_{Y}\times
\Omega_{X}$ by
\[
Q_{1}^{\ast}\left(  B_{1}\times B_{2}\right)  =\int_{B_{1}\times B_{2}%
}q\;d\tilde{Q}_{0}\text{, }B_{1}\in\Omega_{Y}\text{, }B_{2}\in\Omega
_{X}\text{.}%
\]
Let $B_{1}^{\ast}=B_{1}\times X$, $B_{1}\in\Omega_{Y}$ and $B_{2}^{\ast
}=Y\times B_{2}$, $B_{2}\in\Omega_{X}$ be the two cylinder sets and let
$\Omega_{Y}^{\ast}$ and $\Omega_{X}^{\ast}$ be the two sub-sigma-algebras of
$\Omega_{Y}\times\Omega_{X}$ ; we will write $E^{P}$ for the expectation
corresponding to a measure $P$. Now
\begin{align*}
Q_{1}^{\ast}\left(  B_{1}\times B_{2}\right)   &  =E^{\tilde{Q}_{0}}%
\mathbf{1}_{B_{1}\times B_{2}}q=E^{\tilde{Q}_{0}}\mathbf{1}_{B_{1}^{\ast}%
}\mathbf{1}_{B_{2}^{\ast}}q=E^{\tilde{Q}_{0}}E^{\tilde{Q}_{0}}\left(
\mathbf{1}_{B_{1}^{\ast}}\mathbf{1}_{B_{2}^{\ast}}q|\Omega_{X}^{\ast}\right)
\\
&  =E^{\tilde{Q}_{0}}\mathbf{1}_{B_{2}^{\ast}}qE^{\tilde{Q}_{0}}\left(
\mathbf{1}_{B_{1}^{\ast}}|\Omega_{X}^{\ast}\right)
\end{align*}
since $\mathbf{1}_{B_{2}^{\ast}}q$ is $\Omega_{X}^{\ast}$-measurable. Now
\[
E^{\tilde{Q}_{0}}\left(  \mathbf{1}_{B_{1}^{\ast}}|\Omega_{X}^{\ast}\right)
=\tilde{Q}_{0}\left(  B_{1}^{\ast}|y,w\right)  =\kappa\left(  B_{1},w\right)
\text{ a.s. }\left[  \tilde{Q}_{0}\right]
\]
hence
\[
Q_{1}^{\ast}\left(  B_{1}\times B_{2}\right)  =E^{\tilde{Q}_{0}}%
\mathbf{1}_{B_{2}^{\ast}}\left(  Y,W\right)  q\left(  W\right)  \kappa\left(
B_{1},W\right)
\]
where $W$ has the marginal distribution under $\tilde{Q}_{0}$, i.e. $W\sim
Q_{0}$. We also note $\mathbf{1}_{B_{2}^{\ast}}\left(  Y,W\right)
=\mathbf{1}_{B_{2}}\left(  W\right)  $. Hence
\begin{align*}
Q_{1}^{\ast}\left(  B_{1}\times B_{2}\right)   &  =E^{Q_{0}}\mathbf{1}_{B_{2}%
}\left(  W\right)  q\left(  W\right)  \kappa\left(  B_{1},W\right) \\
&  =\int_{B_{2}}q\left(  w\right)  \kappa\left(  B_{1},w\right)  dQ_{0}\left(
w\right) \\
&  =\int_{B_{2}}\kappa\left(  B_{1},w\right)  dQ_{1}\left(  w\right) \\
&  =\tilde{Q}_{1}\left(  B_{1}\times B_{2}\right)  .
\end{align*}
in view of (\ref{extensions-def}), showing that $Q_{1}^{\ast}=\tilde{Q}_{1}$.

This is an "almost" clean textbook proof; the super-rigorous textbooks (e.g.
Witting) might even distinguish between $q\left(  w\right)  $ and a function
of two variables $q^{\ast}\left(  y,w\right)  :=q\left(  w\right)  $ so that
$q^{\ast}$ is $\Omega_{X}^{\ast}$-measurable.

\subsection{The contraction property of state transitions}

Consider a state transition
\[
T:\mathcal{L}^{1}\left(  \mathcal{H}\right)  \rightarrow\mathcal{L}^{1}\left(
\mathcal{K}\right)
\]
i.e a dual channel of TP-CP map, where $\mathcal{H},\mathcal{K}$ are separable
Hilbert spaces. If $\rho,\sigma\in\mathcal{L}^{1}\left(  \mathcal{H}\right)  $
are states then
\begin{equation}
\left\Vert T\left(  \rho\right)  -T\left(  \sigma\right)  \right\Vert _{1}%
\leq\left\Vert \rho-\sigma\right\Vert _{1}. \label{contractivity-1}%
\end{equation}

\begin{proof}
Consider the binary testing problem leading to the Holevo-Helstrom bound: let
$\left(  \mathbf{1}-B,B\right)  $ where is a bounded operator: $B\in
\mathcal{L}\left(  \mathcal{H}\right)  $, with $0\leq B\leq\mathbf{1}$ (where
$B\geq0$ implies $B$ is Hermitian) and consider the error probability
\[
Err\left(  B\right)  =\mathrm{Tr}\;B\rho+\mathrm{Tr}\;\left(  \mathbf{1}%
-B\right)  \sigma.
\]
Then%
\begin{equation}
\inf_{0\leq B\leq\mathbf{1}}Err\left(  B\right)  =1-\frac{1}{2}\left\Vert
\rho-\sigma\right\Vert _{1}. \label{Hol-Helstr}%
\end{equation}
Indeed we note
\begin{align*}
\sigma-\rho &  =\left(  \sigma-\rho\right)  _{+}-\left(  \sigma-\rho\right)
_{-},\\
0  &  =\mathrm{Tr}\;\left(  \sigma-\rho\right)  =\mathrm{Tr}\;\left(
\sigma-\rho\right)  _{+}-\mathrm{Tr}\;\left(  \sigma-\rho\right)  _{-}%
\end{align*}
and furthermore
\begin{align}
\left\Vert \rho-\sigma\right\Vert _{1}  &  =\mathrm{Tr}\;\left\vert
\sigma-\rho\right\vert =\mathrm{Tr}\;\left(  \sigma-\rho\right)
_{+}+\mathrm{Tr}\;\left(  \sigma-\rho\right)  _{-}\nonumber\\
&  =2\mathrm{Tr}\;\left(  \sigma-\rho\right)  _{+}.
\label{L1-norm-repre-positive-part}%
\end{align}
Now we have as a consequence of of (\ref{L1-norm-repre-positive-part})%
\begin{align}
Err\left(  B\right)   &  =1-\mathrm{Tr}\;B\left(  \sigma-\rho\right)
\nonumber\\
&  \geq1-\mathrm{Tr}\;B\left(  \sigma-\rho\right)  _{+}=1-\frac{1}%
{2}\left\Vert \rho-\sigma\right\Vert _{1}. \label{HH-upper-bound-a}%
\end{align}
To obtain the reverse inequality, set
\[
\Pi=\text{\textrm{supp}}\left(  \sigma-\rho\right)  _{+},
\]
the support projection of $\left(  \sigma-\rho\right)  _{+}$. Then
\begin{align}
Err\left(  \Pi\right)   &  =1-\mathrm{Tr}\;\Pi\left(  \sigma-\rho\right)
\nonumber\\
&  =1-\mathrm{Tr}\;\left(  \sigma-\rho\right)  _{+}=1-\frac{1}{2}\left\Vert
\rho-\sigma\right\Vert _{1}. \label{HH-lower-bound}%
\end{align}
Now (\ref{HH-upper-bound-a}) and (\ref{HH-lower-bound}) imply
(\ref{Hol-Helstr}).

Now, using the dual channel to $T$ which is $T_{\ast}:\mathcal{L}\left(
\mathcal{K}\right)  \rightarrow\mathcal{L}\left(  \mathcal{H}\right)  $, a
unital ($T_{\ast}\left(  \mathbf{1}\right)  =\mathbf{1}$) and completely
positive map
\begin{align}
1-\frac{1}{2}\left\Vert T\left(  \rho\right)  -T\left(  \sigma\right)
\right\Vert _{1}  &  =\inf_{B\in\mathcal{L}\left(  \mathcal{K}\right)  ,0\leq
B\leq\mathbf{1}}\left(  \mathrm{Tr}\;BT\left(  \rho\right)  +\mathrm{Tr}%
\;\left(  \mathbf{1}-B\right)  T\left(  \sigma\right)  \right) \nonumber\\
&  =\inf_{B\in\mathcal{L}\left(  \mathcal{K}\right)  ,0\leq B\leq\mathbf{1}%
}\left(  \mathrm{Tr}\;T_{\ast}\left(  B\right)  \rho+\mathrm{Tr}\;\left(
T_{\ast}\left(  \mathbf{1}\right)  -T_{\ast}\left(  B\right)  \right)
\sigma\right) \nonumber\\
&  =\inf_{B\in\mathcal{L}\left(  \mathcal{K}\right)  ,0\leq B\leq\mathbf{1}%
}\left(  \mathrm{Tr}\;T_{\ast}\left(  B\right)  \rho+\mathrm{Tr}\;\left(
\mathbf{1}-T_{\ast}\left(  B\right)  \right)  \sigma\right)  .
\label{special-POVM}%
\end{align}
Here as a consequence of $B\leq\mathbf{1}$ and the positivity of $T_{\ast}$,
we have
\[
\mathbf{1}-T_{\ast}\left(  B\right)  =T_{\ast}\left(  \mathbf{1-}B\right)
\geq0
\]
hence $T_{\ast}\left(  B\right)  \leq\mathbf{1}$, and in (\ref{special-POVM})
$\left(  T_{\ast}\left(  B\right)  ,\mathbf{1}-T_{\ast}\left(  B\right)
\right)  $ is just a special binary POVM with operators from $\mathcal{L}%
\left(  \mathcal{H}\right)  $. We obtain a lower bound of
(\ref{special-e-nu-hat}) by allowing all binary POVM with operators from
$\mathcal{L}\left(  \mathcal{H}\right)  ,$ i.e. we obtain
\begin{align*}
1-\frac{1}{2}\left\Vert T\left(  \rho\right)  -T\left(  \sigma\right)
\right\Vert _{1}  &  \geq\inf_{A\in\mathcal{L}\left(  \mathcal{H}\right)
,0\leq A\leq\mathbf{1}}\left(  \mathrm{Tr}\;A\rho+\mathrm{Tr}\;\left(
\mathbf{1}-A\right)  \sigma\right) \\
&  =1-\frac{1}{2}\left\Vert \rho-\sigma\right\Vert _{1}%
\end{align*}
which proves (\ref{contractivity-1}).

\textbf{Remark. }\texttt{We have assumed here that the transition (TP-CP map)
arises from some channel }$T_{\ast}:\mathcal{L}\left(  \mathcal{K}\right)
\rightarrow\mathcal{L}\left(  \mathcal{H}\right)  $\texttt{, i.e a unital TP
map between the von Neumann algebras. This probably can be proved (to be
filled in here), and it will be an analog of Strasser's one-to-one
correspondence between Markov operators and stochastic operators, for the
commutative case (Strasser, 24.4, 24.5).}

\textbf{The case of observation channels. }Assume the state transition is
\[
T:\mathcal{L}^{1}\left(  \mathcal{H}\right)  \rightarrow L^{1}\left(
\mu\right)
\]
where $\mu$ is a probability measure on a measurable space $\left(
X,\Omega\right)  $. Here total positivity is equivalent to positivity. Then we
have again
\begin{equation}
\left\Vert T\left(  \rho\right)  -T\left(  \sigma\right)  \right\Vert _{1}%
\leq\left\Vert \rho-\sigma\right\Vert _{1}. \label{contractivity-2}%
\end{equation}

\begin{proof}
let $b\in L^{\infty}\left(  \mu\right)  $, then it can be shown, for any $\mu
$-densities $r,s$ and
\[
Err\left(  b\right)  =\int brd\mu+\int\left(  1-b\right)  sd\mu
\]
that%
\[
\inf_{b\in L^{\infty}\left(  \mu\right)  ,0\leq b\leq1}Err\left(  b\right)
=1-\frac{1}{2}\left\Vert r-s\right\Vert _{1}.
\]
Now for states $\rho,\sigma\in\mathcal{L}^{1}\left(  \mathcal{H}\right)  $%
\begin{align*}
1-\frac{1}{2}\left\Vert T\left(  \rho\right)  -T\left(  \sigma\right)
\right\Vert _{1}  &  =\inf_{b\in L^{\infty}\left(  \mu\right)  ,0\leq b\leq
1}\left(  \int bT\left(  \rho\right)  d\mu+\int\left(  1-b\right)  T\left(
\sigma\right)  d\mu\right) \\
&  =\inf_{b\in L^{\infty}\left(  \mu\right)  ,0\leq b\leq1}\left(  \int
T_{\ast}\left(  b\right)  \rho d\mu+\int\left(  \mathbf{1}-T_{\ast}\left(
b\right)  \right)  \sigma d\mu\right)
\end{align*}
using the dual channel $T_{\ast}:L^{\infty}\left(  \mu\right)  \rightarrow
\mathcal{L}\left(  \mathcal{H}\right)  $. Here again $0\mathbf{\leq}T_{\ast
}\left(  b\right)  \leq\mathbf{1}$, so $\left(  \mathbf{1}-T_{\ast}\left(
b\right)  ,T_{\ast}\left(  b\right)  \right)  $ is just a special POVM\ with
values in $\mathcal{L}\left(  \mathcal{H}\right)  $, so we obtain again
\[
1-\frac{1}{2}\left\Vert T\left(  \rho\right)  -T\left(  \sigma\right)
\right\Vert _{1}\geq1-\frac{1}{2}\left\Vert \rho-\sigma\right\Vert _{1}.
\]

\end{proof}
\end{proof}

\bigskip

\textbf{Another point of view. }In\textbf{ }\cite{MR887100}, 2.4.18 it is
stated that the predual $\mathcal{A}_{\ast}$ of a von Neumann algebra
$\mathcal{A}$ is a Banach space in the norm of $\mathcal{A}^{\ast}$, i.e. the
dual of $\mathcal{A}$. Thus the norm of $\omega\in\mathcal{A}_{\ast}$ is
\begin{equation}
\left\Vert \omega\right\Vert _{1}:=\sup_{A\in\mathcal{A}\text{,}\left\Vert
A\right\Vert \leq1}\left\vert \omega\left(  A\right)  \text{.}\right\vert
\label{L1normdef-1}%
\end{equation}

\bigskip(already noted in the paper, Appendix A.1.5) In \cite{MR1741419},
VI.6.9, Corollary 1 it is stated that for any $W^{\ast}$-algebra $W$, the
predual $V$ is unique and is isometrically isomorphic to the space
$W_{n}^{\prime}$ generated by all normal linear forms of $W$. A $W^{\ast}%
$-algebra is an "abstract" von Neumann algebra; a von Neumann algebra is a
$W^{\ast}$-algebra of operators on a Hilbert space. (Normal linear forms are
positive and order continuous).

Then in 2.4.21 in \cite{MR887100} it is stated that $\nu$ is a normal state on
a vN algebra $\mathcal{A}$ acting on a Hilbert space $\mathcal{H}$ (a
[positive] normal linear form taking value 1 on the unit of $\mathcal{A}$) if
there exists a density matrix $\rho$, i.e. a positive trace-class operator on
$\mathcal{H}$ with $\mathrm{Tr}\left(  \rho\right)  =1$ such that
\begin{equation}
\nu\left(  A\right)  =\mathrm{tr}\left(  \rho A\right)  .
\label{repre-predual-1}%
\end{equation}
Note that this result is as yet astract with an abstract trace; it includes
the case that $\mathcal{A}$ acts by multiplication on a space $L^{2}\left(
\mu\right)  $ and hence $\rho$ may be an operator also acting by
multiplication on $L^{2}\left(  \mu\right)  $, and the trace may be an integral.

\begin{proof}
[Remark]This is a bit more complicated; here is the actual characterization of
the predual. If $\mathcal{A=L}\left(  \mathcal{H}\right)  $ then there is a
one-to-one correspondence of $\nu$ to $\rho$. If $\mathcal{A}$ is a subalgebra
of $\mathcal{L}\left(  \mathcal{H}\right)  $ (as in the case where
$\mathcal{H}=L^{2}\left(  \mu\right)  $ and $\mathcal{A}=\left\{  M_{a},a\in
L^{\infty}\left(  \mu\right)  \right\}  $ where $M_{a}$ is the operator of
multiplication of $f\in L^{2}\left(  \mu\right)  $ with the function $a$) then
Theorem 46.4 in \cite{MR1721402} states the existence of a density operator
$\rho$, but not its uniquess. Intuitively, if $\sigma\in\mathcal{L}^{1}\left(
\mathcal{H}\right)  $ is a trace class operator such that
\[
\mathrm{tr}\left(  \sigma A\right)  =0\text{, all }A\in\mathcal{A}%
\]
then $\rho+\sigma$ represents the same linear form on $\mathcal{A}$. If
$\mathcal{A}=\mathcal{L}\left(  \mathcal{H}\right)  $ then $\sigma=0$ and we
have uniqueness of $\rho$ representing $\nu$. If $\mathcal{A}$ is a proper
subset of $\mathcal{L}\left(  \mathcal{H}\right)  $, as in the case
$\mathcal{A}=\left\{  M_{a},a\in L^{\infty}\left(  \mu\right)  \right\}  $,
then nontrivial $\sigma$ exist. Then 2.1.12 of \cite{MR3468018} shows how to
proceed: define
\[
\mathcal{A}^{\bot}:=\left\{  \sigma\in\mathcal{L}^{1}\left(  \mathcal{H}%
\right)  :\mathrm{tr}\left(  \sigma A\right)  =0\text{, }A\in\mathcal{A}%
\right\}
\]
then we have a well-defined linear bijection $\mathcal{L}^{1}\left(
\mathcal{H}\right)  /\mathcal{A}^{\bot}\rightarrow\mathcal{A}_{\ast}$ defined
by $\rho+\mathcal{A}^{\bot}\longmapsto\nu$ where $\nu\left(  A\right)
=\mathrm{tr}\left(  \rho A\right)  $. Therefore $\mathcal{A}_{\ast}$ can be
identified with the Banach space $\mathcal{L}^{1}\left(  \mathcal{H}\right)
/\mathcal{A}^{\bot}$ endowed with the quotient norm. The same fact is stated
in \cite{MR1721402}, before 46.4.

As an example, consider the case $L^{2}\left(  \mu\right)  =\ell^{2}$ with
$\mu$ as the counting measure on $\mathbb{N}$, and $\mathcal{A}=\left\{
M_{a},a\in\ell^{\infty}\right\}  $ where $\ell^{\infty}$ is the set of bounded
sequences. Let $\nu\in\mathcal{A}_{\ast}$, then by \cite{MR1721402}. loc. cit.
we have a $\rho\in\mathcal{L}^{1}\left(  \mathcal{H}\right)  $ (a trace class
infinite dimensional matrix) such that (\ref{repre-predual-1}) holds. What are
the operators $\sigma$? The $A\in\mathcal{A}$ are infinite diagonal matrices
with elements $a_{i}$ say. Then $\sigma A$ is a matrix with elements
$\sigma_{ij}a_{i}$, and $\mathrm{tr}\left(  \sigma A\right)  =\sum_{i}%
\sigma_{ii}a_{i}$. Thus $\mathrm{tr}\left(  \sigma A\right)  =0$ for all
$A\in\mathcal{A}$ if $\sigma_{ii}=0$, all $i$. So if any $\rho+\sigma$
represents the linear form $\nu$ then we may choose $\sigma$ as
\[
\sigma_{ij}=\left\{
\begin{array}
[c]{c}%
-\rho_{ij}\text{, }i\neq j\\
0\text{, }i=j
\end{array}
\right.
\]
making $\rho+\sigma$ a diagonal matrix with diagonal elenents $\rho_{ii}$.
This means we can always choose $\rho$ in (\ref{repre-predual-1}) as a
diagonal matrix of trace class, and then
\[
\nu\left(  A\right)  =\mathrm{tr}\left(  \rho A\right)  =\sum\rho_{ii}a_{i}.
\]
Here $\rho$ can be identified with the multiplication operator $M_{r},r\in
\ell^{1}$ where $r_{i}=\rho_{ii}$.

There should be an obvious analog for a general $L^{2}\left(  \mu\right)  $
which we need to figure out.
\end{proof}

\bigskip

The predual $\mathcal{A}_{\ast}$ is the space of all weak* continuous linear
functionals on $\mathcal{A}$ (Conw00\cite{MR1721402}, before 46.4 ). This is
complex linear hull of the set of positive weak* (=normal) linear forms
(\textit{Schaefer\cite{MR1741419}, VI.6.6.)}, and by \cite{MR1721402}, 54.9 we
have for every $\omega\in\mathcal{A}_{\ast}$%
\begin{equation}
\omega=\omega_{1}-\omega_{2}+i\left(  \omega_{3}-\omega_{4}\right)
\label{decomp-normal}%
\end{equation}
where $\omega_{i}$ are positive weak* continuous linear forms such that
$\omega_{1}\bot\omega_{2}$ and $\omega_{3}\bot\omega_{4}$ (in fact $\omega
_{1}-\omega_{2}$ is an hermitian form and decomposed into its positive and
negative part, likewise with $\omega_{3}-\omega_{4}$). We note that the above
decomposition parallels that for fi.di. $n\times n$ matrices: if $M$ is a
matrix then
\begin{align*}
M  &  =\frac{1}{2}\left(  M+M^{\ast}\right)  +\frac{1}{2}\left(  M-M^{\ast
}\right) \\
&  =\frac{1}{2}\left(  M+M^{\ast}\right)  +i\frac{i}{2}\left(  M^{\ast
}-M\right)
\end{align*}
Here $\frac{1}{2}\left(  M+M^{\ast}\right)  $ is hermitian, so it decomposes
as $A_{1}-A_{2}$ we $A_{1}$ is the positive part and $A_{2}$ is the negative
part. Similarly $\frac{i}{2}\left(  M^{\ast}-M\right)  $ is Hermitian, so it
decomposed into $A_{3}-A_{4}$. Hence
\[
M=A_{1}-A_{2}+i\left(  A_{3}-A_{4}\right)  .
\]
Thus (\ref{decomp-normal}) can be set in relation to $\omega\left(  A\right)
=\mathrm{tr}\left(  \rho A\right)  $ (cf. (\ref{repre-predual-1})) where
$\rho$ is a trace class operator.

Assume $\omega$ is Hermitian. Assume $\omega\left(  A\right)  =\mathrm{tr}%
\left(  \rho A\right)  $ and $\rho=$ $\rho_{+}-\rho_{-}$ such that
\[
\omega\left(  A\right)  =\mathrm{tr}\left(  \rho_{+}A\right)  -\mathrm{tr}%
\left(  \rho_{-}A\right)  .
\]
We conjecture the following: for hermitian $\omega$
\begin{equation}
\left\Vert \omega\right\Vert _{1}=\sup_{A\in\mathcal{A}\text{,}A=A^{\ast
},-\mathbf{1\leq}A\mathbf{\leq1}}\left\vert \omega\left(  A\right)
\right\vert =\mathrm{tr}\left(  \left\vert \rho\right\vert \right)  \text{.}
\label{L1normdef-2}%
\end{equation}
That would make it straighforward to prove the contraction property for TP-CP
maps (though we have a simpler proof below).

\begin{proof}
[Proof for fi.di. matrices]Here $\mathcal{H}=\mathbb{C}^{n}$ and
$\mathcal{A}=\mathcal{L}\left(  \mathcal{H}\right)  $, the set of all $n\times
n$ complex matrices. We first show that if $\omega\geq0$ then
\begin{equation}
\left\Vert \omega\right\Vert _{1}=\sup_{0\mathbf{\leq}B\mathbf{\leq1}}%
\omega\left(  B\right)  . \label{L1normdef-3}%
\end{equation}
Indeed
\[
\omega\left(  B\right)  =\mathrm{tr}\left(  \rho B\right)  .
\]
Since $\rho$ is positive, it is Hermitian. Let $\rho=U\Lambda U^{\ast}$ where
$U$ is unitary and $\Lambda$ is diagonal; then
\[
\mathrm{tr}\left(  \rho B\right)  =\mathrm{tr}\left(  U\Lambda U^{\ast
}B\right)  =\mathrm{tr}\left(  \Lambda U^{\ast}BU\right)  =\mathrm{tr}\left(
\Lambda D\right)
\]
where $D=$ $U^{\ast}BU$. Hence
\begin{equation}
\left\vert \mathrm{tr}\left(  \rho B\right)  \right\vert =\left\vert
\sum_{j=1}^{n}\Lambda_{j}D_{jj}\right\vert \leq\sum_{j=1}^{n}\left\vert
\Lambda_{j}D_{jj}\right\vert . \label{step-diag-1}%
\end{equation}
Now $\left\Vert B\right\Vert =\lambda_{\max}\left(  B^{\ast}B\right)  \leq1$,
hence also $\left\Vert D\right\Vert \leq1$. Now $D_{jj}=e_{j}^{\ast}De_{j}$
where $e_{j}$ is a basis vector, hence
\[
\left\vert D_{jj}\right\vert ^{2}=e_{j}^{\ast}D^{\ast}e_{j}e_{j}^{\ast}%
De_{j}\leq e_{j}^{\ast}D^{\ast}De_{j}\leq1.
\]
Hence from (\ref{step-diag-1}), if $\omega\geq0$
\[
\left\Vert \omega\right\Vert _{1}=\sup_{\left\Vert B\right\Vert \leq
1}\left\vert \mathrm{tr}\left(  \rho B\right)  \right\vert \leq\sum_{j=1}%
^{n}\left\vert \Lambda_{j}\right\vert =\sum_{j=1}^{n}\Lambda_{j}%
=\mathrm{tr}\left(  \rho\right)  .
\]
That bound is attained by $B=\mathbf{1}$, hence (\ref{L1normdef-3}) is shown.

Now for general $\omega$, we have
\begin{align*}
\sup_{B\in\mathcal{A}\text{,}\left\Vert B\right\Vert \leq1}\left\vert
\omega\left(  B\right)  \right\vert  &  =\sup_{B\in\mathcal{A}\text{,}%
\left\Vert B\right\Vert \leq1}\left\vert \mathrm{tr}\left(  \rho_{+}B\right)
-\mathrm{tr}\left(  \rho_{-}B\right)  \right\vert \\
&  \leq\sup_{B\in\mathcal{A}\text{,}\left\Vert B\right\Vert \leq1}\left\vert
\mathrm{tr}\left(  \rho_{+}B\right)  \right\vert +\sup_{B\in\mathcal{A}%
\text{,}\left\Vert B\right\Vert \leq1}\left\vert \mathrm{tr}\left(  \rho
_{-}B\right)  \right\vert \\
&  =\mathrm{tr}\left(  \rho_{+}\right)  +\mathrm{tr}\left(  \rho_{-}\right)
=\mathrm{tr}\left(  \left\vert \rho\right\vert \right)  .
\end{align*}
where $\left\vert \rho\right\vert =\rho_{+}+\rho_{-}=\left(  \rho^{\ast}%
\rho\right)  ^{1/2}$. That bound is attained by $B=$\textrm{supp}$\rho_{+}%
-$\textrm{supp}$\rho_{-}$:
\begin{align*}
\left\vert \mathrm{tr}\left(  \rho_{+}B\right)  -\mathrm{tr}\left(  \rho
_{-}B\right)  \right\vert  &  =\left\vert \mathrm{tr}\left(  \rho
_{+}\mathrm{supp}\rho_{+}\right)  +\mathrm{tr}\left(  \rho_{-}\mathrm{supp}%
\rho_{-}\right)  \right\vert \\
&  =\mathrm{tr}\left(  \rho_{+}\right)  +\mathrm{tr}\left(  \rho_{-}\right)
=\mathrm{tr}\left(  \left\vert \rho\right\vert \right)  .
\end{align*}
Here
\[
-\mathbf{1}\leq\mathrm{supp}\rho_{+}-\mathrm{supp}\rho_{-}\leq\mathbf{1}%
\]
since \textrm{supp}$\rho_{+},$\textrm{supp}$\rho_{-}$ are orthogonal
projections. This proves (\ref{L1normdef-2})
\end{proof}

\textbf{Remark. }The proof of (\ref{L1normdef-2}) is related to the
Holevo-Helstrom bound for quantum testing. It is not equivalent since the
latter bound already assumes that tests are performed by positive operators
$B$ with $0\leq B\leq\mathbf{1}$: the error criterion for testing between
states $\rho_{0}$ and $\rho_{1}$ is
\begin{align*}
Err\left(  B\right)   &  =\mathrm{tr}\rho_{0}B+\mathrm{tr}\rho_{1}\left(
\mathbf{1}-B\right) \\
&  =1-\mathrm{tr}\left(  \rho_{1}-\rho_{0}\right)  B
\end{align*}
such that
\begin{align*}
\inf_{0\leq B\leq\mathbf{1}}Err\left(  B\right)   &  =1-\mathrm{tr}\left(
\rho_{1}-\rho_{0}\right)  _{+}=1-\frac{1}{2}\mathrm{tr}\left\vert \rho
_{1}-\rho_{0}\right\vert \\
&  =1-\frac{1}{2}\left\Vert \rho_{1}-\rho_{0}\right\Vert _{1}.
\end{align*}

\begin{proof}
[Proof for the commutative case ]Here $\mathcal{H}=L^{2}\left(  \mu\right)  $
while $\mathcal{A}=L^{\infty}\left(  \mu\right)  $ and $\mathcal{A}_{\ast
}=L^{1}\left(  \mu\right)  $, and both $\mathcal{A}$ and $\mathcal{A}_{\ast}$
act by multiplication of functions in $L^{2}\left(  \mu\right)  $. We write
the claim (\ref{L1normdef-2}) again, for $\omega\in\mathcal{A}_{\ast}$
\begin{equation}
\left\Vert \omega\right\Vert _{1}=\sup_{a\in\mathcal{A}\text{,}\left\Vert
A\right\Vert \leq1}\left\vert \omega\left(  a\right)  \right\vert =\sup
_{a\in\mathcal{A}\text{,}a=a^{\ast},-\mathbf{1\leq}a\mathbf{\leq1}}\left\vert
\omega\left(  a\right)  \right\vert . \label{L1normdef-4}%
\end{equation}
Let $w\in L^{1}\left(  \mu\right)  $ such that for $a\in\mathcal{A=}$
$L^{\infty}\left(  \mu\right)  $ we have
\[
\omega\left(  a\right)  =\int wad\mu.
\]
Now for $a\in$ $L^{\infty}\left(  \mu\right)  $ (complex valued functions) we
have, for $f,g\in L^{2}\left(  \mu\right)  $%
\[
\left\langle g,af\right\rangle =\int\bar{g}\left(  af\right)  d\mu
=\int\overline{g\bar{a}}fd\mu=\left\langle \bar{a}g,f\right\rangle
\]
hence the adjoint opeator to $a$ is multiplication by $\bar{a}$, thus
$a^{\ast}=\bar{a}$. Hence $a$ is self-ajoint if it is real-valued. So the
claim can be written
\[
\sup_{a\in\mathcal{A}\text{,}\left\Vert a\right\Vert \leq1}\left\vert \int
wad\mu\right\vert =\sup_{a\in\mathcal{A}\text{,}a=a^{\ast},-\mathbf{1\leq
}a\mathbf{\leq1}}\left\vert \int wad\mu\right\vert .
\]
First assume $w\geq0$. Note that
\[
\left\Vert a\right\Vert =\left\Vert a\right\Vert _{\infty}%
=\text{\textrm{esssup}}_{\mu}\left\vert f\left(  x\right)  \right\vert
\]
Then for all a with $\left\Vert a\right\Vert _{\infty}\leq1$%
\[
\left\vert \int wad\mu\right\vert \leq\int\left\vert wa\right\vert d\mu
\leq\left\Vert a\right\Vert _{\infty}\int\left\vert w\right\vert d\mu=\int
wd\mu=\left\Vert w\right\Vert _{1}.
\]
That bound is attained for $a=1$.

The proof in the general case works exactly as for matrices: we have
$w=w_{+}-w_{-}$, and the support projections are defined as appropriate
indicator functions in an obvious way.
\end{proof}

\bigskip

\bigskip

\textbf{Contraction property of TP-CP maps. }Let $\mathcal{\mathcal{H}%
_{\mathcal{A}}}=\mathbb{C}^{n}$ and $\mathcal{\mathcal{H}_{\mathcal{B}}%
}=\mathbb{C}^{m}$, and $\mathcal{A=\mathcal{L}(\mathcal{H}_{\mathcal{A}})}$,
$\mathcal{B=\mathcal{L}(\mathcal{H}_{\mathcal{B}})}$, which are the respective
sets of $n\times n$ and $m\times m$ matrices. \textbf{ }Let $T$ be a TP-CP map
(dual channel) $T:\mathcal{L}^{1}(\mathcal{\mathcal{H}_{\mathcal{A}}%
})\rightarrow\mathcal{L}^{1}(\mathcal{\mathcal{H}_{\mathcal{B}}})$, then we
claim for $\rho\in\mathcal{L}^{1}(\mathcal{\mathcal{H}_{\mathcal{A}}}$
\[
\left\Vert T\left(  \rho\right)  \right\Vert _{1}\leq\left\Vert \rho
\right\Vert _{1}.
\]
Indeed%
\[
\left\Vert T\left(  \rho\right)  \right\Vert _{1}=\sup_{B\in\mathcal{B}%
\text{,}B=B^{\ast},-\mathbf{1\leq}B\mathbf{\leq1}}\mathrm{tr}\left(  T\left(
\rho\right)  B\right)  .
\]
We now assume there is a dual map (a channel ) $T_{\ast}:\mathcal{L}%
(\mathcal{\mathcal{H}_{\mathcal{B}}})\rightarrow\mathcal{L}%
(\mathcal{\mathcal{H}_{\mathcal{A}}})$, such that
\[
\mathrm{tr}\left(  T\left(  \rho\right)  B\right)  =\mathrm{tr}\left(  \rho
T_{\ast}\left(  B\right)  \right)
\]
where $T_{\ast}$ is unital ($T_{\ast}\left(  \mathbf{1}\right)  =\mathbf{1}$)
and completely positive. Hence
\[
\left\Vert T\left(  \rho\right)  \right\Vert _{1}=\sup_{B\in\mathcal{B}%
\text{,}B=B^{\ast},-\mathbf{1\leq}B\mathbf{\leq1}}\mathrm{tr}\left(  \rho
T_{\ast}\left(  B\right)  \right)  .
\]
Since $B\mathbf{\leq1}$, we have $\mathbf{1}-B\geq0$ and hence $T_{\ast
}\left(  \mathbf{1}-B\right)  \geq0,$ hence $T_{\ast}\left(  B\right)
\leq\mathbf{1}$. In the same way it is shown than $T_{\ast}\left(  B\right)
\geq-\mathbf{1}$. It follows that
\[
\left\Vert T\left(  \rho\right)  \right\Vert _{1}\leq\sup_{A\in\mathcal{A}%
\text{,}A=A^{\ast},-\mathbf{1\leq}A\mathbf{\leq1}}\mathrm{tr}\left(  \rho
T_{\ast}\left(  B\right)  \right)  =\left\Vert \rho\right\Vert _{1}.
\]

\textbf{A simpler proof of the contraction property. } The above proof relies
on the existence of the dual channel to $T$. A simpler proof is the following.
For $\rho\in\mathcal{L}^{1}(\mathcal{\mathcal{H}_{\mathcal{A}}})$ we have
\begin{align}
\left\Vert T\left(  \rho\right)  \right\Vert _{1}  &  =\left\Vert T\left(
\rho_{+}-\rho_{-}\right)  \right\Vert _{1}=\left\Vert T\left(  \rho
_{+}\right)  -T\left(  \rho_{-}\right)  \right\Vert _{1}\nonumber\\
&  \leq\left\Vert T\left(  \rho_{+}\right)  \right\Vert _{1}+\left\Vert
T\left(  \rho_{-}\right)  \right\Vert _{1}\text{ (triangle inequality)}%
\nonumber\\
&  =\text{ }\mathrm{tr}\rho_{+}+\mathrm{tr}\rho_{-1}\text{ (}T\text{ is trace
preserving on positives)}\nonumber\\
&  =\mathrm{tr}\left\vert \rho\right\vert =\left\Vert \rho\right\Vert _{1}.
\label{simpler-proof-contraction}%
\end{align}

\textbf{General case: the predual norm }$\left\Vert \omega\right\Vert _{1}%
$\textbf{ is equal to the trace norm, for matrices. }

The norm $\left\Vert \omega\right\Vert _{1}$ of a weak* continuous lienar form
on a von Neumann algebra $\mathcal{A}$ was defined (\ref{L1normdef-1}) as%
\[
\left\Vert \omega\right\Vert _{1}:=\sup_{A\in\mathcal{A}\text{,}\left\Vert
A\right\Vert \leq1}\left\vert \omega\left(  A\right)  \right\vert
\]
We claim that if $\rho$ is a trace class operator on $\mathcal{H}$ such that
$\omega\left(  A\right)  =\mathrm{tr}\left(  \rho A\right)  $ then
\begin{equation}
\left\Vert \omega\right\Vert _{1}=\mathrm{tr}\left\vert \rho\right\vert
=\left\Vert \rho\right\Vert _{1}. \label{L1normdef4a}%
\end{equation}
i.e. $\left\Vert \omega\right\Vert _{1}$ is the trace norm for $\pi$. Here
neither $\omega$ nor $\rho$ are claimed to be hermitian.

\begin{proof}
[Proof for fi.di. matrices]Let $\rho$ be a complex $n\times n$ matrix. A trace
class operator fulfills $\mathrm{tr}\rho^{\ast}\rho<\infty$, and then
$\left\Vert \rho\right\Vert _{1}=\mathrm{tr}\left\vert \rho\right\vert $ where
$\left\vert \rho\right\vert =\left(  \rho^{\ast}\rho\right)  ^{1/2}$. Assume
first that $\rho$ is full rank $n\times n$; then there is a polar
decomposition
\[
\rho=U\left(  \rho^{\ast}\rho\right)  ^{1/2}\text{, }U=\rho\left(  \rho^{\ast
}\rho\right)  ^{-1/2}%
\]
where $U$ is unitary. Then
\[
\left\Vert \omega\right\Vert _{1}=\sup_{A\in\mathcal{A}\text{,}\left\Vert
A\right\Vert \leq1}\left\vert \mathrm{tr}\left(  A\rho\right)  \right\vert
=\sup_{A\in\mathcal{A}\text{,}\left\Vert A\right\Vert \leq1}\left\vert
\mathrm{tr}\left(  AU\left(  \rho^{\ast}\rho\right)  ^{1/2}\right)
\right\vert .
\]
Now $\left\Vert A\right\Vert \leq1$ iff $\left\Vert AU\right\Vert \leq1$
(since $\left\Vert A\right\Vert ^{2}=\lambda_{\max}\left(  AA^{\ast}\right)
=\lambda_{\max}\left(  AUU^{\ast}A^{\ast}\right)  $) and $AU\in\mathcal{A}$,
hence
\[
\left\Vert \omega\right\Vert _{1}=\sup_{B\in\mathcal{A}\text{,}\left\Vert
B\right\Vert \leq1}\left\vert \mathrm{tr}\left(  B\left(  \rho^{\ast}%
\rho\right)  ^{1/2}\right)  \right\vert .
\]
Since $\left(  \rho^{\ast}\rho\right)  ^{1/2}$ is hermitian, we have by
(\ref{L1normdef-2})%
\begin{align}
\left\Vert \omega\right\Vert _{1}  &  =\mathrm{tr}\left(  \left\vert \left(
\rho^{\ast}\rho\right)  ^{1/2}\right\vert \right)  =\mathrm{tr}\left(  \left(
\rho^{\ast}\rho\right)  ^{1/2}\right) \nonumber\\
&  =\mathrm{tr}\left\vert \rho\right\vert =\left\Vert \rho\right\Vert _{1}.
\label{L1normdef-5}%
\end{align}
Now assume \textrm{rank}$\left(  \rho\right)  =r<n$. Let $\mathcal{S}%
=\mathrm{span}\left(  \rho\right)  $ and let $V$ be a $n\times r$ matrix such
that $V^{\ast}V=I_{r}$, $\mathcal{S}=\mathrm{span}\left(  V\right)  .$ Then
there is a $r\times n$ matrix $T,$ \textrm{rank}$\left(  T\right)  =r$ such
that
\[
\rho=VT
\]
(the colums of $V$ span $\mathcal{S}$, hence the columns of $\rho$ are linear
combinations of these). Hence
\begin{equation}
\rho=V\left(  TT^{\ast}\right)  ^{1/2}\left(  TT^{\ast}\right)  ^{-1/2}%
T=V\left(  TT^{\ast}\right)  ^{1/2}S \label{L1normdef-6}%
\end{equation}
where the $r\times n$ matrix $S=\left(  TT^{\ast}\right)  ^{-1/2}T$ fulfills
$SS^{\ast}=I_{r}$. (Here both $V$ and $S$ are partial isometries.). Then for
every $\left\Vert A\right\Vert \leq1$%
\[
\left\vert \mathrm{tr}\left(  A\rho\right)  \right\vert =\left\vert
\mathrm{tr}\left(  AV\left(  TT^{\ast}\right)  ^{1/2}S\right)  \right\vert
=\left\vert \mathrm{tr}\left(  SAV\left(  TT^{\ast}\right)  ^{1/2}\right)
\right\vert .
\]
Here $B=SAV$ is an $r\times r$ matrix fulfilling $\left\Vert B\right\Vert
\leq1$, and for every $B$ with $\left\Vert B\right\Vert \leq1,$ setting
$A=S^{\ast}BV^{\ast}$ we obtain an $A$ such that $B=SAV$ and $\left\Vert
A\right\Vert \leq1$. Hence
\[
\sup_{A\in\mathcal{A}\text{,}\left\Vert A\right\Vert \leq1}\left\vert
\mathrm{tr}\left(  A\rho\right)  \right\vert =\sup_{\left\Vert B\right\Vert
\leq1}\left\vert \mathrm{tr}\left(  B\left(  TT^{\ast}\right)  ^{1/2}\right)
\right\vert =\mathrm{tr}\left(  TT^{\ast}\right)  ^{1/2}%
\]
by (\ref{L1normdef-5}). By (\ref{L1normdef-6}) we have
\[
\mathrm{tr}\left\vert \rho\right\vert =\mathrm{tr}\left(  \rho^{\ast}%
\rho\right)  ^{1/2}=\mathrm{tr}\left(  S^{\ast}TT^{\ast}S\right)
^{1/2}=\mathrm{tr}\left(  TT^{\ast}\right)  ^{1/2}%
\]
because the eigenvalues of $S^{\ast}TT^{\ast}S$ equal those of $TT^{\ast}$.
This proves (\ref{L1normdef4a}).
\end{proof}

\bigskip\newpage

\subsection{Properties of channels related to a general predual}

The von Neumann algebra $\mathcal{A}$ is the dual Banach space of its predual
$\mathcal{A}_{\ast}$. The latter is the Banach space of all weak* continuous
linear functionals on $\mathcal{A}$ (Conw00\cite{MR1721402}, before 46.4 ),
endowed with the norm of $\mathcal{A}^{\ast}$ (the dual Banach space).
$\mathcal{A}_{\ast}$ can also be characterized as the complex linear hull of
the set of positive weak* (=normal) linear forms
(\textit{Schaefer\cite{MR1741419}, VI.6.6.)}, and by \cite{MR1721402}, 54.9 we
have for every $\omega\in\mathcal{A}_{\ast}$%
\begin{equation}
\omega=\omega_{1}-\omega_{2}+i\left(  \omega_{3}-\omega_{4}\right)
\end{equation}
where $\omega_{i}$ are positive weak* continuous linear forms such that
$\omega_{1}\bot\omega_{2}$ and $\omega_{3}\bot\omega_{4}$ (in fact $\omega
_{1}-\omega_{2}$ is an hermitian form and decomposed into its positive and
negative part, likewise with $\omega_{3}-\omega_{4}$).

The two special cases of interest are: (i) the set $\mathcal{L}(\mathcal{H})$
of bounded linear operators on $\mathcal{H}$ (Conw90\cite{MR1070713}, IX.7.2)
), (ii) the set of functions $L^{\infty}\left(  \mu\right)  $ on a $\sigma
$-finite measure space $\left(  X,\Omega,\mu\right)  $, construed as linear
operators on $\mathcal{H}=L^{2}\left(  \mu\right)  $ by pointwise
multiplication (Conw90\cite{MR1070713}, Theorem IX.6.6 ). Since a von Neumann
algebra is a subalgebra of some $\mathcal{L}(\mathcal{H})$, case (i) is a
special case trivially, but case (ii) we have as yet only understood in the
case that in $\left(  X,\Omega,\mu\right)  $, $X$ is a finite set, of
cardinality $n$, and $L^{\infty}\left(  \mu\right)  $ corresponds to the
multiplication operators on $\mathbb{C}^{n}$ corresponding to diagonal
matrices with complex elements. The case of general $L^{\infty}\left(
\mu\right)  $ remains to be explored.

Our aim is to show that to very channel between von Neumann algebras
$\alpha:\mathcal{A}\rightarrow\mathcal{B}$ there corresponds a unique TP-CP
map between the preduals $\alpha_{\ast}:\mathcal{B}_{\ast}\rightarrow
\mathcal{A}_{\ast}$, and vice versa. The framework of the general predual
allows a unified approach; we would otherwise have to deal with 4 cases:
$\mathcal{A}$ or $\mathcal{B}$ are either $\mathcal{L}(\mathcal{H})$ or
$L^{\infty}\left(  \mu\right)  $

\begin{proof}
[Remark]"(but one of these, $\mathcal{A}=L^{\infty}\left(  \mu\right)  $ and
$\mathcal{B}=L^{\infty}\left(  \nu\right)  $, is covered by a theorem of
Strasser on stochastic and Markov operators)". \texttt{That has to be checked,
since originally in Strasser it's all about real vector spaces, but vN
algebras and the predual, also for }$L^{\infty}\left(  \mu\right)  $\texttt{
and }$L^{1}\left(  \mu\right)  $\texttt{, are all defined within the complex
framework. Or does Conway, FuncAna, treat complex and real in parallel? He has
a symbol "}$\mathbb{F}$\texttt{" for this. It might be that the results for
positive maps already imply the complex case.}
\end{proof}

Recall that a quantum channel is a bounded, linear, completely positive,
unital and normal mapping $\alpha:\mathcal{A}\rightarrow\mathcal{B}$.
According to Conw00\cite{MR1721402}, 46.5, a positive linear map
$\alpha:\mathcal{A}\rightarrow\mathcal{B}$ is normal if and only if it is
weak* continuous. Recall that the weak* topology on $\mathcal{A}$ is generated
by all linear functionals on $\mathcal{A}$ which are elements of
$\mathcal{A}_{\ast}$ \footnote{Then by definition the elements of
$\mathcal{A}_{\ast}$ are weak* continuous functionals on $\mathcal{A}$. But
then there is a claim that the elements of $\mathcal{A}_{\ast}$ are
\textbf{all} such functiuonals. We haven't comprehended that argument yet; at
some point the Hahn-Banach Theorem is involved.}.

The section \S 54 in Conw00\cite{MR1721402} contains a detailed discussion of
weak* continuous linear functionals as elements of the predual $\mathcal{A}%
_{\ast}$. Recall that $\mathcal{A}_{\ast}$ is a Banach space under the norm
\[
\left\Vert \omega\right\Vert _{1}=\sup_{a\in\mathcal{A},\left\Vert
a\right\Vert \leq1}\left\vert \omega\left(  a\right)  \right\vert
\]
where $\left\Vert a\right\Vert $ is the operator norm in $\mathcal{L}%
(\mathcal{H})$.

\begin{lemma}
\label{Lem-notes-normlinform}If the linear form $\omega$ is positive on
$\mathcal{A}$, then
\[
\left\Vert \omega\right\Vert _{1}=\omega\left(  \mathbf{1}\right)
\]
where $\mathbf{1}$ is the unit of $\mathcal{A}$.
\end{lemma}

This directly follows Proposition 33.2 (a) in Conw00\cite{MR1721402}. There
$\omega$ is a positive linear form on an operator system $\mathcal{S}$, which
is a linear submanifold of a C$^{\ast}$-algebra with certain properties. Here
$\mathcal{S}$ can be the C$^{\ast}$-algebra itself, and $\mathcal{A}$ as von
Neumann algebra is also a C$^{\ast}$-algebra. Below follows our own proof (to
be erased).

\begin{proof}
Since $\omega$ is positive, it is a bounded linear functional
(Conw00\cite{MR1721402} 33.4), and it is also hermitian ($\omega=\omega^{\ast
}$, cf. Conw00\cite{MR1721402}, next to 54.3). Then according to 54.3 in
Conw00\cite{MR1721402}, $\left\Vert \omega\right\Vert _{1}=\left\Vert
\omega|\operatorname{Re}\mathcal{A}\right\Vert _{1}$ where $\operatorname{Re}%
\mathcal{A}$ is the set of hermitian elements of $\mathcal{A}$. Hence
\[
\left\Vert \omega\right\Vert _{1}=\sup_{a\in\operatorname{Re}\mathcal{A}%
,\left\Vert a\right\Vert \leq1}\left\vert \omega\left(  a\right)  \right\vert
.
\]
According to 3.2 in Conw00\cite{MR1721402}, if $a\in\operatorname{Re}%
\mathcal{A}$ then there are unique positive elements $u$ and $v$ in
$\mathcal{A}$ such that $a=u-v$ and $uv=vu=0$. Now for the operator norm
$\left\Vert a\right\Vert $ we have
\begin{align*}
\left\Vert a\right\Vert ^{2}  &  =\sup_{x\in\mathcal{H},\left\Vert
x\right\Vert \leq1}\left\Vert a\left(  x\right)  \right\Vert ^{2}=\sup
_{x\in\mathcal{H},\left\Vert x\right\Vert \leq1}\left(  x,a^{\ast}ax\right) \\
&  =\sup_{x\in\mathcal{H},\left\Vert x\right\Vert \leq1}\left(  x,\left(
u^{\ast}u+v^{\ast}v\right)  x\right)  \geq\sup_{x\in\mathcal{H},\left\Vert
x\right\Vert \leq1}\left(  x,u^{\ast}ux\right) \\
&  =\left\Vert u\right\Vert ^{2}%
\end{align*}
and in the same way we obtain $\left\Vert a\right\Vert ^{2}\geq\left\Vert
v\right\Vert ^{2}$. Then
\begin{align*}
\left\Vert \omega\right\Vert _{1}  &  =\sup_{a\in\operatorname{Re}%
\mathcal{A},\left\Vert a\right\Vert \leq1}\left\vert \omega\left(  a\right)
\right\vert =\sup_{a\in\operatorname{Re}\mathcal{A},\left\Vert a\right\Vert
\leq1}\left\vert \omega\left(  u\right)  -\omega\left(  v\right)  \right\vert
\\
&  =\max\left(  \sup_{a\in\operatorname{Re}\mathcal{A},\left\Vert a\right\Vert
\leq1}\left(  \omega\left(  u\right)  -\omega\left(  v\right)  \right)
,\sup_{a\in\operatorname{Re}\mathcal{A},\left\Vert a\right\Vert \leq1}\left(
\omega\left(  v\right)  -\omega\left(  u\right)  \right)  \right) \\
&  \leq\max\left(  \sup_{a\in\operatorname{Re}\mathcal{A},\left\Vert
a\right\Vert \leq1}\omega\left(  u\right)  ,\sup_{a\in\operatorname{Re}%
\mathcal{A},\left\Vert a\right\Vert \leq1}\omega\left(  v\right)  \right) \\
&  \leq\max\left(  \sup_{u\geq0,\left\Vert u\right\Vert \leq1}\omega\left(
u\right)  ,\sup_{v\geq0,\left\Vert v\right\Vert \leq1}\omega\left(  v\right)
\right)  =\sup_{u\geq0,\left\Vert u\right\Vert \leq1}\omega\left(  u\right)  .
\end{align*}
On the other hand
\[
\left\Vert \omega\right\Vert _{1}=\sup_{a\in\operatorname{Re}\mathcal{A}%
,\left\Vert a\right\Vert \leq1}\left\vert \omega\left(  a\right)  \right\vert
\geq\sup_{a\geq0,\left\Vert a\right\Vert \leq1}\left\vert \omega\left(
a\right)  \right\vert =\sup_{u\geq0,\left\Vert u\right\Vert \leq1}%
\omega\left(  u\right)
\]
so we have shown
\[
\left\Vert \omega\right\Vert _{1}=\sup_{u\geq0,\left\Vert u\right\Vert \leq
1}\omega\left(  u\right)  .
\]

Now $\left\Vert u\right\Vert \leq1$ means $\left(  x,u^{2}x\right)  \leq1$ for
all $\left\Vert x\right\Vert \leq1$, hence $u^{2}\leq\boldsymbol{1}$. By 3.3
in Conw00\cite{MR1721402}, there is a unique square root of $u^{2}$ which must
be $u$. Now we can argue with the spectral representation of $u$ and $u^{2}$.
If $u$ has a discrete spectal representation $u=\sum\lambda_{i}P_{i}$ where
$P_{i}$ are orthogonal projectors and $\lambda_{i}\geq0$ then $u^{2}%
=\sum\lambda_{i}^{2}P_{i}$, and $\left\Vert u\right\Vert \leq1$ means
$\sup_{i}\lambda_{i}^{2}\leq1$. This is equivalent $\sup_{i}\lambda_{i}\leq1$,
hence $\left\Vert u^{1/2}\right\Vert \leq1$, meaning $\left(  x,ux\right)
\leq1$, thus $0\leq u\leq\mathbf{1}$.

For general positive bounded operators $u$, we have to refer to the Spectral
Theorem. The spectrum $\sigma\left(  a\right)  $ of $a\in\mathcal{A}$ is
\[
\sigma\left(  a\right)  =\left\{  s\in\mathbb{C}:a-s\mathbf{1}\text{ is not
invertible}\right\}  .
\]
(Conw90, VII.3.1). For a von Neumann algebra, $\sigma\left(  a\right)  $ is a
nonempty compact subset of $\mathbb{C}$ (Conw90, VII.3.6). The spectral radius
$r\left(  a\right)  $ is
\[
r\left(  a\right)  =\sup\left\{  \left\vert s\right\vert :s\in\sigma\left(
u\right)  \right\}
\]
which is finite and the supremum is attained. If $a$ is hermitian then
$\sigma\left(  a\right)  \subset\mathbb{R}$ (Conw00\cite{MR1721402}, 1.11),
and
\begin{equation}
\left\Vert a\right\Vert =r\left(  a\right)  \label{spec-radius-norm}%
\end{equation}
(Conw00\cite{MR1721402}, 1.7a). If $a$ is positive then $\sigma\left(
a\right)  \subset\left[  0,\infty\right)  $ (should be obvious, or find a
reference?), hence
\[
r\left(  a\right)  =\sup\left\{  s:s\in\sigma\left(  u\right)  \right\}  .
\]
Now we refer to the functional calculus for normal operators
(Conw00\cite{MR1721402}, after Theorem 2.3). Normal operators are those which
fulfill $aa^{\ast}=a^{\ast}a$, so hermitian and positive operators are a
special case. Let $f$ be a continuous function $f$ on $\sigma\left(  a\right)
$; then it is possible to define the operator $f\left(  a\right)  $. The
Spectral Mapping Theorem then states (Conw00\cite{MR1721402}, 2.9)%
\begin{equation}
\sigma\left(  f\left(  a\right)  \right)  =f\left(  \sigma\left(  a\right)
\right)  . \label{spec-mapping-theorem}%
\end{equation}
Here $f\left(  a\right)  $ is first defined in an abstract way
(Conw00\cite{MR1721402}, after Theorem 2.3), but later the Spectral Theorem
for normal operators is used (Conw00\cite{MR1721402}, relation 10.1): if $a$
is normal then
\begin{equation}
f\left(  a\right)  =\int_{\sigma\left(  a\right)  }f\left(  x\right)
dE\left(  x\right)  \label{spec-repre}%
\end{equation}
where $E$ is a spectral measure (i.e. PVM) on the Borel subsets of
$\sigma\left(  a\right)  $. In our context $\left\Vert u\right\Vert \leq1$
implies $r\left(  u\right)  \leq1$ hence $\sigma\left(  u\right)
\subset\left[  0,1\right]  $. With $f\left(  s\right)  =s^{1/2}$ and
(\ref{spec-mapping-theorem}) we obtain $\sigma\left(  u^{1/2}\right)  =$
$\left(  \sigma\left(  u\right)  \right)  ^{1/2}\subset\left[  0,1\right]  $,
hence by (\ref{spec-radius-norm}) for all $x\in\mathcal{H}$ with $\left\Vert
x\right\Vert \leq1$
\[
\left(  x,ux\right)  \leq\left\Vert u^{1/2}\right\Vert ^{2}=\left(  r\left(
u^{1/2}\right)  \right)  ^{2}=r\left(  u\right)  \leq1=\left(  x,\mathbf{1}%
x\right)
\]
implying $u\leq\mathbf{1}$. (The last relation should follow more directly
from (\ref{spec-repre})).

Now for $u\leq\mathbf{1}$ we have $\mathbf{1}-u\geq0$ hence $0\leq
\omega\left(  \mathbf{1}-u\right)  =\omega\left(  \mathbf{1}\right)
-\omega\left(  u\right)  $, hence $\omega\left(  u\right)  \leq\omega\left(
\mathbf{1}\right)  $. This means that
\[
\left\Vert \omega\right\Vert _{1}=\sup_{u\geq0,\left\Vert u\right\Vert \leq
1}\omega\left(  u\right)  \leq\sup_{u\geq0,u\leq\mathbf{1}}\omega\left(
u\right)  =\omega\left(  \mathbf{1}\right)
\]
since the last supremum is attained at $u=\mathbf{1}$. On the other hand
$\mathbf{1}$ fulfills $\left\Vert \mathbf{1}\right\Vert \leq1$ hence
$\omega\left(  \mathbf{1}\right)  \leq\left\Vert \omega\right\Vert _{1}$ and
we obtain the claim.
\end{proof}

\bigskip

\bigskip For our purpose, we a define a "trace preserving" (on positives) map
(TP-map) $\alpha_{\ast}$ between preduals $\alpha_{\ast}:\mathcal{B}_{\ast
}\rightarrow\mathcal{A}_{\ast}$ as a linear map such that
\[
\left\Vert \alpha_{\ast}\left(  \omega\right)  \right\Vert _{1}=\left\Vert
\omega\right\Vert _{1}\text{ for all }\omega\geq0\text{, }\omega\in
\mathcal{B}_{\ast}.
\]
According to (Conw00\cite{MR1721402}, 46.4), if $\mathcal{A}\subset
\mathcal{L}\left(  \mathcal{H}\right)  $ and $\psi$ is a positive linear
functional, then $\psi$ is weak* continuous if and only if there exists a
positive trace class operator $C$ such that
\[
\psi\left(  a\right)  =\mathrm{tr}\left(  aC\right)  ,\text{ }a\in
\mathcal{A}.
\]
Here $C\in\mathcal{L}^{1}\left(  \mathcal{H}\right)  $ is unique if
$\mathcal{A}=\mathcal{L}\left(  \mathcal{H}\right)  $, since $\mathcal{L}%
\left(  \mathcal{H}\right)  $ is the norm dual of $\mathcal{L}^{1}\left(
\mathcal{H}\right)  $. If $\mathcal{A}$ is a proper subspace of $\mathcal{L}%
\left(  \mathcal{H}\right)  $ then $C$ is not unique; indeed to $C$ we can add
any $D\in\mathcal{L}^{1}\left(  \mathcal{H}\right)  $ such that $\mathrm{tr}%
\left(  aD\right)  =0,$ $a\in\mathcal{A}$. This leads to the representation of
$\mathcal{A}_{\ast}$ as a "factor Banach space", cf. (Conw00\cite{MR1721402},
before 46.4).

For general $a\in\mathcal{A}$ and $\omega\in\mathcal{A}_{\ast}$ we will write
\[
\omega\left(  a\right)  =\left\langle a,\omega\right\rangle .
\]
Let $\alpha$ be a channel between von Neumann algebras $\alpha:\mathcal{A}%
\rightarrow\mathcal{B}$ and let $T$ be a linear map $T:\mathcal{B}_{\ast
}\rightarrow\mathcal{A}_{\ast}$. The pair $\left(  \alpha,T\right)  $ is said
to be a dual pair if
\[
\left\langle \alpha\left(  a\right)  ,\omega\right\rangle =\left\langle
a,T\left(  \omega\right)  \right\rangle \text{, }a\in\mathcal{A}\text{,
}\omega\in\mathcal{A}_{\ast}.
\]

\bigskip

\begin{proposition}
Let $\mathcal{A}$, $\mathcal{B}$ von Neumann algebras acting on Hilbert spaces
$\mathcal{H}_{A},$ $\mathcal{H}_{B}$ respectively, and let $\mathcal{A}_{\ast
}$,$\mathcal{B}_{\ast}$ be the respective preduals. \newline(i) Let $\alpha$
be a channel $\alpha:\mathcal{A}\rightarrow\mathcal{B}$. Then there is a
uniquely determined TP-CP map $T:\mathcal{B}_{\ast}\rightarrow\mathcal{A}%
_{\ast}$ such that $\left(  \alpha,T\right)  $ is a dual pair. \newline(ii)
Let $T:\mathcal{B}_{\ast}\rightarrow\mathcal{A}_{\ast}$ be a TP-CP map. Then
there is a uniquely determined channel $\alpha:\mathcal{A}\rightarrow
\mathcal{B}$ such that $\left(  \alpha,T\right)  $ is a dual pair.
\end{proposition}

\begin{proof}
\textbf{(i)} Consider $\phi\in\mathcal{B}_{\ast}$ and
\[
\left\langle \alpha\left(  a\right)  ,\phi\right\rangle =\phi\left(
\alpha\left(  a\right)  \right)  .
\]
According to Conw00\cite{MR1721402}, 46.5, a positive linear map
$\alpha:\mathcal{A}\rightarrow\mathcal{B}$ is normal if and only if it is
weak* continuous. Hence $\phi\left(  \alpha\left(  a\right)  \right)  $ is a
weak* continuous linear form on $\mathcal{A}$, i.e. an element of
$\mathcal{A}_{\ast}$, which can be written $\omega\left(  a\right)  $,
$\omega\in\mathcal{A}_{\ast}$. Hence
\[
\left\langle \alpha\left(  a\right)  ,\phi\right\rangle =\left\langle
a,\omega\right\rangle .
\]
This $\omega$ is unique, and we can write $\omega=T\left(  \phi\right)  $ when
$\alpha$ is fixed, so we have
\begin{equation}
\left\langle \alpha\left(  a\right)  ,\phi\right\rangle =\left\langle
a,T\left(  \phi\right)  \right\rangle . \label{duality-1}%
\end{equation}
The map $T$ is obviously linear. It is positive: assume $\phi\geq0$; we have
to show that $T\left(  \phi\right)  \geq0$. This means $\left\langle
a,T\left(  \phi\right)  \right\rangle \geq0$ whenever $a\geq0$. But if
$a\geq0$ then $\alpha\left(  a\right)  \geq0$ since $\alpha$ is a positive
map. The latter fact implies $\left\langle \alpha\left(  a\right)
,\phi\right\rangle \geq0$, hence implies $\left\langle a,T\left(  \phi\right)
\right\rangle \geq0$.

For the "trace preserving" property, we have to show
\begin{equation}
\left\Vert T\left(  \phi\right)  \right\Vert _{1}=\left\Vert \phi\right\Vert
_{1}\text{ whenever }\phi\geq0\text{. } \label{TP-1}%
\end{equation}
So let $\phi\geq0$; then $T\left(  \phi\right)  \geq0$ and according to Lemma
\ref{Lem-notes-normlinform}
\begin{align*}
\left\Vert T\left(  \phi\right)  \right\Vert _{1}  &  =T\left(  \phi\right)
\left(  \mathbf{1}\right)  \text{ where }\mathbf{1\in}\mathcal{B}\\
&  =\left\langle \mathbf{1},T\left(  \phi\right)  \right\rangle =\left\langle
\alpha\left(  \mathbf{1}\right)  ,\phi\right\rangle
\end{align*}
by (\ref{duality-1}). But $\alpha\left(  \mathbf{1}\right)  =\mathbf{1\in
}\mathcal{B}$ since $\alpha$ is unital, so
\[
\left\Vert T\left(  \phi\right)  \right\Vert _{1}=\left\langle \mathbf{1}%
,\phi\right\rangle =\left\Vert \phi\right\Vert _{1}%
\]
and (\ref{TP-1}) is shown.

It remains to show complete positivity of the map $T$. A $n\times n$ matrix
$a=\left(  a_{ij}\right)  _{i,j=1}^{n}$ with entries $a_{ij}\in\mathcal{A}$ is
called positive if the associated linear operator $a$ on the $n$-fold direct
sum $\mathcal{H}_{A}^{(n)}:=\mathcal{H}_{A}\oplus\ldots\oplus\mathcal{H}_{A}$
is positive, i.e $a$ is self-adjoint and $\left\langle x|ax\right\rangle
\geq0$ for every $x\in\mathcal{H}_{A}^{(n)}$. A linear map $\alpha
:\mathcal{A}\rightarrow\mathcal{B}$ is completely positive if for every
$n\geq1$ and every positive $a=\left(  a_{ij}\right)  _{i,j=1}^{n}$, with
$a_{ij}\in\mathcal{A}$, the matrix
\[
\alpha_{n}\left(  a\right)  :=\left(  \alpha\left(  a_{ij}\right)  \right)
_{i,j=1}^{n}%
\]
is positive (\cite{MR1721402}, 34.2).

The algebra $M_{n}\left(  \mathcal{A}\right)  $ of all $n\times n$ matrices
with entries from $\mathcal{A}$ acting on $\mathcal{H}_{A}^{(n)}$ is a von
Neumann algebra, with norm derived from its being a subalgebra of
$\mathcal{L}\left(  \mathcal{H}_{A}^{(n)}\right)  $ (\cite{MR1721402}, \S 34,
\S 44). Its predual is seen to be the Banach space $M_{n}\left(
\mathcal{A}\right)  _{\ast}$ of of all $n\times n$ matrices with entries from
$\mathcal{A}_{\ast}$, acting on $M_{n}\left(  \mathcal{A}\right)  $ according
to
\[
\left\langle a,\tau\right\rangle =\sum_{i,j=1}^{n}\left\langle a_{ij}%
,\tau_{ij}\right\rangle \text{, }a\in M_{n}\left(  \mathcal{A}\right)  \text{,
}\tau\in M_{n}\left(  \mathcal{A}\right)  _{\ast}\text{ }%
\]
where $a=\left(  a_{ij}\right)  _{i,j=1}^{n}$, $\tau=\left(  \tau_{ij}\right)
_{i,j=1}^{n}$. Here the norm of $M_{n}\left(  \mathcal{A}\right)  _{\ast}$ is
\[
\left\Vert \tau\right\Vert _{1}=\sup_{a\in M_{n}\left(  \mathcal{A}\right)
,\left\Vert a\right\Vert =1}\left\vert \left\langle a,\tau\right\rangle
\right\vert \text{, }\tau\in M_{n}\left(  \mathcal{A}\right)  _{\ast}.
\]
Let $\mathbf{1}$\textbf{ }be the unit of\textbf{ }$\mathcal{A}$ and let
$\mathbf{1}_{n}$ be the unit of $M_{n}\left(  \mathcal{A}\right)  $, i.e. the
diagonal matrix with diagonal entries all $\mathbf{1}$. Let $\tau\in$
$M_{n}\left(  \mathcal{A}\right)  _{\ast}$, $\tau\geq0$; then%
\[
\left\Vert \tau\right\Vert _{1}=\left\langle \mathbf{1}_{n},\tau\right\rangle
=\sum_{i=1}^{n}\left\langle \mathbf{1},\tau_{ii}\right\rangle =\sum_{i=1}%
^{n}\left\Vert \tau_{ii}\right\Vert _{1}.
\]
A map $T:\mathcal{B}_{\ast}\rightarrow\mathcal{A}_{\ast}$ is completely
positive if for every $n\geq1$ and every $\tau\geq0$, $\tau\in M_{n}\left(
\mathcal{B}\right)  _{\ast}$ $,$ the matrix
\[
T_{n}\left(  \tau\right)  :=\left(  T\left(  \tau_{ij}\right)  \right)
_{i,j=1}^{n}%
\]
is positive in $M_{n}\left(  \mathcal{A}\right)  _{\ast}$. This means that for
every $n\geq1$, the map $T_{n}:$ $M_{n}\left(  \mathcal{B}\right)  _{\ast
}\rightarrow M_{n}\left(  \mathcal{A}\right)  _{\ast}$ is positive. Assume
complete positivity of the map $\alpha$, i.e. positivity of the map
$\alpha_{n}$ for all $n\geq1$. Now
\begin{subequations}
\begin{align}
\left\langle \alpha_{n}\left(  a\right)  ,\tau\right\rangle  &  =\sum
_{i,j=1}^{n}\left\langle \alpha\left(  a_{ij}\right)  ,\tau_{ij}\right\rangle
=\sum_{i,j=1}^{n}\left\langle a_{ij},T\left(  \tau_{ij}\right)  \right\rangle
\nonumber\\
&  =\left\langle a,T_{n}\left(  \tau\right)  \right\rangle .\text{ }
\label{duality-matrix-b}%
\end{align}
Now positivity of the map $T_{n}$ is inferred in the same way as the
positivity of the original map $T$.

\textbf{(ii)} Consider $a\in\mathcal{A}$, $\phi\in\mathcal{B}_{\ast}$ and the
expression
\end{subequations}
\[
\left\langle a,T\left(  \phi\right)  \right\rangle
\]
Every $\phi\in\mathcal{B}_{\ast}$ can be written as a complex linear
combination of positive elements of $\mathcal{B}_{\ast}$:%
\begin{equation}
\phi=\phi_{1}-\phi_{2}+i\left(  \phi_{3}-\phi_{4}\right)
\label{decmpose-into-positives}%
\end{equation}
where $\phi_{i}\geq0$, $i=1,\ldots,4$ and
\[
\phi_{(1)}:=\phi_{1}-\phi_{2}=\frac{1}{2}\left(  \phi+\phi^{\ast}\right)
\text{, }\phi_{(2)}:=\phi_{3}-\phi_{4}=\frac{1}{2i}\left(  \phi-\phi^{\ast
}\right)
\]
are Hermitian elements of $\mathcal{B}_{\ast}$ (\cite{MR1721402} 54.9). Hence
\begin{equation}
\left\Vert T\left(  \phi\right)  \right\Vert _{1}\leq\sum_{i=1}^{4}\left\Vert
T\left(  \phi_{i}\right)  \right\Vert _{1}=\sum_{i=1}^{4}\left\Vert \phi
_{i}\right\Vert _{1}. \label{norm-inequ-transitions}%
\end{equation}
Here
\[
\left\Vert \phi_{1}\right\Vert _{1}+\left\Vert \phi_{2}\right\Vert
_{1}=\left\Vert \phi_{(1)}\right\Vert _{1}%
\]
since $\phi_{1}$, $\phi_{2}$ are orthogonal (\cite{MR1721402} 54.8) ,and
similarly
\[
\left\Vert \phi_{3}\right\Vert _{1}+\left\Vert \phi_{4}\right\Vert
_{1}=\left\Vert \phi_{(2)}\right\Vert _{1}.
\]
But $\left\Vert \phi_{(i)}\right\Vert _{1}\leq\left\Vert \phi\right\Vert _{1}%
$, $i=1,2$ is obvious, and with (\ref{norm-inequ-transitions}) we obtain%
\[
\left\Vert T\left(  \phi\right)  \right\Vert _{1}\leq2\left\Vert
\phi\right\Vert _{1}%
\]
and so $T$ is a norm continuous map between Banach spaces $T:\mathcal{B}%
_{\ast}\rightarrow\mathcal{A}_{\ast}$. Since $a\in\mathcal{A}$ is a norm
continuous linear form on $\mathcal{A}_{\ast}$, it follows that, for given
$a$, the linear form $\left\langle a,T\left(  \phi\right)  \right\rangle $ on
$\phi\in\mathcal{B}_{\ast}$ is norm continuous, and so is an element of
$\mathcal{B}$. Denoting this element $\alpha\left(  a\right)  $, we have
\begin{equation}
\left\langle \alpha\left(  a\right)  ,\phi\right\rangle =\left\langle
a,T\left(  \phi\right)  \right\rangle ,a\in\mathcal{A},\phi\in\mathcal{B}%
_{\ast}. \label{duality--transitions}%
\end{equation}
To show that $\alpha$ is a positive map, we first show for $b\in\mathcal{B}$,
the property
\begin{equation}
\left\langle b,\phi\right\rangle \geq0\text{ for every }\phi\geq0,\phi\in
\phi\in\mathcal{B}_{\ast} \label{implies-positive}%
\end{equation}
implies $b\geq0$. Indeed, for $x\in\mathcal{H}_{B}$, consider a positive
linear map on $\mathcal{B}$%
\[
\phi_{x}\left(  b\right)  =\mathrm{tr}\ b\left\vert x\right\rangle
\left\langle x\right\vert =\left\langle x,bx\right\rangle .
\]
The operator $\left\vert x\right\rangle \left\langle x\right\vert $ on
$\mathcal{H}_{B}$ is trace class, so the map $\phi_{x}$ is weak* continuous
(\cite{MR1721402}, 46.6), hence $\phi_{x}\in\mathcal{B}_{\ast}$. So if
(\ref{implies-positive}) holds then, applying it for all $\phi_{x}$, we see
that $b$ must be positive.

Now to show positivity of $\alpha$, assume $a\geq0$; we have to show that then
$\left\langle \alpha\left(  a\right)  ,\phi\right\rangle \geq0$ for every
$\phi\geq0$. But then $T\left(  \phi\right)  \geq0$ and so $\left\langle
a,T\left(  \phi\right)  \right\rangle \geq0$, swo that
(\ref{duality--transitions}) shows $\left\langle \alpha\left(  a\right)
,\phi\right\rangle \geq0$.

To show that $\alpha$ is unital, set $a=\mathbf{1}$ in
(\ref{duality--transitions}) and let $\phi\geq0$. Then%
\[
\left\langle \alpha\left(  \mathbf{1}\right)  ,\phi\right\rangle =\left\langle
\mathbf{1},T\left(  \phi\right)  \right\rangle =\left\Vert T\left(
\phi\right)  \right\Vert _{1}=\left\Vert \phi\right\Vert _{1}=\left\langle
\mathbf{1},\phi\right\rangle \text{.}%
\]
But (\ref{decmpose-into-positives}) shows that $\left\langle \alpha\left(
\mathbf{1}\right)  ,\phi\right\rangle =\left\langle \mathbf{1},\phi
\right\rangle $ holds for all $\phi\in\mathcal{B}_{\ast}$, which shows that
$\alpha\left(  \mathbf{1}\right)  =\mathbf{1}$.

To show that $\alpha$ is completely positive, we have to show that for every
$n\geq1$, the map $\alpha_{n}:$ $M_{n}\left(  \mathcal{A}\right)  \rightarrow
M_{n}\left(  \mathcal{B}\right)  $ is positive. Assume $a\in M_{n}\left(
\mathcal{A}\right)  $, $a\geq0$; then by (\ref{duality-matrix-b}) for every
$\tau\in M_{n}\left(  \mathcal{B}\right)  _{\ast}$, $\tau\geq0$%
\begin{subequations}
\begin{equation}
\left\langle \alpha_{n}\left(  a\right)  ,\tau\right\rangle =\left\langle
a,T_{n}\left(  \tau\right)  \right\rangle .\nonumber
\end{equation}
But since $T$ is completely positive, we have $T_{n}\left(  \tau\right)
\geq0$, so that $\left\langle \alpha_{n}\left(  a\right)  ,\tau\right\rangle
\geq0$ for every positive $\tau\geq0$. The reasoning in connection with
(\ref{implies-positive}), applied to the von Neumann algebra $M_{n}\left(
\mathcal{B}\right)  $, then shows that $\alpha_{n}\left(  a\right)  \geq0$, so
the map $\alpha_{n}$ is indeed positive.

To show that $\alpha$ is weak* continuous, since $\alpha$ is positive, it
suffices by \cite{MR1721402}, 46.5 to show that $\alpha$ is normal. Let
$\left\{  a_{i}\right\}  $ be an increasing sequence of positive elements of
$\mathcal{A}$ such that $a_{i}\rightarrow a$ (SOT). [As noted in
\S \ref{subsubsec-normal}, since Hilbert spaces are assumed separable, in the
definition of normality, nets can be replaced by sequences.] Then by
\cite{MR1721402}, 43.1 we have $a_{i}\rightarrow a$ in the weak* sense. Then
for every $\psi\in\mathcal{A}_{\ast}$ we have $\left\langle a_{i}%
,\psi\right\rangle \rightarrow\left\langle a,\psi\right\rangle ,$ and applying
(\ref{duality--transitions}) for $\psi=T\left(  \phi\right)  $ we obtain
$\left\langle \alpha\left(  a_{i}\right)  ,\phi\right\rangle \rightarrow
\left\langle \alpha\left(  a\right)  ,\phi\right\rangle $, $\phi\in
\mathcal{B}_{\ast}$. This means $\alpha\left(  a_{i}\right)  \rightarrow
\alpha\left(  a\right)  $ in the weak* sense, and since $\alpha$ is positive,
the sequence $\alpha\left(  a_{i}\right)  $ is increasing and $\sup
_{i}\left\Vert \alpha\left(  a_{i}\right)  \right\Vert \leq\left\Vert
\alpha\left(  a\right)  \right\Vert $. Then by \cite{MR1721402}, 43.1
$\alpha\left(  a_{i}\right)  \rightarrow\alpha\left(  a\right)  $ (SOT), which
shows that $\alpha$ is normal.
\end{subequations}
\end{proof}

\bigskip

\bigskip

\begin{remark}
The proof of complete positivity becomes nearly trivial by using the fact that
$M_{n}\left(  \mathcal{A}\right)  _{\ast}$ is the predual of $M_{n}\left(
\mathcal{A}\right)  $. We haven't proved that, but it seems nearly obvious. To
remove all doubt, below we show it for the special case that $\mathcal{A}%
=\mathcal{L}\left(  \mathcal{H}_{A}\right)  $, the full algebra of bounded
operators on $\mathcal{H}_{A}$.
\end{remark}

\bigskip

Indeed, consider now the special case that $\mathcal{A}=\mathcal{L}\left(
\mathcal{H}\right)  $ so that $\mathcal{A}_{\ast}=\mathcal{L}^{1}\left(
\mathcal{H}\right)  $, the space of trace class operators on $\mathcal{H}$.
Then $M_{n}\left(  \mathcal{A}\right)  $ can be identified with $\mathcal{L}%
\left(  \mathcal{H}^{(n)}\right)  \mathcal{\ }$where $\mathcal{H}^{(n)}$ is
the $n$-fold inflation of $\mathcal{H}$ (the $n$-fold direct sum of copies of
$\mathcal{H}$). Thus an element $A\in$ $\mathcal{L}\left(  \mathcal{H}%
^{(n)}\right)  $ can be written $A=\left(  a_{ij}\right)  _{i,j=1}^{n}$ where
each $a_{ij}$ is an element of $\mathcal{L}\left(  \mathcal{H}\right)  $. An
element $\Phi$ of the predual $\mathcal{L}^{1}\left(  \mathcal{H}%
^{(n)}\right)  $, i.e. a trace class operator on $\mathcal{H}^{(n)}$, can then
also be written as $\Phi=\left(  \phi_{ij}\right)  _{i,j=1}^{n}$.

\begin{lemma}
An operator on $\mathcal{H}^{(n)}:$ $A=$ $\left(  a_{ij}\right)  _{i,j=1}^{n}$
is bounded if and only if each $a_{ij}$ is bounded.
\end{lemma}

\begin{proof}
Indeed let $A_{ij}=\left(  b_{kl}\right)  _{k,l=1}^{n}$ be the $n\times n$
operator matrix such that $b_{kl}=0$ if $(k,l)\neq\left(  i,j\right)  $,
$b_{kl}=a_{ij}$ if $(k,l)=\left(  i,j\right)  $. Then $A=\sum_{i,j=1}%
^{n}A_{ij}$ and if all $A_{ij}$ are bounded then
\[
\left\Vert A\right\Vert \leq\sum_{i,j=1}^{n}\left\Vert A_{ij}\right\Vert .
\]
Now for $y\in\mathcal{H}^{(n)}$, $y=y_{1}\oplus\ldots\oplus y_{n}$
\begin{align}
\left\Vert A_{ij}\right\Vert  &  =\sup_{y\in\mathcal{H}^{(n)},\left\Vert
y\right\Vert \leq1}\left\Vert A_{ij}y\right\Vert =\sup_{y_{j}\in
\mathcal{H},\left\Vert y_{j}\right\Vert \leq1}\left\Vert a_{ij}y_{j}%
\right\Vert \nonumber\\
&  =\left\Vert a_{ij}\right\Vert \label{been-noted-in}%
\end{align}
which shows that if each $a_{ij}$ is bounded then so is $A$. On the other hand%
\[
\left\Vert A\right\Vert ^{2}=\sup_{y\in\mathcal{H}^{(n)},\left\Vert
y\right\Vert \leq1}\left\Vert Ay\right\Vert ^{2}=\sup_{y\in\mathcal{H}%
^{(n)},\left\Vert y\right\Vert \leq1}\sum_{j,k=1}^{n}\sum_{i=1}^{n}%
\left\langle y_{j},a_{ij}^{\ast}a_{ik}y_{k}\right\rangle .
\]
Here we can restrict the supremum to $y\in\mathcal{H}^{(n)}$ of form
$y=y_{1}\oplus\ldots\oplus y_{n}$ where $y_{j}=0$ for $j\neq l$, $y_{j}=x$ for
$j=l$, where $x\in\mathcal{H}$. Then%
\begin{align*}
\left\Vert A\right\Vert ^{2}  &  \geq\sup_{x\in\mathcal{H},\left\Vert
x\right\Vert \leq1}\sum_{i=1}^{n}\left\langle x,a_{il}^{\ast}a_{il}%
x\right\rangle \geq\sup_{x\in\mathcal{H},\left\Vert x\right\Vert \leq
1}\left\langle x,a_{il}^{\ast}a_{il}x\right\rangle \\
&  =\left\Vert a_{il}\right\Vert ^{2}%
\end{align*}
where $i,l\in\left\{  1,\ldots,n\right\}  $ are arbitrary. Hence boundedness
of $A$ implies boundedness of each $a_{il}$.
\end{proof}

\bigskip

\begin{lemma}
A bounded operator $\Phi=\left(  \phi_{ij}\right)  _{i,j=1}^{n}$ on
$\mathcal{H}^{(n)}$ is trace class if and only if each $\phi_{ij}$ is trace
class on $\mathcal{H}$.
\end{lemma}

\begin{proof}
The trace norm of $\Phi$ is
\begin{align}
\left\Vert \Phi\right\Vert _{1}  &  =\mathrm{Tr\,}\left(  \Phi^{\ast}%
\Phi\right)  ^{1/2}=\sup_{A\in\mathcal{L}\left(  \mathcal{H}^{(n)}\right)
,\left\Vert A\right\Vert \leq1}\left\vert \mathrm{Tr\,}A\Phi\right\vert
\nonumber\\
&  =\sup_{A\in\mathcal{L}\left(  \mathcal{H}^{(n)}\right)  ,\left\Vert
A\right\Vert \leq1}\left\vert \sum_{i,j=1}^{n}\mathrm{Tr\,}a_{ij}\phi
_{ij}\right\vert . \label{trace-class-line-2}%
\end{align}
Let $H_{ij}$ be the set of operator matrices $A=\left(  a_{kl}\right)
_{k,l=1}^{n}$ where $a_{kl}=0$ if $\left(  k,l\right)  \neq\left(  i,j\right)
$. Note that for any $A\in H_{ij}$ and $y=y_{1}\oplus\ldots\oplus y_{n}%
\in\mathcal{H}^{(n)}$
\[
\left\Vert A\right\Vert =\sup_{y\in\mathcal{H}^{(n)},\left\Vert y\right\Vert
\leq1}\left\Vert Ay\right\Vert =\sup_{y_{j}\in\mathcal{H},\left\Vert
y_{j}\right\Vert \leq1}\left\Vert a_{ij}y_{j}\right\Vert =\left\Vert
a_{ij}\right\Vert .
\]
as has already been noted in (\ref{been-noted-in}). Then, by restricting the
supremum in (\ref{trace-class-line-2}) to $A\in H_{ij}$
\begin{align}
\left\Vert \Phi\right\Vert _{1}  &  \geq\sup_{A\in H_{ij},\left\Vert
A\right\Vert \leq1}\left\vert \mathrm{Tr\,}a_{ij}\phi_{ij}\right\vert
=\sup_{a_{ij}\in\mathcal{L}\left(  \mathcal{H}\right)  ,\left\Vert
a_{ij}\right\Vert \leq1}\left\vert \mathrm{Tr\,}a_{ij}\phi_{ij}\right\vert
\nonumber\\
&  =\left\Vert \phi_{ij}\right\Vert _{1} \label{trace-class-line-5}%
\end{align}
so that if $\Phi$ is trace class, then so is each $\phi_{ij}$.

To show the converse, assume that $\left\Vert \phi_{ij}\right\Vert _{1}%
<\infty$ for each $i,j$. Let $\Phi_{ij}=\left(  b_{kl}\right)  _{k,l=1}^{n}$
be the $n\times n$ operator matrix such that $b_{kl}=0$ if $(k,l)\neq\left(
i,j\right)  $, $b_{kl}=\phi_{ij}$ if $(k,l)=\left(  i,j\right)  $. Then
$\Phi_{ij}^{\ast}\Phi_{ij}$ is the $n\times n$ operator matrix where the
$(j,j)$ entry is $\phi_{ij}^{\ast}\phi_{ij}$ and all other entries are $0$.
Hence
\[
\left\Vert \Phi_{ij}\right\Vert _{1}=\mathrm{Tr\,}\left(  \Phi_{ij}^{\ast}%
\Phi_{ij}\right)  ^{1/2}=\mathrm{Tr\,}\left(  \phi_{ij}^{\ast}\phi
_{ij}\right)  ^{1/2}=\left\Vert \phi_{ij}\right\Vert _{1},
\]
and then the inequality
\[
\left\Vert \Phi\right\Vert _{1}=\left\Vert \sum_{i,j=1}^{n}\Phi_{ij}%
\right\Vert _{1}\leq\sum_{i,j=1}^{n}\left\Vert \Phi_{ij}\right\Vert _{1}%
=\sum_{i,j=1}^{n}\left\Vert \phi_{ij}\right\Vert _{1}%
\]
shows that if all $\phi_{ij}$ are trace class then so it $\Phi$.
\end{proof}

\subsection{Classical transitions as defining quantum channels}

The book of Strasser \cite{MR812467} gives a comprehensive account of the
theory of classical statistical experiments. The mathematical tools used there
to define equivalence and deficiency relate to spaces of real valued
functions; to embed this in the quantum setting of von Neumann algebras we
need to discuss the relationship to complex function spaces like $L^{1}\left(
X,\Omega,\nu\right)  $. Let $\left(  X,\Omega,\nu\right)  $ be a sigma-finite
measure space; we denote by $L_{r}^{1}\left(  X,\Omega,\nu\right)  $ the
subset of $L^{1}\left(  X,\Omega,\nu\right)  $ consisting of (equivalence
classes of) real valued functions; this is a real Banach space. Analogously
let $L_{r}^{\infty}\left(  X,\Omega,\nu\right)  $ consist of the real valued
functions in $L^{\infty}\left(  X,\Omega,\nu\right)  $, understood as a real
Banach space.

Definitions 24.1 and 24.3 in \cite{MR812467} give the concepts of stochastic
operator and Markov operator. Let $\left(  X_{i},\Omega_{i},\nu_{i}\right)  $,
$i=1,2$ be sigma-finite measure spaces. A \textit{stochastic operator }is a
positive linear operator $M_{1}:L_{r}^{1}\left(  X_{1},\Omega_{1},\nu
_{1}\right)  \rightarrow L_{r}^{1}\left(  X_{2},\Omega_{2},\nu_{2}\right)  $
with the property that $\left\Vert M_{1}g\right\Vert _{1}=\left\Vert
g\right\Vert _{1}$ for every $g\in L_{r+}^{1}\left(  \nu_{1}\right)  $ (the
set of nonnegative elements of $L_{r}^{1}\left(  \nu_{1}\right)  =L_{r}%
^{1}\left(  X_{1},\Omega_{1},\nu_{1}\right)  $). A \textit{Markov
operator\footnote{Not be confused with a Markov kernel, which is given by
(\ref{MK-2}). Rather a Markov kernel gives a special case of a transition, as
seen from (\ref{MK-1}). But it also defines a Markov operator by the mapping
$f\rightarrow\int f\left(  \omega\right)  K\left(  d\omega,\cdot\right)  $ (a
"conditional expectation" of $f$).}} is a positive linear operator
$M_{2}:L_{r}^{\infty}\left(  X_{2},\Omega_{2},\nu_{2}\right)  \rightarrow
L_{r}^{\infty}\left(  X_{1},\Omega_{1},\nu_{1}\right)  $ with $M_{2}%
\mathbf{1}=\mathbf{1}$ and the property that $f_{n}\downarrow0$, $\left(
f_{n}\right)  _{n\in\mathbb{N}}$ $\subset L_{r}^{\infty}\left(  \nu
_{2}\right)  $ implies $M_{2}f_{n}\downarrow0$. Here $L_{r}^{\infty}\left(
\nu_{2}\right)  =L_{r}^{\infty}\left(  X_{2},\Omega_{2},\nu_{2}\right)  $ and
$f_{n}\downarrow0$ means monotone pointwise convergence to zero. In this
setting, $M_{1}$ and $M_{2}$ are a dual pair if
\begin{equation}
\int f\left(  M_{1}g\right)  d\nu_{2}=\int\left(  M_{2}f\right)  gd\nu
_{1},\text{ all }g\in L_{r}^{1}\left(  \nu_{1}\right)  ,\;f\in L_{r}^{\infty
}\left(  \nu_{2}\right)  . \label{duality-classical}%
\end{equation}
(Definition 24.3). Then for every stochastic operator $M_{1}:$ $L_{r}%
^{1}\left(  \nu_{1}\right)  \rightarrow L_{r}^{1}\left(  \nu_{2}\right)  $
there is a uniquely defined Markov operator $M_{2}:$ $L_{r}^{\infty}\left(
\nu_{2}\right)  \rightarrow L_{r}^{\infty}\left(  \nu_{1}\right)  $ which is
the dual to $M_{1}$, and vice versa: for every Markov operator $M_{2}:$
$L_{r}^{\infty}\left(  \nu_{2}\right)  \rightarrow L_{r}^{\infty}\left(
\nu_{1}\right)  $ there is a uniquely defined stochastic operator $M_{1}:$
$L_{r}^{1}\left(  \nu_{1}\right)  \rightarrow L_{r}^{1}\left(  \nu_{2}\right)
$ which is dual to $M_{2}$ (Theorems 24.4, 24.5).

Let $M_{1}$ be a stochastic operator $M_{1}:L_{r}^{1}\left(  \nu_{1}\right)
\rightarrow L_{r}^{1}\left(  \nu_{2}\right)  $; we will extend it to a mapping
$M_{1}^{e}:L^{1}\left(  \nu_{1}\right)  \rightarrow L^{1}\left(  \nu
_{2}\right)  $ which will be seen to be complex linear. Let $f\in L^{1}\left(
\nu_{1}\right)  $ such that $f=f_{1}+if_{2}$ where $f_{1},f_{2}\in L_{r}%
^{1}\left(  \nu_{1}\right)  $. Define
\[
M_{1}^{e}f=M_{1}f_{1}+iM_{1}f_{2}%
\]
Then additivity is trivial: if $g\in L^{1}\left(  \nu_{1}\right)  $ then
$M_{1}^{e}\left(  f+g\right)  =M_{1}^{e}f+M_{1}^{e}g$. The map $M_{1}^{e}$ is
also real homogeneous on $L^{1}\left(  \nu_{1}\right)  $: for $a\in\mathbb{R}$
we have $M_{1}^{e}\left(  af\right)  =aM_{1}^{e}f$, $f\in L^{1}\left(  \nu
_{1}\right)  $. To show that $M_{1}^{e}$ is complex homogeneous, let $z=a+ib$
with $a,b\in\mathbb{R}$. We have for $L^{1}\left(  \nu_{1}\right)  $
\begin{align*}
\left(  a+ib\right)  \left(  f_{1}+if_{2}\right)   &  =af_{1}-bf_{2}+i\left(
af_{2}+bf_{1}\right)  ,\\
M_{1}^{e}\left(  zf\right)   &  =aM_{1}f_{1}-bM_{1}f_{2}+i\left[  aM_{1}%
f_{2}+bM_{1}f_{1}\right]  .
\end{align*}
On the other hand we have
\begin{align*}
zM_{1}^{e}f  &  =\left(  a+ib\right)  \left(  M_{1}f_{1}+iM_{1}f_{2}\right) \\
&  =aM_{1}f_{1}-bM_{1}f_{2}+i\left[  aM_{1}f_{2}+bM_{1}f_{1}\right]
=M_{1}^{e}\left(  zf\right)  .
\end{align*}
showing complex homogeneity of $M_{1}^{e}$. Hence the extended map $M_{1}^{e}$
is complex linear.

Since the properties of $M_{1}^{e}$ of being positive and norm preserving on
positive elements of $L^{1}\left(  \nu_{1}\right)  $ concern only the
restriction of $M_{1}^{e}$ to $L_{r}^{1}\left(  \nu_{1}\right)  $, they
automatically hold for the extended version. Define a complex stochastic
operator $M$ as a complex linear map $M:L^{1}\left(  \nu_{1}\right)
\rightarrow L^{1}\left(  \nu_{2}\right)  $ with the property that its
restriction $M_{r}$ to $L_{r}^{1}\left(  \nu_{1}\right)  $ is a stochastic
operator in the original sense. From the above it follows that $M$ is uniquely
determind by its restiction $M_{r}$.

By understanding elements of $L^{1}\left(  \nu\right)  $ and $L^{\infty
}\left(  \nu\right)  $ as multiplication operators on $L^{2}\left(
\nu\right)  $, we find that $L^{\infty}\left(  \nu\right)  $ is a commutative
von Neumann algebra and $L^{1}\left(  \nu\right)  $ its predual. Thus the
complex stochastic operator $M_{1}^{e}$ is a quantum state transition
$M_{1}^{e}:L^{1}\left(  \nu_{1}\right)  \rightarrow L^{1}\left(  \nu
_{2}\right)  $; indeed the map is also completely positive (cp. XX). By
Theorem XX\ there exists a unique quantum channel $\alpha:L^{\infty}\left(
\nu_{2}\right)  \rightarrow L^{\infty}\left(  \nu_{1}\right)  $ which is dual
to $M_{1}^{e}$, i.e.%
\begin{equation}
\int f\left(  M_{1}^{e}g\right)  d\nu_{2}=\int\left(  \alpha f\right)
gd\nu_{1},\text{ all }g\in L^{1}\left(  \nu_{1}\right)  ,\;f\in L^{\infty
}\left(  \nu_{2}\right)  . \label{duality-commutative-1}%
\end{equation}
By restricting $f$ and $g$ to $f\in L_{r}^{\infty}\left(  \nu_{2}\right)  $
and $g\in L_{r}^{1}\left(  \nu_{1}\right)  $ and comparing to
(\ref{duality-classical}), we see by Theorem 24.4 of \cite{MR812467} that
$\alpha|L_{r}^{\infty}\left(  \nu_{2}\right)  =M_{2}$. We also see that
$M_{1}^{e}=\alpha_{\ast}$.

Let $\mathcal{F}=\left(  q_{\theta},\;\theta\in\Theta\right)  $ be a family of
probability densities on $\left(  X_{1},\Omega_{1},\nu_{1}\right)  $, i.e.
positive elements of $L_{r}^{1}\left(  X_{%
\acute{}%
1},\Omega_{1},\nu_{1}\right)  $ with norm one, and similarly let
$\mathcal{E}=\left(  p_{\theta},\;\theta\in\Theta\right)  $, be a family of
probability densities on $\left(  X_{2},\Omega_{2},\nu_{2}\right)  $. In order
to clarify the relationship of the classical concept of deficiency of
$\mathcal{E}$ with respect to $\mathcal{F}$ to the quantum version discussed
in Section \ref{subsec: Qu-Lecam-distance}, we denote the former temporarily
as $\delta_{r}\left(  \mathcal{E},\mathcal{F}\right)  $ and recall
\begin{equation}
\delta_{r}\left(  \mathcal{E},\mathcal{F}\right)  =\inf_{T}\sup_{\theta
\in\Theta}\left\Vert q_{\theta}-Tp_{\theta}\right\Vert _{1}
\label{real-deficiency-1}%
\end{equation}
where the infimum extends over all stochastic operators $T:$ $L_{r}^{1}\left(
\nu_{2}\right)  \rightarrow L_{r}^{1}\left(  \nu_{1}\right)  $. By the avove
discussion of extensions $T^{e}$, it is clear that in (\ref{real-deficiency-1}%
) the infimum can be taken over all complex stochastic operators
$T:L^{1}\left(  \nu_{1}\right)  \rightarrow L^{1}\left(  \nu_{2}\right)  $.
Now let $Q_{\theta}$ be the state on the von Neumann algebra $\mathcal{A}%
=L^{\infty}\left(  \nu_{1}\right)  $ corresponding to $q_{\theta}$, let
$P_{\theta}$ be the analog for $p_{\theta}$ and $\mathcal{B}=L^{\infty}\left(
\nu_{2}\right)  $ and identify $T$ with a state transition between preduals
$T:\mathcal{B}_{\ast}\rightarrow\mathcal{A}_{\ast}$. Then
(\ref{real-deficiency-1}) can be written
\begin{equation}
\delta_{r}\left(  \mathcal{E},\mathcal{F}\right)  =\inf_{T}\sup_{\theta
\in\Theta}\left\Vert Q_{\theta}-TP_{\theta}\right\Vert _{1}
\label{real-deficiency-2}%
\end{equation}
where $\left\Vert \cdot\right\Vert _{1}$ is the Banach space norm of
$\mathcal{A}_{\ast}$. For any linear form $K\in\mathcal{A}_{\ast}$, write the
action on $a\in\mathcal{A}$ as $K\left(  a\right)  =\left\langle
a,K\right\rangle $ and similarly for $L\in\mathcal{B}_{\ast}$ and $L\left(
b\right)  =\left\langle b,L\right\rangle $, $b\in$ $\mathcal{B}$. Let $\alpha$
be the channel dual to $T$; then
\begin{align*}
\left\langle a,TP_{\theta}\right\rangle  &  =\left\langle \alpha\left(
a\right)  ,P_{\theta}\right\rangle \\
&  =P_{\theta}\left(  \alpha\left(  a\right)  \right)  =\left(  P_{\theta
}\circ\alpha\right)  \left(  a\right)  =\left\langle a,P_{\theta}\circ
\alpha\right\rangle \text{, }a\in\mathcal{A}\text{. }%
\end{align*}
This show that $TP_{\theta}=P_{\theta}\circ\alpha$ and from
(\ref{real-deficiency-2}) we obtain%
\[
\delta_{r}\left(  \mathcal{E},\mathcal{F}\right)  =\inf_{\alpha}\sup
_{\theta\in\Theta}\left\Vert Q_{\theta}-P_{\theta}\circ\alpha\right\Vert
_{1}.
\]
where the infimim is taken over all channels $\alpha:\mathcal{A}%
\rightarrow\mathcal{B}$. It follows that the deficiency
(\ref{real-deficiency-1}) as originally defined by Le Cam in \cite{MR856411}
for the case of dominated experiments is a special case of the quantum
deficiency $\delta\left(  \mathcal{E},\mathcal{F}\right)  $ as defined in
(\ref{deficiency-quantum-def}).

\subsection{ POVM's and state transitions}

A POVM\ (positive operator valued measure) on a measurable space $\left(
X,\Omega\right)  $ is a mapping $M:\Omega\rightarrow\mathcal{\mathcal{L}%
(\mathcal{H})}$ with properties (i) $M\left(  A\right)  \geq0$ (hence
$M\left(  A\right)  $ is self-adjoint), (ii) $M\left(  X\right)  =\mathbf{1}$,
(iii) if $\left\{  A_{j}\right\}  _{j=1}^{\infty}$ are pairwise disjoint sets
from $\Omega$ then
\begin{equation}
M\left(
{\textstyle\bigcup\nolimits_{j=1}^{\infty}}
A_{j}\right)  =%
{\textstyle\sum\nolimits_{j=1}^{\infty}}
M\left(  A_{j}\right)  \label{additive-SOT-convergent}%
\end{equation}
where the r.h.s. is an SOT convergent sum. The following Lemma is related to
Lemma 2.1 in \cite{BGN-QAE}.

\begin{lemma}
\label{lem-POVM-1}For any POVM $M$ over $\left(  X,\Omega\right)  $ and any
state $\rho\in$ $\mathcal{\mathcal{L}}^{1}\mathcal{(\mathcal{H})}$
\begin{equation}
\nu_{\rho}\left(  A\right)  =\mathrm{Tr\,}\rho M\left(  A\right)  \text{,
}A\in\Omega\label{prob-measure-defined-by POVM}%
\end{equation}
is a probability measure on $\Omega$. Furthermore there is a probability
measure $\nu_{0}$ on $\left(  X,\Omega\right)  $ and a state transition%
\[
\mathcal{M}:\mathcal{\mathcal{L}}^{1}\mathcal{(\mathcal{H})}\rightarrow
L^{1}\left(  \nu_{0}\right)
\]
such that for any state $\rho\in$ $\mathcal{\mathcal{L}}^{1}%
\mathcal{(\mathcal{H})}$
\[
\nu_{\rho}\left(  A\right)  =\mathrm{Tr\,}\rho M\left(  A\right)  =\int
_{A}\mathcal{M}\left(  \rho\right)  d\nu_{0}\text{, }\mathcal{M}\left(
\rho\right)  =\frac{d\nu_{\rho}}{d\nu_{0}}\text{, }A\in\Omega.
\]
The measure $\nu_{0}$ can be chosen as $\nu_{0}=\nu_{\rho_{0}}$ for any
faithful state $\rho_{0}\in$ $\mathcal{\mathcal{L}}^{1}\mathcal{(\mathcal{H}%
)}$ (i.e. $\rho_{0}>0$).
\end{lemma}

\begin{proof}
Obviously $\nu_{\rho}\left(  A\right)  \geq0$ and $\nu_{\rho}\left(  X\right)
=\mathrm{Tr\,}\rho M\left(  X\right)  =\mathrm{Tr\,}\rho=1$. It remains to
show sigma-additivity of $\nu_{\rho}$. If $\left\{  A_{j}\right\}
_{j=1}^{\infty}$ are pairwise disjoint sets from $\Omega$ then
\[
\nu_{\rho}\left(
{\textstyle\bigcup\nolimits_{j=1}^{\infty}}
A_{j}\right)  =\mathrm{Tr\,}\rho\left(
{\textstyle\sum\nolimits_{j=1}^{\infty}}
M\left(  A_{j}\right)  \right)  .
\]

Here the sum $%
{\textstyle\sum\nolimits_{j=1}^{\infty}}
M\left(  A_{j}\right)  $ converges in the weak operator toppology (WOT) since
it converges SOT. Furthermore the set $\left\{  M\left(  A_{j}\right)  \text{,
}j\geq1\right\}  $ is bounded in $\mathcal{\mathcal{L}(\mathcal{H})}$ since
$M\left(  A_{j}\right)  \leq\mathbf{1}$, hence by (20.1 (b)) in
\cite{MR1721402} it converges weak*. The latter implies (\S 20 in
\cite{MR1721402}) that
\[
\mathrm{Tr\,}\rho\left(
{\textstyle\sum\nolimits_{j=1}^{\infty}}
M\left(  A_{j}\right)  \right)  =%
{\textstyle\sum\nolimits_{j=1}^{\infty}}
\mathrm{Tr\,}\rho M\left(  A_{j}\right)
\]
so sigma-additivity is shown.

Let $\rho_{0}\in\mathcal{\mathcal{L}}^{1}\mathcal{(\mathcal{H})}$ be the
density operator of a faithful state on $\mathcal{\mathcal{L}(\mathcal{H})}$
($\rho_{0}>0$) and for given $M$, define $\nu_{0}:=\nu_{\rho_{0}}$. Then for
every state $\rho\in\mathcal{\mathcal{L}}^{1}\mathcal{(\mathcal{H})}$ the
measure $\nu_{\rho}\left(  A\right)  =\mathrm{Tr\,}\rho M\left(  A\right)  $
is dominated by $\nu_{0}$. Indeed assume $\nu_{0}\left(  A\right)  =0$; then
in view of $\rho_{0}>0$ we have $M\left(  A\right)  =0$, (cp. 54.7 in
\cite{MR1721402}), hence also $\nu_{\rho}\left(  A\right)  =0$. Define the map
$\mathcal{M}$ on all positive elements of $\mathcal{\mathcal{L}}%
^{1}\mathcal{(\mathcal{H})}$ as a linear map with values in $L^{1}\left(
\nu_{0}\right)  $ by $\mathcal{M}\left(  \rho\right)  =d\nu_{\rho}/d\nu_{0}$.
For $\sigma\in\mathcal{\mathcal{L}}^{1}\mathcal{(\mathcal{H})}$, one defines
$\mathcal{M}$ via the decomposition $\sigma=\sigma_{+}-\sigma_{-}$. Thus
$\mathcal{M}$ is linear on $\mathcal{\mathcal{L}}^{1}\mathcal{(\mathcal{H})}$
and it is obviously positive. The map is norm-preserving on positives since
for any state $\rho\in\mathcal{\mathcal{L}}^{1}\mathcal{(\mathcal{H})}$,
$\mathcal{M}\left(  \rho\right)  $ is a probability density on $\Omega$. Hence
$\mathcal{M}$ is a state transition.
\end{proof}

A converse statement can be given as follows.

\begin{lemma}
\label{Lem-povm-2}Let $\mathcal{M}$ be a state transition
\[
\mathcal{M}:\mathcal{\mathcal{L}}^{1}\mathcal{(\mathcal{H})}\rightarrow
L^{1}\left(  \nu\right)
\]
where $\nu$ is a sigma-finite measure on some $\left(  X,\Omega\right)  $.
\newline Then there is a POVM\ $M$ on $\left(  X,\Omega\right)  $ such that%
\[
\mathrm{Tr\,}\rho M\left(  A\right)  =\int_{A}\mathcal{M}\left(  \rho\right)
d\nu\text{, }A\in\Omega.
\]

\end{lemma}

\begin{proof}
Let the channel $\mu:L^{\infty}\left(  \nu\right)  \rightarrow
\mathcal{\mathcal{L}(\mathcal{H})}$ be given by the duality%
\[
\mathrm{Tr\,}\mu\left(  f\right)  \sigma=\int f\mathcal{M}\left(
\sigma\right)  d\nu\text{, }\sigma\in\mathcal{\mathcal{L}}^{1}%
\mathcal{(\mathcal{H})}\text{, }f\in L^{\infty}\left(  \nu\right)  .
\]
Set $M\left(  A\right)  =\mu\left(  \mathbf{1}_{A}\right)  $ where
$\mathbf{1}_{A}$ is the indicator of $A\in$ $\Omega$. Then $M\left(  A\right)
\geq0$ since $\mu$ is positive, and $M\left(  X\right)  =\mu\left(  1\right)
=\mathbf{1}$ since $\mu$ is unital. Then for any state $\rho\in
\mathcal{\mathcal{L}}^{1}\mathcal{(\mathcal{H})}$, $A\in$ $\Omega$
\begin{align*}
\mathrm{Tr\,}\rho M\left(  A\right)   &  =\mathrm{Tr\,}\rho\mu\left(
\mathbf{1}_{A}\right)  =\int\mathcal{M}\left(  \rho\right)  \mathbf{1}_{A}%
d\nu\text{ }\\
&  =\int_{A}\mathcal{M}\left(  \rho\right)  d\nu.
\end{align*}
Here for any state $\rho$, the set function $\nu_{\rho}\left(  \cdot\right)
=\mathrm{Tr\,}\rho M\left(  \cdot\right)  $ is a probability measure
satisfying $\nu_{\rho}\ll\nu$ since $\mathcal{M}\left(  \rho\right)  \geq0$
and $\left\Vert \mathcal{M}\left(  \rho\right)  \right\Vert _{\nu
,1}=\left\Vert \rho\right\Vert _{1}=1$. It follows that for any $\sigma
\in\mathcal{\mathcal{L}}^{1}\mathcal{(\mathcal{H})}$, the set function
$\nu_{\sigma}\left(  \cdot\right)  =\mathrm{Tr\,}\sigma M\left(  \cdot\right)
$ is sigma-additive: if $\left\{  A_{j}\right\}  _{j=1}^{\infty}$ are pairwise
disjoint sets from $\Omega$ then
\begin{equation}
\mathrm{Tr\,}\sigma M\left(
{\textstyle\bigcup\nolimits_{j=1}^{\infty}}
A_{j}\right)  =\sum_{j=1}^{\infty}\mathrm{Tr\,}\sigma M\left(  A_{j}\right)  .
\label{sigma-additive}%
\end{equation}
It follows that we have (\ref{additive-SOT-convergent}) with weak*
convergence. Set $B_{1,k}=%
{\textstyle\bigcup\nolimits_{j=1}^{k}}
A_{j}$, $B_{2,k}=%
{\textstyle\bigcup\nolimits_{j=k+1}^{\infty}}
A_{j}$; then we have
\'{}%
by additivity for any $\rho\in\mathcal{\mathcal{L}}^{1}\mathcal{(\mathcal{H}%
)}$
\begin{align*}
\mathrm{Tr\,}\rho M\left(
{\textstyle\bigcup\nolimits_{j=1}^{\infty}}
A_{j}\right)   &  =\mathrm{Tr\,}\rho M\left(  B_{1,k}\cup B_{2,k}\right) \\
&  =\mathrm{Tr\,}\rho\left(  \sum_{j=1}^{k}M\left(  A_{j}\right)  +M\left(
B_{2,k}\right)  \right)  .
\end{align*}
If for $Q_{i}\in\mathcal{\mathcal{L}(\mathcal{H})}$, $Q_{i}\geq0$ we have
$\mathrm{Tr\,}\rho Q_{1}=\mathrm{Tr\,}\rho Q_{2}$ for all $\rho\in
\mathcal{\mathcal{L}}^{1}\mathcal{(\mathcal{H})}$ then $Q_{1}=Q_{2}$. Hence we
obtain
\[
M\left(
{\textstyle\bigcup\nolimits_{j=1}^{\infty}}
A_{j}\right)  =\sum_{j=1}^{k}M\left(  A_{j}\right)  +M\left(  B_{2,k}\right)
\]
We now show that $M\left(  B_{2,k}\right)  \rightarrow0$ in SOT (strong
operator topology). Set $\rho=\left\vert u\right\rangle \left\langle
u\right\vert $ for a unit vector $u\in\mathcal{H}$; then by the
sigma-additivity (\ref{sigma-additive}) we have
\[
\left\langle u\right\vert \left.  M\left(  B_{2,k}\right)  u\right\rangle
=\mathrm{Tr\,}\left\vert u\right\rangle \left\langle u\right\vert M\left(
B_{2,k}\right)  =\sum_{j=k}^{\infty}\mathrm{Tr\,}\rho M\left(  A_{j}\right)
\rightarrow0.
\]
But $0\leq M\left(  B_{2,k}\right)  \leq\mathbf{1}$, hence $\left(  M\left(
B_{2,k}\right)  \right)  ^{2}\leq M\left(  B_{2,k}\right)  $ and we obtain
\[
\left\langle h,\left(  M\left(  B_{2,k}\right)  \right)  ^{2}h\right\rangle
=\left\Vert M\left(  B_{2,k}\right)  h\right\Vert ^{2}\rightarrow0\text{ for
}h\in\mathcal{H}\text{.}%
\]
Hence $M\left(  B_{2,k}\right)  \rightarrow0$ in SOT and
(\ref{additive-SOT-convergent}) is shown.

\textbf{ }
\end{proof}

\begin{lemma}
Let $\mathcal{M}$ be a state transition
\[
\mathcal{M}:\mathcal{\mathcal{L}}^{1}\mathcal{(\mathcal{H})}\rightarrow
L^{1}\left(  \nu\right)
\]
where $\nu$ is a sigma-finite measure on some $\left(  X,\Omega\right)  $ and
let $M$ be the POVM\ on $\left(  X,\Omega\right)  $ constructed according to
Lemma \ref{Lem-povm-2}. For the given $M$ and a state $\rho_{0}>0$, let
$\nu_{0}$ be the probability measure and $\mathcal{M}_{0}$ be the transition%
\[
\mathcal{M}_{0}:\mathcal{\mathcal{L}}^{1}\mathcal{(\mathcal{H})}\rightarrow
L^{1}\left(  \nu_{0}\right)
\]
resp., constructed in Lemma \ref{lem-POVM-1}. Then there is an transition
$\mathcal{M}_{1}:L^{1}\left(  \nu_{0}\right)  \rightarrow L^{1}\left(
\nu\right)  $ such that $\mathcal{M=M}_{1}\mathcal{M}_{0}$.
\end{lemma}

\begin{proof}
Since $\nu_{0}=\nu_{\rho_{0}}$ for a faithful state $\rho_{0}\in
\mathcal{L}^{1}\left(  \mathcal{H}\right)  $, we have $\nu_{0}\ll\nu$. Set
\[
g_{0}:=\frac{d\nu_{0}}{d\nu}=\mathcal{M}\left(  \rho_{0}\right)  \in
L^{1}\left(  \nu\right)
\]
and for $h\in L^{1}\left(  \nu_{0}\right)  $, set $\mathcal{M}_{1}\left(
h\right)  =hg_{0}$. The mapping $\mathcal{M}_{1}$ is linear, positive and norm
preserving on positives: if $h\geq0$ then%
\[
\left\Vert \mathcal{M}_{1}\left(  h\right)  \right\Vert _{\nu,1}=\int
hg_{0}d\nu=\int hd\nu_{0}=\left\Vert h\right\Vert _{\nu_{0},1}.
\]
Furthermore for $\rho\in\mathcal{L}^{1}\left(  \mathcal{H}\right)  $,
$\rho\geq0$ we have
\begin{align*}
\nu_{\rho}\left(  A\right)   &  =\int_{A}\mathcal{M}\left(  \rho\right)
d\nu=\int_{A}\frac{d\nu_{\rho}}{d\nu}d\nu=\int_{A}\frac{d\nu_{\rho}}{d\nu_{0}%
}\frac{d\nu_{0}}{d\nu}d\nu\\
&  =\int_{A}\mathcal{M}_{0}\left(  \rho\right)  g_{0}d\nu=\int_{A}%
\mathcal{M}_{1}\left(  \mathcal{M}_{0}\left(  \rho\right)  \right)
d\nu\text{, }A\in\Omega
\end{align*}
which shows $\mathcal{M=M}_{1}\mathcal{M}_{0}$.
\end{proof}

\bigskip

\textbf{POVM and Spectral Measures}

In order for a POVM to give a spectral measure, in the Conway operator book,
9.1, the additional property
\begin{equation}
M\left(  A_{1}\cap A_{2}\right)  =M\left(  A_{1}\right)  M\left(
A_{2}\right)  \label{property-iv}%
\end{equation}
is required in addition to (i)-(iii) and $M\left(  A\right)  $ being
projections; it is also required in Lax, 31.3, Theorem 9. Parthasarathy on p.
23 pretends that the property follows from (i)-(iii) and $M\left(  A\right)  $
being projections. We believe Partha is right; we argue as follows.

A POVM\ (positive operator valued measure) on a measurable space $\left(
X,\Omega\right)  $ is a mapping $M:\Omega\rightarrow\mathcal{\mathcal{L}%
(\mathcal{H})}$ with properties (i) $M\left(  A\right)  \geq0$ (hence
$M\left(  A\right)  $ is self-adjoint), (ii) $M\left(  X\right)  =\mathbf{1}$,
(iii) if $\left\{  A_{j}\right\}  _{j=1}^{\infty}$ are pairwise disjoint sets
from $\Omega$ then
\begin{equation}
M\left(
{\textstyle\bigcup\nolimits_{j=1}^{\infty}}
A_{j}\right)  =%
{\textstyle\sum\nolimits_{j=1}^{\infty}}
M\left(  A_{j}\right)
\end{equation}
where the r.h.s. is an SOT convergent sum. Now suppose all $M\left(  A\right)
$ are projections. Consider $A\in\Omega$ and its complement $A^{c}$; then
\[
M\left(  A\right)  +M\left(  A^{c}\right)  =M\left(  X\right)  =\mathbf{1}%
\text{. }%
\]
Multiplying both sides by $M\left(  A\right)  $ gives
\[
M\left(  A\right)  +M\left(  A^{c}\right)  M\left(  A\right)  =M\left(
A\right)
\]
implying $M\left(  A^{c}\right)  M\left(  A\right)  =0$. Now let $B\in\Omega$
such that $A\cap B=\emptyset$; then $B\subset A^{c}$ and hence $A^{c}%
=B+\left(  A^{c}\diagdown B\right)  $. Then%
\begin{align*}
M\left(  A^{c}\right)   &  =M\left(  B\right)  +M\left(  A^{c}\diagdown
B\right)  ,\\
0  &  =M\left(  A^{c}\right)  M\left(  A\right)  =\left(  M\left(  B\right)
+M\left(  A^{c}\diagdown B\right)  \right)  M\left(  A\right) \\
&  =M\left(  B\right)  M\left(  A\right)  +M\left(  A^{c}\diagdown B\right)
M\left(  A\right)  .
\end{align*}
Multiply both sides above from the left with $M\left(  A\right)  $; we obtain
\[
0=M\left(  A\right)  M\left(  B\right)  M\left(  A\right)  +M\left(  A\right)
M\left(  A^{c}\diagdown B\right)  M\left(  A\right)  .
\]
Since all $M\left(  \cdot\right)  \geq0$, both operators on the r.h.s. above
are positive; it follows that $M\left(  A\right)  M\left(  B\right)  M\left(
A\right)  =0$ whenever $A\cap B=\emptyset$. But
\begin{align*}
M\left(  A\right)  M\left(  B\right)  M\left(  A\right)   &  =M\left(
A\right)  M\left(  B\right)  M\left(  B\right)  M\left(  A\right) \\
&  =V^{\ast}V\text{ for }V=M\left(  B\right)  M\left(  A\right)
\end{align*}
hence
\[
0=V=M\left(  B\right)  M\left(  A\right)
\]
whenever $A\cap B=\emptyset.$ Now for arbitrary $A_{1},A_{2}$ note that
\[
M\left(  A_{1}\right)  =M\left(  A_{1}\diagdown A_{2}\right)  +M\left(
A_{1}\cap A_{2}\right)  ,
\]
and (\ref{property-iv}) is shown by
\begin{align*}
M\left(  A_{1}\right)  M\left(  A_{2}\right)   &  =M\left(  A_{1}\diagdown
A_{2}\right)  M\left(  A_{2}\right)  +M\left(  A_{1}\cap A_{2}\right)
M\left(  A_{2}\right) \\
&  =M\left(  A_{1}\cap A_{2}\right)  M\left(  A_{2}\right)  \text{ since
}\left(  A_{1}\diagdown A_{2}\right)  \cap A_{2}=\emptyset\\
&  =M\left(  A_{1}\cap A_{2}\right)  \left(  M\left(  A_{2}\diagdown
A_{1}\right)  +M\left(  A_{1}\cap A_{2}\right)  \right) \\
&  =M\left(  A_{1}\cap A_{2}\right)  M\left(  A_{1}\cap A_{2}\right) \\
&  =M\left(  A_{1}\cap A_{2}\right)  .
\end{align*}

\subsection{Remarks on notation}

For the trace of an operator $A$, we adopted the notation $\mathrm{Tr\,}A$,
following \cite{MR1230389}, with a bracket only if necessary, as in
$\mathrm{Tr\,}\left(  A-B\right)  $. Our paper \cite{BGN-QAE} uses
$\mathrm{Tr\,}\left(  A\right)  $ even if no bracket is necessary, cf also
\cite{MR2506764}, \cite{MR887100}. The lower case notation $\mathrm{tr\,}A$ is
found in \cite{MR3012668}; with brackets like $\mathrm{tr\,}\left(  A\right)
$ it is in \cite{MR1721402} (mostly) and in \cite{MR1070713}, \cite{MR3468018}.

\bigskip

\cite{MR1230389}Ohya, Petz, Entropy

\cite{MR1222649} Meyer, Q-Proba

\cite{MR3468018}Chang, Q-Stochastics

\cite{CIT-006} Gray, Toeplitz

\cite{MR2510896} Mosonyi

\cite{MR618863} D. W. M\"{u}ller

\cite{MR2346393}Guta, Jencova, LAN in q-stat

\cite{MR4224167}H\"{o}rmann, Dirac and normal states on Weyl--von Neumann
algebras, 2021

\cite{MR2256497}Jencova, Petz, Sufficiency, examples

\cite{MR2207329}Jencova, Petz, Sufficiency, basic

\cite{MR1057180}Petz, CCR

\cite{Matsumoto-randomization} Matsumoto-randomization

\cite{MR1070713} Conway, Func Ana

\cite{MR1721402}Conway, Operator Theory

\cite{MR887100} Bratteli 1

\cite{MR1441540}Bratteli 2

\cite{MR1093459}Brockwell, Davis, Time Series

\cite{MR3060648}Derezinski, G\'{e}rard

\cite{MR0982264}Dudley, Real Analysis

\cite{MR3012668}Parthasarathy

\cite{MR1490835}Sakai

\cite{MR1741419}Schaefer, Topological Vector Spaces

\cite{MR856411}Le Cam book

\cite{MR3930599}Shiryaev, Proba 2, 3rd Ed

\cite{MR812467} Strasser

\cite{MR4137283}McLaren, Plosker, Ramsey

\cite{MR1892228}Lax, Func Ana

\cite{MR2797301}Holevo, Quaderni, Pisa

\cite{MR4319036}Nikolski, Toeplitz

\cite{MR0374877}Nikolski, Approximation of Functions

\cite{MR2724359}Tsybakov, NP estimation

\cite{RevModPhys.84.621}Weedbrock, Pirandola et al, Gaussian quantum information

\cite{BGN-QAE}BGN\ -QAE

\cite{MR1633574} Grama, N, Expon. family regression

\cite{MR1900972}Grama, N. Asy. equiv. for NP regression

\cite{MR1425959}N 96, Ay. equiv. density estimation and white noise

\cite{MR2589320} Gol, N. Zhou, Spec 2010

\cite{GolNussbZ-specpaper-preprint} Gol, N. Zhou, Spec preprint arXiv 2009

\cite{MR1425958}Brown, Low, Equivalence NP regression

\cite{MR2767163}Korostelev, Math Stat

\cite{MR120319}Dieudonn\'{e}, Foundations of Modern Analysis (vol I)

\cite{MR2506764}Kahn, Guta, qu-LAN\ for fidi systems

\cite{MR3098996}Cohn, Measure Theory

\cite{MR2377635}Audenaert, N, S, Verstraete, Asy Error Rates in Q hypotheses testing

\cite{MR2502660}N\&S, Chernoff lower bound, AS, 2010

\cite{MR840521}Mammen, Additional information

\bigskip

\bigskip

\bigskip%

\end{privatenotes}%

\bigskip%

\begin{privatenotes}%

Text%

\end{privatenotes}%

\bigskip%

\begin{privatenotes}
\begin{boxedminipage}{\textwidth}%

\begin{sfblock}
Text
\end{sfblock}

%

\end{boxedminipage}
\end{privatenotes}%

\end{document}